\documentclass[12pt,a4paper,french,english,PSfrag,BCOR8mm, most, usenames, dvipsnames]{article}
\usepackage[T1]{fontenc}
\usepackage[latin9]{inputenc}
\usepackage{xcolor}
\usepackage{babel}
\makeatletter
\addto\extrasfrench{%
   \providecommand{\fg}{\ifdim\lastskip>\z@\unskip\fi~\frqq}%
}

\makeatother
\usepackage{calc}
\usepackage{tcolorbox}
\usepackage{amsmath}
\usepackage{amsthm}
\usepackage{amssymb}
\usepackage{graphicx}
\usepackage{wasysym}
\usepackage[all]{xy}
\usepackage[unicode=true,pdfusetitle,
 bookmarks=true,bookmarksnumbered=false,bookmarksopen=false,
 breaklinks=false,pdfborder={0 0 0},pdfborderstyle={},backref=false,colorlinks=true]
 {hyperref}

\makeatletter

\pdfpageheight\paperheight
\pdfpagewidth\paperwidth

\newcommand{\noun}[1]{\textsc{#1}}

\numberwithin{equation}{section}
\numberwithin{figure}{section}
\theoremstyle{plain}
\newtheorem{thm}{\protect\theoremname}[section]
\theoremstyle{remark}
\newtheorem{rem}[thm]{\protect\remarkname}
\theoremstyle{definition}
\newtheorem{problem}[thm]{\protect\problemname}
\theoremstyle{plain}
\newtheorem{cor}[thm]{\protect\corollaryname}
\theoremstyle{remark}
\newtheorem*{acknowledgement*}{\protect\acknowledgementname}
\theoremstyle{definition}
\newtheorem{example}[thm]{\protect\examplename}
\theoremstyle{plain}
\newtheorem{lem}[thm]{\protect\lemmaname}
\theoremstyle{definition}
\newtheorem{defn}[thm]{\protect\definitionname}
\theoremstyle{remark}
\newtheorem{notation}[thm]{\protect\notationname}
\theoremstyle{plain}
\newtheorem{prop}[thm]{\protect\propositionname}

\usepackage[a4paper, total={17cm, 27cm}]{geometry}
\usepackage{color}
\usepackage{graphics}
\usepackage{stackrel} 

\usepackage{stmaryrd}
\tcbuselibrary{breakable}

\makeatother

\addto\captionsenglish{\renewcommand{\acknowledgementname}{Acknowledgement}}
\addto\captionsenglish{\renewcommand{\corollaryname}{Corollary}}
\addto\captionsenglish{\renewcommand{\definitionname}{Definition}}
\addto\captionsenglish{\renewcommand{\examplename}{Example}}
\addto\captionsenglish{\renewcommand{\lemmaname}{Lemma}}
\addto\captionsenglish{\renewcommand{\notationname}{Notation}}
\addto\captionsenglish{\renewcommand{\problemname}{Problem}}
\addto\captionsenglish{\renewcommand{\propositionname}{Proposition}}
\addto\captionsenglish{\renewcommand{\remarkname}{Remark}}
\addto\captionsenglish{\renewcommand{\theoremname}{Theorem}}
\addto\captionsfrench{\renewcommand{\acknowledgementname}{Remerciements}}
\addto\captionsfrench{\renewcommand{\corollaryname}{Corollaire}}
\addto\captionsfrench{\renewcommand{\definitionname}{Définition}}
\addto\captionsfrench{\renewcommand{\examplename}{Exemple}}
\addto\captionsfrench{\renewcommand{\lemmaname}{Lemme}}
\addto\captionsfrench{\renewcommand{\notationname}{Notation}}
\addto\captionsfrench{\renewcommand{\problemname}{Problème}}
\addto\captionsfrench{\renewcommand{\propositionname}{Proposition}}
\addto\captionsfrench{\renewcommand{\remarkname}{Remarque}}
\addto\captionsfrench{\renewcommand{\theoremname}{Théorème}}
\providecommand{\acknowledgementname}{Acknowledgement}
\providecommand{\corollaryname}{Corollary}
\providecommand{\definitionname}{Definition}
\providecommand{\examplename}{Example}
\providecommand{\lemmaname}{Lemma}
\providecommand{\notationname}{Notation}
\providecommand{\problemname}{Problem}
\providecommand{\propositionname}{Proposition}
\providecommand{\remarkname}{Remark}
\providecommand{\theoremname}{Theorem}

\begin{document}
\selectlanguage{english}
\title{Micro-local analysis of contact Anosov flows and band structure of
the Ruelle spectrum}
\author{\href{https://www-fourier.ujf-grenoble.fr/~faure/}{Frédéric Faure}\\
Institut Fourier, UMR 5582, Laboratoire de Mathématiques\\
Université Grenoble Alpes, CS 40700, 38058 Grenoble cedex 9, France\\
\href{mailto:frederic.faure@univ-grenoble-alpes.fr}{frederic.faure@univ-grenoble-alpes.fr}\\
\and\\
\href{https://tsujiimasato.wordpress.com/}{Masato Tsujii}\\
Department of Mathematics, Kyushu University,\\
Moto-oka 744, Nishi-ku, Fukuoka, 819-0395, Japan \\
\href{mailto:tsujii@math.kyushu-u.ac.jp}{tsujii@math.kyushu-u.ac.jp}}

\maketitle
\selectlanguage{english}

\selectlanguage{french}%

\newtcolorbox{cBoxA}[1][]{enhanced, frame style={purple!80}, interior style={red!0}, #1}

\newtcolorbox{cBoxB}[2][]{enhanced, frame style={teal!80}, interior style={cyan!0}, #2}

\selectlanguage{english}
\begin{abstract}
We develop a geometric micro-local analysis of contact Anosov flow,
such as the geodesic flow on negatively curved manifold. This microlocal
analysis is based on the wavepacket transform discussed in \cite{faure_tsujii_Ruelle_resonances_density_2016}.
The main result is that the transfer operator is well approximated
(in the high frequency limit) by the quantization of the Hamiltonian
flow naturally defined from the contact Anosov flow and extended to
some vector bundle over the symplectization set. This has a few important
consequences: the discrete eigenvalues of the generator of transfer
operators, called Ruelle spectrum, are structured into vertical bands.
If the right-most band is isolated from the others, most of the Ruelle
spectrum in it concentrate along a line parallel to the imaginary
axis and, further, the density satisfies a Weyl law as the imaginary
part tend to infinity. Some of these results were announced in \cite{faure_tsujii_band_CRAS_2013}.
\end{abstract}
\footnote{2010 Mathematics Subject Classification:

37D20 Uniformly hyperbolic systems (expanding, Anosov, Axiom A, etc.)

37D35 Thermodynamic formalism, variational principles, equilibrium
states

37C30 Zeta functions, (Ruelle-Frobenius) transfer operators, and other
functional analytic techniques in dynamical systems

81Q20 Semiclassical techniques, including WKB and Maslov methods

81Q50 Quantum chaos

Keywords: Transfer operator; Ruelle resonances; decay of correlations;
Semi-classical analysis.}

\global\long\def\eq#1{\underset{(#1)}{=}}%
\global\long\def\ineq#1{\underset{(#1)}{\leq}}%
\global\long\def\ineqs#1{\underset{(#1)}{\geq}}%

\begin{rem}
On this pdf file, you can click on the colored words, they contain
an hyper-link to \href{https://en.wikipedia.org/wiki/Main_Page}{wikipedia}
or other multimedia contents. Appendix \ref{sec:General-notations-used}
contains some convention of notations used in this paper.
\end{rem}

\newpage{}

\tableofcontents{}

\newpage{}

\section{Introduction}

\subsection{Broad mathematical context}

In this paper we consider a \textbf{\href{https://en.wikipedia.org/wiki/Chaos_theory}{deterministic chaotic dynamical system}}
defined by a contact Anosov vector field $X$ on a closed manifold
$M$ that generates a flow $\phi^{t}:M\rightarrow M$. A typical example
is the \textbf{geodesic flow for a negatively curved Riemannian manifold}.
These flows possess the Anosov property also known as uniform hyperbolicity:
each trajectory is very unstable with \href{https://en.wikipedia.org/wiki/Chaos_theory\#Sensitivity_to_initial_conditions}{sensitivity to initial conditions},
and most of them look complex and random (i.e. chaotic) even though
this system is deterministic. Long time prediction of the dynamics
therefore seems difficult. For such chaotic dynamical systems a fruitful
approach initiated in the 1970s has been to consider not individual
orbits but a smooth collection of them, i.e. the evolution of a cloud
of points under the flow. It appears that this cloud of points equidistributes
over long time, i.e., converges weakly to some measure called equilibrium
and thus providing a predictive tool for these systems. We can moreover
refine this description by considering corrections to this evolution
towards the equilibrium, giving an asymptotic series that describes
the long time behavior of the cloud of points. Technically the clouds
of points will be modeled by a smooth function $u\in C^{\infty}\left(M;\mathbb{C}\right)$
(or more generally by a smooth section of some arbitrary vector bundle)
and we are interested in describing the evolution $u\circ\phi^{t}$
for long times $t\gg1$ in the weak sense, i.e., ``correlation functions''
$t\in\mathbb{R}\mapsto\langle v|u\circ\phi^{t}\rangle\in\mathbb{C}$
for any smooth test function $v$.

Important progress has been made by Pollicott, Ruelle and others who
were able to describe these correlation functions using an asymptotic
expansion over some discrete set of ``Pollicott-Ruelle resonances'',
which are the poles of the resolvent $z\in\mathbb{C}\rightarrow\left(z-X\right)^{-1}$
and more recently, that appear to be the discrete eigenvalues of the
vector field $X:\mathcal{H}\rightarrow\mathcal{H}$ in a suitable
Hilbert space $\mathcal{H}$ (an anisotropic Sobolev space) \cite{pollicott1986meromorphic,ruelle_86,Ruelle1991},
\cite{baladi_livre_00}.

More recently researchers have been able to provide more precise information
on the location of these eigenvalues and associated eigenvectors for
some specific models. This paper presents the special case where the
vector field $X$ is an Anosov contact vector field. We show that
the Pollicott-Ruelle eigenvalues are organized into vertical bands
in the spectral plane $\mathbb{C}$ and provide consequences for the
description of correlation functions $t\mapsto\langle v|u\circ\phi^{t}\rangle$
defined above. The main result presented here is that these corrections
to the equilibrium are small fluctuations that behave as if governed
by a quantum evolution equation, cf. Theorem \ref{thm:1_emergence_of_QM}.
We can say roughly that a\textbf{ quantum dynamics emerges from the
correlation functions of a deterministic chaotic dynamical system}.
Technically we use \textbf{microlocal analysis}, which appears to
be a very fruitful approach in this domain \cite{fred-roy-sjostrand-07,fred_flow_09}\cite{dyatlov_zworski_zeta_2013}.

\subsubsection*{\label{subsec:Related-papers}Related papers}

We present now previous works specifically related to the topics of
band structure of the Ruelle spectrum for contact Anosov dynamics
and the emergence of an effective quantum dynamics as expressed in
Eq.(\ref{eq:emerging-1}), using micro-local analysis. This paper
is a continuation along the following series of papers.
\begin{itemize}
\item In \cite{fred-PreQ-06}, the simplest toy model has been considered.
This dynamical system is a contact $U\left(1\right)$-extension of
the Arnold's \href{https://en.wikipedia.org/wiki/Arnold\%27s_cat_map}{cat map}
$M=\left(\begin{array}{cc}
2 & 1\\
1 & 1
\end{array}\right)$ on $\mathbb{T}^{2}$ (or any hyperbolic $M\in SL_{2}\left(\mathbb{Z}\right)$).
Band structure of the Ruelle spectrum and emergence of the quantum
cat map dynamics has been shown. This model is also called ``prequantum
cat map'' since its construction follows the \href{https://en.wikipedia.org/wiki/Geometric_quantization\#Prequantization}{prequantization}
procedure of Kostant, Souriau, Kirillov that is equivalent to a contact
$U\left(1\right)$-extension of the Hamiltonian dynamics.
\begin{rem}
\label{rem:The-symplectic-form}The symplectic form $\Omega$ on $\mathbb{T}^{2}$
is not \href{https://en.wikipedia.org/wiki/Closed_and_exact_differential_forms}{exact}
so we need to consider a non trivial line bundle $L\rightarrow\mathbb{T}^{2}$
with curvature $\Omega$, called prequantum line bundle. On the opposite,
in the model of contact Anosov flows considered in this paper the
symplectic form $\Omega$ on $\Sigma\subset T^{*}M$ is exact since
$\Omega=d\theta$ with $\theta$ being the Liouville one-form (\ref{eq:canonical_Liouville}).
Hence the prequantum line bundle $L\rightarrow\Sigma$ is trivial
and we have decided to ignore it, although its presence manifests
all along the paper, for example  in the definition of the twist operator
(\ref{eq:def_twisted_pull_back}). For a more correct geometrical
description, we have to twist by $L$ in the tensor product $\mathcal{F}\otimes L\rightarrow\Sigma$
in (\ref{eq:def_F-1}).
\end{rem}

\item In \cite{faure-tsujii_prequantum_maps_12}, the previous model is
generalized. We consider a $U\left(1\right)$-extension of an arbitrary
symplectic Anosov map $\phi:M\rightarrow M$ (also called prequantum
Anosov map as explained before). Band structure of the Ruelle spectrum
and emergence of the quantum dynamics has been shown. The models in
\cite{fred-PreQ-06,faure-tsujii_prequantum_maps_12} can be considered
as toy models (or simplified models) for contact Anosov flows considered
in this paper because they are analogous (but not equivalent) to a
time one contact Anosov flow map. These toy models are indeed at the
core of this paper, we will come back to them in the appendix \ref{sec:Bargmann-transform-and}
and \ref{sec:Linear-expanding-maps}.
\item The paper \cite{faure_tsujii_band_CRAS_2013} announced the band spectrum
for contact Anosov flows, (less precise) Weyl law and the method presented
in this paper.
\item In the paper \cite{faure-tsujii_anosov_flows_13}, one has considered
the interesting choice of the bundle $F=\left|\mathrm{det}E_{s}\right|^{1/2}$,
giving the trivial bundle (\ref{eq:trivial bundle}). This model is
treated in relation with the ``semi-classical zeta function'' that
generalizes the \href{https://en.wikipedia.org/wiki/Selberg_zeta_function}{Selberg zeta function}
to non constant curvature.  Technically, the case of the Hölder continuous
bundle $F=\left|\mathrm{det}E_{s}\right|^{1/2}$ is more tricky than
the case a smooth bundle $F$ and needs to consider an extension of
the dynamics to a Grassmanian bundle.
\item C. Guillarmou and M. Cekic in \cite{guillarmou2020band} prove the
first band for contact Anosov flows in dimension 3 using horocycle
operators.
\item For the special case of hyperbolic manifolds, endowed with an algebraic
structure, ie. high symmetry described by Lie groups, band spectrum
of Anosov dynamics has been studied in \cite{dyatlov_faure_guillarmou_2014},
\cite{guillarmou_weich_resonances_16}, \cite{hilgert2021higher},
\cite{barkhofen2022semiclassical}.
\item In different situations, band structure and Weyl law for the spectrum
of resonances has been studied in \cite{stefanov1995distribution},\cite{sjostrand1999asymptotic}
for convex obstacles and by S. Dyatlov in \cite{dyatlov_resonance_band_2013}
for regular normally hyperbolic trapped sets.
\end{itemize}

\subsection{Model considered in this paper}

\paragraph{Contact Anosov flows:}

In this paper we consider a smooth \textbf{\href{https://en.wikipedia.org/wiki/Contact_geometry\#Reeb_vector_field}{contact}
\href{https://en.wikipedia.org/wiki/Anosov_diffeomorphism\#Anosov_flow}{Anosov}
vector field} $X$ on a closed \href{https://en.wikipedia.org/wiki/Contact_geometry}{contact manifold}
$\left(M,\mathcal{A}\right)$ with contact one form $\mathcal{A}$.
A particular and important example is with $\left(\mathcal{N},g_{\mathcal{N}}\right)$
a closed Riemannian manifold with negative curvature, the \href{https://en.wikipedia.org/wiki/Geodesic\#Geodesic_flow}{geodesic vector field}
$X$ is a contact Anosov vector field on the \href{https://en.wikipedia.org/wiki/Unit_tangent_bundle}{unit cotangent bundle}
$M=T_{1}^{*}\mathcal{N}$, where the \href{https://en.wikipedia.org/wiki/Contact_geometry}{contact one form}
$\mathcal{A}=pdq$ is the \href{https://en.wikipedia.org/wiki/Tautological_one-form}{Liouville one-form}.

The vector field $X$ is considered as a \href{https://en.wikipedia.org/wiki/Vector_field\#Vector_fields_on_manifolds}{derivation}
$X:C^{\infty}\left(M\right)\rightarrow C^{\infty}\left(M\right)$
and generates the pullback operators by the flow $\phi^{t}:M\rightarrow M$,
i.e. $e^{tX}u=u\circ\phi^{t}$ with $u\in C^{\infty}\left(M\right)$
and $t\in\mathbb{R}$. More generally, if $F$ is a vector bundle
over $M$, we will consider a derivation $X_{F}$ acting on smooth
sections $C^{\infty}\left(M;F\right)$ of $F$ over\footnote{See Definition (\ref{eq:def_XF}).}
$X$, generator of the pull back operator 
\begin{equation}
e^{tX_{F}}:C^{\infty}\left(M;F\right)\rightarrow C^{\infty}\left(M;F\right).\label{eq:pull_back_exp_t_XF}
\end{equation}
As an example, we may consider $F=\Lambda^{\bullet}\left(M\right)$,
the bundle of \href{https://en.wikipedia.org/wiki/Differential_form}{differential forms},
and $X_{F}$ being the \href{https://en.wikipedia.org/wiki/Lie_derivative}{Lie derivative}. 

\paragraph{Question of long time behavior:}

Anosov flows are a typical model of \href{https://en.wikipedia.org/wiki/Chaos_theory}{chaotic dynamics}
with sensitivity to initial conditions. A typical question in dynamical
systems theory that is addressed in this paper is to describe the
long time action of the flow generated by $X$ on smooth sections,
i.e. for any given smooth sections $u,v\in C^{\infty}\left(M;F\right)$,
describe the correlation $\langle v|e^{tX_{F}}u\rangle_{L^{2}}$ for
$t\rightarrow\infty$? Here $L^{2}=L_{dm}^{2}\left(M;F\right)$ with
the invariant contact volume $dm$ (\ref{eq:volume_form_M}) on $M$
and an arbitrary hermitian metric on $F$. In other words we are looking
for a ``good description'' of the pull back operator $e^{tX_{F}}$
for $t\rightarrow\infty$.

In the case of a trivial bundle $F=\mathbb{C}$, it is known from
Sinaï \cite{sinai1961geodesic} that contact Anosov flows are \textbf{\href{https://en.wikipedia.org/wiki/Mixing_(mathematics)}{mixing},
}i.e. for any smooth functions $u,v\in C^{\infty}\left(M;\mathbb{C}\right)$,
we have $\langle v|e^{tX}u\rangle_{L^{2}}\underset{t\rightarrow+\infty}{\rightarrow}\langle v|\boldsymbol{1}\rangle_{L^{2}}\langle\frac{\boldsymbol{1}}{\mathrm{Vol}\left(M\right)}|u\rangle_{L^{2}}$.
This mixing property dominates the long time behavior and says that
the operator $e^{tX}$ converges (in the weak sense) to the rank one
operator $\Pi_{0}=\boldsymbol{1}\langle\frac{\boldsymbol{1}}{\mathrm{Vol}\left(M\right)}|.\rangle_{L^{2}}$. 
\begin{problem}
\label{Aim}(Aim) In this paper we will be interested in finding a
description of the operator $e^{tX_{F}}$ in the limit of large $t$,
up to an error that decays as $e^{-\Lambda t}$ with an arbitrary
large rate $\Lambda>0$.
\end{problem}

\paragraph{Discrete Ruelle spectrum and anisotropic Sobolev space:}

Although the space $L^{2}\left(M;F\right)$ may be natural to consider
in order to describe the operator $e^{tX_{F}}$, this is not what
we will do. We will instead consider a family of generalized Sobolev
spaces $\mathcal{H}_{W}\left(M;F\right)$, defined from a weight function
$W$ on $T^{*}M$. By choosing a good weight $W$, we can get that
the essential spectrum of the generator $X_{F}$ is in $\left\{ z\in\mathbb{C},\mathrm{Re}\left(z\right)<-\Lambda\right\} $
with arbitrary large $\Lambda>0$ \cite[thm 2.11]{faure_tsujii_Ruelle_resonances_density_2016},
revealing some \textbf{intrinsic discrete spectrum} called \textbf{Ruelle
resonances} in $\left\{ z\in\mathbb{C},\mathrm{Re}\left(z\right)>-\Lambda\right\} $.
See Figure \ref{fig:Band-structure-of}. This fact has been essentially
first obtained by Butterley and Liverani in \cite{liverani_butterley_07}.
The function space $\mathcal{H}_{W}\left(M;F\right)$ is called anisotropic
Sobolev space \cite{liverani_04}\cite{Baladi05}\cite{fred-roy-sjostrand-07,fred_flow_09}
because it is a Hilbert space of distributional sections that contains
smooth sections $C^{\infty}\left(M;F\right)$ and its order depends
on the stable/unstable directions of the dynamics.

This discrete spectrum governs the long time behavior of $e^{tX_{F}}$
but only little\footnote{Ruelle spectrum can be calculated explicitly in only exceptional cases.
Appendix \ref{subsec:Discrete-Ruelle-Pollicott-spectr} presents the
Ruelle spectrum of a very simple toy model on \noun{$\mathbb{R}$}.
See \cite[Prop. 4.1]{faure_tsujii_band_CRAS_2013} for the Ruelle
spectrum of the geodesic flow on a compact hyperbolic surface.} is known about it for a general Anosov flow (or Axiom A flow). In
the special case of a contact Anosov vector field $X$, much more
can be said about the Ruelle spectrum of $X_{F}$ and about the effective
action of the operator $e^{tX_{F}}$. This is the subject of this
paper.

\subsection{\label{subsec:Main-result.-Emergence}Main result: emergence of an
effective quantum dynamics}

\subsubsection{Known results beyond the exponential mixing description}

In order to motivate the analysis and few results presented in this
paper, let us discuss again about the question of long time behavior
of the dynamics $e^{tX}$ in the case of a contact Anosov vector field
$X$ (i.e. trivial bundle $F=\mathbb{C}$ for this discussion) and
known results so far. Recall the result of \textbf{mixing} given above,
that is the weak convergence $e^{tX}\underset{t\rightarrow+\infty}{\rightarrow}\Pi_{0}=\boldsymbol{1}\langle\frac{\boldsymbol{1}}{\mathrm{Vol}\left(M\right)}|.\rangle_{L^{2}}$.
Beyond this, there is \textbf{exponential mixing} obtained by C. Liverani
\cite{liverani_contact_04}, that is $\exists\epsilon>0,\exists C>0,\forall t\geq0$,
$\left\Vert e^{tX}-\Pi_{0}\right\Vert _{\mathcal{H}_{W}\left(M\right)}\leq Ce^{-t\epsilon}$.
Even better, there is an explicit description of the \textbf{exponential
small corrections \cite{tsujii_08} in terms of a finite number of
Ruelle eigenvalues} of $X$ in $\mathcal{H}_{W}\left(M\right)$, $z_{j}=a_{j}+i\omega_{j}\in\mathbb{C}$
with real part $\gamma_{0}^{+}+\epsilon<\ldots a_{j+1}\leq a_{j}\leq\ldots<a_{1}<a_{0}=0$
and associated finite rank spectral projectors $\Pi_{j}$, that can
be expressed as follows:
\begin{equation}
\exists\epsilon>0,C_{\epsilon}>0,\forall t\geq0,\left\Vert e^{tX}-R_{t}\right\Vert _{\mathcal{H}_{W}\left(M\right)}\leq C_{\epsilon}e^{t\left(\gamma_{0}^{+}+\epsilon\right)},\label{eq:expon_terms}
\end{equation}
with a finite rank operator\footnote{In case of a simple eigenvalue $z_{j}$ this gives $e^{tX}\Pi_{j}=e^{z_{j}t}\Pi_{j}.$
Notice also that $z_{0}=0$ is simple, hence $e^{tX}\Pi_{0}=\Pi_{0}$
corresponding to the ``mixing main term'' so called ``equilibrium''.}
\begin{equation}
R_{t}=\sum_{z_{j},\mathrm{Re}\left(z_{j}\right)>\gamma_{0}^{+}+\epsilon}e^{tX}\Pi_{j},\label{eq:R_t}
\end{equation}
with the explicit threshold 
\begin{equation}
\gamma_{0}^{+}:=\lim_{t\rightarrow+\infty}\log\left\Vert e^{tX_{\mathcal{F}_{0}}}\right\Vert _{L^{\infty}}^{1/t}<0,\label{eq:def_gamma_0+}
\end{equation}
where $X_{\mathcal{F}_{0}}$ is the Lie derivative of the vector field
acting on section of the line bundle $\mathcal{F}_{0}\left(E_{s}\right):=\left|\mathrm{det}E_{s}\right|^{-1/2}\rightarrow M$,
$E_{s}\subset TM$ is the vector bundle over $M$ of the stable directions
of the Anosov flow, $\mathrm{det}E_{s}=E_{s}^{\wedge^{d}}$ is the
\href{https://en.wikipedia.org/wiki/Line_bundle\#Determinant_bundles}{determinant bundle}
from $E_{s}$, with $d=\mathrm{dim}E_{s}$, $\left|\mathrm{det}E_{s}\right|^{1/2}$
is the \href{https://en.wikipedia.org/wiki/Density_on_a_manifold}{half densities bundle}
and $\left|\mathrm{det}E_{s}\right|^{-1/2}=\left(\left|\mathrm{det}E_{s}\right|^{1/2}\right)^{*}$
is its dual.
\begin{rem}
~
\begin{itemize}
\item We have $\gamma_{0}^{+}<0$ because $E_{s}$ is the stable bundle
and therefore under the pullback operator $e^{tX}$, the induced action
of the line bundle $\left|\mathrm{det}E_{s}\right|$ is expanding
hence on $\mathcal{F}_{0}\left(E_{s}\right)=\left|\mathrm{det}E_{s}\right|^{-1/2}$
it is contracting. We have the estimate $\gamma_{0}^{+}<-\frac{d}{2}\lambda_{\mathrm{min}}$
with the minimal Lyapounov exponent $\lambda_{\mathrm{min}}>0$ and
$d=\mathrm{dim}E_{s}=\frac{1}{2}\left(\mathrm{dim}M-1\right)$.
\item Except for the obvious eigenvalue $z_{0}=0$ of $X$ (whose eigenspace
is constant functions), the set of other eigenvalues $\left\{ z_{j}\right\} _{j}$
that enter in the sum (\ref{eq:R_t}) is quite unknown in general
and may be empty.
\item Considering an arbitrary vector bundle $F\rightarrow M$ changes the
operator $R_{t}$ and the estimate $\gamma_{0}^{+}$, see below.
\end{itemize}
\end{rem}

\subsubsection{Looking for a better description of the dynamics up to arbitrary
small rate $e^{-\Lambda t}$}

In this paper our aim \ref{Aim} is to give a description of the dynamics
$e^{tX}$ similar to (\ref{eq:expon_terms}) but with an error term
$Ce^{-\Lambda t}$ with arbitrary large rate $\Lambda>0$. So we have
to overcome the threshold $\gamma_{0}^{+}$. For this we do some micro-local
analysis of the pullback operator $e^{tX}$ which means that we describe
precisely its action on functions (or sections) with high frequencies
(i.e. roughly speaking considering in any local chart only Fourier
components with large Fourier variable $\left|\xi\right|\gg1$, $\xi\in\mathbb{R}^{\mathrm{dim}M}$).

With the \textbf{micro-local approach}, it is a general fact that
we can describe the operator $e^{tX}$ in terms of the action of the
flow induced on the cotangent bundle $T^{*}M$ (in other terms $e^{tX}$
is considered as a \href{https://en.wikipedia.org/wiki/Fourier_integral_operator}{Fourier Integral Operator}
F.I.O.). In the special case of contact Anosov flow we show that only
a neighborhood of the symplectic subset $\Sigma=\mathbb{R}^{*}\mathcal{A}\subset T^{*}M$
has to be considered, where $\mathcal{A}$ is the contact one-form,
because (in some sense explained in this paper) ``outside of $\Sigma$,
the norm of $e^{tX}$'' decays faster than $e^{-\Lambda t}$ with
arbitrary large $\Lambda>0$. As a result, we obtain an effective
description of $e^{tX}$ in terms of an Hamiltonian flow on the symplectic
subset $\Sigma$ but considering also its action on the symplectic
normal bundle $N\rightarrow\Sigma$, needed to get the neighborhood
of $\Sigma$. This micro-local description of the operator $e^{tX}$
is in terms of \textbf{\href{https://en.wikipedia.org/wiki/Quantization_(physics)}{quantization}}
of the ``classical Hamiltonian action'' $e^{tX_{\mathcal{F}}}$
giving an operator $\mathrm{Op}_{\Sigma}\left(e^{tX_{\mathcal{F}}}\right)$. 

The \textbf{main drawback} of the micro-local approach is its inability
to describe the ``low frequency subspace of functions'' (that corresponds
effectively to a finite dimensional subspace since $M$ is compact)
and for this reason, our approximation below in Theorem \ref{thm:1_emergence_of_QM}
will contain an unknown finite rank operator $R_{\sigma,t}$, similar
to the operator $R_{t}$ in (\ref{eq:expon_terms}). Optimistically
we can consider that a finite rank operator is negligible compared
to the infinite rank quantum operator $\mathrm{Op}_{\Sigma}\left(e^{tX_{\mathcal{F}}}\right)$
that will enter in the description.

Let us now explain more precisely this micro-local approach, the effective
classical dynamics that appears in a first result presented in Theorem
\ref{thm:1_emergence_of_QM} below.

\subsubsection{Effective classical dynamics}

\paragraph{Bundle $\mathcal{F}_{k}\left(E_{s}\right)\rightarrow M$ and its
dynamics:}

We now go to our first aim \ref{Aim} and consider again a general
vector bundle $F\rightarrow M$. For this, for every $k\in\mathbb{N}$,
we introduce the Hölder continuous finite rank vector bundle over
$M$:
\begin{equation}
\mathcal{F}_{k}\left(E_{s}\right):=\left|\mathrm{det}E_{s}\right|^{-1/2}\otimes\mathrm{Pol}_{k}\left(E_{s}\right)\otimes F\rightarrow M\label{eq:def_Fk}
\end{equation}
where $\mathrm{Pol}_{k}\left(E_{s}\right)$ is the bundle of \href{https://en.wikipedia.org/wiki/Homogeneous_polynomial}{homogeneous polynomials}
of degree $k$ on $E_{s}$. The bundle $E_{s}\rightarrow M$ is Hölder
continuous but smooth along the flow and $E_{s}$ directions. Let
$C^{\beta}\left(M;\mathcal{F}_{k}\left(E_{s}\right)\right)$ be the
space of $\beta-$Hölder continuous sections. Then we consider $X_{\mathcal{F}_{k}}:C^{\beta}\left(M;\mathcal{F}_{k}\left(E_{s}\right)\right)\rightarrow C^{\beta}\left(M;\mathcal{F}_{k}\left(E_{s}\right)\right)$
being the derivation of sections of $\mathcal{F}_{k}\left(E_{s}\right)$
over $X$ ($C^{1}$ along the flow direction). $X_{\mathcal{F}_{k}}$
is naturally induced from the initial derivation $X_{F}$ and is the
generator of the group of operators $e^{tX_{\mathcal{F}_{k}}}:C^{\beta}\left(M;\mathcal{F}_{k}\left(E_{s}\right)\right)\rightarrow C^{\beta}\left(M;\mathcal{F}_{k}\left(E_{s}\right)\right),\,t\in\mathbb{R}$.
We introduce the related quantities that generalize (\ref{eq:def_gamma_0+})
\begin{equation}
\gamma_{k}^{\pm}:=\lim_{t\rightarrow\pm\infty}\log\left\Vert e^{tX_{\mathcal{F}_{k}}}\right\Vert _{L^{\infty}\left(M;\mathcal{F}_{k}\left(E_{s}\right)\right)}^{1/t}.\label{eq:def_gamma_-2}
\end{equation}

\begin{rem}
We can give obvious (but rough) estimates for $\gamma_{k}^{\pm}$
as follows. If $0<\lambda_{\mathrm{min}}\leq\lambda_{\mathrm{max}}$
denote the minimal and maximal Lyapounov exponents of the Anosov flow
$\phi^{t}$, we have
\begin{equation}
-\left(\frac{d}{2}+k\right)\lambda_{\mathrm{max}}+c_{F}\leq\gamma_{k}^{-}\leq\gamma_{k}^{+}\leq-\left(\frac{d}{2}+k\right)\lambda_{\mathrm{min}}+C_{F},\label{eq:estimates_lambda_k}
\end{equation}
where $\mathrm{dim}M=2d+1$ and $c_{F}\leq C_{F}$ depend on $X_{F}$.
For example, if $X_{F}=X$ is the vector field itself with $F=\mathbb{C}$,
we have $c_{F}=C_{F}=0$ and from (\ref{eq:estimates_lambda_k}),
the pinching condition $\frac{\lambda_{\mathrm{max}}}{\lambda_{\mathrm{min}}}<1+\frac{1}{k}$
implies $\gamma_{k+1}^{+}<\gamma_{k}^{-}$ that we call the \textbf{``spectral
gap assumption''} later.
\end{rem}

\paragraph{A notation $\mathcal{S}_{\beta}\left(E_{s}\right)$:}

For every point $m\in M$, $\mathcal{S}\left(E_{s}\left(m\right)\right)$
is the space of Schwartz functions on the vector space $E_{s}\left(m\right)$.
However the bundle $E_{s}\rightarrow M$ is only $\beta-$Hölder,
i.e. the map $m\in M\rightarrow E_{s}\left(m\right)$ is Hölder continuous.
For simplicity we will denote 
\begin{equation}
\mathcal{S}_{\beta}\left(E_{s}\right):=C^{\beta}\left(M;\mathcal{S}\left(E_{s}\right)\right)\label{eq:def_S_beta}
\end{equation}
the space of $\beta-$Hölder continuous sections valued in $\mathcal{S}\left(E_{s}\left(m\right)\right)$
at each point $m\in M$. Similarly $\mathcal{S}'_{\beta}\left(E_{s}\right):=C^{\beta}\left(M;\mathcal{S}'\left(E_{s}\right)\right)$
is the space of sections valued in Schwartz distributions $\mathcal{S}'\left(E_{s}\left(m\right)\right)$,
or more precisely distributional extensions of $\mathcal{S}\left(E_{s}\left(m\right)\right)$,
the dual of Schwartz forms $\mathcal{S}\left(E_{s}\left(m\right);\Lambda^{d}\left(E_{s}\left(m\right)\right)\right)$,
so that polynomials are included in $\mathcal{S}'\left(E_{s}\left(m\right)\right)$.

\paragraph{Taylor projectors $T_{k}$:}

The homogeneous polynomials $\mathrm{Pol}_{k}\left(E_{s}\right)$
are obtained as the image of the ``Taylor projector'' (see section
\ref{subsec:Taylor-operators-} for a definition) $T_{k}:\mathcal{S}_{\beta}\left(E_{s}\right)\rightarrow\mathrm{Pol}_{k}\left(E_{s}\right)\subset\mathcal{S}'_{\beta}\left(E_{s}\right)$
that we extend to a \href{https://en.wikipedia.org/wiki/Bundle_map}{bundle map}
$T_{k}\in L\left(\mathcal{F}\left(E_{s}\right),\mathcal{F}'\left(E_{s}\right)\right)$,
\begin{equation}
T_{k}:\mathcal{F}\left(E_{s}\right):=\left|\mathrm{det}E_{s}\right|^{-1/2}\otimes\mathcal{S}_{\beta}\left(E_{s}\right)\otimes F\rightarrow\mathcal{F}'\left(E_{s}\right):=\left|\mathrm{det}E_{s}\right|^{-1/2}\otimes\mathcal{S}'_{\beta}\left(E_{s}\right)\otimes F\label{eq:def_F_cal}
\end{equation}
For any given $K\in\mathbb{N}$, we denote $T_{\left[0,K\right]}:=\bigoplus_{k=0}^{K}T_{k}$.

\paragraph{Bundle $\mathcal{F}_{k}\left(N_{s}\right)\rightarrow\Sigma$:}

We consider the symplectic sub-manifold
\[
\Sigma:=\mathbb{R}^{*}\mathcal{A}=\left\{ \omega\mathcal{A}\left(m\right),\omega\in\mathbb{R}^{*},m\in M\right\} 
\]
being the \textbf{\href{https://en.wikipedia.org/wiki/Symplectization}{symplectization}}
of the contact manifold $\left(M,\mathcal{A}\right)$. This manifold
$\Sigma$, as a sub-manifold of $T^{*}M$, is also the trapped set
(or \href{https://en.wikipedia.org/wiki/Wandering_set}{non wandering set})
of the induced dynamics $\tilde{\phi}^{t}=\left(d\phi^{t}\right)^{*}$
on $T^{*}M$ and therefore invariant. We consider the symplectic normal
bundle $N=\left(T\Sigma\right)^{\perp}$, its stable component $N_{s}\subset N$
and construct the bundle\footnote{From the precise definition of $N_{s}$ in Lemma \ref{lem:LetEach-component-},
it will appear that this bundle $\mathcal{F}_{k}\left(N_{s}\right)\rightarrow\Sigma$
is isomorphic to the the pull back of $\mathcal{F}_{k}\left(E_{s}\right)\rightarrow M$
under the projection $\pi:\Sigma\rightarrow M$.} $\mathcal{F}_{k}\left(N_{s}\right):=\left|\mathrm{det}N_{s}\right|^{-1/2}\otimes\mathrm{Pol}_{k}\left(N_{s}\right)\otimes F\rightarrow\Sigma$.
The operator $e^{tX_{\mathcal{F}_{k}}}$ defined above extends to
\begin{equation}
e^{tX_{\mathcal{F}_{k}}}:C^{\beta}\left(\Sigma;\mathcal{F}_{k}\left(N_{s}\right)\right)\rightarrow C^{\beta}\left(\Sigma;\mathcal{F}_{k}\left(N_{s}\right)\right)\label{eq:classical_dyn}
\end{equation}
over the Hamiltonian flow $\tilde{\phi}^{t}=\left(d\phi^{t}\right)^{*}:\Sigma\rightarrow\Sigma$.
The family of operators $\left(e^{tX_{\mathcal{F}_{k}}}\right)_{t\in\mathbb{R}}$
in (\ref{eq:classical_dyn}), will be called the \textbf{classical
dynamics}. For any $t\in\mathbb{R}$, we also denote the flow $e^{tX_{\mathcal{F}}}:C^{\beta}\left(\Sigma;\mathcal{F}\left(N_{s}\right)\right)\rightarrow C^{\beta}\left(\Sigma;\mathcal{F}\left(N_{s}\right)\right)$
with similar definition $\mathcal{F}\left(N_{s}\right):=\left|\mathrm{det}N_{s}\right|^{-1/2}\otimes\mathcal{S}_{\beta}\left(N_{s}\right)\otimes F$
and that $e^{tX_{\mathcal{F}}}\tilde{T}_{\left[0,K\right]}=\bigoplus_{k=0}^{K}e^{tX_{\mathcal{F}_{k}}}$
with Taylor projectors $\tilde{T}_{k}:\mathcal{S}_{\beta}\left(N_{s}\right)\rightarrow\mathrm{Pol}_{k}\left(N_{s}\right)$
and $\tilde{T}_{\left[0,K\right]}:=\bigoplus_{k=0}^{K}\tilde{T}_{k}$.

\subsubsection{General remarks about quantization}

Before continuing the presentation, we first make a pause to say few
important words about quantization. These remarks will be important
to understand the content of the next theorem.
\begin{itemize}
\item In the most common situation, on a given manifold $M$, quantization
denoted $\mathrm{Op}\left(.\right)$ maps a function $a\in\mathcal{S}\left(T^{*}M;\mathbb{C}\right)$
called ``symbol'' to a bounded operator $\mathrm{Op}\left(a\right):L^{2}\left(M\right)\rightarrow L^{2}\left(M\right)$
called \href{https://en.wikipedia.org/wiki/Pseudo-differential_operator}{Pseudo Differential Operator}
P.D.O. There are many different definitions for $\mathrm{Op}\left(.\right)$,
and it is important to known that \textbf{they are all equivalent
at leading order in the limit of high frequencies}. At leading order
the map $\mathrm{Op}\left(.\right)$ satisfies some very interesting
and useful universal properties as the ``boundness theorem'', ``composition
theorem'' that is homomorphism of algebra, ``trace formula'', etc
(under some additional hypothesis called ``symbol classes'' \cite[(1.4) p.3]{taylor_tome2})
and our analysis will rely on these properties. For these reasons
the precise definition of $\mathrm{Op}\left(.\right)$ is not very
crucial for the moment.
\item The simplest example of $\mathrm{Op}\left(.\right)$ called ordinary
quantization, is given in local charts $x\in\mathbb{R}^{n}$ on $M$
with dual coordinates $\xi\in\mathbb{R}^{n}$ on $T_{x}^{*}M$, by
\cite[(1.3) p.2]{taylor_tome2}
\[
\left(\mathrm{Op}\left(a\right)u\right)\left(x\right)=\int a\left(x,\xi\right)\left(\mathcal{F}u\right)\left(\xi\right)e^{i\langle\xi|x\rangle}d\xi,
\]
where $\left(\mathcal{F}u\right)\left(\xi\right)$is Fourier transform
of $u\in\mathcal{S}\left(\mathbb{R}^{n}\right)$. \href{https://en.wikipedia.org/wiki/Wigner\%E2\%80\%93Weyl_transform}{Weyl quantization}
is different and has the advantage that a real valued symbol $a$
gives a self-adjoint operator $\mathrm{Op}\left(a\right)$, \cite[(14.1) p.67]{taylor_tome2}.
\item In geometric quantization people introduce an operator $\mathcal{T}:C^{\infty}\left(M\right)\rightarrow\mathcal{S}\left(T^{*}M\right)$,
called Bargman transform or FBI transform, wavelet transform, wave-packet
transform (depending on the context, and additional structures for
its construction) and they define Toeplitz quantization by
\[
\mathrm{Op}\left(a\right):=\mathcal{T}^{\dagger}a\mathcal{T}
\]
where $a$ on the right hand side stands for the multiplication operator
by $a$ and $\mathcal{T}^{\dagger}:\mathcal{S}\left(T^{*}M\right)\rightarrow C^{\infty}\left(M\right)$
is its $L^{2}$ adjoint satisfying $\mathrm{Op}\left(1\right)=\mathcal{T}^{\dagger}\mathcal{T}=\mathrm{Id}$.
We will mainly use this quantization in this paper.
\item It is also possible to define the quantization of a symplectomorphism
$\tilde{\phi}:T^{*}M\rightarrow T^{*}M$ giving an operator $\mathrm{Op}\left(\tilde{\phi}\right):L^{2}\left(M\right)\rightarrow L^{2}\left(M\right)$
called \href{https://en.wikipedia.org/wiki/Fourier_integral_operator}{Fourier Integral Operator F.I.O.}
(e.g. \cite{zelditch_Index_98}). In geometric quantization the definition
is
\begin{equation}
\mathrm{Op}\left(\tilde{\phi}\right):=\mathcal{T}^{\dagger}\Upsilon_{\tilde{\phi}}^{1/2}\tilde{\phi}^{-\circ}\mathcal{T}\label{eq:def_OIF}
\end{equation}
where $\tilde{\phi}^{-\circ}:\mathcal{S}\left(T^{*}M\right)\rightarrow\mathcal{S}\left(T^{*}M\right)$
is the push-forward operator and $\Upsilon_{\tilde{\phi}}^{1/2}\in C^{\infty}\left(T^{*}M;\mathbb{R}^{+}\right)$
is the ``metaplectic correction'' (\ref{eq:def_d}) to get an almost
unitary operator.
\item It is possible to consider a vector bundle $F\rightarrow M$ and define
the quantization of a symbol $a\in\mathcal{S}\left(T^{*}M;\mathrm{End}\left(F\right)\right)$
giving $\mathrm{Op}\left(a\right):L^{2}\left(M;F\right)\rightarrow L^{2}\left(M;F\right)$.
\end{itemize}

\subsubsection{Quantization in our setting}

In this paper in order to study the operator $e^{tX_{F}}:C^{\infty}\left(M;F\right)\rightarrow C^{\infty}\left(M;F\right)$
in (\ref{eq:pull_back_exp_t_XF}), we will need all these previous
and different variations around quantization and moreover another
variation: instead of considering the whole cotangent bundle $T^{*}M$
as the symplectic manifold, we will restrict to a vicinity of the
symplectic sub-manifold $\Sigma:=\mathbb{R}^{*}\mathcal{A}\subset T^{*}M$.
To have a good description of a vicinity of $\Sigma$, we will use
its normal bundle $N=\left(T\Sigma\right)^{\perp}\rightarrow\Sigma$
and the exponential map $\exp_{N}:N\rightarrow T^{*}M$ that is a
local diffeomorphism from a vicinity of the zero section of $N$ to
a vicinity of $\Sigma$ in $T^{*}M$. See Figure \ref{fig:N_Sigma}.

Then we will use a wave-packet transform given by some operators $\mathcal{T}_{N_{s}}:C^{\infty}\left(M;F\right)\rightarrow C^{\beta}\left(\Sigma;\mathcal{F}\left(N_{s}\right)\right)$
and a converse $\mathcal{T}_{N}^{\Delta}:C^{\beta}\left(\Sigma;\mathcal{F}\left(N_{s}\right)\right)\rightarrow C^{\infty}\left(M;F\right)$
defined by 
\[
\mathcal{T}_{N_{s}}\eq{\ref{eq:def_T_Ns}}\mathcal{B}_{N_{s},N_{u}}^{\dagger}\chi_{\Sigma}^{\mu}\widetilde{\left(\exp_{N}\right)^{\circ}}\mathcal{T}
\]
and
\[
\mathcal{T}_{N_{s}}^{\Delta}\eq{\ref{eq:def_T_Ns_Delta}}\mathcal{T}^{\dagger}\widetilde{\left(\exp_{N}^{-1}\right)^{\circ}}\chi_{\Sigma}^{\mu}\mathcal{B}_{N_{s},N_{u}},
\]
where $\mathcal{T}:C^{\infty}\left(M;F\right)\rightarrow\mathcal{S}\left(T^{*}M;F\right)$
is a wave-packet transform, $\widetilde{\left(\exp_{N}\right)^{\circ}}:\mathcal{S}\left(T^{*}M;F\right)\rightarrow\mathcal{S}\left(\Sigma;\mathcal{S}\left(N\right)\right)$
is the pull back of $\exp_{N}$ with some additional phase in (\ref{eq:def_twisted_push_forward-1}),
$\chi_{\Sigma}^{\mu}$ is some cutoff function to a vicinity of the
zero section of $N$ at distance $\omega^{\mu/2}$ in (\ref{eq:def_Chi_mu})
and $B_{N_{s},N_{u}}:C^{\beta}\left(\Sigma;\mathcal{S}_{\beta}\left(N_{s}\right)\right)\rightarrow C^{\beta}\left(\Sigma;\mathcal{S}\left(N\right)\right)$
in (\ref{eq:def_B_Ns_Nu}) is a bundle-wise Bargman transform. With
this, we define the quantization of the classical dynamics (\ref{eq:classical_dyn})
with the additional symbol $\tilde{T}_{\left[0,K\right]}$ by
\begin{equation}
\mathrm{Op}_{\Sigma}\left(e^{tX_{\mathcal{F}}}\tilde{T}_{\left[0,K\right]}\right)\eq{\ref{eq:def_Op_a}}\mathcal{T}_{N_{s}}^{\Delta}\Upsilon_{t}^{1/2}e^{tX_{\mathcal{F}}}\tilde{T}_{\left[0,K\right]}\mathcal{T}_{N_{s}}\quad:C^{\infty}\left(M;F\right)\rightarrow C^{\infty}\left(M;F\right)\label{eq:OP_sigma}
\end{equation}
called \textbf{quantum evolution operators}. In the analysis, the
semi-classical parameter is the frequency $\omega$ along the flow
direction (equivalently the parameter along $\Sigma$) and the \textbf{semi-classical
limit} is $\left|\omega\right|\rightarrow\infty$ (usually denoted
$2\pi\hbar=1/\left|\omega\right|\rightarrow0$). 

\subsubsection{Emergence of an effective quantum dynamics}

To express the result, we also need an operator $\mathrm{Op}\left(\chi_{\Sigma,\sigma}\right):C^{\infty}\left(M;F\right)\rightarrow C^{\infty}\left(M;F\right)$
defined in (\ref{eq:Op_Sigma_sigma}) that restricts functions microlocally
to their vicinity of $\Sigma$ at distance $\sigma>0$. The next theorem
shows how the \href{https://en.wikipedia.org/wiki/Pullback_(differential_geometry)}{pull-back operator}
$e^{tX_{F}}$ in (\ref{eq:pull_back_exp_t_XF}) is well approximated
by the quantum evolution operator $\mathrm{Op}_{\Sigma}\left(e^{tX_{\mathcal{F}}}T_{\left[0,K\right]}\right)$
in (\ref{eq:OP_sigma}).

\begin{cBoxB}{}
\begin{thm}[\textbf{Emergence of quantum dynamics}]
\label{thm:1_emergence_of_QM}For any $K\in\mathbb{N}$ and $\epsilon>0$,
we can choose an anisotropic Sobolev space $\mathcal{H}_{W}\left(M;F\right)$,
such that $\exists C>0$, $\forall t>0,\exists\sigma_{t}>0,\forall\sigma>\sigma_{t}$,
there exists a finite rank operator $R_{\sigma,t}$ and
\begin{equation}
\left\Vert e^{tX_{F}}-\mathrm{Op}\left(\chi_{\Sigma,\sigma}\right)\mathrm{Op}_{\Sigma}\left(e^{tX_{\mathcal{F}}}\tilde{T}_{\left[0,K\right]}\right)\mathrm{Op}\left(\chi_{\Sigma,\sigma}\right)-R_{\sigma,t}\right\Vert _{\mathcal{H}_{W}\left(M;F\right)}\leq Ce^{t\left(\gamma_{K+1}^{+}+\epsilon\right)}.\label{eq:emerging-1}
\end{equation}
\end{thm}

\end{cBoxB}

The proof of Theorem \ref{thm:1_emergence_of_QM} is given in Section
\ref{subsec:Proof-of-Theorem-2}.
\begin{rem}
From (\ref{eq:estimates_lambda_k}) we have $\gamma_{k}^{+}\underset{k\rightarrow\infty}{\rightarrow}-\infty$,
hence the term $e^{t\left(\gamma_{K+1}^{+}+\epsilon\right)}$ in (\ref{eq:emerging-1})
decays very fast for $t\rightarrow+\infty$ when $K$ is large, so
this theorem is a kind of answer to our aim \ref{Aim}. One interpretation
of Theorem \ref{thm:1_emergence_of_QM} is that for large time $t\gg1$,
an \textbf{effective quantum dynamics emerges from the contact Anosov
dynamics}. This quantum dynamics is the \href{https://en.wikipedia.org/wiki/Quantization_(physics)}{quantization}
of the Hamiltonian dynamics (\ref{eq:classical_dyn}) that takes place
on $\Sigma$.
\end{rem}

To show (\ref{eq:emerging-1}) we will use \href{https://en.wikipedia.org/wiki/Microlocal_analysis}{microlocal analysis}
directly on $T^{*}M$ following the approach proposed in \cite{faure_tsujii_Ruelle_resonances_density_2016},
using a metric $g$ on $T^{*}M$ compatible with the symplectic form,
that measures the size of wave-packets (sometimes called \href{https://en.wikipedia.org/wiki/Coherent_state}{coherents states})
in accordance with the \href{https://en.wikipedia.org/wiki/Uncertainty_principle}{uncertainty principle}.

\subsubsection{\label{subsec:Few-comments-about}Few comments about the result (\ref{eq:emerging-1})}
\begin{itemize}
\item As explained above, in the case of a trivial bundle $F=\mathbb{C}$,
contact Anosov flows are \textbf{mixing} $e^{tX}\stackrel[t\rightarrow\infty]{\mathrm{weak}}{\rightarrow}\Pi_{0}=\boldsymbol{1}\langle\frac{\boldsymbol{1}}{\mathrm{Vol}\left(M\right)}|.\rangle_{L^{2}}$.
Although this mixing property dominates the long time behavior, the
limit rank one operator $\boldsymbol{1}\langle\frac{\boldsymbol{1}}{\mathrm{Vol}\left(M\right)}|.\rangle_{L^{2}}$,
belongs to the finite rank operator $R_{\sigma,t}$ in (\ref{eq:emerging-1})
that we do not describe in this paper. We instead consider terms that
are exponentially small compared to it, but that belong to an infinite
dimensional effective space. In our analysis the rank of the operator
$R_{\sigma,t}$ may increase with $t$ because we use micro-local
analysis (i.e. high frequency analysis) where the cutoff $\omega$
we need in frequency increases with $t$. The operator $R_{\sigma,t}$
represents the remainder in low frequencies that we are not able to
describe with this approach. See also footnote \ref{fn:Remark-also-that}
for more comments.
\item In (\ref{eq:def_Fk}), the vector bundle $F\rightarrow M$ is arbitrary
but notice that the \textbf{special choice} of the bundle $F=\left|\mathrm{det}E_{s}\right|^{1/2}$
is particularly interesting because it gives
\begin{equation}
\mathcal{F}_{0}\eq{\ref{eq:def_F-1}}\left|\mathrm{det}E_{s}\right|^{-1/2}\otimes\mathrm{Pol}_{0}\left(E_{s}\right)\otimes\left|\mathrm{det}E_{s}\right|^{1/2}=\mathbb{C}\label{eq:trivial bundle}
\end{equation}
i.e. the trivial bundle over $\Sigma$. The paper \cite{faure-tsujii_anosov_flows_13}
is devoted to that case\footnote{$F=\left|\mathrm{det}E_{s}\right|^{1/2}$ is not a smooth bundle but
only Hölder continuous and technically we consider an smooth extension
to a Grassmanian bundle in \cite{faure-tsujii_anosov_flows_13}.}. Eq.(\ref{eq:def_gamma_-2}) gives $\gamma_{0}^{\pm}=0$ and $\gamma_{1}^{+}<0$.
For that case and taking $K=0$, Eq.(\ref{eq:emerging-1}) shows that
for $t\gg1$, the operator $e^{tX_{F}}$ is well described by the
dominant term $\mathrm{Op}_{\Sigma}\left(e^{tX_{\mathcal{F}_{0}}}T_{0}\right)$
that is the ``quantization'' of the dynamics $X$ itself. From Theorem
\ref{Thm: bands} below, for $\omega=\mathrm{Im}\left(z\right)\rightarrow\infty$,
the eigenvalues of $X_{F}$ accumulate on the imaginary axis $\mathrm{Re}\left(z\right)=\gamma_{0}^{\pm}=0$
with density given by the Weyl law, separated by a uniform spectral
gap $\gamma_{1}^{+}<0$. See Figure \ref{fig:Band-structure-of}(b).
\item In the special case of the \textbf{geodesic flow} on a surface $\mathcal{N}$
of constant and negative curvature, i.e. \emph{$\mathcal{N}=\Gamma\backslash\mathrm{SL}_{2}\left(\mathbb{R}\right)/\mathrm{SO}_{2}\left(\mathbb{R}\right)$,
}giving \emph{$M=\Gamma\backslash\mathrm{SL}_{2}\left(\mathbb{R}\right)$},
with $\Gamma$ being a co-compact subgroup of $\mathrm{SL}_{2}\left(\mathbb{R}\right)$
(that contains $-\mathrm{Id}$), there is no remainder operator $R_{t}$
in (\ref{eq:emerging-1}) and the dominant quantum operator $\mathrm{Op}_{\Sigma}\left(e^{tX_{\mathcal{F}}}\tilde{T}_{0}\right)$
has the same spectrum as the wave operators $\exp\left(\pm it\sqrt{\Delta-\frac{1}{4}}\right)$
on $\mathcal{S}\left(\mathcal{N}\right)$, \cite{faure_tsujii_band_CRAS_2013,dyatlov_faure_guillarmou_2014},
see Figure 4.1(a) in \cite[Fig 4.1(a) p.389]{faure_tsujii_band_CRAS_2013}.
This operator is indeed considered in physics as the \href{https://en.wikipedia.org/wiki/Schr\%C3\%B6dinger_equation}{Schrödinger evolution operator},
giving the quantum description of a \href{https://en.wikipedia.org/wiki/Free_particle}{free particle}.
For the case of geodesic flow on a non-constant negatively curved
Riemannian manifold $\left(\mathcal{N},g\right)$, let us observe
that the symplectic phase space $\Sigma$ is (symplectically) isomorphic
to a double cover of the cotangent bundle $T^{*}\mathcal{N}\backslash\left\{ 0\right\} $
and that the principal symbol of the generator of the leading quantum
operator $\mathrm{Op}_{\Sigma}\left(e^{tX_{\mathcal{F}}}\tilde{T}_{0}\right)$
in (\ref{eq:emerging-1}) is the frequency $\omega$, equal to the
principal symbol of $\sqrt{\Delta}$ on $T^{*}\mathcal{N}$ under
this isomorphism, where $\Delta=d^{\dagger}d$ is the \href{https://en.wikipedia.org/wiki/Laplace\%E2\%80\%93Beltrami_operator}{Laplace Beltrami operator}
on $\mathcal{S}\left(\mathcal{N}\right)$. However we do not expect
that the spectra of both operators coincide in general. We will investigate
this question in a future work. See also \cite{barkhofen2022semiclassical}
for related results and discussions.
\item In semi-classical analysis, the Egorov Theorem \cite[p.26]{taylor_tome2}
or the \href{https://en.wikipedia.org/wiki/WKB_approximation}{WKB approximation}
\cite[p.11]{weinstein_97} show that \textbf{classical Hamiltonian
dynamics emerges in the high frequency limit (and finite time) of
quantum dynamics}. This is how in physics, \href{https://en.wikipedia.org/wiki/Hamiltonian_optics}{geometrical optics}
is derived from \href{https://en.wikipedia.org/wiki/Maxwell\%27s_equations}{electromagnetic waves},
\href{https://en.wikipedia.org/wiki/Hamiltonian_mechanics}{Newtonian (and Hamilton) mechanics}
is derived from quantum waves mechanics of Schrödinger, etc. Inspired
from these physical phenomena this is how quantization has been defined
in mathematics, for example standard quantization \cite[p.2]{taylor_tome2}
that defines a pseudo-differential operator on $\mathcal{S}\left(\mathbb{R}^{n}\right)$
from a Hamiltonian function (symbol) on $\mathbb{R}^{2n}$, or geometric
quantization in a more geometric setting \cite{woodhouse2}. In this
paper we have exhibited the converse and maybe \textbf{unexpected
phenomena: how quantum mechanics emerges from the classical mechanics}
when this later is chaotic. From the mathematical point of view, an
interesting consequence is that it furnishes a \textbf{natural quantization}
of a given classical dynamics among all possible quantizations. This
natural quantization has indeed the preferable properties (that characterize
it) that the semi-classical Van-Vleck formula or semi-classical Trace
formula are asymptotically exact, i.e. they have error terms that
decay exponentially fast with $t\rightarrow\infty$ at large but fixed
$\omega$. This provides a kind of generalization of Selberg theory
to non constant curvature. For more discussions, see \cite[ Section 1.5,1.6,1.7]{faure-tsujii_prequantum_maps_12}
and \cite{faure-tsujii_anosov_flows_13}. From a physical point of
view one may wonder if quantum phenomena observed in experimental
data may emergence from an underlying deterministic but chaotic dynamics,
see e.g. the last paragraph in \cite{nelson2012review}. In section
\ref{subsec:Discussion-about-emergence} we give Theorem \ref{cor:emergence_contact}
that concerns a general contact flow (i.e. without Anosov assumption)
so this may concern a general geodesic flow and show the emergence
of quantum dynamics near the invariant set $\Sigma$.
\end{itemize}

\subsection{Sketch of proof}

The proof of the emerging quantum dynamics in (\ref{eq:emerging-1})
can be summarized by the following mechanisms that will be explained
in details in this paper and illustrated in Figure \ref{fig:N_Sigma}.
\begin{enumerate}
\item In the limit of high frequencies, evolution of functions (or sections)
by the pull-back operator $e^{tX_{F}}$ is well described on the cotangent
bundle $T^{*}M$ with the induced flow $\tilde{\phi}^{t}:=\left(d\phi^{t}\right)^{*}$,
$t\in\mathbb{R}$. This is because $e^{tX_{F}}$ is a \href{https://en.wikipedia.org/wiki/Fourier_integral_operator}{Fourier integral operator}
(i.e. has micro-local property (\ref{eq:Psi_tilde})). We introduce
a specific metric $g$ on $T^{*}M$, compatible with the symplectic
form and then we define an $L^{2}$-isometric ``wave-packet transform''
$\mathcal{T}:C^{\infty}\left(M;F\right)\rightarrow\mathcal{S}\left(T^{*}M:F\right)$,
that allows to use \textbf{micro-local analysis} on $T^{*}M$ for
the pull back operator $e^{tX_{F}}$. The unit boxes for the metric
$g$ correspond to the effective size of wave-packets and reflect
the uncertainty principle.
\item For a contact Anosov flow, the dynamics induced on the cotangent bundle
$T^{*}M$ is a ``scattering dynamics'' on the trapped set $\Sigma=\mathbb{R}^{*}\mathcal{A}\subset T^{*}M$
and $\Sigma$ is symplectic and normally hyperbolic. See Figure \ref{fig:N_Sigma}.
In terms of dynamics, the subset $\Sigma$ is the \href{https://en.wikipedia.org/wiki/Wandering_set}{non wandering set}
for the flow $\tilde{\phi}^{t}$ on $T^{*}M$ and the orbits on the
outside of $\Sigma$ go to infinity either as $t\rightarrow+\infty$
or $t\rightarrow-\infty$. As a consequence, for large time $\left|t\right|\gg1$,
the outer part of the trapped set $\Sigma$ has a negligible contribution,
because information escapes to infinity (i.e. the Sobolev norm measured
by the specific weight $W$ decays). This will be given in Theorem
\ref{thm:decay}. So only the dynamics on $\Sigma$ plays a role for
our purpose. But due to the uncertainty principle in $T^{*}M$, we
still have to consider a neighborhood of $\Sigma$ in $T^{*}M$. We
will hence consider a neighborhood of $\Sigma$ of a given size $\left\langle \omega\right\rangle ^{\mu/2}$
in (\ref{eq:def_Chi_mu}) (measured by the metric $g$), with some
$0<\mu<1$. It is important to remark that for large frequencies $\omega$
this neighborhood of size $\left\langle \omega\right\rangle ^{\mu/2}$
projected down on $M$ is scaled by $\left\langle \omega\right\rangle ^{-1/2}$
(from the definition of the metric $g$ on $T^{*}M$ in (\ref{eq:metric_g_tilde_in_coordinates}),
see Remark \ref{rem:an-important-property}) and get size $\asymp\omega^{-\left(1-\mu\right)/2}\underset{\omega\rightarrow+\infty}{\rightarrow}0$
that goes to zero as $\omega$ goes to infinity. This will allow us
to use \textbf{the linearization of the dynamics} $\tilde{\phi}^{t}$
as a local approximation.
\item In the neighborhood of size $\left\langle \omega\right\rangle ^{\mu/2}$
of the trapped set $\Sigma$ (i.e. the \href{https://en.wikipedia.org/wiki/Wandering_set}{non wandering set})
that matters, there is a micro-local decoupling between the directions
tangent to $\Sigma$ and those (symplectically) normal to $\Sigma$,
represented by a normal vector bundle denoted $N=\left(T\Sigma\right)^{\perp}\rightarrow\Sigma$.
The dynamics on the normal direction $N$ is hyperbolic and responsible
for the emergence of polynomial functions along the stable direction
$N_{s}$ (that projects to $E_{s}$). By considering the Taylor operators
$\tilde{T}_{k}$ as symbols, we introduce approximate projectors $\mathrm{Op}_{\Sigma}\left(\tilde{T}_{k}\right)$
on $C^{\infty}\left(M;F\right)$ that micro-locally (i.e. seen in
$\mathcal{S}\left(T^{*}M;F\right)$ after the wave-packet transform
$\mathcal{T}$) restricts functions to the symplectic trapped set
$\Sigma$ and that are polynomial valued along $N_{s}$ with degree
$k$. This projector plays a similar role as the Bergman projector
(or Szegö projector) in geometric quantization. What remains for large
time, is an effective Hilbert space of ``quantum waves'' that live
on the trapped set $\Sigma$, valued in the vector bundle $\mathcal{F}_{k}$
defined in (\ref{eq:def_Fk}).
\end{enumerate}
The very simple toy model that is useful to have in mind is given
in Section \ref{subsec:Discrete-Ruelle-Pollicott-spectr}. It explains
in particular the emergence of polynomials $\mathrm{Pol}_{k}\left(E_{s}\right)$
as in (\ref{eq:def_F-1}).

\begin{figure}
\centering{}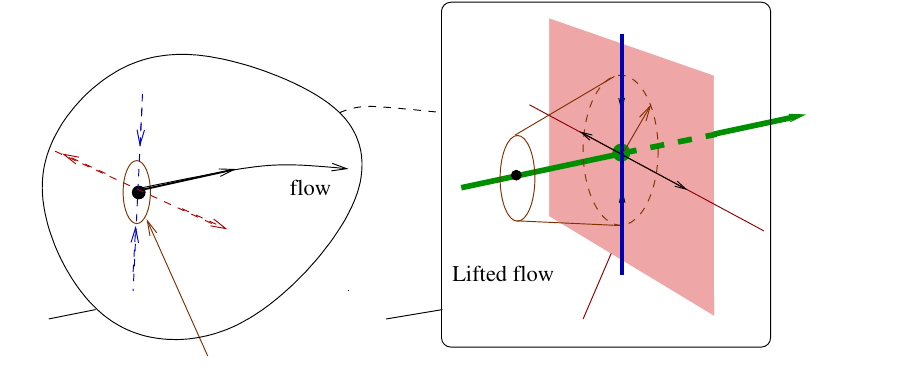\caption{\label{fig:N_Sigma}The dynamics induced on $T^{*}M$ scatters on
the trapped set $\Sigma\subset T^{*}M$ defined in (\ref{eq:def_Sigma}).
$\Sigma$ is a line bundle over $M$, a symplectic sub-manifold of
$T^{*}M$ and at every point $\rho=\omega\mathcal{A}\left(m\right)\in\Sigma$,
where $\omega$ called frequency is the coordinate along the line,
the symplectic-normal bundle $N\left(\rho\right)=\left(T_{\rho}\Sigma\right)^{\perp}$
(a symplectic linear subspace of $T_{\rho}T^{*}M$) splits into unstable/stable
subspaces, $N\left(\rho\right)=N_{u}\left(\rho\right)\oplus N_{s}\left(\rho\right)$,
see (\ref{eq:decomp_K_K0_N}). The \textbf{main geometrical object}
considered in this paper is this fibration $N_{s}\rightarrow\Sigma\rightarrow M$.
Beware that for a geodesic flow on $\left(\mathcal{N},g\right)$,
this fibration sequence continues with $M=\left(T^{*}\mathcal{N}\right)_{1}\rightarrow\mathcal{N}$.}
\end{figure}

\subsection{\label{subsec:Consequences-and-other}Consequences and other results}
\begin{center}
\begin{figure}
\begin{centering}
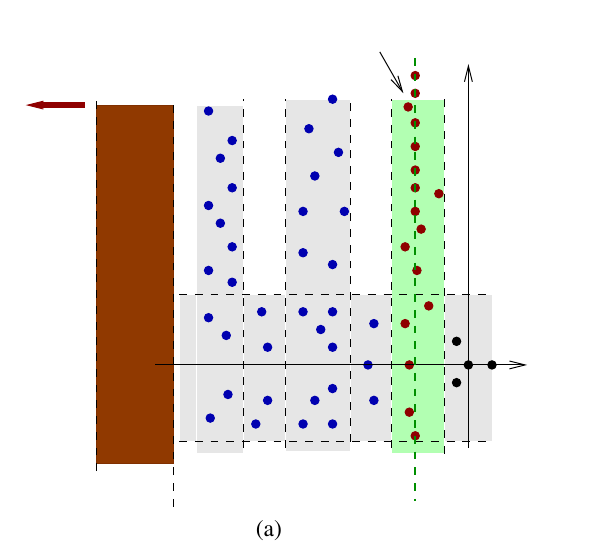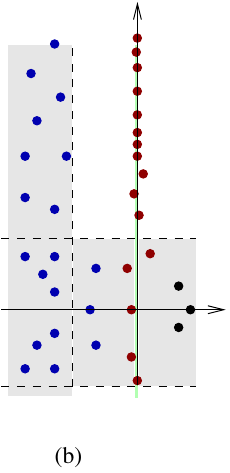
\par\end{centering}
\caption{\label{fig:Band-structure-of}The dots represent the intrinsic Ruelle
discrete spectrum of the derivation $X_{F}$ in a Sobolev space $\mathcal{H}_{W}\left(M;F\right)$.
From \cite[thm 2.11]{faure_tsujii_Ruelle_resonances_density_2016},
the essential spectrum is in a (brown) vertical band that can be moved
arbitrarily far on the left by changing the weight $W$, and reveals
this intrinsic discrete spectrum. The right most eigenvalues in the
first band $B_{0}$ dominate the emerging behavior of $e^{tX_{F}}$
for $t\gg1$. On figure (b) for the special case of the bundle $F=\left|\mathrm{det}E_{s}\right|^{1/2}$,
the first band coincides with the imaginary axis.}
\end{figure}
\par\end{center}

There are many consequences of the effective description (\ref{eq:emerging-1})
of the dynamics. In this paper we describe a few of them that are
illustrated on Figure \ref{fig:Band-structure-of}. We will follow
general ideas from classical-quantum correspondence principles for
quantization in the case where the classical symbol is an operator
valued function on a symplectic manifold $\Sigma$ (this situation
is also present in physics for \href{https://en.wikipedia.org/wiki/Linear_elasticity\#Elastic_wave}{elastic waves},
\href{https://en.wikipedia.org/wiki/Electromagnetic_radiation}{electromagnetic waves},
\href{https://en.wikipedia.org/wiki/Dirac_equation}{Dirac equation}
etc), and we will derive the following results for the operator $X_{F}$.

\subsubsection{Discrete spectrum in vertical bands}

The generator of the classical dynamics (\ref{eq:classical_dyn})
is $X_{\mathcal{F}_{k}}$ and its spectrum (some essential spectrum
in fact) in $L^{2}\left(\Sigma;\mathcal{F}_{k}\right)$ is contained
in the vertical band $B_{k}:=\left[\gamma_{k}^{-},\gamma_{k}^{+}\right]\times i\mathbb{R}$
with $\gamma_{k}^{\pm}$ defined in (\ref{eq:def_gamma_-2}). We always
have $\gamma_{k+1}^{-}<\gamma_{k}^{-}$ and $\gamma_{k+1}^{+}<\gamma_{k}^{+}$
but this does not guarantee that the bands are separated by gaps i.e.
that $\gamma_{k+1}^{+}<\gamma_{k}^{-}$ , $B_{k+1}\cap B_{k}=\emptyset$. 

In the semi-classical limit, i.e. for $\omega=\mathrm{Im}\left(z\right)\rightarrow\pm\infty$
on the spectral plane $z\in\mathbb{C}$, Theorem \ref{Thm: bands}
below shows that the spectrum of $X_{F}$ is discrete and asymptotically
contained in these vertical bands $\bigcup_{k\in\mathbb{N}}B_{k}$
and the resolvent is uniformly bounded in the gaps between the bands
(if they exist). This is illustrated on Figure \ref{fig:Band-structure-of}.

\begin{cBoxB}{}
\begin{thm}[\textbf{Band structure of the Ruelle spectrum}]
\label{Thm: bands}For any $\epsilon>0$, $C>0$, there exists $C_{\epsilon}>0,\omega_{\epsilon}>0$
such that the Ruelle eigenvalues $\mathrm{spect}\left(X_{F}\right)$
are contained in the following spectral domain that consists of a
union of a \textbf{``low frequency horizontal band''} and \textbf{``vertical
bands''}:
\begin{equation}
\left(\mathrm{spect}\left(X_{F}\right)\cap\left\{ \mathrm{Re}\left(z\right)>-C\right\} \right)\subset\left\{ \left|\mathrm{Im}\left(z\right)\right|\leq\omega_{\epsilon}\right\} \cup\left(\bigcup_{k\in\mathbb{N}}\left\{ \mathrm{Re}\left(z\right)\in\left[\gamma_{k}^{-}-\epsilon,\gamma_{k}^{+}+\epsilon\right]\right\} \right),\label{eq:gaps-1}
\end{equation}
and the resolvent operator is uniformly bounded in the ``gaps'':
\begin{equation}
\left\Vert \left(z-X_{F}\right)^{-1}\right\Vert _{\mathcal{H}_{W}}\leq C_{\epsilon},\qquad\forall z\in\mathbb{C}\text{ s.t. }\left|\mathrm{Im}\left(z\right)\right|>\omega_{\epsilon},\mathrm{Re}\left(z\right)\in\left[\gamma_{k+1}^{+}+\epsilon,\gamma_{k}^{-}-\epsilon\right].\label{eq:resolvent-1}
\end{equation}
\end{thm}

\end{cBoxB}

The proof is given in Section \ref{subsec:Band-spectrum-of}.

Notice that if the gap does not exist, i.e. if the band overlap ($\gamma_{k+1}^{+}\geq\gamma_{k}^{-}$
for every $k\in\mathbb{N}$) then the content of Theorem \ref{Thm: bands}
is empty except for the right most limit $\gamma_{0}^{+}$.

From a general theorem in semi-group theory \cite[p.276, Thm1.18 p.307, prop 1.15p.305]{engel_1999},
Theorem \ref{Thm: bands} that concerns the generator $X_{F}$ is
equivalent to the following result that concerns the semi-group $\left(e^{tX_{F}}\right)_{t\geq0}$
and illustrated on Figure \ref{fig:Spectrum-of-annuli}.

\begin{figure}
\begin{centering}
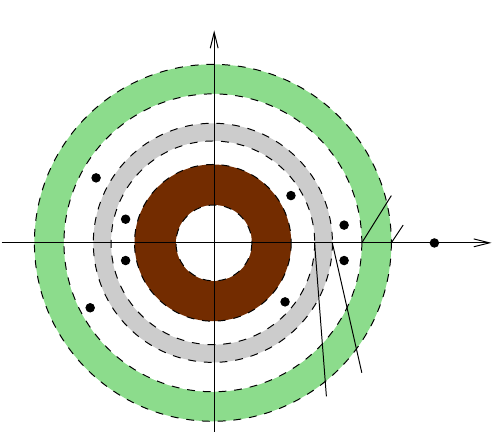
\par\end{centering}
\caption{\label{fig:Spectrum-of-annuli}Band spectrum of $e^{tX_{F}}$ in $\mathcal{H}_{W}\left(M;F\right)$
for $t>0$, from Theorem \ref{thm:Discrete-band_spectrum-1}. For
$k\in\mathbb{N}$, and $\gamma_{k}^{\pm}$ defined in (\ref{eq:def_gamma_-2}),
the operator $e^{tX_{F}}$ has discrete spectrum in the annulus $e^{t\gamma_{k+1}^{+}}<\left|z\right|<e^{t\gamma_{k}^{-}}$
(if $\gamma_{k+1}^{+}<\gamma_{k}^{-}$) called \textquotedblleft spectral
gap\textquotedblright . It has also discrete spectrum on $\left|z\right|>e^{t\gamma_{0}^{+}}$
and possibly essential spectrum elsewhere (the colored annuli). This
picture is somehow the exponential of Figure \ref{fig:Band-structure-of}.}
\end{figure}
\begin{cBoxB}{}
\begin{thm}[\textbf{Ring spectrum}]
\label{thm:Discrete-band_spectrum-1}Let $k\in\mathbb{N}$. For any
$t>0$, the operator $e^{tX_{F}}:\mathcal{H}_{W}\left(M;F\right)\rightarrow\mathcal{H}_{W}\left(M;F\right)$
has discrete spectrum on the ring $\left\{ z\in\mathbb{C},\quad e^{t\gamma_{k+1}^{+}}<\left|z\right|<e^{t\gamma_{k}^{-}}\right\} $
(if non empty, i.e. if $\gamma_{k+1}^{+}<\gamma_{k}^{-}$) and on
$\left|z\right|>e^{t\gamma_{0}^{+}}$.
\end{thm}

\end{cBoxB}

\begin{rem}
An immediate consequence of Theorem \ref{thm:Discrete-band_spectrum-1},
if $\gamma_{k+1}^{+}<\gamma_{k}^{-}$ for some $k\in\mathbb{N}$,
is the possibility to perform a contour integral of the resolvent
of $e^{tX_{F}}$ at fixed value of $t>0$, on the circle $C_{t,\epsilon}:=\left\{ z\in\mathbb{C}\mathrm{\,s.t.\,}\left|z\right|=e^{t\left(\gamma_{k+1}^{+}+\epsilon\right)}\right\} $
with $\epsilon>0$ such that $\gamma_{k+1}^{+}+\epsilon<\gamma_{k}^{-}$
and such that there is no eigenvalue on this circle. Using \href{https://en.wikipedia.org/wiki/Resolvent_formalism}{Cauchy integral}
and \href{https://en.wikipedia.org/wiki/Holomorphic_functional_calculus}{holomorphic functional calculus},
we get a spectral projector $\Pi_{\left[0,k\right]}=\mathrm{Id}-\frac{1}{2\pi i}\varoint_{C_{t,\epsilon}}\left(z-e^{tX_{F}}\right)^{-1}dz$
on the first $k$ external bands and one can approximate $e^{tX_{F}}$
for $t>0$ by the spectral restriction $e^{tX_{F}}\Pi_{\left[0,k\right]}$
with the following estimate:
\begin{equation}
\left\Vert e^{tX_{F}}-e^{tX_{F}}\Pi_{\left[0,k\right]}\right\Vert _{\mathcal{H}_{W}\left(M;F\right)}\leq C_{\epsilon}e^{t\left(\gamma_{k+1}^{+}+\epsilon\right)},\label{eq:decay}
\end{equation}
where $e^{tX_{F}}\Pi_{\left[0,k\right]}$ is compared to a quantum
evolution operator in this paper. This operator can be used to describe
with great accuracy the \textbf{decay of correlations} $\langle v|e^{tX_{F}}u\rangle_{L^{2}\left(M;F\right)}$
as $t\rightarrow\infty$, that is the question\footnote{\label{fn:Remark-also-that}Remark that, thanks to the spectral gap
assumption $\gamma_{k+1}^{+}<\gamma_{k}^{-}$ and thanks to the abstract
definition of the spectral projector $\Pi_{\left[0,k\right]}$ we
obtain Eq.(\ref{eq:decay}) without being embarrassed by the finite
rank operator $R_{\sigma,t}$ described in section \ref{subsec:Few-comments-about},
whose rank may grows with $t$.} raised in the beginning of this paper.
\end{rem}

~
\begin{rem}
\label{rem:norm_resolvent_and_W}The discrete spectrum of $X_{F}$
(and eigenspace) are said ``intrinsic'' because they do not depend
on the space $\mathcal{H}_{W}\left(M;F\right)$, see \cite[Thm 1.5]{fred_flow_09}.
However more refined spectral properties as the norm of the resolvent
$\left\Vert \left(z-X_{F}\right)^{-1}\right\Vert _{\mathcal{H}_{W}\left(M;F\right)}$
depend on the choice of the space $\mathcal{H}_{W}\left(M;F\right)$
and play an important role in this paper. For example the uniform
boundness of the norm of the resolvent in the gaps given in Theorem
\ref{Thm: bands} (implying Theorem \ref{thm:Discrete-band_spectrum-1})
is obtained with the weight function $W$ and Sobolev space $\mathcal{H}_{W}\left(M;F\right)$
from \cite{faure_tsujii_Ruelle_resonances_density_2016} and defined
below in (\ref{eq:def_W}) because we have decay outside a parabolic
neighborhood of the trapped set $\Sigma$, saturating the uncertainty
principle. This property of uniform boundness is not true with the
weight function and Sobolev space defined in \cite{fred_flow_09}
where the decay takes place only outside a (much bigger) conical neighborhood
of the trapped set $\Sigma$. 
\end{rem}

\subsubsection{Weyl law}

Recall that $\mathrm{dim}M=2d+1$, with $d=\mathrm{dim}E_{s,u}=\mathrm{dim}N_{u,s}$.
For some given $k\in\mathbb{N}$, suppose $B_{k}\cap\left(\bigcup_{k'\neq k}B_{k'}\right)=\emptyset$,
i.e. that the band $B_{k}$ is isolated. Theorem \ref{thm:Weyl-law-for}
below shows that the density of discrete eigenvalues of $X_{F}$ in
band $B_{k}$ in the limit $\omega\rightarrow\infty$ converges to
$\mathrm{rank}\left(\mathcal{F}_{k}\right)\mathrm{Vol}\left(M\right)\frac{\omega^{d}}{\left(2\pi\right)^{d+1}}$
with 
\begin{equation}
\mathrm{rank}\left(\mathcal{F}_{k}\right)\eq{\ref{eq:def_Fk}}\binom{k+d-1}{d-1}\mathrm{rank}\left(F\right).\label{eq:rank_F_k}
\end{equation}
It is analogous to the usual \href{https://en.wikipedia.org/wiki/Weyl_law}{Weyl law}.

\begin{cBoxB}{}
\begin{thm}[\textbf{Weyl law for isolated bands}]
\label{thm:Weyl-law-for}For any $k\in\mathbb{N}$ and $\epsilon>0$,
such that $\gamma_{k+1}^{+}<\gamma_{k}^{-}-\epsilon$ and $\gamma_{k}^{+}+\epsilon<\gamma_{k-1}^{-}$
(this last condition is only for $k\geq1$)
\begin{align}
\lim_{\delta\rightarrow+\infty}\limsup_{\omega\rightarrow\pm\infty}\left|\frac{1}{\left|\omega\right|^{d}\delta}\sharp\left\{ \mathrm{spect}\left(X_{F}\right)\cap\mathrm{Box}_{\omega,\delta}\right\} -\frac{\mathrm{rank}\left(\mathcal{F}_{k}\right)\mathrm{Vol}\left(M\right)}{\left(2\pi\right)^{d+1}}\right| & =0\label{eq:Weyl_law}
\end{align}
with the spectral box
\[
\mathrm{Box}_{\omega,\delta}:=\left\{ z\in\left(\left[\gamma_{k}^{-}-\epsilon,\gamma_{k}^{+}+\epsilon\right]\times i\left[\omega,\omega+\delta\right]\right)\right\} .
\]
\end{thm}

\end{cBoxB}

The proof is obtained in Section \ref{sec:Proof-of-Theorem_Weyl}.
\begin{rem}
As the proof shows, the result holds true more generally for a group
of isolated bands indexed by $k\in\left[k_{1},k_{2}\right]$ with
$k_{1}\leq k_{2}$, i.e. assuming $\gamma_{k_{2}+1}^{+}<\gamma_{k_{2}}^{-}$
and $\gamma_{k_{1}}^{+}<\gamma_{k_{1}-1}^{-}$ . For the case $k_{1}=0$
we only have to assume $\gamma_{k_{2}+1}^{+}<\gamma_{k_{2}}^{-}$.
Then in (\ref{eq:Weyl_law}), the interval $\left[\gamma_{k}^{-}-\epsilon,\gamma_{k}^{+}+\epsilon\right]$
has to be replaced by $\left[\gamma_{k_{2}}^{-}-\epsilon,\gamma_{k_{1}}^{+}+\epsilon\right]$
and $\mathrm{rank}\left(\mathcal{F}_{k}\right)$ has to be replaced
by $\mathrm{rank}\left(\bigoplus_{k=k_{1}}^{k_{2}}\mathcal{F}_{k}\right)$.
\end{rem}

\subsubsection{Ergodic concentration of the spectrum}

We define in (\ref{eq:def_Lyapounov_exp}) the \textbf{maximal and
minimal exponents} $\check{\gamma}_{k}^{-}\leq\check{\gamma}_{k}^{+}$
of the bundle maps $e^{tX_{\mathcal{F}_{k}}}$ with respect to the
contact volume on $\Sigma$, satisfying $\left[\check{\gamma}_{k}^{-},\check{\gamma}_{k}^{+}\right]\subset\left[\gamma_{k}^{-},\gamma_{k}^{+}\right]$.
The following theorem in addition to Theorem \ref{thm:Weyl-law-for}
shows that most of the eigenvalues in band $B_{k}$ belong to the
narrower band $\mathrm{Re}\left(z\right)\in\left[\check{\gamma}_{k}^{-},\check{\gamma}_{k}^{+}\right]$.

\begin{cBoxB}{}
\begin{thm}[\textbf{Ergodic concentration of the spectrum}]
\textbf{\label{thm:Ergodic-concentration-of}}Let $k\in\mathbb{N}$.
Assume that $\gamma_{k+1}^{+}<\gamma_{k}^{-}$ and $\gamma_{k}^{+}<\gamma_{k-1}^{-}$
(this last condition is only for $k\geq1$). For any $\epsilon>0$
small enough, we have
\begin{equation}
\lim_{\delta\rightarrow+\infty}\limsup_{\omega\rightarrow\pm\infty}\frac{1}{\left|\omega\right|^{d}\delta}\sharp\left\{ \mathrm{spect}\left(X_{F}\right)\cap\mathrm{Strips}_{\omega,\delta}\right\} =0\label{eq:accumulation-1-1}
\end{equation}
with
\[
\mathrm{Strips}_{\omega,\delta}:=\left\{ z\in\left(\left[\gamma_{k}^{-}-\epsilon,\gamma_{k}^{+}+\epsilon\right]\backslash\left[\check{\gamma}_{k}^{-}-\epsilon,\check{\gamma}_{k}^{+}+\epsilon\right]\right)\times i\left[\omega,\omega+\delta\right]\right\} .
\]
\end{thm}

\end{cBoxB}

The proof is given in Section \ref{sec:Proof-of-Theorem-accumulation}.
For example, in the special case of band $k=0$ and $\mathrm{rank}\left(F\right)=1$
we have $\mathrm{rank}\left(\mathcal{F}_{0}\right)\eq{\ref{eq:rank_F_k}}1$,
and from ergodicity of $X$, most of eigenvalues accumulate on the
vertical line
\[
\check{\gamma}_{0}^{-}=\check{\gamma}_{0}^{+}=\frac{1}{\mathrm{Vol}\left(M\right)}\int_{M}Ddm
\]
being the space average of the ``damping function'' $D\in C\left(M;\mathbb{R}\right)$
defined by $D\left(m\right):=V\left(m\right)+\frac{1}{2}\mathrm{div}X_{/E_{s}}\left(m\right)$
with the potential function $V$ given in (\ref{eq:XF_X_V}), see
\cite{faure-tsujii_prequantum_maps_12}.

\subsubsection{Emerging quantum dynamics on vector bundles}

The band structure of the Ruelle spectrum and Weyl law described above
reflect in fact a deeper geometric phenomenon that can be explained
in terms of geometric quantization. Let us explain this.

\paragraph{Wave-packet transform $\mathcal{T}$:}

In paper \cite[def. 4.23]{faure_tsujii_Ruelle_resonances_density_2016},
see also (\ref{eq:def_T}), we introduce a wave-packet transform (here
$\mathcal{S}\left(T^{*}M;F\right)$ is the space of Schwartz sections
of the bundle $F\rightarrow T^{*}M$)
\begin{equation}
\mathcal{T}:C^{\infty}\left(M;F\right)\rightarrow\mathcal{S}\left(T^{*}M;F\right),\label{eq:Tau}
\end{equation}
satisfying $\mathcal{T}^{\dagger}\mathcal{T}=\mathrm{Id}_{L^{2}\left(M;F\right)}$
and that is the basic tool used for micro-local analysis. For example,
using the characteristic function $\chi_{\Sigma,\sigma}=\boldsymbol{1}_{\left\{ \left\Vert \rho_{u}+\rho_{s}\right\Vert _{g_{\rho}}\leq\sigma\right\} }:T^{*}M\rightarrow\left[0,1\right]$
defined in (\ref{eq:def_Lambda}), for the neighborhood of the set
$\Sigma=\mathbb{R}^{*}\mathcal{A}\subset T^{*}M$ at a distance $\sigma>0$
measured by a specific metric $g$ (given in Lemma \ref{lem:There-exists-a}),
we define the operator 
\begin{equation}
\mathrm{Op}\left(\chi_{\Sigma,\sigma}\right)\eq{\ref{eq:def_op_Lambda}}\mathcal{T}^{\dagger}\chi_{\Sigma,\sigma}\mathcal{T}:C^{\infty}\left(M;F\right)\rightarrow C^{\infty}\left(M;F\right)\label{eq:Op_Sigma_sigma}
\end{equation}
 that restricts functions to their micro-local components near $\Sigma$.

Theorem \ref{thm:decay} will show that the norm decays exponentially
fast outside the trapped set $\Sigma$ with an exponential rate, i.e.
we have $\left\Vert e^{tX_{F}}\left(\mathrm{Id}-\mathrm{Op}\left(\chi_{\Sigma,\sigma}\right)\right)\right\Vert _{\mathcal{H}_{W}\left(M\right)}\leq Ce^{-t\Lambda}$
for $t\geq0$, with a rate $\Lambda>0$ arbitrarily large if $\sigma$
is large enough (depending on $t$). This implies that for our study,
we can consider only the component of the dynamics $e^{tX}\mathrm{Op}\left(\chi_{\Sigma,\sigma}\right)$
near the trapped set.

\paragraph{Quantization:}

Recall the vector bundle $\mathcal{F}\left(E_{s}\right)\rightarrow M$
defined in (\ref{eq:def_F_cal}) and the lifted bundle $\mathcal{F}\left(N_{s}\right)\rightarrow\Sigma$.
For a map $a\in C^{\beta}\left(M;\mathrm{End}\left(\mathcal{F}\left(E_{s}\right)\right)\right)$
(or even $a\in C^{\beta}\left(M;L\left(\mathcal{F}\left(E_{s}\right),\mathcal{F}'\left(E_{s}\right)\right)\right)$)
and time $t\in\mathbb{R}$, we will define a quantum operator in Definition
\ref{def:The-quantization-of-1}
\[
\mathrm{Op}_{\Sigma}\left(e^{tX_{\mathcal{F}}}\tilde{a}\right):C^{\infty}\left(M;F\right)\rightarrow C^{\infty}\left(M;F\right)
\]
where $\tilde{a}\in C^{\beta}\left(\Sigma;L\left(\mathcal{F}\left(N_{s}\right)\right)\right)$
is the lifted map.

For two operators $A,B$, we define in (\ref{eq:def_approx}) an equivalence
relation $A\approx B$ if their difference in any $\sigma-$neighborhood
of the trapped set $\Sigma$ and for large enough frequencies becomes
``negligible''. A consequence is: if $A\approx B$ then for any
$\sigma>0$,
\begin{equation}
\left\Vert \mathrm{Op}\left(\chi_{\Sigma,\sigma}\right)\left(A-B\right)\mathrm{Op}\left(\chi_{\Sigma,\sigma}\right)\left(\mathrm{Id}-\mathrm{Op}\left(\chi_{\omega}\right)\right)\right\Vert _{\mathcal{H}_{W}\left(M\right)}\underset{\omega\rightarrow\infty}{\rightarrow}0,\label{eq:r_sigma-1}
\end{equation}
where $\mathrm{Op}\left(\chi_{\Sigma,\sigma}\right)$ is given in
(\ref{eq:Op_Sigma_sigma}) and for $\omega>0$, $\mathrm{Op}\left(\chi_{\omega}\right):\mathcal{S}'\left(M;F\right)\rightarrow\mathcal{S}\left(M;F\right)$
is a smoothing operator hence compact that removes frequencies larger
than $\omega$, defined as follows: let $g_{M}$ be an arbitrary smooth
metric on $M$. Let $\chi_{\omega}:T^{*}M\rightarrow\mathbb{R}^{+}$
given by 
\begin{equation}
\chi_{\omega}\left(.\right)=\boldsymbol{1}_{\left\{ \left\Vert .\right\Vert _{g_{M}}\leq\omega\right\} }\label{eq:def_Chi_omega}
\end{equation}
and let
\begin{equation}
\mathrm{Op}\left(\chi_{\omega}\right):=\mathcal{T}^{\dagger}\chi_{\omega}\mathcal{T}.\label{eq:def_Xi_low}
\end{equation}
In Theorem \ref{thm:Approximation-of-the} we will obtain

\begin{cBoxB}{}
\begin{thm}[\textbf{Approximation of the dynamics by quantum operator}]
\textbf{\label{thm:Approximation-of-the-1}}For any $t\in\mathbb{R}$,
\begin{align}
e^{tX_{F}} & \approx\mathrm{Op}_{\Sigma}\left(e^{tX_{\mathcal{F}}}\right).\label{eq:expression-1}
\end{align}
\end{thm}

\end{cBoxB}

In Theorem \ref{thm:Composition-formula-For} we will obtain\footnote{Recall from section \ref{subsec:Main-result.-Emergence} that a symbol
$a\in C^{\beta}\left(M;\mathrm{End}\left(\mathcal{F}_{\left[0,K\right]}\left(E_{s}\right)\right)\right)$can
be extended to $a\circ T_{\left[0,K\right]}\in C^{\beta}\left(M;L\left(\mathcal{F},\mathcal{F}'\right)\right)$
using Taylor projectors $T_{\left[0,K\right]}:\mathcal{F}\left(E_{s}\right)\rightarrow\mathcal{F}_{\left[0,K\right]}\left(E_{s}\right)\subset\mathcal{F}'\left(E_{s}\right)$.}

\begin{cBoxB}{}
\begin{thm}[\textbf{Composition formula}]
\textbf{\label{thm:Composition-formula-For-1}}For any symbols $a,b\in C^{\beta}\left(M;L\left(\mathcal{F}_{\left[0,K\right]}\left(E_{s}\right)\right)\right)$,
any $t,t'\in\mathbb{R}$, we have
\begin{equation}
\mathrm{Op}_{\Sigma}\left(e^{tX_{\mathcal{F}}}\tilde{a}\right)\mathrm{Op}_{\Sigma}\left(e^{t'X_{\mathcal{F}}}\tilde{b}\right)\approx\mathrm{Op}_{\Sigma}\left(e^{tX_{\mathcal{F}}}\tilde{a}e^{t'X_{\mathcal{F}}}\tilde{b}\right).\label{eq:compos_operators-1}
\end{equation}
\end{thm}

\end{cBoxB}

In particular, since for any $k,k'\in\mathbb{N}$ we have for the
Taylor projectors $\left[T_{k},T_{k'}\right]=\delta_{k=k'}T_{k}$
and $\left[T_{k},e^{tX_{\mathcal{F}}}\right]=0$, we will get in corollary
\ref{cor:As-particular-cases} that
\begin{equation}
\mathrm{Op}_{\Sigma}\left(e^{tX_{\mathcal{F}}}\tilde{T}_{k}\right)\mathrm{Op}_{\Sigma}\left(e^{t'X_{\mathcal{F}}}\tilde{T}_{k'}\right)\approx\delta_{k=k'}\mathrm{Op}_{\Sigma}\left(e^{\left(t+t'\right)X_{\mathcal{F}}}\tilde{T}_{k}\right).\label{eq:alg-1}
\end{equation}

We will get a general continuity Theorem \ref{thm:continuity_thm},

\begin{cBoxB}{}
\begin{thm}[continuity theorem for F.I.O.]
\label{thm:continuity_thm-1}There exists $C>0$, such that for any
symbol $a\in C^{\beta}\left(M;L\left(\mathcal{F}_{\left[0,K\right]}\left(E_{s}\right)\right)\right)$
and $t\in\mathbb{R}$,
\begin{equation}
\left\Vert \mathrm{Op}_{\Sigma}\left(e^{tX_{\mathcal{F}}}\tilde{a}\right)\right\Vert _{\mathcal{H}_{W}\left(M\right)}\leq C\left\Vert e^{tX_{\mathcal{F}}}\tilde{a}\right\Vert _{\mathcal{H}_{\mathcal{W}}\left(\mathcal{F}\right)}.\label{eq:bound_OP_Sigma-2}
\end{equation}
\end{thm}

\end{cBoxB}

As a corollary we will obtain in corollary \ref{Thm:Boundness} some
estimates that reflect Eq.(\ref{eq:def_gamma_-2}) at the quantum
level, such as

\begin{cBoxB}{}
\begin{cor}
For any $k\in\mathbb{N}$, $\forall\epsilon>0,\exists C_{k,\epsilon}>0$,$\forall t\geq0$,
\begin{equation}
\left\Vert \mathrm{Op}_{\Sigma}\left(e^{tX_{\mathcal{F}}}\tilde{T}_{k}\right)\right\Vert _{\mathcal{H}_{W}\left(M\right)}\leq C_{k,\epsilon}e^{t\left(\gamma_{k}^{+}+\epsilon\right)},\label{eq:bound-4-1-1}
\end{equation}
\begin{equation}
\left\Vert \mathrm{Op}_{\Sigma}\left(e^{-tX_{\mathcal{F}}}\tilde{T}_{k}\right)\right\Vert _{\mathcal{H}_{W}\left(M\right)}\leq C_{k,\epsilon}e^{-t\left(\gamma_{k}^{-}-\epsilon\right)}.\label{eq:bound-4-2}
\end{equation}
\end{cor}

\end{cBoxB}

The band spectrum of the Ruelle spectrum in Theorem \ref{Thm: bands}
and Theorem \ref{thm:Discrete-band_spectrum-1} is a direct manifestation
of these boundness estimates together with the algebraic structure
of (\ref{eq:alg-1}).

In Theorem \ref{thm:Take-.-There} we will get some trace formula
that will be used to prove Weyl law in section \ref{sec:Proof-of-Theorem_Weyl}.

\subsubsection{Horocycle operators}

In the special case of the geodesic flow on $M=\Gamma\backslash\mathrm{SL}_{2}\left(\mathbb{R}\right)$,
the Lie \href{https://en.wikipedia.org/wiki/Special_linear_Lie_algebra}{algebra}
of the $\mathrm{sl}_{2}\left(\mathbb{R}\right)$ provides left invariant
vector fields $u\in C^{\infty}\left(M;E_{u}\right)$, $s\in C^{\infty}\left(M;E_{s}\right)$
that satisfies
\begin{equation}
\left[s,u\right]=X\label{eq:sl2R}
\end{equation}
and called respectively unstable/stable \href{https://en.wikipedia.org/wiki/Horocycle}{horocycle}
vector fields (and more generally for the algebra of $\mathrm{so}\left(n,1\right)$),
see \cite{flaminio_forni_2003,dyatlov_faure_guillarmou_2014,guillarmou_weich_resonances_16}
where it is shown that these operators $s,u$ map respectively band
$B_{k}$ to band $B_{k-1},B_{k+1}$.

In the general case of a contact Anosov vector field $X$, we propose
below some \textbf{``approximate horocycle operators}'' $\mathrm{Op}_{\Sigma}\left(\iota_{s}\right),\mathrm{Op}_{\Sigma}\left(u\varodot\right)$
that satisfy some commutator relation similar to (\ref{eq:sl2R})
but in the semiclassical limit $\omega\gg1$. They also respectively
map band $B_{k}$ to $B_{k\mp1}$ (at first order).

If $u\in C^{\beta}\left(M;E_{u}\right)$, $s\in C^{\beta}\left(M;E_{s}\right)$
are respectively Hölder continuous sections of the unstable/stable
bundles of the Anosov dynamics, we naturally associate to them the
following ``operator-valued symbols'' (see Section \ref{sec:Horocycle-operators}
for a more detailed definition) 
\begin{equation}
\iota_{s}:C^{\beta}\left(\Sigma;\mathcal{F}_{k}\left(N_{s}\right)\right)\rightarrow C^{\beta}\left(\Sigma;\mathcal{F}_{k-1}\left(N_{s}\right)\right)\label{eq:iota_s}
\end{equation}
by point-wise tensor-contraction with the polynomial in (\ref{eq:def_Fk})
and
\begin{equation}
u\varodot:C^{\beta}\left(\Sigma;\mathcal{F}_{k}\left(N_{s}\right)\right)\rightarrow C^{\beta}\left(\Sigma;\mathcal{F}_{k+1}\left(N_{s}\right)\right)\label{eq:ext_u}
\end{equation}
by point-wise symmetric tensor product. We will see in Lemma \ref{lem:Weyl-algebra.-For}
that they satisfy the point-wise \href{https://en.wikipedia.org/wiki/Weyl_algebra}{Weyl algebra}
(also called symplectic Clifford algebra)
\begin{equation}
\left[\iota_{s},u\varodot\right]\tilde{T}_{k}=\boldsymbol{\omega}\left(.\right)\left(\left(d\mathcal{A}\right)\left(s,u\right)\right)\tilde{T}_{k}.\label{eq:weyl_algebra}
\end{equation}
where $\boldsymbol{\omega}\left(.\right):\omega\mathcal{A}\left(m\right)\in\Sigma\rightarrow\omega\in\mathbb{R}$
is the frequency map, $d\mathcal{A}$ is the symplectic form on $E_{u}\oplus E_{s}$
and $\tilde{T}_{k}$ the projector on $\mathcal{F}_{k}$.

\begin{cBoxB}{}
\begin{thm}[\textbf{Quantized Weyl algebra}]
\label{thm:Quantization-of-()}Let $K\in\mathbb{N}$ and $\tilde{T}_{\left[0,K\right]}$
the projector on $\mathcal{F}_{\left[0,K\right]}$. For $k<K$ we
have
\begin{equation}
\left[\mathrm{Op}_{\Sigma}\left(\iota_{s}\tilde{T}_{\left[0,K\right]}\right),\mathrm{Op}_{\Sigma}\left(\left(u\varodot\right)\tilde{T}_{\left[0,K\right]}\right)\right]\mathrm{Op}_{\Sigma}\left(\tilde{T}_{k}\right)\approx\mathrm{Op}_{\Sigma}\left(\boldsymbol{\omega}\left(.\right)\left(d\mathcal{A}\right)\left(s,u\right)\tilde{T}_{k}\right)\label{eq:weyl_algebra_quantized}
\end{equation}
and for any $t\in\mathbb{R}$, 
\begin{equation}
e^{tX_{F}}\mathrm{Op}_{\Sigma}\left(\iota_{s}\tilde{T}_{k}\right)\approx\mathrm{Op}_{\Sigma}\left(\iota_{d\phi^{t}s}\tilde{T}_{k}\right)e^{tX_{F}}\label{eq:commut-1-1-2}
\end{equation}
\begin{equation}
e^{tX_{F}}\mathrm{Op}_{\Sigma}\left(\left(u\varodot\right)\tilde{T}_{k}\right)\approx\mathrm{Op}_{\Sigma}\left(\left(\left(d\phi^{t}u\right)\varodot\right)\tilde{T}_{k}\right)e^{tX_{F}}\label{eq:commut-1-1-1-1}
\end{equation}
\end{thm}

\end{cBoxB}

The proof is obtained in Section \ref{sec:Horocycle-operators}.
\begin{rem}
To compare (\ref{eq:weyl_algebra_quantized}) with (\ref{eq:sl2R}),
notice that $\boldsymbol{\omega}\left(.\right)$ is the principal
symbol of $X$, see \cite[prop.4.42]{faure_tsujii_Ruelle_resonances_density_2016}.
\end{rem}

\subsection{Organization of the paper}

\paragraph{First part, setup of a symbolic calculus for contact Anosov vector
field. }

The first part of paper is organized by starting with a general model
and then adding more and more assumptions to arrive at the model we
are interested in and get more specific results.

In section \ref{sec:Micro-local-analysis-of} we present the micro-local
analysis of a general non vanishing vector field $X$ on a manifold
$M$. The strategy is to lift the analysis of the pullback operator
$e^{tX}$ on the cotangent bundle $T^{*}M$ using the wave-packet
transform $\mathcal{T}:C^{\infty}\left(M\right)\rightarrow\mathcal{S}\left(T^{*}M\right)$.
In the high frequency limit (semi-classical limit), the linearization
of the flow $\tilde{\phi}^{t}$ induced on $T^{*}M$, gives a good
description of the pullback operator $e^{tX}$ in term of the differential
map $d\tilde{\phi}^{t}$ on $TT^{*}M$. This is obtained in Theorem
\ref{thm:propagation_singularities_2}.

In section \ref{sec:Micro-local-analysis-of-1} we assume furthermore
that $X$ is a \href{https://en.wikipedia.org/wiki/Reeb_vector_field}{Reeb vector field},
i.e. defined from smooth contact one-form $\mathcal{A}$ on $M$.
The purpose of this section is to obtain a precise description of
the dynamics $e^{tX}$ not on all $T^{*}M$ but only in a vicinity
of the symplectization subset $\Sigma:=\left(\mathbb{R}\backslash\left\{ 0\right\} \right)\mathcal{A}\subset T^{*}M$.
This subset $\Sigma$ is an invariant set for the dynamics $\tilde{\phi}^{t}$
induced on $T^{*}M$. Our description of the operator $e^{tX}$ in
Theorem \ref{Thm:For-any-} uses the differential $d\tilde{\phi}^{t}$
on the vector bundle $TT^{*}M$ restricted to the base $\Sigma\subset T^{*}M$
and denoted $T_{\Sigma}T^{*}M$. We get an alternative description
in Theorem \ref{Thm:approx_exptX_from_N} in terms of the normal bundle
$N:=\left(T\Sigma\right)^{\perp}\rightarrow\Sigma$. Both descriptions
will be used later.

In section \ref{sec:Micro-local-analysis-of-1-1} we assume furthermore
that $X$ is a contact Anosov vector field on $M,$ the main subject
of this paper. First, from this assumption, we have that the outside
of the set $\Sigma$ is negligible for the description of $e^{tX}$
for large time. This is Theorem \ref{thm:decay} and obtained using
escape functions, anisotropic Sobolev norm, results from previous
papers. Secondly, the normal bundle over $\Sigma$ splits as $N=N_{s}\oplus N_{u}$
with stable and unstable invariant sub-bundles. It is there that we
introduce the vector bundles $\mathcal{F}_{k}\left(N_{s}\right)=\left|\mathrm{det}N_{s}\right|^{-1/2}\otimes\mathrm{Pol}_{k}\left(N_{s}\right)\otimes F\rightarrow\Sigma$
for every $k\in\mathbb{N}$. In Theorem \ref{Thm:approx_exptX_from_Ns}
we obtain a description of the operator $e^{tX}$ in terms of the
dynamics induced on these bundles $\mathcal{F}_{k}$. We show that
this description is a ``geometric quantization framework'' with
some adapted symbolic calculus. We introduce symbols in Definition
\ref{def:symbols} and quantization in Definition \ref{def:The-quantization-of-1}
giving PDO and FIO. The rest of the section is to obtain classical
and useful results in symbolic calculus, as continuity Theorem \ref{thm:continuity_thm},
composition Theorem \ref{thm:Composition-formula-For}, Egorov Theorem
\ref{cor:As-particular-cases} Trace formula \ref{thm:Take-.-There}.

\paragraph{Second part, some consequences.}

In the second part of the paper we use the setup of the first part
to deduce some specific spectral properties of a contact Anosov vector
field $X$. In section \ref{subsec:Band-spectrum-of}, we show how
to use these properties to deduce the band structure of the Ruelle
spectrum of the vector field $X$. In section \ref{subsec:Proof-of-Theorem-2},
we deduce the Theorem \ref{thm:1_emergence_of_QM} about emergence
of quantum dynamics. Section \ref{sec:Proof-of-Theorem_Weyl} gives
the proof of the Weyl law and section \ref{sec:Proof-of-Theorem-accumulation}
explains the concentration in narrower bands. Section \ref{sec:Horocycle-operators}
gives definitions and proofs of Theorem \ref{thm:Quantization-of-()}
for approximate horocycle operators.

\paragraph{Appendices:}

Appendix \ref{sec:General-notations-used} gives some convention of
notations used in this paper. Appendix \ref{sec:More-informations-on}
contains additional comments about flows. Appendices \ref{sec:Bargmann-transform-and}
and \ref{sec:Linear-expanding-maps} are quite important and sustain
the main argument used in this paper: for the analysis of the differential
$d\tilde{\phi}^{t}$ on $T_{\Sigma}T^{*}M$, we use Bargmann transform,
i.e. wave-packet transform on an Euclidean vector space with metaplectic
operators $\mathrm{Op}\left(d\tilde{\phi}^{t}\right)$ obtained by
quantization of the bundle map of linear symplectic maps $d\tilde{\phi}^{t}:TT^{*}M\rightarrow TT^{*}M$.
All the details of this are given in appendix \ref{sec:Bargmann-transform-and}.
We also use results for linear expanding maps given in appendix \ref{sec:Linear-expanding-maps}.
\begin{acknowledgement*}
F. Faure acknowledges Claude Gignoux, Colin Guillarmou, Victor Maucout,
Malik Mezzadri, Stéphane Nonnenmacher for their support, for interesting
and motivating discussions. F. Faure acknowledges M.S.R.I. and organizers
of the \href{https://www.msri.org/programs/315}{micro-local semester 2020}
where a part of this work has been developed. F. Faure acknowledges
partially supported by French \href{https://www.math.sciences.univ-nantes.fr/~riviere-g/ANR-adyct.html}{ANR Adyct},
Grant Number ANR-20-CE40-0017 during this work. M. Tsujii acknowledges
partially supported by JSPS KAKENHI Grant Number 15H03627 and 22340035
during this work.
\end{acknowledgement*}

\section{\label{sec:Micro-local-analysis-of}Micro-local analysis of a non
vanishing vector field $X$}

In this section we consider a general smooth and non vanishing vector
field $X$ (not necessarily Anosov) on a smooth closed manifold $M$
(not necessarily contact). We review some results given in \cite{faure_tsujii_Ruelle_resonances_density_2016}.
As in (\ref{eq:Tau}), we define the wave packet transform that is
a linear operator $\mathcal{T}:C^{\infty}\left(M\right)\rightarrow\mathcal{S}\left(T^{*}M\right)$
and an isometry from\footnote{Here $dm$ is an arbitrary smooth measure on $M$ and $d\rho$ the
canonical Liouville measure on $T^{*}M$.} $L^{2}\left(M;dm\right)$ to $L^{2}\left(T^{*}M;\frac{d\rho}{\left(2\pi\right)^{\mathrm{dimM}}}\right)$,
i.e. satisfies $\mathcal{T}^{\dagger}\mathcal{T}=\mathrm{Id}$ where
$\mathcal{T}^{\dagger}$ is the $L^{2}$-adjoint operator. Then we
recall the Theorem of propagation of singularities \ref{thm:Microlocality-of-the_TO},
saying that the Schwartz kernel of the lifted pull back operator $\mathcal{T}e^{tX}\mathcal{T}^{\dagger}$
on $\mathcal{S}\left(T^{*}M\right)$, with $t\in\mathbb{R}$, decays
very fast outside the graph of the flow $\tilde{\phi}^{t}:=\left(d\phi^{t}\right)^{*}:T^{*}M\rightarrow T^{*}M$
acting on the cotangent space $T^{*}M$. Moreover by linearization,
we give in Theorem \ref{thm:propagation_singularities_2} an approximate
description of this Schwartz kernel in the neighborhood of the graph
of $\tilde{\phi}^{t}$.

\subsection{\label{subsec:Anosov_def}Vector field $X$ and pull back operator
$e^{tX}$}

Let $M$ be a $C^{\infty}$ closed connected manifold. Let $X\in C^{\infty}\left(M;TM\right)$
be a $C^{\infty}$ vector field on $M$, considered as a first order
differential operator acting on smooth functions $C^{\infty}\left(M\right)$,
i.e. in local coordinates $y=\left(y_{1},y_{2},\ldots,y_{\mathrm{dim}M}\right)$,
$X=\sum_{j=1}^{\mathrm{dim}M}X_{j}\left(y\right)\frac{\partial}{\partial y_{j}}.$
For $t\in\mathbb{R}$, let
\begin{equation}
\phi^{t}:M\rightarrow M\label{eq:flow_map}
\end{equation}
be the $C^{\infty}$ flow defined by $\frac{d\left(u\circ\phi^{t}\right)}{dt}=Xu$,
$\forall u\in C^{\infty}\left(M\right)$, i.e.
\begin{equation}
e^{tX}u:=u\circ\phi^{t}=\left(\phi^{t}\right)^{\circ}u,\label{eq:def_exptX}
\end{equation}
where $\left(\phi^{t}\right)^{\circ}=e^{tX}$ denotes the \textbf{\href{https://en.wikipedia.org/wiki/Pullback_(differential_geometry)}{pull back operator}}
acting on functions (following notations in Appendix \ref{sec:General-notations-used}).
See Appendix \ref{subsec:Transfer-operator-1} for additional remarks
about the transfer operator.

\subsection{\label{subsec:Transfer-operator}More general pull back operator
$e^{tX_{F}}$ with generator $X_{F}$}

It will be interesting to consider the following more general situation.
Let $\pi:F\rightarrow M$ be a smooth complex \href{https://en.wikipedia.org/wiki/Vector_bundle}{vector bundle}
of finite rank over $M$. Let $X_{F}$ a \textbf{\href{https://en.wikipedia.org/wiki/Derivation_(differential_algebra)}{derivation}
over $X$}, i.e. $X_{F}$ is a first order differential operator acting
on \href{https://en.wikipedia.org/wiki/Vector_bundle\#Sections_and_locally_free_sheaves}{sections}
of $F$
\begin{equation}
X_{F}:C^{\infty}\left(M;F\right)\rightarrow C^{\infty}\left(M;F\right)\label{eq:def_X_F}
\end{equation}
satisfying the Leibniz rule
\begin{equation}
X_{F}\left(fu\right)=X\left(f\right)u+fX_{F}\left(u\right),\quad\forall f\in C^{\infty}\left(M;\mathbb{C}\right),\forall u\in C^{\infty}\left(M;F\right).\label{eq:def_XF}
\end{equation}

\begin{example}
If $F$ is the bundle of \href{https://en.wikipedia.org/wiki/Differential_form}{differential forms}
$F=\Lambda^{\bullet}\left(TM\right)$ (or a more general \href{https://en.wikipedia.org/wiki/Tensor_field}{tensor field}).
The \href{https://en.wikipedia.org/wiki/Lie_derivative}{Lie derivative}
$X_{F}$ (obtained from the differential $d\phi^{t}$) satisfies (\ref{eq:def_XF}),
so $X_{F}$ is a derivation over $X$.
\end{example}

~
\begin{example}
A simple but useful example is the trivial rank one bundle $F=M\times\mathbb{C}$.
Let $V\in C^{\infty}\left(M;\mathbb{C}\right)$ called the \textbf{potential}
function. The operator 
\begin{equation}
X_{F}=X+V\label{eq:XF_X_V}
\end{equation}
satisfies (\ref{eq:def_XF}), where $V$ is seen as the multiplication
operator by the function $V$. By integration we get, with $u\in C^{\infty}\left(M\right)$,
\begin{equation}
e^{tX_{F}}u=\underbrace{e^{\int_{0}^{t}V\circ\phi^{s}ds}}_{\mathrm{amplitude}}\underbrace{u\circ\phi^{t}}_{\mathrm{transport}}.\label{eq:transport_amplitude}
\end{equation}
More remarks about general pull back operators $e^{tX_{F}}$ are given
in Section \ref{subsec:More-general-pull}.
\end{example}

\subsection{Lifted flow $\tilde{\phi}^{t}:T^{*}M\rightarrow T^{*}M$ and its
generator $\tilde{X}$}

In this section we define and describe the flow $\tilde{\phi}^{t}=\left(d\phi^{t}\right)^{*}:T^{*}M\rightarrow T^{*}M$
induced from the flow $\phi^{t}$ in (\ref{eq:flow_map}). It will
appear in Theorem \ref{thm:Microlocality-of-the_TO} below, that for
the analysis of $e^{tX}$, this flow $\tilde{\phi}^{t}$ on $T^{*}M$
plays a major role.

\subsubsection{Frequency function $\boldsymbol{\omega}$}

We define the \textbf{frequency function} $\boldsymbol{\omega}\in C^{\infty}\left(T^{*}M;\mathbb{R}\right)$
by
\begin{equation}
\forall\rho\in T^{*}M,\quad\boldsymbol{\omega}\left(\rho\right):=\rho\left(X\right).\label{eq:omega_function}
\end{equation}
that measures the oscillations of $\rho$ along the flow direction
$X$.
\begin{rem}
In terms of micro-local analysis, the function $i\boldsymbol{\omega}$
is the \textbf{\href{https://en.wikipedia.org/wiki/Symbol_of_a_differential_operator}{principal symbol}}
of the vector field $X$, see \cite[footnote page 332.]{fred_flow_09}.
\end{rem}

\subsubsection{The lifted flow $\tilde{\phi}^{t}:T^{*}M\rightarrow T^{*}M$}

For $m\in M$, we denote $d\phi^{t}:T_{m}M\rightarrow T_{\phi^{t}\left(m\right)}M$,
the differential of the flow map $\phi^{t}:M\rightarrow M$. By duality
we get a map
\begin{equation}
\tilde{\phi}^{t}:=\left(d\phi^{t}\right)^{*}\quad:T_{m}^{*}M\rightarrow T_{\phi^{-t}\left(m\right)}^{*}M.\label{eq:lifted_flow}
\end{equation}
defined by 
\begin{equation}
\langle v|\left(d\phi^{t}\right)^{*}\left(\rho\right)\rangle_{T_{m}M\times T_{m}^{*}M}=\langle d\phi^{t}v|\rho\rangle_{T_{\phi^{t}\left(m\right)}M\times T_{\phi^{t}\left(m\right)}^{*}M},\qquad\forall\rho\in T_{\phi^{t}\left(m\right)}^{*}M,\quad v\in T_{m}M,\quad m\in M.\label{eq:def_flow}
\end{equation}
Notice that $\tilde{\phi}^{t}$ is a lift of the inverse map $\phi^{-t}$
map, i.e. 
\begin{equation}
\pi\left(\tilde{\phi}^{t}\left(\rho\right)\right)=\phi^{-t}\left(\pi\left(\rho\right)\right),\label{eq:project_phi}
\end{equation}
with the bundle projection map $\pi:T^{*}M\rightarrow M$.

\subsubsection{The canonical Liouville one form $\theta$ and the symplectic two
form $\Omega$}

Let $\theta$ be the \href{https://en.wikipedia.org/wiki/Tautological_one-form}{canonical Liouville one form}
on the vector bundle $\pi:T^{*}M\rightarrow M$ \cite[p.90]{mac_duff_98}\cite[p.7]{da_silva_01}\cite[p.56]{duistermaat_FIO_96}:
at any point $\rho\in T^{*}M$, $\theta_{\rho}$ is defined by 
\begin{equation}
\forall V\in T_{\rho}T^{*}M,\quad\theta_{\rho}\left(V\right):=\rho\left(d\pi\left(V\right)\right).\label{eq:def_theta}
\end{equation}
Using the local coordinates $\left(y_{j}\right)_{j=1,\ldots,\mathrm{dim}M}$
on $M$ and dual coordinates $\left(\eta_{j}\right)_{j}$ on $T^{*}M$,
this gives
\begin{equation}
\theta=\sum_{j=1}^{\mathrm{dim}M}\eta_{j}dy_{j}.\label{eq:canonical_Liouville}
\end{equation}
Let $\Omega$ be the canonical symplectic form
\begin{equation}
\Omega:=d\theta=\sum_{j=1}^{\mathrm{dim}M}d\eta_{j}\wedge dy_{j},\label{eq:def_Omega-1}
\end{equation}
that gives a bundle isomorphism $\check{\Omega}:TT^{*}M\rightarrow T^{*}T^{*}M$
defined by
\[
\forall V\in TT^{*}M,\quad\check{\Omega}\left(V\right):=\Omega\left(V,.\right)\in T^{*}T^{*}M.
\]
Using notation in section \ref{sec:General-notations-used} for the
pullback, the Definition (\ref{eq:def_theta}) of $\theta$ on $T^{*}M$
is equivalent to the fact that for any smooth one form $\alpha\in C^{\infty}\left(M;T^{*}M\right)$,
considered as a map $\alpha:M\rightarrow T^{*}M$, considering its
differential $d\alpha:TM\rightarrow TT^{*}M$, we have
\begin{equation}
\left(d\alpha\right)^{\circ}\theta=\alpha,\label{eq:alpha_theta}
\end{equation}
and $\mathrm{Ker}\theta=\mathrm{Ker}\left(d\pi\right)$ away from
the zero section. Taking the differential we get
\begin{equation}
\left(d\alpha\right)^{\circ}\Omega=d\alpha.\label{eq:dalpha_Omega}
\end{equation}
We define the \textbf{Euler (or Liouville) vector field} on $T^{*}M$
\begin{equation}
\mathcal{E}:=\check{\Omega}^{-1}\left(\theta\right).\label{eq:def_Euler}
\end{equation}
In coordinates, $\mathcal{E}=\sum_{j=1}^{\mathrm{dim}M}\eta_{j}\frac{\partial}{\partial\eta_{j}}$.
Using \href{https://en.wikipedia.org/wiki/Lie_derivative}{Cartan formula}
we get for the Lie derivatives
\[
\mathcal{L}_{\mathcal{E}}\theta=\theta,\quad\mathcal{L}_{\mathcal{E}}\Omega=\Omega.
\]

\subsubsection{\label{subsec:The-vector-field}The vector field $\tilde{X}$ on
$T^{*}M$}

Let us write $\tilde{X}$ for the vector field on $T^{*}M$ that generates
the flow $\tilde{\phi}^{t}$ (\ref{eq:lifted_flow}). Notice from
(\ref{eq:project_phi}) that
\begin{equation}
\left(d\pi\right)\left(\tilde{X}\right)=-X.\label{eq:dp_Xtilde}
\end{equation}

\begin{cBoxB}{}
\begin{lem}
\label{lem:one_form_theta}\cite[Thm 2.124 p.112, in coordinates]{gallot1990riemannian},\cite[prop. 1.21 p.15]{paternain2012geodesic}.
The Lie derivative of $\theta$ by the vector field $\tilde{X}$ vanishes:
\begin{equation}
\mathcal{L}_{\tilde{X}}\theta=0,\label{eq:derivative_one_theta}
\end{equation}
and $\tilde{X}$ is determined by
\begin{equation}
\tilde{X}=\check{\Omega}^{-1}\left(d\boldsymbol{\omega}\right),\label{eq:eq5}
\end{equation}
i.e. $\tilde{X}$ is the Hamiltonian vector field on $\left(T^{*}M,\Omega\right)$,
with Hamiltonian function $\boldsymbol{\omega}$, Eq.(\ref{eq:omega_function}).
\end{lem}

\end{cBoxB}

\begin{proof}
\cite[Thm 2.124 p.112, for a proof using  coordinates]{gallot1990riemannian},\cite[Prop. 1.21 p.15]{paternain2012geodesic}
Let $\rho\in T^{*}M$ and $V\in T_{\rho}\left(T^{*}M\right)$. For
any $t\in\mathbb{R}$ we have
\begin{align*}
\left(\left(d\tilde{\phi}^{t}\right)^{*}\theta_{\tilde{\phi}^{t}\left(\rho\right)}\right)\left(V\right) & =\theta_{\tilde{\phi}^{t}\left(\rho\right)}\left(\left(d\tilde{\phi}^{t}\right)V\right)\underset{(\ref{eq:def_theta})}{=}\tilde{\phi}^{t}\left(\rho\right)\left(d\pi\left(d\tilde{\phi}^{t}\left(V\right)\right)\right)\\
 & \eq{\ref{eq:lifted_flow}}\langle\left(d\phi^{t}\right)^{*}\left(\rho\right)|d\phi^{-t}\left(d\pi\left(V\right)\right)\rangle\underset{(\ref{eq:def_flow})}{=}\rho\left(d\pi\left(V\right)\right)=\theta_{\rho}\left(V\right)
\end{align*}
\[
\Leftrightarrow\left(d\tilde{\phi}^{t}\right)^{*}\theta_{\tilde{\phi}^{t}\left(\rho\right)}=\theta_{\rho}\underset{(\ref{eq:evol_XF})}{\Leftrightarrow}e^{t\mathcal{L}_{\tilde{X}}}\theta=\theta\Leftrightarrow\mathcal{L}_{\tilde{X}}\theta=0.
\]
This gives (\ref{eq:derivative_one_theta}). Notice that
\begin{equation}
\boldsymbol{\omega}\left(\rho\right)\eq{\ref{eq:omega_function}}\rho\left(X\right)\eq{\ref{eq:dp_Xtilde}}-\rho\left(d\pi\left(\tilde{X}\right)\right)\eq{\ref{eq:def_theta}}-\theta_{\rho}\left(\tilde{X}\right)=-\iota_{\tilde{X}}\theta_{\rho}.\label{eq:omega}
\end{equation}
Hence, using \href{https://en.wikipedia.org/wiki/Lie_derivative}{Cartan formula}
$\mathcal{L}_{\tilde{X}}\theta=d\left(\iota_{\tilde{X}}\theta\right)+\iota_{\tilde{X}}d\theta$
we get
\[
d\boldsymbol{\omega}\eq{\ref{eq:omega}}-d\left(\iota_{\tilde{X}}\theta\right)=-\mathcal{L}_{\tilde{X}}\theta+\iota_{\tilde{X}}d\theta=\Omega\left(\tilde{X},.\right),
\]
that gives (\ref{eq:eq5}).
\end{proof}
\begin{rem}
As a consequence of (\ref{eq:eq5}) we have $\tilde{X}\left(\boldsymbol{\omega}\right)\eq{\ref{eq:eq5}}\Omega\left(\tilde{X},\tilde{X}\right)=0$,
i.e. for every $\omega\in\mathbb{R}$, the frequency level
\begin{equation}
\Sigma_{\omega}:=\boldsymbol{\omega}^{-1}\left(\omega\right)=\left\{ \rho\in T^{*}M,\quad X\left(\rho\right)=\omega\right\} \subset T^{*}M\label{eq:def_Sigma_omega}
\end{equation}
that is a affine sub-bundle of $T^{*}M$ is preserved by $\tilde{\phi}^{t}$.
\end{rem}

~
\begin{rem}
\label{rem:For-,-the}For $\omega\neq0$, the one form $-\frac{1}{\omega}\theta$
is a contact one form on the frequency level $\Sigma_{\omega}$ and
$\tilde{X}$ is its \href{https://en.wikipedia.org/wiki/Reeb_vector_field}{Reeb vector field},
because $\left(-\frac{1}{\omega}\theta\right)\left(\tilde{X}\right)\eq{\ref{eq:omega}}1$
and $d\left(-\frac{1}{\omega}\theta\right)\left(\tilde{X},.\right)=-\frac{1}{\omega}\Omega\left(\tilde{X},.\right)=-\frac{1}{\omega}d\boldsymbol{\omega}\underset{(\ref{eq:def_Sigma_omega})}{=}0$.
\end{rem}

\subsection{\label{subsec:Metric--on}Metric $g$ on $T^{*}M$}

Let us suppose that we have an atlas $M=\bigcup_{j=1}^{J}U_{j}$ where
$U_{j}$ are open subset on $M$ such that for $j=1,\ldots J$,
\begin{equation}
\kappa_{j}:m\in U_{j}\rightarrow y=\left(x,z\right)\in V_{j}\subset\mathbb{R}_{x}^{\mathrm{dim}M-1}\times\mathbb{R}_{z}\label{eq:kappa_j}
\end{equation}
are flow box coordinate, i.e. $\left(d\kappa_{j}\right)\left(X\right)=\frac{\partial}{\partial z}$.
On each chart $U_{j}$, let us denote $\eta=\left(\xi,\omega\right)\in\mathbb{R}^{\mathrm{dim}M-1}\times\mathbb{R}$
the dual coordinates associated to $\left(x,z\right)$. Hence $\left(y,\eta\right)=\left(x,z,\xi,\omega\right)$
are coordinates on $T^{*}\left(\mathbb{R}_{x}^{\mathrm{dim}M-1}\times\mathbb{R}_{z}\right)$.

Let us recall the definition and properties of the metric $g$ on
$T^{*}M$ introduced in \cite[section 4.1.2]{faure_tsujii_Ruelle_resonances_density_2016}
with parameters that we take here $\alpha^{\perp}=\frac{1}{2}$ and
$\alpha^{\parallel}=0$. For $s\in\mathbb{R}$, we set
\begin{equation}
\langle s\rangle:=\left(1+s^{2}\right)^{1/2}\underset{\left|s\right|\gg1}{\sim}\left|s\right|.\label{eq:def_japanese_bracket}
\end{equation}
For $\eta\in\mathbb{R}^{\mathrm{dim}M}$, let
\begin{equation}
\delta^{\perp}\left(\eta\right):=\left\langle \left|\eta\right|\right\rangle ^{-1/2}.\label{eq:def_delta}
\end{equation}
In coordinates, we define the metric $g$ on $T^{*}\mathbb{R}^{\mathrm{dim}M}=\mathbb{R}_{\varrho}^{2\mathrm{dim}M}$
at a point $\varrho=\left(y,\eta\right)\in\mathbb{R}^{\mathrm{dim}M}\times\mathbb{R}^{\mathrm{dim}M}$
by
\begin{equation}
g\left(\varrho\right):=\left(\frac{dx}{\delta^{\perp}\left(\eta\right)}\right)^{2}+\left(\frac{d\xi}{\delta^{\perp}\left(\eta\right)^{-1}}\right)^{2}+dz^{2}+d\omega^{2}.\label{eq:metric_g_tilde_in_coordinates}
\end{equation}
This gives a metric $g_{j}:=\left(d\kappa_{j}\right)^{*}g$ on $T^{*}U_{j}$.
It is shown in \cite[Lemma 4.6]{faure_tsujii_Ruelle_resonances_density_2016}
that on each intersection $U_{j}\cap U_{j'}$, the metric $g_{j}$
and $g_{j'}$ are relatively bounded by each other uniformly in $\eta$.
Each metric $g_{j}$ is compatible with the canonical symplectic form
$\Omega$ on $T^{*}M$ given in (\ref{eq:def_Omega-1}).
\begin{rem}
\label{rem:an-important-property}An important property of the metric
$g$ that will be essential in our analysis is that the size of a
unit ball projects on the coordinates $x$ as a set of size of order
$\delta^{\perp}\left(\eta\right)=\left\langle \left|\eta\right|\right\rangle ^{-1/2}$
that goes to zero as $\left|\eta\right|\rightarrow+\infty$. See \cite[Figure 2.3]{faure_tsujii_Ruelle_resonances_density_2016}.
\end{rem}

We recall the definition that $\Omega,g$ are compatible in Definition
\ref{def:Compatible}. 

\begin{cBoxB}{}
\begin{lem}
\label{lem:There-exists-a}There exists a global smooth metric $g$
on $T^{*}M$ that is also compatible with $\Omega$ and relatively
bounded with respect to every metric $g_{j}$ on any chart $T^{*}U_{j}$. 
\end{lem}

\end{cBoxB}

\begin{proof}
Let $G$ be a global metric on $T^{*}M$ defined by
\begin{equation}
G:=\sum_{j=1}^{J}\left|\mathrm{det}d\kappa_{j}\right|\left(\chi_{j}^{2}\circ\kappa_{j}\circ\pi\right)\cdot g_{j}\label{eq:global_metric}
\end{equation}
where $\pi:T^{*}M\to M$ denotes the bundle projection and $\chi_{j}\in C_{0}^{\infty}\left(\kappa_{j}\left(U_{j}\right);\mathbb{R}^{+}\right)$
form a quadratic partition of unity on $M$, i.e. $\sum_{j}\chi_{j}^{2}\circ\kappa_{j}\left|\mathrm{det}d\kappa_{j}\right|=\boldsymbol{1}$,
see \cite[Lemma 4.2]{faure_tsujii_Ruelle_resonances_density_2016}.
Recall that on each intersection $U_{j}\cap U_{j'}$, the metric $g_{j}$
and $g_{j'}$ are uniformly equivalent (i.e. relatively bounded) that
we write $g_{j}\asymp g_{j'}$. Consequently, from finitness of the
sum (\ref{eq:global_metric}), on any chart $T^{*}U_{j}$, we have
\begin{equation}
G\asymp g_{j}.\label{eq:G_equiv}
\end{equation}
We define
\begin{equation}
K:=\check{\Omega}^{-1}\check{G}\quad:TT^{*}M\rightarrow TT^{*}M,\label{eq:def_K-1}
\end{equation}
where $\check{G},\check{\Omega}:TT^{*}M\rightarrow T^{*}T^{*}M$ are
defined from $G,\Omega$ as in (\ref{eq:def_Omega_check}). We consider
the \href{https://en.wikipedia.org/wiki/Polar_decomposition}{polar decomposition}
of $K$ written $K=J\left|K\right|$ where $\left|K\right|=\sqrt{K^{\dagger_{g}}K}$
is positive definite Hermitian (bundle maps over $T^{*}M$) and $J$
is an \href{https://en.wikipedia.org/wiki/Almost_complex_manifold}{almost complex structure}
because, using that $\check{G}^{*}=\check{G}$ and $\check{\Omega}^{*}=-\check{\Omega}$
we get $K^{\dagger_{g}}\eq{\ref{eq:adjoint}}\check{G}^{-1}K^{*}\check{G}\eq{\ref{eq:def_K-1}}\check{G}^{-1}\check{G}^{*}\check{\Omega}^{*-1}\check{G}=-\check{\Omega}^{-1}\check{G}=-K$,
hence $J^{2}=\frac{K^{2}}{\left(K^{\dagger_{g}}K\right)}=-\mathrm{Id}$.
Since $g_{j}$ is compatible with $\Omega$ then $\check{g}_{j}=\check{\Omega}J_{j}$
with $J_{j}^{2}=-\mathrm{Id}$. Thus $\check{\Omega}K\eq{\ref{eq:def_K-1}}\check{G}\underset{(\ref{eq:G_equiv})}{\asymp}g_{j}=\check{\Omega}J_{j}$
gives $K\asymp J_{j}$ then 
\begin{equation}
\left|K\right|\asymp\mathrm{Id}.\label{eq:K_equiv}
\end{equation}
We define the metric $g$ by
\begin{equation}
\check{g}:=\check{\Omega}J\quad:TT^{*}M\rightarrow T^{*}T^{*}M.\label{eq:def_g-global}
\end{equation}
By construction, the global metric $g$ is smooth, compatible with
$\Omega$ and relatively bounded with respect to every $g_{j}$ on
any chart $T^{*}U_{j}$ because $\check{g}=\check{\Omega}J\underset{(\ref{eq:K_equiv})}{\asymp}\check{\Omega}J\left|K\right|=\check{\Omega}K\eq{\ref{eq:def_K-1}}\check{G}\asymp\check{g}_{j}$.
\end{proof}
We will denote $\mathrm{dist}_{g}\left(\rho',\rho\right)$ the distance
between two points $\rho,\rho'\in T^{*}M$ according to the metric
$g$. It is shown in \cite[Lemma 4.5]{faure_tsujii_Ruelle_resonances_density_2016}
that the metric $g$ is \textbf{geodesically complete}.

We will also use a property of slow variation of the metric $g$ given
in \cite[Lemma 4.12]{faure_tsujii_Ruelle_resonances_density_2016}
and recalled later in (\ref{eq:g_moderate and temperate}). Finally
we have from \cite[Lemma 4.6]{faure_tsujii_Ruelle_resonances_density_2016}
the ``Lipschitz property of $g$'', namely, for any $t\in\mathbb{R}$,
there exists a constant $C_{t}>0$ such that for any $\rho,\rho'\in T^{*}M$,
\begin{equation}
\mathrm{dist}_{g}\left(\tilde{\phi}^{t}(\rho),\tilde{\phi}^{t}(\rho')\right)\le C_{t}\mathrm{dist}_{g}\left(\rho,\rho'\right).\label{eq:invariance_g}
\end{equation}
where $\tilde{\phi}^{t}:T^{*}M\rightarrow T^{*}M$ is defined in (\ref{eq:lifted_flow}).

\subsection{\label{subsec:Wave-packet-transform}Wave-packet transform $\mathcal{T}$}

As in \cite[def. 4.23]{faure_tsujii_Ruelle_resonances_density_2016},
for every coordinate chart $U_{j}\subset M$ with $j\in\left\{ 1,\ldots,J\right\} $
and $\rho\in T^{*}U_{j}$ we define a wave packet\footnote{We refer to \cite[def. 4.23]{faure_tsujii_Ruelle_resonances_density_2016}
for the precise expression of a wave packet $\Phi_{j,\rho}$. Here
it is enough to say that on a chart $U_{j}$, for a given $\varrho=\left(y,\eta\right)\in T^{*}\mathbb{R}^{n+1}$
and in local coordinates $y'\in\mathbb{R}^{n+1}$, then for large
$\left|\eta\right|\gg1$, the function $\Phi_{j,\rho}$ is equivalent
to a \textbf{Gaussian wave packet} in ``vertical Gauge'':
\begin{equation}
\Phi_{j,\varrho}\left(y'\right)\sim a_{\varrho}\chi\left(y'-y\right)\exp\left(i\eta.\left(y'-y\right)-\left\Vert y'-y\right\Vert _{g_{\varrho}}^{2}\right)\label{eq:wave_packet_1}
\end{equation}
with $y'=\left(x',z'\right)\in\mathbb{R}^{n+1}$ and where $\chi\in C_{0}^{\infty}\left(\mathbb{R}^{n+1}\right)$
is some cut-off function with $\chi\equiv1$ near the origin and $a_{\varrho}>0$
is such that $\left\Vert \Phi_{j,\varrho}\right\Vert _{L^{2}\left(\mathbb{R}^{n+1}\right)}=1$
and $\left\Vert y'-y\right\Vert _{g_{\varrho}}^{2}=\left|z'-z\right|^{2}+\left|\frac{\left(x'-x\right)}{\delta^{\perp}\left(\eta\right)}\right|^{2}$
is obtained from the metric (\ref{eq:metric_g_tilde_in_coordinates}).} $\Phi_{j,\rho}\in C^{\infty}\left(M;\mathbb{C}\right)$ and the \textbf{``wave
packet transform''}
\begin{equation}
\mathcal{T}:\begin{cases}
C^{\infty}\left(M;\mathbb{C}\right) & \rightarrow\mathcal{S}\left(T^{*}M;\mathbb{C}^{J}\right)\\
u & \rightarrow\left(\langle\Phi_{j,\rho}|u\rangle_{L^{2}\left(M\right)}\right)_{\rho\in T^{*}M,j\in\left\{ 1\ldots J\right\} }
\end{cases},\label{eq:def_T}
\end{equation}
that satisfies the ''\textbf{resolution of identity on $C^{\infty}\left(M\right)$}''
\cite[prop. 4.24]{faure_tsujii_Ruelle_resonances_density_2016}
\begin{align}
\mathrm{Id}_{/C^{\infty}\left(M\right)} & =\mathcal{T}^{\dagger}\mathcal{T}\label{eq:resol_ident_Pi_rho}
\end{align}
where 
\[
\mathcal{T}^{\dagger}:L^{2}\left(T^{*}M\otimes\mathbb{C}^{J};\frac{d\rho}{\left(2\pi\right)^{\mathrm{dim}M}}\right)\rightarrow L^{2}\left(M;dm\right)
\]
is the $L^{2}$-adjoint of $\mathcal{T}$ where $d\rho\equiv dyd\eta$
is the canonical measure on $T^{*}M$ and $dm$ an arbitray smooth
measure on $M$. We will introduce the \textbf{wave packet projector
}onto the image of $\mathcal{T}$, see \cite[prop. 4.27]{faure_tsujii_Ruelle_resonances_density_2016}:
\begin{equation}
\mathcal{P}:=\mathcal{T}\mathcal{T}^{\dagger}.\label{eq:def_P_wave_packet_projector}
\end{equation}

\begin{rem}
Below for simplicity, we will ignore $\otimes\mathbb{C}^{J}$ in the
notation, i.e. ignore chart components, though keeping it in mind.
The operator $\mathcal{T}^{\dagger}:\mathcal{S}'\left(T^{*}M\right)\rightarrow C^{\infty}\left(M\right)$
is smoothing. 
\end{rem}

\subsection{Description of the operator $e^{tX}$ with the bundle $TT^{*}M$}
\begin{rem}
From now on we will often use the knowledge about Bargmann transform
that is summarized in Section \ref{sec:Bargmann-transform-and} of
the appendix and we ask the readers to check the notation and their
knowledge. We will often refer to this Section \ref{sec:Bargmann-transform-and}.
\end{rem}

We first introduce a definition that will be used to express that
some operator $R$ is ``under control'' or ``negligible'' in our
analysis. It will mean that the Schwartz kernel of $\mathcal{T}R\mathcal{T}^{\dagger}$
decays very fast outside the graph of $\tilde{\phi}^{t}$ defined
in (\ref{eq:lifted_flow}) and moreover that on the graph, the Schwartz
kernel is bounded by the polynomial growth $\left\langle \left|\rho\right|\right\rangle ^{m}$.

\begin{cBoxA}{}
\begin{defn}
\label{def:Let--be}Let $t\in\mathbb{R}$ and $m\in\mathbb{R}$. We
define $\Psi_{\tilde{\phi}^{t}}^{m}$ as the set of operators $R:C^{\infty}\left(M\right)\rightarrow\mathcal{S}'\left(M\right)$
such that for any $N>0$, there exists a constant $C_{N,t}>0$ such
that for any $\rho,\rho'\in T^{*}M$
\begin{align}
\left|\langle\delta_{\rho'}|\mathcal{T}R\mathcal{T}^{\dagger}\delta_{\rho}\rangle_{L^{2}\left(T^{*}M\right)}\right| & \leq C_{N,t}\left\langle \mathrm{dist}_{g}\left(\rho',\tilde{\phi}^{t}\left(\rho\right)\right)\right\rangle ^{-N}\left\langle \left|\rho\right|\right\rangle ^{m}.\label{eq:microl_estimate-2}
\end{align}
\end{defn}

\end{cBoxA}

A useful consequence of (\ref{eq:microl_estimate-2}) is the following
Lemma that uses the truncation operator $\mathrm{Op}\left(\chi_{\omega}\right)$
defined in (\ref{eq:def_Xi_low}).

\begin{cBoxB}{}
\begin{lem}
\label{Lem:A-useful-consequence}If $R\in\Psi_{\tilde{\phi}^{t}}^{m}$
with some $m<0$, then $\exists C_{t}>0$,
\begin{equation}
\left\Vert R\left(\mathrm{Id}-\mathrm{Op}\left(\chi_{\omega}\right)\right)\right\Vert _{L^{2}\left(M\right)}\leq C_{t}\omega^{m}\underset{\omega\rightarrow+\infty}{\rightarrow}0.\label{eq:result_of_Shur}
\end{equation}
\end{lem}

\end{cBoxB}

\begin{proof}
We follow the same notations and techniques as in \cite[section 4.2]{faure_tsujii_Ruelle_resonances_density_2016}.
We have that
\begin{align*}
\left\Vert R\left(\mathrm{Id}-\mathrm{Op}\left(\chi_{\omega}\right)\right)\right\Vert _{L^{2}\left(M\right)} & \eq{\ref{eq:resol_ident_Pi_rho}}\left\Vert \mathcal{T}R\left(\mathrm{Id}-\mathrm{Op}\left(\chi_{\omega}\right)\right)\mathcal{T}^{\dagger}\right\Vert _{L^{2}\left(T^{*}M\right)}\\
 & \eq{\ref{eq:def_P_wave_packet_projector}}\left\Vert \mathcal{T}R\mathcal{T}^{\dagger}\left(\mathcal{P}-\mathcal{P}\chi_{\omega}\mathcal{P}\right)\right\Vert _{L^{2}\left(T^{*}M\right)}.
\end{align*}
From \cite[Prop. 4.27]{faure_tsujii_Ruelle_resonances_density_2016},
we have $\forall N\geq0,\exists C_{N}>0,\forall\rho,\rho'\in T^{*}M$,
\begin{equation}
\left|\langle\delta_{\rho'}|\mathcal{P}\delta_{\rho}\rangle_{L^{2}\left(T^{*}M\right)}\right|\leq C_{N}\left\langle \mathrm{dist}_{g}\left(\rho',\rho\right)\right\rangle ^{-N},\label{eq:estimate_Bergman_kernel}
\end{equation}
i.e. the Schwartz kernel of \emph{$\mathcal{P}$ }decays fast outside
the diagonal and is uniformly bounded on the diagonal. Using definition
of $\chi_{\omega}$ in (\ref{eq:def_Chi_omega}) we have (the exponent
$1/2$ comes from the metric $g$ in (\ref{eq:metric_g_tilde_in_coordinates})),
$\forall N\geq0,\exists C_{N}>0,\forall\rho,\rho'\in T^{*}M$,
\[
\left|\langle\delta_{\rho'}|\mathcal{P}\chi_{\omega}\mathcal{P}\delta_{\rho}\rangle_{L^{2}\left(T^{*}M\right)}\right|\leq C_{N}\left\langle \mathrm{dist}_{g}\left(\rho',\rho\right)\right\rangle ^{-N}\left\langle \mathrm{max}\left(0,\omega^{-1/2}\left(\omega-\left\Vert \rho\right\Vert _{g_{M}}\right)\right)\right\rangle ^{-N}.
\]
Together with (\ref{eq:microl_estimate-2}) we deduce that $\forall N\geq0,\exists C_{N,t}>0,\forall\rho,\rho'\in T^{*}M$,
\begin{align*}
\left|\langle\delta_{\rho'}|\mathcal{T}R\left(\mathrm{Id}-\mathrm{Op}\left(\chi_{\omega}\right)\right)\mathcal{T}^{\dagger}\delta_{\rho}\rangle_{L^{2}\left(T^{*}M\right)}\right| & \leq C_{N,t}\left\langle \mathrm{dist}_{g}\left(\rho',\tilde{\phi}^{t}\left(\rho\right)\right)\right\rangle ^{-N}\left\langle \left|\rho\right|\right\rangle ^{m}\\
 & \qquad\left\langle \mathrm{max}\left(0,\omega^{-1/2}\left(\omega-\left\Vert \rho\right\Vert _{g_{M}}\right)\right)\right\rangle ^{-N}.
\end{align*}
Using Shur Lemma \cite[Lemma 4.38]{faure_tsujii_Ruelle_resonances_density_2016},
that estimates the $L^{2}$ norm operator from the Schwartz kernel,
we deduce that if $m<0$, then $\exists C_{t}>0$,
\[
\left\Vert R\left(\mathrm{Id}-\mathrm{Op}\left(\chi_{\omega}\right)\right)\right\Vert _{L^{2}\left(M\right)}\leq C_{t}\omega^{m}\underset{\omega\rightarrow+\infty}{\rightarrow}0.
\]
\end{proof}

\subsubsection{Propagation of singularities. First estimate for the wave front set.}

The next theorem is similar to the description of evolution of the
\textbf{\href{https://en.wikipedia.org/wiki/Wave_front_set}{wave-front set}}
in micro-local analysis as in \cite[Prop. 9.5, page 29.]{taylor_tome2}
and often called ``\href{https://fr.wikipedia.org/wiki/Th\%C3\%A9or\%C3\%A8me_de_propagation_des_singularit\%C3\%A9s}{propagation of singularities}''.
It gives some first description of the pull back operator $e^{tX}$
that we will need to improve later. Recall the notation $\Psi_{\tilde{\phi}^{t}}^{0}$
from definition \ref{def:Let--be}. 

\begin{cBoxB}{}
\begin{thm}[\textbf{Propagation of singularities}]
\label{thm:Microlocality-of-the_TO}\cite[Thm 4.51]{faure_tsujii_Ruelle_resonances_density_2016}.
For each $t\in\mathbb{R}$,
\begin{equation}
e^{tX}\in\Psi_{\tilde{\phi}^{t}}^{0}.\label{eq:Psi_tilde}
\end{equation}
\end{thm}

\end{cBoxB}

\begin{rem}
If $\rho'$ and $\tilde{\phi}^{t}\left(\rho\right)$ are in bounded
distance from each other, Theorem \ref{thm:Microlocality-of-the_TO}
says nothing very informative for the value of the Schwartz kernel
$\langle\delta_{\rho'}|\mathcal{T}e^{tX}\mathcal{T}^{\dagger}\delta_{\rho}\rangle_{L^{2}\left(T^{*}M\right)}$.
The next Theorem \ref{thm:propagation_singularities_2} below will
complete this lack of information by giving an approximate expression
for the Schwartz kernel of $\mathcal{T}e^{tX}\mathcal{T}^{\dagger}$
in the neighborhood of the graph of $\tilde{\phi}^{t}$, and this
will be useful later to get Theorem \ref{Thm:approx_exptX_from_N}.
We first need to introduce some operators.
\end{rem}

\subsubsection{\label{subsec:Definition-of-some}Definition of some operators}

Below we view the (usual) \href{https://en.wikipedia.org/wiki/Schwartz_space}{space of Schwartz functions}
$\mathcal{S}\left(TT^{*}M\right)$ as the set of functions $u:\rho\in T^{*}M\rightarrow u_{\rho}\in\mathcal{S}\left(T_{\rho}T^{*}M\right)$,
i.e. we have a natural identification, with induced topology,
\begin{equation}
\mathcal{S}\left(TT^{*}M\right)\equiv\mathcal{S}\left(\rho\in T^{*}M;\mathcal{S}\left(T_{\rho}T^{*}M\right)\right).\label{eq:def_S-1}
\end{equation}

\paragraph{The operator $\widetilde{\exp^{\circ}}$.}

The Riemannian manifold $\left(T^{*}M,g\right)$ is geodesically complete
\cite[Lemma 4.5]{faure_tsujii_Ruelle_resonances_density_2016}. Let
us write
\begin{equation}
\exp:TT^{*}M\rightarrow T^{*}M\label{eq:def_exp}
\end{equation}
for the \href{https://en.wikipedia.org/wiki/Exponential_map_(Riemannian_geometry)}{exponential map}
associated to the metric $g$ on $T^{*}M$ of Lemma \ref{lem:There-exists-a}.
From this map we have the pull back operator $\exp^{\circ}:\mathcal{S}\left(T^{*}M\right)\rightarrow\mathcal{S}\left(TT^{*}M\right)$
but we consider instead the \textbf{twisted pull back operator }defined
as follows
\begin{equation}
\widetilde{\exp^{\circ}}:=e^{i\varphi}\exp^{\circ}\quad:C^{\infty}\left(T^{*}M\right)\rightarrow C^{\infty}\left(TT^{*}M\right),\label{eq:def_twisted_pull_back}
\end{equation}
with a phase function $\varphi:TT^{*}M\rightarrow\mathbb{R}$ that
corresponds to a change of trivialization to pass from the vertical
gauge given by the wave-packet transform (\ref{eq:wave_packet_1})
to the local radial gauge as explained in section \ref{subsec:Bargmann-Transform}.
Explicitly, this phase function $\varphi$ is given by, for $\rho\in T^{*}M$,
$v\in T_{\rho}T^{*}M$,
\begin{equation}
\varphi\left(v\right):=-\theta\left(v\right)-\frac{1}{2}D^{2}f\left(v\right)\label{eq:def_phase_phi}
\end{equation}
with the Liouville one form $\theta\left(v\right)=\rho\left(d\pi\left(v\right)\right)$
given in (\ref{eq:def_theta}) and in the second part the function
$f:=\theta\circ\left(d\exp\right)_{\rho}-\theta_{\rho}:T_{\rho}T^{*}M\rightarrow\mathbb{R}$
made with the differential $\left(d\exp\right)_{\rho}:T_{\rho}T^{*}M\rightarrow TT^{*}M$
(with the canonical identification of vector spaces $TT_{\rho}T^{*}M\equiv T_{\rho}T^{*}M$).
One has $f\left(0\right)=0$ and $Df\left(0\right)=0$, so the quadratic
form $D^{2}f:T_{\rho}T^{*}M\rightarrow\mathbb{R}$ called \href{https://en.wikipedia.org/wiki/Hessian_matrix\#Generalizations_to_Riemannian_manifolds}{Hessian},
is well defined.
\begin{rem}
Geometrically the phase $\varphi$ reflects the existence of a (trivial)
complex line bundle $L$ over $T^{*}M$ with connection given by the
canonical Liouville one form with respect to a global trivialization.
This line bundle is usually called the prequantum line bundle. As
mentioned before in Remark \ref{rem:The-symplectic-form}, we can
ignore this line bundle since it is trivial so we can use a global
section, giving phases like in (\ref{eq:def_twisted_pull_back}).To
understand better the phase $\varphi$ in (\ref{eq:def_twisted_pull_back})
it may be useful to express it in local coordinates. Let $y\in\mathbb{R}^{\mathrm{dim}M}$
be local coordinates on $M$ as in (\ref{eq:kappa_j}) and $\eta\in\mathbb{R}^{\mathrm{dim}M}$
be dual coordinates on $T^{*}M$. For a (local) flat metric $g$ on
$T^{*}M$ one has $\rho'=\exp_{\rho}\left(v\right)=\left(y_{\rho}+y_{v},\eta_{\rho}+\eta_{v}\right)$
hence $v'=\left(d\exp\right)_{\rho}\left(v\right)=\left(y_{v},\eta_{v}\right)$.
The Liouville one-form is $\theta_{\rho}\left(v'\right)=\eta_{\rho}y_{v}$
and $\theta_{\rho'}\left(v'\right)=\eta_{\rho'}y_{v}=\left(\eta_{\rho}+\eta_{v}\right)y_{v}$
giving the function $f\left(v\right)=\eta_{v}y_{v}$ which is quadratic,
i.e. $D^{2}f=f$ (due to absence of non linearities in this simple
setting of flat metric). The first term in (\ref{eq:def_phase_phi})
is $e^{-i\theta_{\rho}\left(v\right)}=e^{-i\eta_{\rho}y_{v}}$ and
is there to remove high oscillations coming from wave-packet transform
at point $\rho\in T^{*}M$. The second term in (\ref{eq:def_phase_phi})
is $e^{-\frac{i}{2}D^{2}f\left(v\right)}=e^{-\frac{i}{2}\eta_{v}y_{v}}$
and is used to pass from vertical gauge to radial gauge in accordance
with formula (\ref{eq:Gaussian_wave_packet-V-R}). In summary, for
Euclidean metric the phase function is
\begin{equation}
\varphi_{\rho}\left(y_{v},\eta_{v}\right)=-\eta_{\rho}y_{v}-\frac{1}{2}\eta_{v}y_{v}.\label{eq:phase_phi}
\end{equation}
\end{rem}

\paragraph{Truncation operator $\chi^{\lambda}$ .}

Let $0<\lambda<1/2$ that is fixed\footnote{From the result (\ref{eq:Rt}) and estimate (\ref{eq:result_of_Shur}),
we see that $\lambda$ close to zero gives better results, but in
this paper it is enough to take $\lambda=1/4$.} in this paper. The operator $\chi^{\lambda}$ truncates functions
in a ball of radius\footnote{Here and after, $\left|\rho\right|=\left\Vert \rho\right\Vert _{g_{M}}$
is measured with respect to an arbitrary metric $g_{M}$ on $M$.} $\left\langle \left\Vert \rho\right\Vert _{g_{M}}\right\rangle ^{\lambda/2}$
centered on the zero section of $TT^{*}M$:

\begin{equation}
\chi^{\lambda}:=\boldsymbol{1}_{\left\{ \left\Vert .\right\Vert _{g_{\rho}}\leq\left\langle \left\Vert \rho\right\Vert _{g_{M}}\right\rangle ^{\lambda/2}\right\} }:\mathcal{S}\left(TT^{*}M\right)\rightarrow\mathcal{S}'\left(TT^{*}M\right),\label{eq:def_Chi_sigma-2}
\end{equation}
i.e. for $u\in\mathcal{S}\left(TT^{*}M\right)$, $\rho\in T^{*}M$,
$v\in T_{\rho}T^{*}M$,
\begin{equation}
\left(\chi^{\lambda}u\right)\left(v\right)=\begin{cases}
u\left(v\right) & \text{ if }\left\Vert v\right\Vert _{g_{\rho}}\leq\left\langle \left\Vert \rho\right\Vert _{g_{M}}\right\rangle ^{\lambda/2}\\
0 & \text{ otherwise.}
\end{cases}.\label{eq:def_Chi_sigma}
\end{equation}

\paragraph{The restriction operator $r_{0}$.}

Let
\[
r_{0}:\mathcal{S}\left(TT^{*}M\right)\rightarrow\mathcal{S}\left(T^{*}M\right)
\]
be the map that restricts a function $u\in\mathcal{S}\left(TT^{*}M\right)$
(written $u\left(\rho,v\right)$, $\rho\in T^{*}M$, $v\in T_{\rho}T^{*}M$),
to its value at the zero section $\left(r_{0}u\right)\left(\rho\right):=u\left(\rho,0\right)$.

\paragraph{The metaplectic operator $\tilde{\mathrm{Op}}\left(d\tilde{\phi}^{t}\right)$.}

Recall the map $\tilde{\phi}^{t}:T^{*}M\rightarrow T^{*}M$ defined
in (\ref{eq:lifted_flow}). Its differential $d\tilde{\phi}^{t}:TT^{*}M\rightarrow TT^{*}M$
gives a push-forward operator that we denote $\left(d\tilde{\phi}^{t}\right)^{-\circ}:\mathcal{S}\left(TT^{*}M\right)\rightarrow\mathcal{S}\left(TT^{*}M\right)$.
As in (\ref{eq:def_d}) we define the ``metaplectic correction''

\[
\Upsilon\left(d\tilde{\phi}^{t}\right):=\left(\mathrm{det}\left(\frac{1}{2}\left(\mathrm{Id}+\left(\left(d\tilde{\phi}^{t}\right)^{-1}\right)^{\dagger}\left(d\tilde{\phi}^{t}\right)^{-1}\right)\right)\right)^{1/2},
\]
that is a positive function on $T^{*}M$. We also denote $\mathcal{P}:\mathcal{S}\left(TT^{*}M\right)\rightarrow\mathcal{S}\left(TT^{*}M\right)$
the Bergman projector in radial gauge as defined in (\ref{eq:def_P}),
that is here a fiber-wise operator. As in Definition \ref{eq:def_op_tilde},
we define the fiber-wise metaplectic operator over the map $\tilde{\phi}^{t}$:
\begin{align}
\tilde{\mathrm{Op}}\left(d\tilde{\phi}^{t}\right) & \eq{\ref{eq:def_op_tilde}}\left(\Upsilon\left(d\tilde{\phi}^{t}\right)\right)^{1/2}\,\mathcal{P}\left(d\tilde{\phi}^{t}\right)^{-\circ}\mathcal{P}.\label{eq:Op_tilde}
\end{align}

\subsubsection{Linear approximation of the kernel}

In the next theorem we use and compose the operators defined in the
previous section \ref{subsec:Definition-of-some} with the operator
$\mathcal{T}$ in (\ref{eq:def_T}), and obtain a good approximation
of the pull back operator $e^{tX}$ that improves the information
given in Theorem \ref{thm:Microlocality-of-the_TO}. The rough idea
of this approximation is based on two properties: first, for $\rho\in T^{*}M$,
with $\left\Vert \rho\right\Vert _{g_{M}}\gg1$, the metric is slowly
varying hence the metric $g$ can be approximated by a Euclidean metric
in a neighborhood of $\rho$. Second, the unit balls of the metric
$g$, transversely to the flow, have very small size projected on
the base $M$ (see Remark \ref{rem:an-important-property}), hence
the dynamics can be approximated by its linearization i.e. by the
action of the differential $d\tilde{\phi}^{t}$.

\begin{cBoxB}{}
\begin{thm}[\textbf{``More precise expression for the propagation of singularities''}]
\label{thm:propagation_singularities_2} For any $0<\lambda<\frac{1}{2}$,
for any $t\in\mathbb{R}$, we have
\begin{equation}
e^{tX}=\left(\mathcal{T}^{\dagger}r_{0}\right)\tilde{\mathrm{Op}}\left(d\tilde{\phi}^{t}\right)\left(\chi^{\lambda}\widetilde{\exp^{\circ}}\mathcal{T}\right)+R_{t}\label{eq:etX}
\end{equation}

with some operator $R_{t}$ that satisfies
\begin{equation}
R_{t}\in\Psi_{\tilde{\phi}^{t}}^{-\frac{1}{2}+\lambda}.\label{eq:Rt}
\end{equation}
\end{thm}

\end{cBoxB}

\begin{rem}
Since $-\frac{1}{2}+\lambda<0$, we can apply Lemma \ref{Lem:A-useful-consequence}
that guaranties that the remainder operator $R_{t}$ restricted to
high frequencies is arbitrary small in operator norm.
\end{rem}

~
\begin{rem}
\label{rem:T*M_or_TT*M}The proof below shows that we also have the
simpler expression
\begin{equation}
e^{tX}=\mathcal{T}^{\dagger}\left(\Upsilon\left(d\tilde{\phi}^{t}\right)\right)^{1/2}\left(\tilde{\phi}^{t}\right)^{-\circ}\mathcal{T}+R_{t}\label{eq:etX2}
\end{equation}
with the push-forward operator $\left(\tilde{\phi}^{t}\right)^{-\circ}:\mathcal{S}\left(T^{*}M\right)\rightarrow\mathcal{S}\left(T^{*}M\right)$
and the multiplication operator by the metaplectic correction $\left(\Upsilon\left(d\tilde{\phi}^{t}\right)\right)^{1/2}:\mathcal{S}\left(T^{*}M\right)\rightarrow\mathcal{S}\left(T^{*}M\right)$.
We already have commented this expression in (\ref{eq:def_OIF}).
One reason for us to later use (\ref{eq:etX}) instead of (\ref{eq:etX2})
is that (\ref{eq:etX}) does not contains $\left(\Upsilon\left(d\tilde{\phi}^{t}\right)\right)^{1/2}$,
but contains instead the operator $\tilde{\mathrm{Op}}\left(d\tilde{\phi}^{t}\right)$
that is fiberwise unitary in $L^{2}$, see appendix \ref{subsec:Metaplectic-decomposition-of}
(in fact $\tilde{\mathrm{Op}}\left(d\tilde{\phi}^{t}\right)$ contains
$\left(\Upsilon\left(d\tilde{\phi}^{t}\right)\right)^{1/2}$). On
the other hand, one advantage of (\ref{eq:etX2}) compare to (\ref{eq:etX})
is that it deals directly with function in $\mathcal{S}\left(T^{*}M\right)$
instead of the more elaborate space $\mathcal{S}\left(TT^{*}M\right)$.
Later we will meet a similar situation, see Remark \ref{rem:One-advantage-of}.
\end{rem}

\begin{proof}
Let $0<\lambda<\frac{1}{2}$ and $t\in\mathbb{R}$. We consider the
operator $R_{t}$ defined in (\ref{eq:etX}), lifted to the cotangent
space:

\begin{align}
\tilde{R}_{t} & :=\mathcal{T}R_{t}\mathcal{T}^{\dagger}\nonumber \\
 & \eq{\ref{eq:etX}}\mathcal{T}e^{tX}\mathcal{T}^{\dagger}-\left(\mathcal{T}\mathcal{T}^{\dagger}\right)r_{0}\tilde{\mathrm{Op}}\left(d\tilde{\phi}^{t}\right)\left(\chi^{\lambda}\widetilde{\exp^{\circ}}\right)\left(\mathcal{T}\mathcal{T}^{\dagger}\right).\label{eq:Rtilde_t}
\end{align}
According to Definition \ref{def:Let--be}, to get (\ref{eq:Rt}),
we have to show the following estimate for the Schwartz kernel of
$\tilde{R}_{t}$, that for any $N>0$, there exists a constant $C_{N,t}>0$
such that for any $\rho,\rho'\in T^{*}M$
\begin{align}
\left|\langle\delta_{\rho'}|\tilde{R}_{t}\delta_{\rho}\rangle\right| & \leq C_{N,t}\left\langle \mathrm{dist}_{g}\left(\rho',\tilde{\phi}^{t}\left(\rho\right)\right)\right\rangle ^{-N}\left\langle \left|\rho\right|\right\rangle ^{-\frac{1}{2}+\lambda}.\label{eq:microl_estimate-2-1}
\end{align}
We will split the computation in two parts: (A) far from the graph
of $\tilde{\phi}^{t}$ and (B) near the graph.

\paragraph{Step (A), far from the graph of $\tilde{\phi}^{t}$.}

Let us show that $\forall\lambda'>\lambda$, $\exists C_{t,\lambda'}$,
$\forall\rho,\rho'\in T^{*}M$, 
\begin{equation}
\mathrm{dist}_{g}\left(\rho',\tilde{\phi}^{t}\left(\rho\right)\right)>C_{t,\lambda'}\left\langle \left|\rho\right|\right\rangle ^{\lambda'/2}\Rightarrow\langle\delta_{\rho'}|r_{0}\tilde{\mathrm{Op}}\left(d\tilde{\phi}^{t}\right)\left(\chi^{\lambda}\widetilde{\exp^{\circ}}\right)\delta_{\rho}\rangle=0.\label{eq:impliq}
\end{equation}
From definition of the exponential map we have $\forall\rho''\in T^{*}M,v''\in T_{\rho''}T^{*}M$,
\begin{equation}
\mathrm{dist}_{g}\left(\rho'',\text{exp}\left(\left(\rho'',v''\right)\right)\right)=\left\Vert v''\right\Vert _{g_{\rho''}}.\label{eq:dist_g}
\end{equation}
From definitions of $\chi^{\lambda}$ and $\widetilde{\exp^{\circ}}$,
we have that $\langle\delta_{\rho'',v''}|\left(\chi^{\lambda}\widetilde{\exp^{\circ}}\right)\delta_{\rho}\rangle\neq0$
iff $\text{exp}\left(\left(\rho'',v''\right)\right)=\rho$ and
\begin{equation}
\text{\textrm{dist}}_{g}\left(\rho'',\rho\right)\eq{\ref{eq:dist_g}}\left\Vert v''\right\Vert _{g_{\rho''}}\underset{(\ref{eq:def_Chi_sigma-2})}{\leq}\left\langle \left|\rho''\right|\right\rangle ^{\lambda/2}.\label{eq:dist_}
\end{equation}
The operator $\tilde{\mathrm{Op}}\left(d\tilde{\phi}^{t}\right)$
is a bundle map over the map $\tilde{\phi}^{t}$ hence $\langle\delta_{\rho',v'}|\tilde{\mathrm{Op}}\left(d\tilde{\phi}^{t}\right)\delta_{\rho'',v''}\rangle$
vanishes if $\rho'\neq\tilde{\phi}^{t}\left(\rho''\right)$. Hence,
setting $\rho''=\tilde{\phi}^{-t}\left(\rho'\right)$, we have that
\begin{equation}
\langle\delta_{\rho'}|r_{0}\tilde{\mathrm{Op}}\left(d\tilde{\phi}^{t}\right)\left(\chi^{\lambda}\widetilde{\exp^{\circ}}\right)\delta_{\rho}\rangle\neq0\quad\Rightarrow\quad\text{\textrm{dist}}_{g}\left(\rho',\tilde{\phi}^{t}\left(\rho\right)\right)\underset{(\ref{eq:invariance_g})}{\asymp}\text{\textrm{dist}}_{g}\left(\tilde{\phi}^{-t}\left(\rho'\right),\rho\right)\ineq{\ref{eq:dist_}}\left\langle \left|\rho''\right|\right\rangle ^{\lambda/2}.\label{eq:bound-5}
\end{equation}
From \cite[Lemma 4.14]{faure_tsujii_Ruelle_resonances_density_2016}
we have that $\forall0<\lambda<1$, $\forall\epsilon>0,\exists C_{\epsilon}>0$,
$\forall\rho,\rho'$,
\[
\text{\textrm{dist}}_{g}\left(\rho'',\rho\right)\leq\left\langle \left|\rho''\right|\right\rangle ^{\lambda/2}\Rightarrow\left\langle \left|\rho''\right|\right\rangle \leq C_{\epsilon}\left\langle \left|\rho\right|\right\rangle ^{1+\epsilon},
\]
hence the right hand side of (\ref{eq:bound-5}) becomes $\text{\textrm{dist}}_{g}\left(\rho',\tilde{\phi}^{t}\left(\rho\right)\right)\leq C_{t,\lambda'}\left\langle \left|\rho\right|\right\rangle ^{\lambda'/2}$
for any $\lambda'>\lambda$ and we get (\ref{eq:impliq}).

We assume that $\mathrm{dist}_{g}\left(\rho',\tilde{\phi}^{t}\left(\rho\right)\right)>C_{t,\lambda'}\left\langle \left|\rho\right|\right\rangle ^{\lambda'/2}$
i.e. points $\left(\rho,\rho'\right)\in T^{*}M\times T^{*}M$ are
``far'' from the graph of $\tilde{\phi}^{t}$. We have that $e^{tX}\in\Psi_{\tilde{\phi}^{t}}^{0}$
and $\left(\mathcal{T}^{\dagger}r_{0}\right)\tilde{\mathrm{Op}}\left(d\tilde{\phi}^{t}\right)\left(\chi^{\lambda}\widetilde{\exp^{\circ}}\mathcal{T}\right)\in\Psi_{\tilde{\phi}^{t}}^{0}$.
This implies that $R_{t}\in\Psi_{\tilde{\phi}^{t}}^{0}$. For any
$N>0$, take $M>\frac{1}{\lambda'}$ and $N'=N+M$. We have
\begin{align*}
\left|\langle\delta_{\rho'}|\tilde{R}_{t}\delta_{\rho}\rangle\right| & \underset{(\ref{eq:Psi_tilde})}{\leq}C'_{N',t}\left\langle \mathrm{dist}_{g}\left(\rho',\tilde{\phi}^{t}\left(\rho\right)\right)\right\rangle ^{-N'}\leq C''_{N',t}\left\langle \mathrm{dist}_{g}\left(\rho',\tilde{\phi}^{t}\left(\rho\right)\right)\right\rangle ^{-N}\left\langle \mathrm{dist}_{g}\left(\rho',\tilde{\phi}^{t}\left(\rho\right)\right)\right\rangle ^{-M}\\
 & \ineq{\mathrm{hyp.}}C''_{N',t}\left\langle \mathrm{dist}_{g}\left(\rho',\tilde{\phi}^{t}\left(\rho\right)\right)\right\rangle ^{-N}\left(C_{t,\lambda'}\left\langle \left|\rho\right|\right\rangle ^{\lambda'/2}\right)^{-M}\\
 & \leq C'''_{N',t,\lambda'}\left\langle \mathrm{dist}_{g}\left(\rho',\tilde{\phi}^{t}\left(\rho\right)\right)\right\rangle ^{-N}\left\langle \left|\rho\right|\right\rangle ^{-M\lambda'/2}\\
 & \leq C_{N,t}\left\langle \mathrm{dist}_{g}\left(\rho',\tilde{\phi}^{t}\left(\rho\right)\right)\right\rangle ^{-N}\left\langle \left|\rho\right|\right\rangle ^{-1/2}
\end{align*}
that gives (\ref{eq:microl_estimate-2-1}).

\paragraph{Step (B), near the graph of $\tilde{\phi}^{t}$.}

We assume now that 
\begin{equation}
\mathrm{dist}_{g}\left(\rho',\tilde{\phi}^{t}\left(\rho\right)\right)\leq C_{t,\lambda'}\left\langle \left|\rho\right|\right\rangle ^{\lambda'/2}\label{eq:near_the_graph}
\end{equation}
with $\lambda<\lambda'<1$, i.e. points $\left(\rho,\rho'\right)\in T^{*}M\times T^{*}M$
are ``near'' the graph of $\tilde{\phi}^{t}$. Let us explain the
strategy that we will pursue. Using the slow variation of the metric
$g$ given in \cite[Lemma 4.12]{faure_tsujii_Ruelle_resonances_density_2016},
we will (1) approximate the metric $g_{\rho'}$ by the Euclidean metric
at point $\tilde{\phi}^{t}\left(\rho\right)$. From (\ref{eq:metric_g_tilde_in_coordinates}),
the projected points $\pi\left(\rho'\right),\pi\left(\tilde{\phi}^{t}\left(\rho\right)\right)$
on $M$ are at distance from each other that is $O\left(\left\langle \left|\rho\right|\right\rangle ^{\lambda'/2-1/2}\right)$
transversely to the flow direction and this distance goes to zero
if $\left|\rho\right|\rightarrow\infty$. So we will (2) approximate
the map $\tilde{\phi}^{t}$ by its differential, i.e. neglect non-linear
terms. Finally (3), even for a linear map on an Euclidean space, we
need to show that the effect of the cutoff $\chi^{\lambda}$ is negligible.

Having these three approximations in mind we write
\begin{align*}
\left|\langle\delta_{\rho'}|\tilde{R}_{t}\delta_{\rho}\rangle\right| & \underset{(\ref{eq:Rtilde_t})}{\leq}R_{1}+R_{2}+R_{3},
\end{align*}
with
\[
R_{1}=\left|\langle\delta_{\rho'}|\left(\mathcal{T}e^{tX}\mathcal{T}^{\dagger}-\mathcal{B}_{\tilde{\phi}^{t}\left(\rho\right)}e^{tX}\mathcal{B}_{\rho}^{\dagger}\right)\delta_{\rho}\rangle\right|
\]
\begin{equation}
R_{2}=\left|\langle\delta_{\rho'}|\left(\mathcal{B}_{\tilde{\phi}^{t}\left(\rho\right)}e^{tX}\mathcal{B}_{\rho}^{\dagger}-e^{-i\varphi_{\tilde{\phi}^{t}\left(\rho\right)}}\tilde{\mathrm{Op}}_{\rho}\left(d\tilde{\phi}_{\rho}^{t}\right)e^{i\varphi_{\rho}}\right)\delta_{\rho}\rangle\right|\label{eq:def_R2}
\end{equation}
\[
R_{3}=\left|\langle\delta_{\rho'}|\left(e^{-i\varphi_{\tilde{\phi}^{t}\left(\rho\right)}}\tilde{\mathrm{Op}}_{\rho}\left(d\tilde{\phi}_{\rho}^{t}\right)e^{i\varphi_{\rho}}-\left(\mathcal{T}\mathcal{T}^{\dagger}\right)r_{0}\tilde{\mathrm{Op}}\left(d\tilde{\phi}^{t}\right)\left(\chi^{\lambda}\widetilde{\exp^{\circ}}\right)\left(\mathcal{T}\mathcal{T}^{\dagger}\right)\right)\delta_{\rho}\rangle\right|,
\]
where $\mathcal{B}_{\rho}=\mathcal{B}_{\left(V\right),\rho}:\mathcal{S}\left(\mathbb{R}^{\mathrm{dim}M}\right)\rightarrow\mathcal{S}\left(\mathbb{R}^{2\mathrm{dim}M}\right)$
is the Bargman transform in vertical gauge (\ref{eq:def_Bargman_B})
using local charts on $M$ and defined from the Euclidean metric $g_{\rho}$.
We have the metaplectic operator $\tilde{\mathrm{Op}}_{\rho}\left(d\tilde{\phi}_{\rho}^{t}\right)\eq{\ref{eq:def_op_tilde}}\left(\Upsilon\left(d\tilde{\phi}_{\rho}^{t}\right)\right)^{1/2}\,\mathcal{P}_{\tilde{\phi}^{t}\left(\rho\right)}\left(d\tilde{\phi}_{\rho}^{t}\right)^{-\circ}\mathcal{P}_{\rho}$
with the Bergman projector in radial gauge $\mathcal{P}_{\rho}\eq{\ref{eq:Bergman_projector}}\mathcal{B}_{\left(R\right)\rho}\mathcal{B}_{\left(R\right),\rho}^{\dagger}$.
We have the phase function $\varphi_{\rho}$ defined in (\ref{eq:def_phase_phi})
and expressed in (\ref{eq:phase_phi}) for the Euclidean metric, that
is to pass from vertical gauge to radial gauge. 
\begin{enumerate}
\item In the term $R_{1}$, the operator $\mathcal{T}$ is constructed from
the metric $g$ whereas the operator $\mathcal{B}_{g}$ is constructed
from the local Euclidean metric $g_{\rho}$ in $\mathcal{B}_{\rho}$.
We use \cite[Lemma 4.12]{faure_tsujii_Ruelle_resonances_density_2016}
that the metric $g$ varies slowly on $T^{*}M$: for any $0\leq\gamma<1$,
there exist $N>0$ and $C>0$ such that for any $\varrho,\varrho'\in\mathbb{R}^{2\left(n+1\right)}$
and $v\in\mathbb{R}^{2\left(n+1\right)}$,
\begin{align}
\max\left\{ \frac{\|v\|_{g_{\varrho'}}}{\|v\|_{g_{\varrho}}},\frac{\|v\|_{g_{\varrho}}}{\|v\|_{g_{\varrho'}}}\right\} \le1+C\left\langle \left|\rho\right|\right\rangle ^{-\left(1-\gamma\right)/2}\left\langle \left\langle \left|\varrho\right|\right\rangle ^{-\gamma/2}\left\Vert \varrho'-\varrho\right\Vert _{g_{\varrho}}\right\rangle ^{N}.\label{eq:g_moderate and temperate}
\end{align}
If we take $\gamma=\lambda+\epsilon$ with $\epsilon>0$, we get that
for $\left\Vert \varrho'-\varrho\right\Vert _{g_{\varrho}}<\left\langle \left|\varrho\right|\right\rangle ^{\frac{1}{2}\left(\lambda+\epsilon\right)}$
then
\[
\left|\frac{\|v\|_{g_{\varrho'}}}{\|v\|_{g_{\varrho}}}-1\right|\leq C\left\langle \left|\rho\right|\right\rangle ^{-\frac{1}{2}\left(1-\lambda-\epsilon\right)},
\]
hence $\left|R_{1}\right|\leq C_{N,t}\left\langle \mathrm{dist}_{g}\left(\rho',\tilde{\phi}^{t}\left(\rho\right)\right)\right\rangle ^{-N}\left\langle \left|\rho\right|\right\rangle ^{-\frac{1}{2}+\lambda}$.
\item For the term $R_{2}$, we start from
\[
\mathcal{B}_{\tilde{\phi}^{t}\left(\rho\right)}e^{tX}\mathcal{B}_{\rho}^{\dagger}\eq{\ref{eq:def_exptX}}\mathcal{B}_{\tilde{\phi}^{t}\left(\rho\right)}\left(\phi^{t}\right)^{\circ}\mathcal{B}_{\rho}^{\dagger}
\]
and in local chart in a small neighborhood of $\pi\left(\rho\right),\pi\left(\tilde{\phi}^{t}\left(\rho\right)\right)\in M$
of size 
\begin{equation}
O_{t}\left(\left\langle \left|\rho\right|\right\rangle ^{\frac{\lambda'}{2}}\delta^{\perp}\left(\rho\right)\right)\eq{\ref{eq:def_delta}}O_{t}\left(\left\langle \left|\rho\right|\right\rangle ^{-\left(1-\lambda'\right)/2}\right),\label{eq:size}
\end{equation}
we approximate $\phi^{t}$ by its differential $d\phi_{\rho}^{t}$
at $\rho$, replace $\left(\phi^{t}\right)^{\circ}$ by $\left(d\phi^{t}\right)_{\rho}^{\circ}$.
In the next lines, $\left(V\right),\left(R\right)$ denote respectively
vertical and radial gauge defined in section \ref{subsec:Bargmann-Transform}.
In the first equality, we substract the main oscillatory term $e^{-i\theta_{\rho}\left(.\right)}$.
In the last equality we use that $\mathrm{det}\left(d\phi^{t}\right)\eq{\ref{eq:volume_form_M}}1$.
We write
\begin{align*}
\mathcal{B}_{\tilde{\phi}^{t}\left(\rho\right)}\left(d\phi^{t}\right)_{\rho}^{\circ}\mathcal{B}_{\rho}^{\dagger} & =e^{i\theta_{\tilde{\phi}^{t}\left(\rho\right)}\left(.\right)}\mathcal{B}_{\tilde{\phi}^{t}\left(\rho\right)}^{\left(V\right)}\left(d\phi^{t}\right)_{\rho}^{\circ}\mathcal{B}_{\rho}^{\dagger\left(V\right)}e^{-i\theta_{\rho}\left(.\right)}\eq{\ref{eq:Gaussian_wave_packet-V-R},\ref{eq:def_Bargman_B},\ref{eq:phase_phi}}e^{-i\varphi_{\tilde{\phi}^{t}\left(\rho\right)}}\mathcal{B}_{\tilde{\phi}^{t}\left(\rho\right)}^{\left(R\right)}\left(d\phi^{t}\right)_{\rho}^{\circ}\mathcal{B}_{\rho}^{\dagger\left(R\right)}e^{i\varphi_{\rho}}\\
 & \eq{\ref{eq:Op_tilde_PHI}}e^{-i\varphi_{\tilde{\phi}^{t}\left(\rho\right)}}\tilde{\mathrm{Op}}_{\rho}\left(d\tilde{\phi}_{\rho}^{t}\right)e^{i\varphi_{\rho}}
\end{align*}
In the approximation replacing $\mathcal{B}_{\tilde{\phi}^{t}\left(\rho\right)}e^{tX}\mathcal{B}_{\rho}^{\dagger}$
by $e^{-i\varphi_{\tilde{\phi}^{t}\left(\rho\right)}}\tilde{\mathrm{Op}}_{\rho}\left(d\tilde{\phi}_{\rho}^{t}\right)e^{i\varphi_{\rho}}$,
we have neglected the non linearity of the map $\phi^{t}$. From Taylor
expansion this gives that 
\[
\left|R_{2}\right|\ineq{\ref{eq:size}}C_{N,t}\left\langle \mathrm{dist}_{g}\left(\rho',\tilde{\phi}^{t}\left(\rho\right)\right)\right\rangle ^{-N}\left(\left\langle \left|\rho\right|\right\rangle ^{-\left(1-\lambda'\right)/2}\right)^{2},
\]
smaller than previous terms. 
\item For the term $R_{3}$, one first operation is to pass to the tangent
bundle with $\widetilde{\exp^{\circ}}$ and with a truncation $\chi^{\lambda}$
and past restriction $r_{0}$. For this, in the linear setting, Lemma
\ref{lem:Let-us-consider} shows that in case of a linear map $A:E\rightarrow E$
and Euclidean metric $g$ on $E\oplus E^{*}$, a cutoff $\chi_{\sigma}$
gives an error term bounded by $C_{N}\sigma^{-N},\forall N$. Here
$\sigma=\left\langle \left|\rho\right|\right\rangle ^{\lambda/2}$,
so this gives $O_{N,t}\left(\left\langle \mathrm{dist}_{g}\left(\rho',\tilde{\phi}^{t}\left(\rho\right)\right)\right\rangle ^{-N}\left\langle \left|\rho\right|\right\rangle ^{-N}\right)$,
smaller than previous terms. Another operation is to pass again to
the manifold with the metric $g$ as in step 1. We have seen that
it gives $\left|R_{3}\right|\leq C_{N,t}\left\langle \mathrm{dist}_{g}\left(\rho',\tilde{\phi}^{t}\left(\rho\right)\right)\right\rangle ^{-N}\left\langle \left|\rho\right|\right\rangle ^{-\frac{1}{2}+\lambda}$.
\end{enumerate}
We have checked (\ref{eq:microl_estimate-2-1}) that gives (\ref{eq:Rt}).
\end{proof}

\section{\label{sec:Micro-local-analysis-of-1}Micro-local analysis of a contact
vector field $X$ on $\left(M,\mathcal{A}\right)$ near $\Sigma=\mathbb{R}\mathcal{A}\backslash\left\{ 0\right\} $}

In this section we pursue the analysis of a smooth non vanishing vector
field $X$ on a closed manifold $M$ done in section \ref{sec:Micro-local-analysis-of})
and furthermore we assume that $X$ is a \textbf{\href{https://en.wikipedia.org/wiki/Reeb_vector_field}{Reeb vector field}},
i.e. there is a smooth contact one form $\mathcal{A}$ on $M$ such
that for every $m\in M$, the linear space $\mathrm{Ker}\left(\mathcal{A}\left(m\right)\right)$
endowed with the two form $\left(d\mathcal{A}\right)\left(m\right)$
is a linear symplectic space and 
\begin{equation}
\mathcal{A}\left(X\right)=1,\qquad d\mathcal{A}\left(X,.\right)=0.\label{eq:Ker_A}
\end{equation}
This implies that $\mathcal{A}$ is invariant under the flow $\phi^{t}$
generated by X. Indeed the Lie derivative vanishes: $\mathcal{L}_{X}\mathcal{A}=\iota_{X}d\mathcal{A}+d\iota_{X}\mathcal{A}=d\mathcal{A}\left(X,.\right)+d1=0$.
Also $\mathrm{dim}M=2d+1$, with $2d=\mathrm{dim}\left(\mathrm{Ker}\mathcal{A}\right)$.
We will write
\begin{equation}
dm:=\frac{1}{d!}\mathcal{A}\wedge\left(d\mathcal{A}\right)^{\wedge d}\label{eq:volume_form_M}
\end{equation}
for the corresponding smooth and non-degenerate volume form on $M$
invariant by the flow $\phi^{t}$.
\begin{rem}
For the moment we do not assume that $X$ is Anosov, but later in
section \ref{sec:Micro-local-analysis-of-1-1}, this one form $\mathcal{A}$
will be determined by $\mathrm{Ker}\mathcal{A}=E_{u}\oplus E_{s}$
in (\ref{eq:one_form}). A typical example of contact vector field
$X$ on $\left(M,\mathcal{A}\right)$ is a general geodesic vector
field on the unit cotangent bundle $M=\left(T^{*}\mathcal{N}\right)_{1}$
of a Riemannian manifold $\mathcal{N}$ with $\mathcal{A}$ being
the \href{https://en.wikipedia.org/wiki/Tautological_one-form}{Liouville one-form}.

In $T^{*}M$ we define
\begin{align}
E_{0}^{*} & :=\left(\mathrm{Ker}\mathcal{A}\right)^{\perp}=\left\{ \rho\in T^{*}M,\quad\mathrm{Ker}\rho\supset\mathrm{Ker}\mathcal{A}\right\} \label{eq:def_E_0*}\\
 & \underset{(\ref{eq:Ker_A})}{=}\mathbb{R}\mathcal{A}=\left\{ \omega\mathcal{A}\left(m\right),\quad m\in M,\omega\in\mathbb{R}\right\} .
\end{align}
$E_{0}^{*}$ is a rank $1$ sub-bundle of $T^{*}M$ over $M$. We
have
\begin{equation}
T^{*}M=\mathrm{Ker}X\oplus E_{0}^{*},\label{eq:decomp_KerX_E0*}
\end{equation}
where 
\begin{equation}
\mathrm{Ker}X=\left\{ \rho\in T^{*}M,\quad\boldsymbol{\omega}\left(\rho\right)\eq{\ref{eq:omega_function}}\rho\left(X\right)=X\left(\rho\right)=0\right\} \label{eq:def_Ker_X}
\end{equation}
is a rank $2d$ sub-bundle of $T^{*}M$ over $M$.
\end{rem}

\subsection{The symplectization $\Sigma=\mathbb{R}\mathcal{A}\backslash\left\{ 0\right\} $}

We first recall the important following lemma. It shows that $\Sigma:=\mathbb{R}\mathcal{A}\backslash\left\{ 0\right\} =E_{0}^{*}\backslash\left\{ 0\right\} $
is a symplectic sub-manifold of $T^{*}M$, called the \href{https://en.wikipedia.org/wiki/Symplectization}{symplectization}
of $M$ \cite[Section 11.2]{da_silva_01}\cite[Appendix 4.]{arnold-mmmc}.
We will denote $\pi:T^{*}M\rightarrow M$ the projection  and $\pi^{\circ}\left(dm\right)$
the form $dm$ (\ref{eq:volume_form_M}) pulled back on $T^{*}M$.

\begin{cBoxB}{}
\begin{lem}
\label{lem:The-canonical-symplectic}The set 
\begin{equation}
\Sigma:=\mathbb{R}\mathcal{A}\backslash\left\{ 0\right\} :=\left\{ \omega\mathcal{A}\left(m\right)\in T^{*}M\,\mid\quad m\in M,\omega\in\mathbb{R}\backslash\left\{ 0\right\} \right\} \label{eq:def_Sigma}
\end{equation}
is a smooth \textbf{symplectic sub-manifold} of $T^{*}M$ with $\mathrm{dim}\Sigma=2\left(d+1\right)$.
The induced volume form on $\Sigma$ is given by
\begin{equation}
d\varrho:=\frac{1}{\left(d+1\right)!}\left(d\theta\right)^{\wedge\left(d+1\right)}=\omega^{d}\left(d\omega\right)\wedge\pi^{\circ}\left(dm\right)\label{eq:dvol_E0*}
\end{equation}
where $\theta$ is the Liouville form (\ref{eq:def_theta}).
\end{lem}

\end{cBoxB}

\begin{proof}
For $m\in M$, $\omega\in\mathbb{R}\backslash\left\{ 0\right\} $,
at point $\rho=\omega\mathcal{A}\left(m\right)\in\Sigma\subset T^{*}M$
we have
\begin{equation}
\theta_{\rho}\eq{\ref{eq:def_theta}}\omega\left(\rho\right)\cdot\pi^{\circ}\mathcal{A}\label{eq:theta_A}
\end{equation}
 where $\pi^{\circ}$ is the pull back map on forms. Hence $d\theta=d\left(\omega\left(\pi^{\circ}\mathcal{A}\right)\right)=d\omega\wedge\left(\pi^{\circ}\mathcal{A}\right)+\omega\left(\pi^{\circ}d\mathcal{A}\right)$
giving the following volume form on $E_{0}^{*}$
\begin{align*}
d\varrho: & =\frac{1}{\left(d+1\right)!}\left(d\theta\right)^{\wedge\left(d+1\right)}=\frac{1}{\left(d+1\right)!}\left(d+1\right)\omega^{d}\cdot d\omega\wedge\left(\left(\pi^{\circ}\mathcal{A}\right)\wedge\left(\pi^{\circ}d\mathcal{A}\right)^{\wedge d}\right)\\
 & \underset{(\ref{eq:volume_form_M})}{=}\omega^{d}\left(d\omega\right)\wedge\pi^{\circ}\left(dm\right)
\end{align*}
which does not vanish on $\Sigma=E_{0}^{*}\backslash\left\{ 0\right\} $
since for the line bundle $E_{0}^{*}\rightarrow M$, $dm$ is a measure
on the base $M$ and $d\omega$ is a measure on the fibers. Consequently
$d\theta$ restricted to $\Sigma$ is \href{https://en.wikipedia.org/wiki/Symplectic_vector_space\#Volume_form}{non degenerate},
$\Sigma$ is a smooth sub-manifold of $T^{*}M$.
\end{proof}
Since the one form $\mathcal{A}$ is invariant under the flow $\phi^{t}$,
the set $\Sigma$ is invariant under the lifted flow $\tilde{\phi}^{t}$
in $T^{*}M$ defined in (\ref{eq:lifted_flow}). The purpose of this
section \ref{sec:Micro-local-analysis-of-1} is to provide a useful
approximation of the pullback operator $e^{tX}$ restricted to a micro-local
neighborhood of $\Sigma$ (i.e. vicinity of $\Sigma$ in $T^{*}M$)
in terms of the differential map $d\tilde{\phi}_{N}^{t}$ that is
$d\tilde{\phi}^{t}$ restricted to the symplectic normal bundle $N$
of $\Sigma$.
\begin{rem}
Later in section \ref{sec:Micro-local-analysis-of-1-1} we will consider
a contact Anosov flow and see in Lemma \ref{lem:trapped_set} that
$\Sigma$ is the trapped set or non wandering set for the lifted flow
$\tilde{\phi}^{t}$. Consequently we will obtain in Theorem \ref{thm:decay}
that the dynamics $e^{tX}$ outside $\Sigma$ is negligible in some
sense, so the only important part is indeed a micro-local neighborhood
of $\Sigma$ that we consider in this section.
\end{rem}

\subsection{The symplectic normal bundle $N\rightarrow\Sigma$}

\subsubsection{Metaplectic decomposition $K\protect\overset{\perp_{\Omega}}{\oplus}N$}

We have seen in Lemma \ref{lem:The-canonical-symplectic} that $\Sigma:=E_{0}^{*}\backslash\left\{ 0\right\} $
is symplectic. We will denote
\[
T_{\Sigma}T^{*}M:=\left\{ T_{\rho}\left(T^{*}M\right)\,\mid\quad\rho\in\Sigma\right\} 
\]
that is the tangent bundle $T\left(T^{*}M\right)$ restricted to the
base space $\Sigma$. The following Lemma provides a decomposition
of $T_{\Sigma}T^{*}M$ into symplectic sub-bundles, invariant under
the flow map $d\tilde{\phi}^{t}:T_{\Sigma}T^{*}M\rightarrow T_{\Sigma}T^{*}M$
with $\tilde{\phi}^{t}$ defined in (\ref{eq:lifted_flow}). For the
one form $\mathcal{A}$ seen as a map $\mathcal{A}:M\rightarrow T^{*}M$,
we will use its differential map $d\mathcal{A}:TM\rightarrow TT^{*}M$.
Recall that $\tilde{X}$ was defined in section \ref{subsec:The-vector-field}.

\begin{cBoxB}{}
\begin{lem}
\label{def:KNA}Let $\tilde{E}_{0}:=\mathbb{R}\tilde{X}$ and $\tilde{E}_{0}^{*}:=T\mathbb{R}\mathcal{A}$
be rank 1 sub-bundles of $T_{\Sigma}T^{*}M$. Then
\begin{equation}
K_{0}:=\tilde{E}_{0}\oplus\tilde{E}_{0}^{*}\label{eq:def_K0}
\end{equation}
is a symplectic sub-bundle of $T_{\Sigma}T^{*}M$ and we have a decomposition
of $T_{\Sigma}T^{*}M$ as an orthogonal sum of symplectic spaces $K,K_{0},N$:
\begin{equation}
T_{\Sigma}T^{*}M=\underbrace{K\overset{\perp_{\Omega}}{\oplus}K_{0}}_{T\Sigma}\overset{\perp_{\Omega}}{\oplus}N\label{eq:decomp_K_K0_N}
\end{equation}
\[
\mathrm{dim}K=2d,\quad\mathrm{dim}K_{0}=2,\quad\mathrm{dim}N=2d,
\]
invariant under the flow map $d\tilde{\phi}^{t}:T_{\Sigma}T^{*}M\rightarrow T_{\Sigma}T^{*}M$.
\end{lem}

\end{cBoxB}

\begin{rem}
Notice that $N=\left(T\Sigma\right){}^{\perp_{\Omega}}$ is the $\Omega$-symplectic
orthogonal to $T\Sigma=K\overset{\perp_{\Omega}}{\oplus}K_{0}$ in
$T_{\Sigma}T^{*}M$. To explain the notation of $\tilde{E}_{0}^{*}$,
notice that $K_{0}=\tilde{E}_{0}\oplus\tilde{E}_{0}^{*}$ is $\Omega$-symplectic,
hence $\tilde{E}_{0}^{*}$ is isomorphic to the dual of $\tilde{E}_{0}$.
\end{rem}

\begin{proof}
For fixed $\omega\in\mathbb{R}\backslash\left\{ 0\right\} $, let
$T\left(\omega\mathcal{A}\right)\subset TT^{*}M$ be the tangent space
to the graph of $\omega\mathcal{A}$:
\[
T\left(\omega\mathcal{A}\right):=\left(d\left(\omega\mathcal{A}\right)\right)\left(TM\right).
\]
For any $\rho=\omega\mathcal{A}\left(m\right)\in\Sigma$, $m=\pi\left(\rho\right)\in M$,
we have that $d\pi:T\left(\omega\mathcal{A}\right)\left(\rho\right)\rightarrow TM\left(m\right)$
is an isomorphism. We decompose
\begin{equation}
T\left(\omega\mathcal{A}\right)=K\oplus\tilde{E}_{0}\label{eq:TomegaA}
\end{equation}
with $\tilde{E}_{0}=\mathbb{R}\tilde{X}=T\left(\omega\mathcal{A}\right)\cap d\pi^{-1}\left(\mathbb{R}X\right),$
$\mathrm{dim}\tilde{E}_{0}=1,$ and $K:=T\left(\omega\mathcal{A}\right)\cap d\pi^{-1}\left(\mathrm{Ker}\mathcal{A}\right)$,
$\mathrm{dim}K=2d$. We have $\tilde{E}_{0}^{*}=\mathbb{R}\mathcal{A}=T\Sigma\cap d\pi^{-1}\left(\left\{ 0\right\} \right)$,
$\mathrm{dim}\tilde{E}_{0}^{*}=1$ and 
\begin{equation}
K_{0}=\tilde{E}_{0}\oplus\tilde{E}_{0}^{*}=T\Sigma\cap d\pi^{-1}\left(\mathbb{R}X\right).\label{eq:def_K0-1}
\end{equation}
We have $T\Sigma=T\left(\omega\mathcal{A}\right)\oplus\tilde{E}_{0}^{*}=K\oplus K_{0}.$
We have to prove that $K,K_{0}$ are symplectic and that $K\perp_{\Omega}K_{0}$.
We first show the following Lemma, where for fixed $\omega$, the
one form $\omega\mathcal{A}$ is seen as a map $\omega\mathcal{A}:M\rightarrow T^{*}M$,
with differential $d\left(\omega\mathcal{A}\right):TM\rightarrow TT^{*}M$.

\begin{cBoxB}{}
\begin{lem}
At point $\omega\mathcal{A}\left(m\right)\in\Sigma$, the linear map
\begin{equation}
p:=\left(d\left(\omega\mathcal{A}\right)\right)\circ\left(d\pi\right)\qquad:T_{\Sigma}T^{*}M\rightarrow T_{\Sigma}T^{*}M\label{eq:def_projector_p}
\end{equation}
is a projector with 
\begin{equation}
\mathrm{Ker}p=\mathrm{Ker}\left(d\pi\right),\quad\mathrm{Im}p=T\left(\mathrm{Im}\left(\omega\mathcal{A}\right)\right).\label{eq:Ker_p_Ker_dpi}
\end{equation}
We have
\begin{equation}
\left(d\pi\right)^{-1}\left(\mathrm{Ker}\mathcal{A}\right)=\left(T\mathbb{R}\mathcal{A}\right)^{\perp_{\Omega}}\eq{\mathrm{def}}\left(\tilde{E}_{0}^{*}\right)^{\perp_{\Omega}}.\label{eq:lemma_KerA}
\end{equation}
\end{lem}

\end{cBoxB}

\begin{proof}
We have $\pi\circ\mathcal{A}=\mathrm{Id}_{M}$ hence $\left(d\pi\right)\circ\left(d\mathcal{A}\right)=\mathrm{Id}_{TM}$
and
\[
p^{2}=\left(d\left(\omega\mathcal{A}\right)\right)\underbrace{\left(d\pi\right)\left(d\left(\omega\mathcal{A}\right)\right)}_{\mathrm{Id}_{TM}}\left(d\pi\right)=p.
\]
We have
\begin{equation}
\Omega_{/\mathrm{Ker}\left(d\pi\right)}\eq{\ref{eq:def_Omega-1}}0.\label{eq:Omega_ker_dpi}
\end{equation}
For any $U\in T_{\omega\mathcal{A}}T^{*}M$,
\begin{align*}
\omega\mathcal{A}\left(\left(d\pi\right)\left(U\right)\right) & \eq{\ref{eq:alpha_theta}}\theta\left(\left(d\left(\omega\mathcal{A}\right)\right)\circ\left(d\pi\right)\left(U\right)\right)\eq{\ref{eq:def_Euler}}\Omega\left(\mathcal{E},\left(d\left(\omega\mathcal{A}\right)\right)\circ\left(d\pi\right)\left(U\right)\right)\\
 & \eq{\ref{eq:def_projector_p}}\Omega\left(\mathcal{E},p\left(U\right)\right)=\Omega\left(\mathcal{E},\left(p\left(U\right)-U\right)+U\right)\eq{\ref{eq:Ker_p_Ker_dpi},\ref{eq:Omega_ker_dpi}}\Omega\left(\mathcal{E},U\right).
\end{align*}
On $\Sigma=\mathbb{R}\mathcal{A}$ we have $\mathbb{R}\mathcal{E}=\mathbb{R}\mathcal{A}$,
we deduce (\ref{eq:lemma_KerA}).
\end{proof}
We have an isomorphism
\[
d\left(\omega\mathcal{A}\right):TM=\mathrm{Ker}\mathcal{A}\oplus\mathbb{R}X\rightarrow T\left(\omega\mathcal{A}\right)\eq{\ref{eq:TomegaA}}K\oplus\tilde{E}_{0},
\]
and
\begin{equation}
\left(d\left(\omega\mathcal{A}\right)\right)^{\circ}\Omega\eq{\ref{eq:dalpha_Omega}}d\left(\omega\mathcal{A}\right).\label{eq:dA_Omega}
\end{equation}
Hence
\[
d\left(\omega\mathcal{A}\right):\left(\mathrm{Ker}\mathcal{A},\omega d\mathcal{A}\right)\rightarrow\left(K,\Omega\right)
\]
is a symplectomorphism for the respective symplectic structures. Since
$\mathrm{Ker}\mathcal{A}$ is $d\mathcal{A}-$symplectic, this implies
that $K$ is $\Omega-$symplectic. This also gives that $\Omega\left(K,\tilde{E}_{0}\right)=\omega d\mathcal{A}\left(\mathrm{Ker}\mathcal{A},\mathbb{R}X\right)=0$.
By definition we have $\tilde{E}_{0}^{*}=\mathbb{R}\mathcal{A}$ and
$K=T\left(\omega\mathcal{A}\right)\cap d\pi^{-1}\left(\mathrm{Ker}\mathcal{A}\right)$.
From (\ref{eq:lemma_KerA}) we have that $\Omega\left(K,\tilde{E}_{0}^{*}\right)=0$.
Since $K_{0}=\tilde{E}_{0}\oplus\tilde{E}_{0}^{*}$, we have obtained
that $K\perp_{\Omega}K_{0}$ and that $K_{0}$ is $\Omega-$symplectic.
This complete the proof of Lemma \ref{def:KNA}.
\end{proof}
The next Lemma concerns specifically the subspace $K\overset{\perp_{\Omega}}{\oplus}N$
in (\ref{eq:decomp_K_K0_N}). We show that this space is also the
sum of canonical Lagrangian spaces $H\oplus V$ (horizontal and vertical).
This is illustrated on figure \ref{fig:N_K_V_H}. Recall the frequency
function $\boldsymbol{\omega}:\rho\in T^{*}M\rightarrow\boldsymbol{\omega}\left(\rho\right)=\rho\left(X\right)\in\mathbb{R}$
defined in (\ref{eq:omega_function}) and that $\left(\mathrm{Ker}\mathcal{A},d\mathcal{A}\right)$
is a linear symplectic space from section \ref{subsec:Contact-Anosov-flow}.

\begin{cBoxA}{}
\begin{defn}
For $u\in\mathrm{Ker}\mathcal{A}$, we define the maps
\begin{align}
\overline{K}\left(u\right) & :=\left(\left(d\pi^{-1}\right)\left(u\right)\right)\cap K,\quad\overline{N}\left(u\right):=\left(\left(d\pi^{-1}\right)\left(u\right)\right)\cap N\nonumber \\
\overline{H}\left(u\right) & :=\frac{1}{2}\left(\overline{K}\left(u\right)+\overline{N}\left(u\right)\right),\quad\overline{V}\left(u\right):=\frac{1}{2}\left(\overline{K}\left(u\right)+\overline{N}\left(-u\right)\right),\label{eq:def_H_V}
\end{align}
that define the horizontal and the vertical spaces:
\begin{align*}
H & :=\left\{ \overline{H}\left(u\right),u\in\mathrm{Ker}\mathcal{A}\right\} ,\\
V & :=\left\{ \overline{V}\left(u\right),u\in\mathrm{Ker}\mathcal{A}\right\} 
\end{align*}
that are subspaces of $K\overset{\perp_{\Omega}}{\oplus}N$.
\end{defn}

\end{cBoxA}
\begin{cBoxB}{}
\begin{lem}[Metaplectic decomposition]
\label{lem:KN}At any point $\omega\mathcal{A}\in\Sigma$ with $\omega\neq0$,
let
\begin{equation}
\widetilde{\mathrm{Ker}\mathcal{A}}:=\left(d\pi^{-1}\left(\mathrm{Ker}\mathcal{A}\right)\right)\cap\mathrm{Ker}\left(d\boldsymbol{\omega}\right).\label{eq:Ker_tilde_A}
\end{equation}
$\widetilde{\mathrm{Ker}\mathcal{A}}$ is a $\Omega-$symplectic subspace
of $T_{\Sigma}T^{*}M$ and we have the decompositions
\begin{equation}
\widetilde{\mathrm{Ker}\mathcal{A}}=K\overset{\perp_{\Omega}}{\oplus}N=H\oplus V,\label{eq:K_N_HA_VA}
\end{equation}
with $\Omega-$symplectic and orthogonal subspaces $K,N$ and $\Omega-$Lagrangian
subspaces $H,V$ with $\mathrm{dim}V=\mathrm{dim}H=2d$.

We have $\left(d\pi\right)\left(H\right)=\mathrm{Ker}\mathcal{A}$
and $\left(d\pi\right)\left(V\right)=\left\{ 0\right\} $. The linear
maps
\begin{equation}
\sqrt{\omega}d\pi:\left(K,\Omega\right)\rightarrow\left(\mathrm{Ker}\mathcal{A},d\mathcal{A}\right)\label{eq:dpi_K}
\end{equation}
\begin{equation}
\sqrt{\omega}d\pi:\left(N,\Omega\right)\rightarrow\left(\mathrm{Ker}\mathcal{A},-d\mathcal{A}\right)\label{eq:dpi_N}
\end{equation}
are symplectomorphism for the respective symplectic structures.
\end{lem}

\end{cBoxB}

\begin{figure}
\centering{}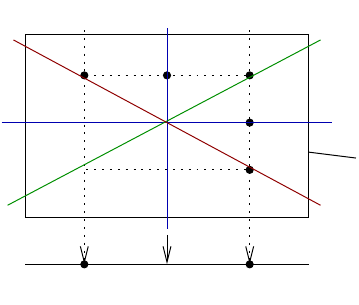\caption{\label{fig:N_K_V_H}Illustration of Lemma \ref{lem:KN} about the
vector bundle $\widetilde{\mathrm{Ker}\mathcal{A}}\subset T_{\Sigma}T^{*}M\protect\overset{\pi}{\rightarrow}M$.
We have the linear map $\widetilde{\mathrm{Ker}\mathcal{A}}\protect\overset{d\pi}{\rightarrow}\mathrm{Ker}\mathcal{A}\subset TM$
and $\widetilde{\mathrm{Ker}\mathcal{A}}$ is a sum of two symplectic
subspaces $K\oplus N$ and two Lagrangian subspaces $H\oplus V$.
This structure is preserved by the dynamics.}
\end{figure}

\begin{proof}
We have
\begin{align}
\widetilde{\mathrm{Ker}\mathcal{A}} & \eq{\ref{eq:Ker_tilde_A}}\left(d\pi^{-1}\left(\mathrm{Ker}\mathcal{A}\right)\right)\cap\mathrm{Ker}\left(d\boldsymbol{\omega}\right)\eq{\ref{eq:lemma_KerA}}\left(\tilde{E}_{0}^{*}\right)^{\perp_{\Omega}}\cap\mathrm{Ker}\left(d\boldsymbol{\omega}\right)\nonumber \\
 & \eq{\ref{eq:decomp_K_K0_N}}\left(K\oplus\tilde{E}_{0}^{*}\oplus N\right)\cap\mathrm{Ker}\left(d\boldsymbol{\omega}\right)\label{eq:sum}
\end{align}
We have $d\boldsymbol{\omega}\left(N\right)\eq{\ref{eq:eq5}}\Omega\left(\tilde{X},N\right)\eq{\mathrm{def}}\Omega\left(\tilde{E}_{0},N\right)\eq{\ref{eq:decomp_K_K0_N}}0$.
This gives $N\subset\mathrm{Ker}\left(d\boldsymbol{\omega}\right)$
and similarly $K\subset\mathrm{Ker}\left(d\boldsymbol{\omega}\right)$.
We have $\tilde{E}_{0}^{*}=\mathbb{R}\mathcal{E}$ and 
\[
d\boldsymbol{\omega}\left(\mathcal{E}\right)\eq{\ref{eq:eq5}}\Omega\left(\tilde{X},\mathcal{E}\right)\eq{\ref{eq:def_Euler}}\theta\left(\tilde{X}\right)\eq{\ref{eq:omega}}-\omega\neq0,
\]
hence $\tilde{E}_{0}^{*}\cap\mathrm{Ker}\left(d\boldsymbol{\omega}\right)=0$.
We deduce the decomposition $\widetilde{\mathrm{Ker}\mathcal{A}}=K\overset{\perp_{\Omega}}{\oplus}N$
in (\ref{eq:K_N_HA_VA}).

On the space $K$ we have
\begin{align*}
\Omega_{/K} & \eq{\ref{eq:def_projector_p},\ref{eq:TomegaA}}p^{\circ}\Omega_{/K}\eq{\ref{eq:def_projector_p}}\left(\left(d\left(\omega\mathcal{A}\right)\right)\circ\left(d\pi\right)\right)^{\circ}\Omega_{/K}\\
 & =\left(d\pi\right)^{\circ}\left(d\left(\omega\mathcal{A}\right)\right)^{\circ}\Omega_{/K}\eq{\ref{eq:dA_Omega}}\left(d\pi\right)^{\circ}d\left(\omega\mathcal{A}\right)=\omega\left(d\pi\right)^{\circ}\left(d\mathcal{A}\right).
\end{align*}
This gives (\ref{eq:dpi_K}) since $d\mathcal{A}$ is bilinear. For
$u\in\mathrm{Ker}\mathcal{A}$, we have $\left(d\pi\right)\left(\overline{V}\left(u\right)\right)\eq{\ref{eq:def_H_V}}\frac{1}{2}\left(u-u\right)=0$,
hence $V$ is vertical hence $\Omega-$Lagrangian. For the map $\overline{V}:\mathrm{Ker}\mathcal{A}\rightarrow V$,
we have $\overline{V}\circ d\pi_{V\rightarrow\mathrm{Ker}\mathcal{A}}=\mathrm{Id}_{/V}$
and $V$ is $\Omega-$Lagrangian hence
\[
0=\Omega_{/V}=\left(d\pi_{V\rightarrow\mathrm{Ker}\mathcal{A}}\right)^{\circ}\circ\left(\overline{V}\right)^{\circ}\Omega\eq{\ref{eq:def_H_V}}\frac{1}{2}\left(d\pi_{V\rightarrow\mathrm{Ker}\mathcal{A}}\right)^{\circ}\circ\left(\left(\overline{K}\right)^{\circ}\left(\Omega\right)+\left(-\mathrm{Id}\right)^{\circ}\circ\left(\overline{N}\right)^{\circ}\left(\Omega\right)\right).
\]
Since $\overline{K}=\left(d\pi\right)_{\mathrm{Ker}\mathcal{A}\rightarrow K}^{-1}$
and $\left(-\mathrm{Id}\right)^{\circ}=\mathrm{Id}$ on bilinear forms
we deduce
\[
\left(\overline{N}\right)^{\circ}\left(\Omega\right)=-\left(\overline{K}\right)^{\circ}\left(\Omega\right)=-\left(d\pi\right)_{\mathrm{Ker}\mathcal{A}\rightarrow K}^{-1}\left(\Omega\right)\eq{\ref{eq:dpi_K}}-\omega d\mathcal{A},
\]
giving (\ref{eq:dpi_N}) because $\overline{N}=\left(d\pi\right)_{\mathrm{Ker}\mathcal{A}\rightarrow N}^{-1}$.
For the map $\overline{H}:\mathrm{Ker}\mathcal{A}\rightarrow H$,
we have $\overline{H}\circ d\pi_{H\rightarrow\mathrm{Ker}\mathcal{A}}=\mathrm{Id}_{/H}$
hence
\begin{align*}
\Omega_{/H} & =\left(d\pi_{H\rightarrow\mathrm{Ker}\mathcal{A}}\right)^{\circ}\circ\left(\overline{H}\right)^{\circ}\Omega\eq{\ref{eq:def_H_V}}\frac{1}{2}\left(d\pi_{H\rightarrow\mathrm{Ker}\mathcal{A}}\right)^{\circ}\circ\left(\left(\overline{K}\right)^{\circ}\left(\Omega\right)+\left(\overline{N}\right)^{\circ}\left(\Omega\right)\right)\\
 & \eq{\ref{eq:dpi_K},\ref{eq:dpi_N}}\frac{1}{2}\left(d\pi_{H\rightarrow\mathrm{Ker}\mathcal{A}}\right)^{\circ}\left(\omega d\mathcal{A}-\omega d\mathcal{A}\right)=0.
\end{align*}
Hence $H$ is $\Omega-$Lagrangian.
\end{proof}

\subsubsection{Modified metric $\tilde{g}$ on the set $\Sigma$}

In (\ref{eq:decomp_K_K0_N}), the decomposition $K\overset{\perp_{\Omega}}{\oplus}K_{0}\overset{\perp_{\Omega}}{\oplus}N$
is orthogonal for the symplectic form $\Omega$. But a priori it is
not orthogonal for the metric $g$ given in (\ref{eq:metric_g_tilde_in_coordinates}).
Since this last property would be useful later, we will modify the
metric $g$ according to the next lemma to get $g-$orthogonality
as well.

\begin{cBoxB}{}
\begin{lem}
\label{thm:For-every-point}On the  set $\Sigma$, we construct a
modified metric $\tilde{g}$ from the existing metric $g$ as follows.
Let $\tilde{g}_{/N}=g_{/N}$, $\tilde{g}_{/K}=g_{/K}$, $\tilde{g}_{/K_{0}}=g_{/K_{0}}$
and assume $K\overset{\perp_{\tilde{g}}}{\oplus}K_{0}\overset{\perp_{\tilde{g}}}{\oplus}N$.
Then
\begin{enumerate}
\item $\tilde{g}$ is compatible with $\Omega$.
\item The decomposition (\ref{eq:decomp_K_K0_N}) is orthogonal with respect
to the metric $\tilde{g}$ (and w.r.t. $\Omega$ as before):
\begin{equation}
T_{\Sigma}T^{*}M=K\overset{\perp_{\Omega,\tilde{g}}}{\oplus}K_{0}\overset{\perp_{\Omega,\tilde{g}}}{\oplus}N\label{eq:decomp_K_N}
\end{equation}
\item The modified metric $\tilde{g}$ is equivalent to the metric $g$,
uniformly with respect to $\rho\in\Sigma$ for $\left|\omega\right|\geq1$.
\end{enumerate}
\end{lem}

\end{cBoxB}

\begin{proof}
By construction we have claim 1,2. In local coordinates $\left(x,z,\xi,\omega\right)$
used in section \ref{subsec:Metric--on}, and for some $\omega_{0}\gg1$,
consider the change of coordinates $x'=\omega_{0}^{1/2}x,z'=z$, $\xi'=\omega_{0}^{-1/2}\xi$,
$\omega'=\omega$ so that $\Omega=dx'\wedge d\xi'+dz'\wedge d\omega'$.
Let $\rho=\omega\mathcal{A}\left(m\right)\in\Sigma$ with $\left|\omega\right|\asymp\omega_{0}$.
From (\ref{eq:def_delta}) we have $\delta^{\perp}\left(\rho\right)\asymp\left|\rho\right|^{-1/2}\asymp\left|\omega\right|^{-1/2}\asymp\omega_{0}^{-1/2}$.
Then $g\underset{(\ref{eq:metric_g_tilde_in_coordinates})}{\asymp}dx'^{2}+d\xi'^{2}+dz'^{2}+d\omega'^{2}$,
where the equivalence are uniform w.r.t. $\left|\omega\right|\geq1$.
Using this new system of coordinates we get that the subspaces $K,K_{0},N$
are uniformly apart form each over with respect to $\rho\in\Sigma$,
i.e. for the pair of spaces $\left(K,N\right)$, we have $\exists0<C<1,$
$\forall\rho\in\Sigma$,$\forall u\in K_{\rho}$,$\forall v\in N_{\rho}$,
\begin{equation}
\left|g_{\rho}\left(u,v\right)\right|<C\left\Vert u\right\Vert _{g_{\rho}}\left\Vert v\right\Vert _{g_{\rho}},\label{eq:g_C}
\end{equation}
and similarly for other pair of space $\left(K,K_{0}\right),\left(K_{0},N\right)$.
This gives Claim 3.
\end{proof}
For simplicity of notations, from now on, \textbf{we assume that the
metric $g$ on $T^{*}M$ has the properties 1,2,3 }of Lemma \ref{thm:For-every-point}
on $\Sigma$.

\subsubsection{Factorization formula for $\tilde{\mathrm{Op}}\left(\left(d\tilde{\phi}^{t}\right)_{/\Sigma}\right)$}

We consider now the operator $\tilde{\mathrm{Op}}\left(d\tilde{\phi}^{t}\right)$
that appears on the right hand side of (\ref{eq:etX}). Since $\Sigma$
in invariant by $\tilde{\phi}^{t}$ and $\tilde{\mathrm{Op}}\left(d\tilde{\phi}^{t}\right)$
is fiber-wise, we can consider its restriction to $\mathcal{S}\left(T_{\Sigma}T^{*}M\right)$.
It is denoted and given by
\[
\tilde{\mathrm{Op}}\left(\left(d\tilde{\phi}^{t}\right)_{/\Sigma}\right)\eq{\ref{eq:Op_tilde}}\left(\Upsilon\left(d\tilde{\phi}^{t}\right)\right)^{1/2}\,\mathcal{P}\left(d\tilde{\phi}^{t}\right)^{-\circ}\mathcal{P}\quad:\mathcal{S}\left(T_{\Sigma}T^{*}M\right)\rightarrow\mathcal{S}\left(T_{\Sigma}T^{*}M\right)
\]
The decomposition (\ref{eq:decomp_K_K0_N}) is orthogonal for both
the symplectic form and the modified metric $\tilde{g}$, and preserved
by the differential of the flow, i.e.
\begin{equation}
\left(d\tilde{\phi}^{t}\right)_{/\Sigma}=\underbrace{d\tilde{\phi}^{t}{}_{K}+d\tilde{\phi}^{t}{}_{K_{0}}}_{d\tilde{\phi}_{T\Sigma}^{t}}+d\tilde{\phi}^{t}{}_{N}.\label{eq:dPhi_Sigma}
\end{equation}

\begin{rem}
\label{rem:Eq.-()-implies}Eq. (\ref{eq:decomp_K_K0_N}) implies that
on the  set $\Sigma$ we have a natural identification
\begin{equation}
\mathcal{S}\left(T_{\Sigma}T^{*}M\right)=\mathcal{S}\left(K\right)\otimes_{\Sigma}\mathcal{S}\left(K_{0}\right)\otimes_{\Sigma}\mathcal{S}\left(N\right).\label{eq:K_N}
\end{equation}
where the right hand side is a notation for the space of sections
$\mathcal{S}\left(\Sigma;\mathcal{S}\left(K\right)\otimes\mathcal{S}\left(K_{0}\right)\otimes\mathcal{S}\left(N\right)\right)$,
i.e. the tensor product is fiber-wise over $\Sigma$ and not global.
\end{rem}

From invariant Lagrangian decomposition (\ref{eq:K_N_HA_VA}), (\ref{eq:def_K0})
and Lemma \ref{lem:We-have-} we have
\[
\mathcal{P}=\mathcal{P}_{K}\otimes\mathcal{P}_{K_{0}}\otimes\mathcal{P}_{N}.
\]
Then (\ref{eq:def_op_tilde}) is used to define $\tilde{\mathrm{Op}}\left(d\tilde{\phi}_{K}^{t}\right)\eq{\ref{eq:def_op_tilde}}\left(\Upsilon\left(d\tilde{\phi}_{K}^{t}\right)\right)^{1/2}\,\mathcal{P}_{K}\left(d\tilde{\phi}_{K}^{t}\right)^{-\circ}\mathcal{P}_{K}$
and similarly $\tilde{\mathrm{Op}}\left(d\tilde{\phi}_{K_{0}}^{t}\right),\tilde{\mathrm{Op}}\left(d\tilde{\phi}_{N}^{t}\right)$.

\begin{cBoxB}{}
\begin{thm}[\textbf{Factorization formula}]
\textbf{\label{thm:On-the-set}}On the set $\Sigma$, with respect
to the identification (\ref{eq:K_N}) we have
\begin{equation}
\tilde{\mathrm{Op}}\left(\left(d\tilde{\phi}^{t}\right)_{/\Sigma}\right)=\underbrace{\tilde{\mathrm{Op}}\left(d\tilde{\phi}_{K}^{t}\right)\otimes_{\Sigma}\tilde{\mathrm{Op}}\left(d\tilde{\phi}_{K_{0}}^{t}\right)}_{\tilde{\mathrm{Op}}\left(d\tilde{\phi}_{T\Sigma}^{t}\right)}\otimes_{\Sigma}\tilde{\mathrm{Op}}\left(d\tilde{\phi}_{N}^{t}\right).\label{eq:fact}
\end{equation}
\end{thm}

\end{cBoxB}

\begin{proof}
Use previous factorization and (\ref{eq:product}) to get (\ref{eq:fact}).
\end{proof}

\subsection{Equivalent family of operators}

As we have said, we want to describe only the restriction of the operator
$e^{tX}$ to a small neighborhood of the set $\Sigma\subset T^{*}M$,
i.e. to consider Schwartz kernel of operators in this neighborhood
and not on the whole manifold $T^{*}M$. For this purpose, we introduce
here an equivalence relation between operators denoted $\approx$
that we will often use later. The equivalence $A\approx B$ will mean
roughly that the two operators $A,B$ are approximately the same at
high frequency and near the set $\Sigma$.

Let $\sigma>0$ that will be chosen large enough later. Define the
following characteristic function $\chi_{\Sigma,\sigma}:T^{*}M\rightarrow\mathbb{R}^{+}$
as follows. We write $\rho=\rho_{*}+\rho_{0}\in T^{*}M$ with transverse
component $\rho_{*}\in\mathrm{Ker}X$ and frequency component $\rho_{0}\in E_{0}^{*}$,
see (\ref{eq:decomp_KerX_E0*}). We set
\begin{align}
\chi_{\Sigma,\sigma}\left(\rho\right) & :=\boldsymbol{1}_{\left\{ \left\Vert \rho_{*}\right\Vert _{g_{\rho}}\leq\sigma\right\} },\label{eq:def_Lambda}
\end{align}
and
\begin{equation}
\mathrm{Op}\left(\chi_{\Sigma,\sigma}\right):=\mathcal{T}^{\dagger}\chi_{\Sigma,\sigma}\mathcal{T}\quad:C^{\infty}\left(M\right)\rightarrow C^{\infty}\left(M\right)\label{eq:def_op_Lambda}
\end{equation}
be the corresponding P.D.O. operator, as defined in \cite[def. 4.28]{faure_tsujii_Ruelle_resonances_density_2016}
that extracts components at distance less than $\sigma$ from the
set $\Sigma$.

Recall $\Psi_{\tilde{\phi}^{t}}^{m}$ that has been defined in Definition
\ref{def:Let--be}.

\begin{cBoxA}{}
\begin{defn}
\label{def:For-two-family}For two operators $A,B:C^{\infty}\left(M\right)\rightarrow\mathcal{S}'\left(M\right)$
we write
\begin{equation}
A\approx B\label{eq:def_approx}
\end{equation}
if there exists $m<0$, $t\in\mathbb{R}$, such that for any $\sigma>0$,
\[
\mathrm{Op}\left(\chi_{\Sigma,\sigma}\right)\left(A-B\right)\mathrm{Op}\left(\chi_{\Sigma,\sigma}\right)\in\Psi_{\tilde{\phi}^{t}}^{m}.
\]
\end{defn}

\end{cBoxA}

\begin{rem}
From $m<0$, we recall that consequently we have the estimate (\ref{eq:result_of_Shur})
for $R=\mathrm{Op}\left(\chi_{\Sigma,\sigma}\right)\left(A-B\right)\mathrm{Op}\left(\chi_{\Sigma,\sigma}\right)$,
where the constant $C_{t}$ depends on $\sigma$ also.
\end{rem}

\subsection{Description of the operator $e^{tX}$ near the  set $\Sigma$ with
the bundle $N$}

We first define some operators. The construction that we pursue below
is represented on Figure \ref{fig:N_Sigma}. Using the metric $g$
on $T^{*}M$, and considering the sub-bundle $N\subset T_{\Sigma}T^{*}M$,
we have the exponential map 
\begin{equation}
\exp_{N}:N\rightarrow T^{*}M.\label{eq:exp_N}
\end{equation}
Due to the slow variation of the metric $g$ at high frequencies,
given in (\ref{eq:g_moderate and temperate}), we have that for any
$0\leq\mu<1$, and any frequency $\left|\omega\right|\geq\omega_{0}$
with $\omega_{0}$ large enough, $\exp_{N}$ is a diffeomorphism on
the neighborhood $\left\Vert v\right\Vert _{g}\leq\left\langle \omega\right\rangle ^{\mu/2}$.
To express this, let $0<\mu<1$ and $\omega_{0}\geq1$ for this property.
We introduce the following cut-off function in frequency $\chi_{\Sigma}^{\mu}:N\rightarrow\left[0,1\right]$
defined as follows. For $\rho\in\Sigma$, $v\in N_{\rho}$ and $\omega=\boldsymbol{\omega}\left(\rho\right)$,
\begin{equation}
\chi_{\Sigma}^{\mu}\left(v\right)=\boldsymbol{1}_{\left\{ \left\Vert v\right\Vert _{g}\leq\left\langle \omega\right\rangle ^{\mu/2},\left|\omega\right|\geq\omega_{0}\right\} }.\label{eq:def_Chi_mu}
\end{equation}
The multiplication operator by $\chi_{\Sigma}^{\mu}$ is also denoted
$\chi_{\Sigma}^{\mu}:\mathcal{S}\left(N\right)\rightarrow\mathcal{S}'\left(N\right)$.
Later, in Theorem \ref{thm:Composition-formula-For}, we will need
to take $\mu$ small enough depending in particular on the Hölder
exponents of the Anosov dynamics.

Similarly to (\ref{eq:def_twisted_pull_back}), we define the twisted
pull back and push forward operators
\[
\chi_{\Sigma}^{\mu}\widetilde{\left(\exp_{N}\right)^{\circ}}:\mathcal{S}\left(T^{*}M\right)\rightarrow\mathcal{S}'\left(N\right),
\]
\[
\widetilde{\left(\exp_{N}^{-1}\right)^{\circ}}\chi_{\Sigma}^{\mu}:\mathcal{S}\left(N\right)\rightarrow\mathcal{S}'\left(T^{*}M\right).
\]
Namely, for $u\in\mathcal{S}\left(T^{*}M\right)$, $w\in\mathcal{S}\left(N\right)$,
let $\rho\in T^{*}M$ with $\rho=\exp_{N}\left(v_{2}\right)$, $v_{2}\in N_{\rho_{2}}$,
$\rho_{2}\in\Sigma$, $\left|\boldsymbol{\omega}\left(\rho_{2}\right)\right|\geq\omega_{0}$,
and $\left\Vert v_{2}\right\Vert _{g}\leq\left\langle \boldsymbol{\omega}\left(\rho_{2}\right)\right\rangle ^{\mu/2}$,
\begin{equation}
\left(\widetilde{\left(\exp_{N}\right)^{\circ}}u\right)\left(v_{2}\right):=e^{i\varphi\left(v_{2}\right)}u\left(\rho\right),\label{eq:def_twisted_push_forward-1}
\end{equation}
\begin{equation}
\left(\widetilde{\left(\exp_{N}^{-1}\right)^{\circ}}w\right)\left(\rho\right):=e^{-i\varphi\left(v_{2}\right)}w\left(v_{2}\right),\label{eq:def_twisted_push_forward}
\end{equation}
with phase function $\varphi$ defined in (\ref{eq:def_phase_phi}),
to pass from vertical gauge to radial gauge. In the next definition,
we combine these previous operators together with $\mathcal{T}:C^{\infty}\left(M\right)\rightarrow\mathcal{S}\left(T^{*}M\right)$
in (\ref{eq:def_T}).

\begin{cBoxA}{}
\begin{defn}
\label{def:We-define-the}We define the operators

\begin{equation}
\mathcal{T}_{N}:=\chi_{\Sigma}^{\mu}\widetilde{\left(\exp_{N}\right)^{\circ}}\mathcal{T}\quad:C^{\infty}\left(M\right)\rightarrow\mathcal{S}'\left(N\right)\label{eq:def_T_N}
\end{equation}
\begin{equation}
\mathcal{T}_{N}^{\Delta}:=\mathcal{T}^{\dagger}\widetilde{\left(\exp_{N}^{-1}\right)^{\circ}}\chi_{\Sigma}^{\mu}\quad:\mathcal{S}\left(N\right)\rightarrow C^{\infty}\left(M\right).\label{eq:def_T_N_Delta}
\end{equation}
\end{defn}

\end{cBoxA}

Using the operators (\ref{eq:def_T_N}),(\ref{eq:def_T_N_Delta}),
the next theorem gives a good approximation of the operator $e^{tX}$
in a neighborhood of the set $\Sigma$ in terms of the differential
map $d\tilde{\phi}_{N}^{t}$ on the normal bundle $N$. We will use
the notation $\approx$ defined in (\ref{eq:def_approx}) that compare
operators near $\Sigma$. We also use the shortened notation
\begin{equation}
\Upsilon_{t}:=\Upsilon\left(d\tilde{\phi}_{T\Sigma}^{t}\right),\label{eq:def_Upsilon_t}
\end{equation}
that is the metaplectic correction (\ref{eq:def_d}) of the differential
restricted to $T\Sigma$, $d\tilde{\phi}_{T\Sigma}^{t}=d\tilde{\phi}^{t}{}_{K}+d\tilde{\phi}^{t}{}_{K_{0}}$
that appears in (\ref{eq:dPhi_Sigma}). In the next theorem we view
$\Upsilon_{t}^{1/2}>0$ as a multiplication operator over $\Sigma$.

\begin{cBoxB}{}
\begin{thm}
\label{Thm:approx_exptX_from_N}For any $t\in\mathbb{R}$, we have
\begin{equation}
e^{tX}\approx\mathcal{T}_{N}^{\Delta}\,\Upsilon_{t}^{1/2}\tilde{\mathrm{Op}}\left(d\tilde{\phi}_{N}^{t}\right)\mathcal{T}_{N}\label{eq:exp_tX_N}
\end{equation}
\end{thm}

\end{cBoxB}

\begin{rem}
Let us compare formula (\ref{eq:exp_tX_N}) with (\ref{eq:etX}).
Both formula express the pull back operator $e^{tX}$ from some quantum
(metaplectic) operator $\tilde{\mathrm{Op}}\left(d\tilde{\phi}^{t}\right)$
or $\tilde{\mathrm{Op}}\left(d\tilde{\phi}_{N}^{t}\right)$ acting
on some vector bundle over $T^{*}M$. The difference is that (\ref{eq:etX})
is valid microlocally on the whole $T^{*}M$ (but at high frequencies)
and general to any non vanishing vector fields $X$, whereas (\ref{eq:exp_tX_N})
is specific to contact (or Reeb) vector field and valid microlocally
near $\Sigma\subset T^{*}M$ (and high frequencies).
\end{rem}

~
\begin{rem}
As a particular case, taking $t=0$ in (\ref{eq:exp_tX_N}) gives
that
\begin{equation}
\mathcal{T}_{N}^{\Delta}\mathcal{T}_{N}\approx\mathrm{Id}.\label{eq:TDelta_T_Id}
\end{equation}
\end{rem}

\begin{proof}
As we already did in the proof of Theorem \ref{thm:propagation_singularities_2},
we use $\mathcal{T}:C^{\infty}\left(M\right)\rightarrow\mathcal{S}\left(T^{*}M\right)$
to lift the analysis to $T^{*}M$ and consider the Schwartz kernel
\[
\tilde{R}:=\left|\langle\delta_{\rho'}|\mathcal{T}\left(e^{tX}-\mathcal{T}_{N}^{\Delta}\Upsilon_{t}^{1/2}\tilde{\mathrm{Op}}\left(d\tilde{\phi}_{N}^{t}\right)\mathcal{T}_{N}\right)\mathcal{T}^{\dagger}\delta_{\rho}\rangle\right|,
\]
that we will estimate for points $\rho',\rho\in T^{*}M$ close to
the set $\Sigma$ as needed by Definition \ref{def:For-two-family}.
For this, let $\sigma>0$. We write $\rho=\rho_{*}+\rho_{0}$ with
transverse component $\rho_{*}\in\mathrm{Ker}X$ and frequency component
$\rho_{0}\in E_{0}^{*}$, see (\ref{eq:decomp_KerX_E0*}), similarly
for $\rho'$ and we assume that
\begin{equation}
\left\Vert \rho_{*}\right\Vert _{g_{\rho}}\leq\sigma,\quad\left\Vert \rho'_{*}\right\Vert _{g_{\rho}}\leq\sigma.\label{eq:assumpt_rho}
\end{equation}
We also assume that points $\left(\rho,\rho'\right)\in T^{*}M\times T^{*}M$
are ``near'' the graph of $\tilde{\phi}^{t}$, i.e. that (\ref{eq:near_the_graph})
holds true. Otherwise the proof of Theorem \ref{thm:propagation_singularities_2},
(step A), shows that far from the graph, the Schwartz kernel is totally
negligible.

As before, the strategy is to approximate the dynamical operator $e^{tX}$
to its linear part, approximate the metric to an Euclidean metric,
and then check that formula holds true exactly in this linear and
Euclidean setting. For this, we use local coordinates in $\mathbb{R}^{2\left(2d+1\right)}$
on $T^{*}M$ and use the operator $\tilde{\mathrm{Op}}_{\rho}\left(d\tilde{\phi}_{\rho}^{t}\right)$
and phase $\varphi_{\rho}$ that have already been defined in (\ref{eq:def_R2}).
We introduce intermediary operators

\[
O_{1}:=\mathcal{T}e^{tX}\mathcal{T}^{\dagger}\qquad O_{2}:=e^{-i\varphi_{\tilde{\phi}^{t}\left(\rho\right)}}\tilde{\mathrm{Op}}_{\rho}\left(d\tilde{\phi}_{\rho}^{t}\right)e^{i\varphi_{\rho}},
\]
\[
O_{3}:=e^{-i\varphi_{\tilde{\phi}^{t}\left(\rho\right)}}\mathcal{P}_{\tilde{\phi}^{t}\left(\rho\right)}\widetilde{\left(\exp_{N}^{-1}\right)^{\circ}}\Upsilon_{t}^{1/2}\tilde{\mathrm{Op}}_{\rho}\left(d\tilde{\phi}_{N}^{t}\right)\widetilde{\left(\exp_{N}\right)^{\circ}}\mathcal{P}_{\rho}e^{i\varphi_{\rho}},
\]
\begin{align*}
O_{4} & :=\mathcal{T}\mathcal{T}^{\dagger}\widetilde{\left(\exp_{N}^{-1}\right)^{\circ}}\chi_{\Sigma}^{\mu}\Upsilon_{t}^{1/2}\tilde{\mathrm{Op}}\left(d\tilde{\phi}_{N}^{t}\right)\chi_{\Sigma}^{\mu}\widetilde{\left(\exp_{N}\right)^{\circ}}\mathcal{T}\mathcal{T}^{\dagger}\\
 & \eq{\ref{eq:def_T_N},\ref{eq:def_T_N_Delta}}\mathcal{T}\mathcal{T}_{N}^{\Delta}\,\Upsilon_{t}^{1/2}\tilde{\mathrm{Op}}\left(d\tilde{\phi}_{N}^{t}\right)\mathcal{T}_{N}\mathcal{T}^{\dagger},
\end{align*}
and their differences
\begin{equation}
R_{1}:=\left|\langle\delta_{\rho'}|\left(O_{1}-O_{2}\right)\delta_{\rho}\rangle\right|,\qquad R_{2}:=\left|\langle\delta_{\rho'}|\left(O_{2}-O_{3}\right)\delta_{\rho}\rangle\right|,\label{eq:def_R2-1}
\end{equation}
\[
R_{3}:=\left|\langle\delta_{\rho'}|\left(O_{3}-O_{4}\right)\delta_{\rho}\rangle\right|.
\]
So we have
\[
\tilde{R}\leq R_{1}+R_{2}+R_{3}.
\]
For the term $R_{1}$, in the proof of Theorem \ref{thm:propagation_singularities_2},
(step B), we have already shown that for any $0<\lambda<\frac{1}{2}$,
\[
R_{1}\leq C_{N,t}\left\langle \mathrm{dist}_{g}\left(\rho',\tilde{\phi}^{t}\left(\rho\right)\right)\right\rangle ^{-N}\left\langle \left|\rho\right|\right\rangle ^{-\frac{1}{2}+\lambda}.
\]
For the term $R_{2}$, we decompose in $O_{2}$ the term
\[
\tilde{\mathrm{Op}}_{\rho}\left(d\tilde{\phi}_{\rho}^{t}\right)\eq{\ref{eq:fact}}\left(\tilde{\mathrm{Op}}_{\rho}\left(d\tilde{\phi}_{K}^{t}\right)\otimes\tilde{\mathrm{Op}}_{\rho}\left(d\tilde{\phi}_{N}^{t}\right)\right)\otimes\tilde{\mathrm{Op}}_{\rho}\left(d\tilde{\phi}_{K_{0}}^{t}\right).
\]
For the transverse part $\tilde{\Phi}=d\tilde{\phi}_{K}^{t}+d\tilde{\phi}_{N}^{t}$,
Lemma \ref{lem:decomp} in the appendix gives that (and this is the
main relation that underlies Theorem \ref{Thm:approx_exptX_from_N})
\[
\tilde{\mathrm{Op}}_{\rho}\left(d\tilde{\phi}_{K}^{t}\right)\otimes\tilde{\mathrm{Op}}_{\rho}\left(d\tilde{\phi}_{N}^{t}\right)=\tilde{\mathrm{Op}}_{\rho}\left(\tilde{\Phi}\right)\eq{\ref{eq:expression_Phi_circ-1}}\mathcal{P}_{\tilde{\phi}^{t}\left(\rho\right)}\widetilde{\left(\exp_{N}^{-1}\right)^{\circ}}\left(\Upsilon\left(d\tilde{\phi}_{K}^{t}\right)\right)^{1/2}\tilde{\mathrm{Op}}_{\rho}\left(d\tilde{\phi}_{N}^{t}\right)\widetilde{\left(\exp_{N}\right)^{\circ}}\mathcal{P}_{\rho}.
\]
If we add the neutral component $\tilde{\mathrm{Op}}_{\rho}\left(d\tilde{\phi}_{K_{0}}^{t}\right)$
that has no effect this gives
\[
\tilde{\mathrm{Op}}_{\rho}\left(d\tilde{\phi}_{\rho}^{t}\right)=\mathcal{P}_{\tilde{\phi}^{t}\left(\rho\right)}\widetilde{\left(\exp_{N}^{-1}\right)^{\circ}}\Upsilon_{t}^{1/2}\tilde{\mathrm{Op}}_{\rho}\left(d\tilde{\phi}_{N}^{t}\right)\widetilde{\left(\exp_{N}\right)^{\circ}}\mathcal{P}_{\rho}.
\]
This shows that $R_{2}\eq{\ref{eq:def_R2-1}}0$.Finally for the term
$R_{3}$, we add the cut-off $\chi_{\Sigma}^{\mu}$ that has no effect
from our assumptions on points $\rho,\rho'$, we pass from the Euclidean
metric to the metric $g$ repeating arguments with the slow variation
of the metric (\ref{eq:g_moderate and temperate}) and get that $R_{3}\leq C_{N,t}\left\langle \mathrm{dist}_{g}\left(\rho',\tilde{\phi}^{t}\left(\rho\right)\right)\right\rangle ^{-N}\left\langle \left|\rho\right|\right\rangle ^{-\frac{1}{2}+\lambda}$
for any $0<\lambda<\frac{1}{2}$, as well. 

From these estimates and assumption (\ref{eq:assumpt_rho}) on $\rho,\rho'$,
we can add the cut-off $\mathrm{Op}\left(\chi_{\Sigma,\sigma}\right)$
and deduce that for any $0<\lambda<\frac{1}{2}$, 
\[
\mathrm{Op}\left(\chi_{\Sigma,\sigma}\right)\left(e^{tX}-\mathcal{T}_{N}^{\Delta}\,\Upsilon_{t}^{1/2}\tilde{\mathrm{Op}}\left(d\tilde{\phi}_{N}^{t}\right)\mathcal{T}_{N}\right)\mathrm{Op}\left(\chi_{\Sigma,\sigma}\right)\in\Psi_{\tilde{\phi}^{t}}^{-\frac{1}{2}+\lambda}.
\]
From Definition \ref{def:For-two-family}, this gives (\ref{eq:exp_tX_N}). 
\end{proof}

\subsection{Description of the operator $e^{tX}$ near the  set $\Sigma$ with
the bundle $T_{\Sigma}T^{*}M$}

In this section we provide another expression for the operator $e^{tX}$
similar to (\ref{eq:exp_tX_N}) but sometimes more convenient, see
Remark \ref{rem:One-advantage-of} below. The little difference with
(\ref{eq:exp_tX_N}) is that we will use the bundle $T_{\Sigma}T^{*}M\eq{\ref{eq:decomp_K_K0_N}}T\Sigma\oplus N\rightarrow\Sigma$
instead of the sub-bundle $N\rightarrow\Sigma$.

Similarly to (\ref{eq:def_Chi_mu}), we introduce the following cut-off
function in frequency $\chi_{\Sigma}^{\mu}:T_{\Sigma}T^{*}M\rightarrow\left[0,1\right]$
defined as follows. For $v\in T_{\Sigma}T^{*}M$, and $\omega=\boldsymbol{\omega}\left(\pi\left(v\right)\right)$,
\begin{equation}
\chi_{\Sigma}^{\mu}\left(v\right)=\boldsymbol{1}_{\left\{ \left\Vert v\right\Vert _{g}\leq\left\langle \omega\right\rangle ^{\mu/2},\left|\omega\right|\geq\omega_{0}\right\} }.\label{eq:def_Chi_mu-1}
\end{equation}
The multiplication operator by $\chi_{\Sigma}^{\mu}$ is also denoted
$\chi_{\Sigma}^{\mu}:\mathcal{S}\left(T_{\Sigma}T^{*}M\right)\rightarrow\mathcal{S}'\left(T_{\Sigma}T^{*}M\right)$. 

We denote $r_{/\Sigma}:\mathcal{S}\left(TT^{*}M\right)\rightarrow\mathcal{S}\left(T_{\Sigma}T^{*}M\right)$
the restriction operator to the base space $\Sigma\subset T^{*}M$.
We denote 
\begin{equation}
r_{N}:\mathcal{S}\left(T_{\Sigma}T^{*}M\right)\rightarrow\mathcal{S}\left(N\right)\label{eq:def_rN}
\end{equation}
the restriction operator to the sub-bundle $N\subset T_{\Sigma}T^{*}M$.

\begin{cBoxA}{}
\begin{defn}
\label{def:We-define-the-2}We define the operators

\begin{equation}
\mathcal{T}_{\Sigma}:=\chi_{\Sigma}^{\mu}r_{/\Sigma}\left(\widetilde{\exp^{\circ}}\right)\mathcal{T}\quad:C^{\infty}\left(M\right)\rightarrow\mathcal{S}'\left(T_{\Sigma}T^{*}M\right)\label{eq:def_T_Sigma}
\end{equation}
\begin{equation}
\mathcal{T}_{\Sigma}^{\Delta}:=\mathcal{T}^{\dagger}\widetilde{\left(\exp_{N}^{-1}\right)^{\circ}}r_{N}\chi_{\Sigma}^{\mu}\quad:\mathcal{S}\left(T_{\Sigma}T^{*}M\right)\rightarrow C^{\infty}\left(M\right).\label{eq:def_T_sigma_Delta}
\end{equation}
\end{defn}

\end{cBoxA}

\begin{cBoxB}{}
\begin{thm}
\label{Thm:For-any-}For any $t\in\mathbb{R}$, we have
\begin{equation}
e^{tX}\approx\mathcal{T}_{\Sigma}^{\Delta}\tilde{\mathrm{Op}}\left(\left(d\tilde{\phi}^{t}\right)_{/\Sigma}\right)\mathcal{T}_{\Sigma}\label{eq:exp_tX_Lamda}
\end{equation}
\end{thm}

\end{cBoxB}

\begin{rem}
\label{rem:One-advantage-of}This remark is very similar to the previous
Remark \ref{rem:T*M_or_TT*M}. One advantage of (\ref{eq:exp_tX_Lamda})
compared to (\ref{eq:exp_tX_N}) is that it does not contain the time
dependent metaplectic correction $\Upsilon_{t}^{1/2}$, but contains
instead, through the decomposition (\ref{eq:fact}), the operator
$\tilde{\mathrm{Op}}\left(d\tilde{\phi}_{T\Sigma}^{t}\right)$ that
itself contains $\Upsilon_{t}^{1/2}$. Later in the proof of Theorem
\ref{thm:continuity_thm}, we will use and take advantage of some
unitary properties of this operator $\tilde{\mathrm{Op}}\left(d\tilde{\phi}_{T\Sigma}^{t}\right)$.
On the other hand, one advantage of (\ref{eq:exp_tX_N}) compare to
(\ref{eq:exp_tX_Lamda}) is that it deals directly with the normal
sub-bundle $N$ on which we aim to work, so it is easier at first
sight.
\end{rem}

\begin{proof}
The proof follows the same lines as the proof of Theorem \ref{Thm:approx_exptX_from_N}.
The only difference is that instead of using Lemma \ref{lem:decomp},
we use Lemma \ref{lem:For-the-symplectic} in the appendix for the
Euclidean and linear case.
\end{proof}

\subsection{\label{subsec:Discussion-about-emergence}Emergence of quantum mechanics
near the set $\Sigma$ for a general contact vector field}

A direct consequence of Theorem \ref{Thm:approx_exptX_from_N} is
the following theorem that shows that at any time $t$ the dynamics
of the pull back operator $e^{tX_{F}}$ is well approximated microlocally
near the set $\Sigma$ (at distance less than $\sigma$) and at high
frequencies $\omega$, by a quantum operator $\mathcal{T}_{N}^{\Delta}\,\Upsilon_{t}^{1/2}\tilde{\mathrm{Op}}\left(d\tilde{\phi}_{N}^{t}\right)\mathcal{T}_{N}$
(this is a F.I.O., compare with (\ref{eq:def_OIF})).

This result is less precise than Theorem \ref{thm:1_emergence_of_QM}
that gives an approximation of the full operator $e^{tX_{F}}$, whereas
here we have an approximation for a restriction to a micro-local neighborhood
of $\Sigma$. However this may be interesting because $\Sigma$ is
an invariant set for the dynamics. We may wonder if with some additional
generic assumptions, this result can be improved towards Theorem \ref{thm:1_emergence_of_QM}.

\begin{cBoxB}{}
\begin{thm}
\label{cor:emergence_contact}$\forall t\in\mathbb{R}$, $\exists m<0,$
$\forall\sigma>0$,$\exists C_{t,\sigma}>0$ $\forall\omega>0$,{\small{}
\[
\left\Vert \mathrm{Op}\left(\chi_{\Sigma,\sigma}\right)\left(e^{tX_{F}}-\mathcal{T}_{N}^{\Delta}\,\Upsilon_{t}^{1/2}\tilde{\mathrm{Op}}\left(d\tilde{\phi}_{N}^{t}\right)\mathcal{T}_{N}\right)\mathrm{Op}\left(\chi_{\Sigma,\sigma}\right)\left(\mathrm{Id}-\mathrm{Op}\left(\chi_{\omega}\right)\right)\right\Vert _{L^{2}\left(M;F\right)}\leq C_{t}\omega^{m}\underset{\omega\rightarrow+\infty}{\rightarrow}0.
\]
}
\end{thm}

\end{cBoxB}

\begin{proof}
From Theorem \ref{Thm:approx_exptX_from_N}, Definition \ref{def:For-two-family}
and Lemma \ref{Lem:A-useful-consequence}.
\end{proof}

\section{\label{sec:Micro-local-analysis-of-1-1}Micro-local analysis of a
contact Anosov vector field $X$ near $\Sigma=\mathbb{R}\mathcal{A}\backslash\left\{ 0\right\} $}

In this section, as a last step, we consider the model of interest
for this paper: a smooth contact Anosov vector field $X$ on a closed
manifold $M$. In other words, we add the hypothesis of Anosov dynamics
to the previous section \ref{sec:Micro-local-analysis-of-1}. We first
recall the precise definitions and notations.

\subsection{Definitions and notations}

\subsubsection{\label{subsec:Anosov-vector-field}Anosov vector field}

We make the hypothesis that $X$ is a smooth \textbf{Anosov vector
field} on a closed manifold $M$. This means that for every point
$m\in M$ the tangent vector $X\left(m\right)$ is non zero and we
have a continuous splitting of the tangent space
\begin{equation}
T_{m}M=E_{u}\left(m\right)\oplus E_{s}\left(m\right)\oplus\underbrace{E_{0}\left(m\right)}_{\mathbb{R}X\left(m\right)}\label{eq:decomp_TM}
\end{equation}
that is invariant under the action of the differential of the flow
map $d\phi^{t}:TM\rightarrow TM$ and there exist $\lambda_{\mathrm{min}}>0$,
$C>0$ and a smooth metric $g_{M}$ on $M$ such that

\begin{equation}
\forall t\geq0,\forall m\in M,\quad\left\Vert d\phi_{/E_{u}\left(m\right)}^{-t}\right\Vert _{g_{M}}\leq Ce^{-\lambda_{\mathrm{min}}t},\quad\left\Vert d\phi_{/E_{s}\left(m\right)}^{t}\right\Vert _{g_{M}}\leq Ce^{-\lambda_{\mathrm{min}}t}.\label{eq:hyperbolicity}
\end{equation}
See Figure \ref{fig:Anosov-flow.}. The linear subspace $E_{u}\left(m\right),E_{s}\left(m\right)\subset T_{m}M$
are called the unstable/stable spaces and the one dimensional space
$E_{0}\left(m\right):=\mathbb{R}X\left(m\right)$ is called the neutral
direction or flow direction. For a general Anosov vector field, the
maps $m\rightarrow E_{u}\left(m\right)$, $m\rightarrow E_{s}\left(m\right)$
and $m\rightarrow E_{u}\left(m\right)\oplus E_{s}\left(m\right)$
are only Hölder continuous with (a priori different) Hölder exponents
respectively \cite{hurder-90}: 
\begin{equation}
\beta_{u},\beta_{s},\beta_{0}\in]0,1],\quad\beta:=\mathrm{min}\left(\beta_{u},\beta_{s}\right).\label{eq:Holder_exp}
\end{equation}

\begin{figure}
\centering{}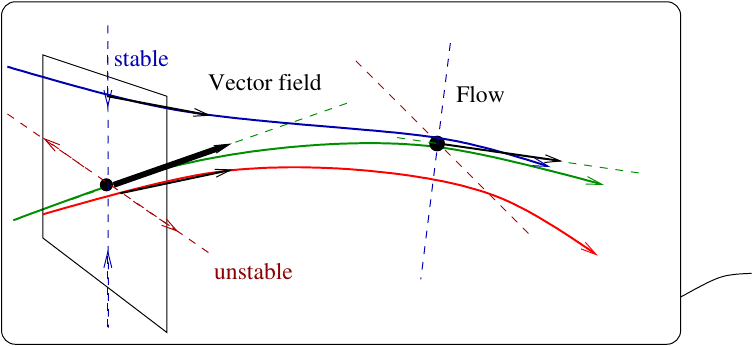\caption{\label{fig:Anosov-flow.}Anosov flow $\phi^{t}$ (in solid line) generated
by a vector field $X$ on a compact manifold $M$.}
\end{figure}
Let $\mathcal{A}\in C^{0}\left(M;T^{*}M\right)$ be the continuous
one form on $M$ called \textbf{Anosov one form }defined for every
$m\in M$ by the conditions
\begin{equation}
\mathcal{A}\left(m\right)\left(X\left(m\right)\right)=1\quad\text{and}\quad\mathrm{Ker}\left(\mathcal{A}\left(m\right)\right)=E_{u}\left(m\right)\oplus E_{s}\left(m\right).\label{eq:one_form}
\end{equation}
From this definition, $\mathcal{A}$ is preserved\footnote{Notice from (\ref{eq:hyperbolicity}), that $\mathcal{A}$ is the
unique continuous one form preserved by the flow with the normalization
condition $\mathcal{A}\left(X\right)=1$.} by the flow $\phi^{t}$ and the map $m\rightarrow\mathcal{A}\left(m\right)$
is Hölder continuous with exponent $\beta_{0}$.

\subsubsection{\label{subsec:Contact-Anosov-flow}Contact Anosov vector field}

In addition to section \ref{subsec:Anosov-vector-field} we assume
in this paper that the Anosov vector field $X$ is \textbf{\href{https://en.wikipedia.org/wiki/Contact_geometry}{contact}},
i.e. that the Anosov one form $\mathcal{A}$ defined in (\ref{eq:one_form}),
is a \textbf{smooth contact one form} on $M$. This means that the
distribution $E_{u}\left(m\right)\oplus E_{s}\left(m\right)$ is smooth
(i.e. $C^{\infty}$) w.r.t. $m\in M$ (hence with Hölder exponent
$\beta_{0}=1$) and that the linear space $\mathrm{Ker}\left(\mathcal{A}\left(m\right)\right)=E_{u}\left(m\right)\oplus E_{s}\left(m\right)$
endowed with the two form $d\mathcal{A}\left(m\right)$ is a linear
symplectic space for every $m\in M$.

Since $\mathcal{A}$ is invariant by the flow, for any $u_{1},u_{2}\in E_{s}\left(m\right)$
we have $\forall t\in\mathbb{R}$, $\left(d\mathcal{A}\right)\left(u_{1},u_{2}\right)=\left(d\left(\left(\phi^{t}\right)^{*}\mathcal{A}\right)\right)\left(u_{1},u_{2}\right)=\left(d\mathcal{A}\right)\left(d\phi^{t}\left(u_{1}\right),d\phi^{t}\left(u_{2}\right)\right)\underset{t\rightarrow+\infty}{\rightarrow}0$
hence $d\mathcal{A}_{/E_{s}}=0$ meaning that $E_{s}\left(m\right)$
is $d\mathcal{A}-$isotropic. Similarly $E_{u}\left(m\right)$ is
$d\mathcal{A}-$isotropic hence $E_{u}\left(m\right),E_{s}\left(m\right)$
are $d\mathcal{A}-$Lagrangian and
\begin{equation}
d:=\mathrm{dim}E_{u}\left(m\right)=\mathrm{dim}E_{s}\left(m\right),\label{eq:dim_d}
\end{equation}
hence $\mathrm{dim}M=2d+1$ is odd.
\begin{rem}
The vector field $X$ is the \textbf{\href{https://en.wikipedia.org/wiki/Contact_geometry}{Reeb vector field}}
of $\mathcal{A}$, i.e. $X$ is determined from $\mathcal{A}$ by
$\mathcal{A}\left(X\right)=1$ and\footnote{From Cartan formula and invariance of $\mathcal{A}$: $0=\mathcal{L}_{X}\mathcal{A}=\iota_{X}d\mathcal{A}+d\iota_{X}\mathcal{A}$
and $\iota_{X}\mathcal{A}=1$ give $\left(d\mathcal{A}\right)\left(X,.\right)=0$.} 
\begin{equation}
\left(d\mathcal{A}\right)\left(X,.\right)=0,\label{eq:dalpha_X}
\end{equation}
as the setting of section \ref{sec:Micro-local-analysis-of-1}. 
\end{rem}

\subsubsection{Dual decomposition}

The dual bundle decomposition of (\ref{eq:decomp_TM}) is
\begin{equation}
T^{*}M=E_{u}^{*}\oplus E_{s}^{*}\oplus E_{0}^{*}\label{eq:decomp*}
\end{equation}
with
\begin{align}
E_{u}^{*} & :=\left(E_{s}\oplus E_{0}\right)^{\perp}=\left\{ \rho\in T^{*}M,\quad\mathrm{Ker}\rho\supset E_{s}\oplus E_{0}\right\} ,\label{eq:def_Eu*_Es*}\\
E_{s}^{*} & :=\left(E_{u}\oplus E_{0}\right)^{\perp}=\left\{ \rho\in T^{*}M,\quad\mathrm{Ker}\rho\supset E_{u}\oplus E_{0}\right\} ,
\end{align}

\begin{align}
E_{0}^{*} & :=\left(E_{u}\oplus E_{s}\right)^{\perp}=\left\{ \rho\in T^{*}M,\quad\mathrm{Ker}\rho\supset E_{u}\oplus E_{s}\right\} \label{eq:def_E_0*-1}\\
 & \underset{(\ref{eq:one_form})}{=}\mathbb{R}\mathcal{A}=\left\{ \omega\mathcal{A}\left(m\right),\quad m\in M,\omega\in\mathbb{R}\right\} .
\end{align}
Hence
\[
\mathrm{dim}E_{u}^{*}\left(m\right)\eq{\ref{eq:dim_d}}\mathrm{dim}E_{s}^{*}\left(m\right)=d,\quad\mathrm{dim}E_{0}^{*}\left(m\right)=1.
\]
From (\ref{eq:Holder_exp}), the map $m\in M\rightarrow E_{u}^{*}\left(m\right)$
is $\beta_{s}$-Hölder continuous and $m\in M\rightarrow E_{s}^{*}\left(m\right)$
is $\beta_{u}$-Hölder continuous. However the maps $m\in M\rightarrow E_{0}^{*}\left(m\right)=\mathbb{R}\mathcal{A}\left(m\right)$
and $m\in M\rightarrow\left(E_{u}^{*}\oplus E_{s}^{*}\right)\left(m\right)=E_{0}^{\perp}\left(m\right)=\mathrm{Ker}\left(X\right)\left(m\right)$
are smooth and have already been defined in (\ref{eq:def_E_0*}) and
(\ref{eq:def_Ker_X}).

A cotangent vector $\rho\in T^{*}M$ is decomposed accordingly to
the dual decomposition (\ref{eq:decomp*}) as
\begin{equation}
\rho=\rho_{u}+\rho_{s}+\rho_{0}\label{eq:decomp_Xi}
\end{equation}
with components
\[
\rho_{u}\in E_{u}^{*},\rho_{s}\in E_{s}^{*},\quad\rho_{0}=\omega\mathcal{A}\left(m\right)\in E_{0}^{*}
\]
with $m=\pi\left(\rho\right)$ and the frequency $\omega\eq{\ref{eq:omega_function}}\boldsymbol{\omega}\left(\rho\right)\in\mathbb{R}$.

By duality, the hyperbolicity assumption (\ref{eq:hyperbolicity})
gives that the components (\ref{eq:decomp_Xi}) of $\rho\left(t\right)=\tilde{\phi}^{t}\left(\rho\right)$
satisfy 
\begin{align}
\exists C & >0,\forall\rho\left(0\right)\in T^{*}M,\forall t\geq0,\nonumber \\
 & \left\Vert \rho_{u}\left(t\right)\right\Vert _{g_{M}}\geq\frac{1}{C}e^{\lambda_{\mathrm{min}}t}\left\Vert \rho_{u}\left(0\right)\right\Vert _{g_{M}},\qquad\left\Vert \rho_{s}\left(t\right)\right\Vert _{g_{M}}\leq Ce^{-\lambda_{\mathrm{min}}t}\left\Vert \rho_{s}\left(0\right)\right\Vert _{g_{M}},\quad\omega\left(t\right)=\omega\left(0\right).\label{hyperbolicity_xi}
\end{align}
See figure \ref{fig:The-Anosov-flow-on_T*M}.

\begin{figure}
\centering{}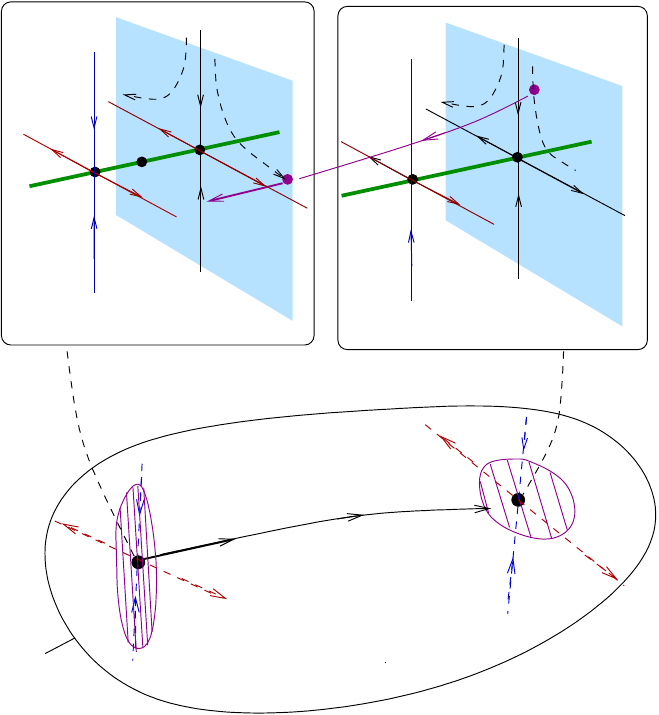\caption{\label{fig:The-Anosov-flow-on_T*M}The Anosov flow $\phi^{t}=\exp\left(tX\right)$
on $M$ induces a Hamiltonian flow $\tilde{\phi}^{t}=\exp\left(t\tilde{X}\right)$
in cotangent space $T^{*}M$. The magenta lines on $M$ represent
\textquotedblleft internal oscillations\textquotedblright{} of a function
$u$, that is a wave packet and is supposed here to have a small support
near $\phi^{t}\left(m\right)$. These oscillations correspond to a
cotangent vector $\rho\in T_{\phi^{t}\left(m\right)}^{*}M$. Transported
by the flow the oscillations of $e^{tX}u=u\circ\phi^{t}$ increase
and the wave front of these oscillations become parallel to $E_{s}\oplus E_{0}$
equivalently $\rho\left(t\right)=\tilde{\phi^{t}}\left(\rho\right)$
converges to the direction of $E_{u}^{*}\subset T^{*}M$ and remains
in the frequency level $\Sigma_{\omega}:=\boldsymbol{\omega}^{-1}\left(\omega\right)$.
The \textbf{trapped set} of the lifted flow $\tilde{\phi^{t}}$ is
the rank one vector bundle $E_{0}^{*}=\mathbb{R}\mathcal{A}$ (green
line) where $\mathcal{A}$ is the Anosov one form. From Lemma \ref{lem:The-canonical-symplectic},
$\Sigma:=E_{0}^{*}\backslash\left\{ 0\right\} $ is a symplectic sub-manifold
of $T^{*}M$, with $\mathrm{dim}\Sigma=2\left(d+1\right)$.}
\end{figure}

\subsubsection{Trapped set}

From (\ref{hyperbolicity_xi}) we deduce the following Lemma.

\begin{cBoxB}{}
\begin{lem}
\label{lem:trapped_set}The set $E_{0}^{*}\eq{\ref{eq:def_E_0*-1}}\mathbb{R}\mathcal{A}$
is the\textbf{ trapped set} (or \href{https://en.wikipedia.org/wiki/Wandering_set\#Non-wandering_points}{non wandering set})
of the flow $\tilde{\phi}^{t}$ in the sense that
\begin{align}
E_{0}^{*}\eq{\ref{eq:def_E_0*-1}}\mathbb{R}\mathcal{A}\eq{\ref{eq:def_Sigma}}\Sigma\cup\left\{ 0\right\}  & =\left\{ \rho\in T^{*}M\,\,\mid\quad\exists\mathcal{K}\subset T^{*}M\mbox{ compact, s.t. }\tilde{\phi}^{t}\left(\rho\right)\in\mathcal{K},\quad\forall t\in\mathbb{R}\right\} ,\label{eq:e_0*}
\end{align}
and $E_{0}^{*}$ is \textbf{transversely hyperbolic}.
\end{lem}

\end{cBoxB}

\begin{proof}
From (\ref{eq:decomp_Xi}), the trapped set $E_{0}^{*}$ is characterized
by $\rho_{u}=0,\rho_{s}=0$ and from (\ref{hyperbolicity_xi}), $E_{0}^{*}$
is transversely hyperbolic.
\end{proof}

\subsection{\label{sec:Anisotropic-Sobolev-space}Anisotropic Sobolev space,
decay of norm outside the trapped set and discrete Ruelle spectrum}

In Lemma \ref{lem:trapped_set} we have observed that the trapped
set $E_{0}^{*}$ is transversely hyperbolic. This will permit us to
define a weight function or escape function $W:T^{*}M\rightarrow\mathbb{R}^{+}$
that decays outside the trapped set (i.e. that is a \href{https://en.wikipedia.org/wiki/Lyapunov_function}{Lyapounov function}
for the flow $\tilde{\phi}^{t}$), following the constructions in
\cite{fred_flow_09} or \cite[Def 5.6]{faure_tsujii_Ruelle_resonances_density_2016}.
Then we define the anisotropic Sobolev space $\mathcal{H}_{W}\left(M\right)$
and obtain that the norm of the pull back operator $e^{tX}$ restricted
to the outside of $E_{0}^{*}$ decays exponentially fast with an arbitrarily
large rate. This is the important property that will enable us to
restrict the analysis of the operator $e^{tX}$ for $t>0$, to a neighborhood
of the trapped set $E_{0}^{*}$ in the forthcoming sections. In this
section we recall all these results.

\subsubsection{\label{subsec:Properties-of-the}Definition and properties of the
escape function $W$}

Following \cite[Def 5.6]{faure_tsujii_Ruelle_resonances_density_2016},
we consider the positive valued continuous function $W\in C\left(T^{*}M;\mathbb{R}^{+}\backslash\left\{ 0\right\} \right)$
called \textbf{weight function} defined as follows. Let $R\geq0$,
$\gamma\in[1-\beta,1[$ with $\beta$ given in (\ref{eq:Holder_exp}).
For $\rho\in T^{*}M$, with stable/unstable components $\rho_{s}\in E_{s}^{*},\rho_{u}\in E_{u}^{*}$
given in (\ref{eq:decomp_Xi}), we define

\begin{equation}
W\left(\rho\right):=\frac{\left\langle h_{\gamma}\left(\rho\right)\left\Vert \rho_{s}\right\Vert _{g_{\rho}}\right\rangle ^{R}}{\left\langle h_{\gamma}\left(\rho\right)\left\Vert \rho_{u}\right\Vert _{g_{\rho}}\right\rangle ^{R}}\label{eq:def_W}
\end{equation}
with
\begin{equation}
h_{\gamma}\left(\rho\right):=\left\langle \left(\left\Vert \rho_{u}\right\Vert _{g_{\rho}}^{2}+\left\Vert \rho_{s}\right\Vert _{g_{\rho}}^{2}\right)^{1/2}\right\rangle ^{-\gamma}\label{eq:def_h}
\end{equation}

\begin{rem}
In the previous expressions, the vertical vectors $\rho_{u},\rho_{s}\in T_{m}^{*}M$
are naturally identified with vectors in the tangent space $T_{\rho}\left(T^{*}M\right)$
of $T^{*}M$ at point $\rho\in T^{*}M$, so we can get their norm,
$\left\Vert \rho_{u}\right\Vert _{g_{\rho}},\left\Vert \rho_{s}\right\Vert _{g_{\rho}}$. 
\end{rem}

The function $W:\rho\in T^{*}M\rightarrow W\left(\rho\right)\in\mathbb{R}$
in (\ref{eq:def_W}) is positive and Hölder continuous, since the
decomposition (\ref{eq:decomp_Xi}) is Hölder continuous. $W$ has
a few important properties given in \cite[Thm 5.9]{faure_tsujii_Ruelle_resonances_density_2016}
as \textbf{decay property} and \textbf{slow variation properties}.

\subsubsection{\label{subsec:Anisotropic-spaces-}Anisotropic spaces $\mathcal{H}_{W}\left(M;F\right)$}

We will define the anisotropic Sobolev spaces $\mathcal{H}_{W}\left(M;F\right)$
used in this paper, in order to get Theorem \ref{thm:If--is} below,
saying that the operator $X_{F}$ has \textbf{intrinsic discrete spectrum}
in $\mathcal{H}_{W}\left(M;F\right)$.
\begin{rem}
\label{subsec:Essential-spectrum-of}Let us remark why to get discrete
spectrum, we cannot consider the Hilbert space $L^{2}\left(M;F\right)$.
If we choose an arbitrary smooth Hermitian metric on the vector bundle
$F\rightarrow M$ and using the measure $dm$ on $M$, (\ref{eq:volume_form_M}),
we can define the Hilbert space $L^{2}\left(M;F\right)$ of sections
of $F$. One can show that the operator $X_{F}$, Eq.(\ref{eq:def_X_F}),
acting in $L^{2}\left(M;F\right)$ has only \href{https://en.wikipedia.org/wiki/Essential_spectrum}{essential spectrum}
contained in some vertical band (the detailed argument will be explained
in a forthcoming paper):
\[
\mathrm{spec}\left(X_{F}:L^{2}\left(M;F\right)\rightarrow L^{2}\left(M;F\right)\right)=\left\{ z\in\mathbb{C},\mathrm{Re}\left(z\right)\in\left[\gamma_{L^{2}}^{-},\gamma_{L^{2}}^{+}\right]\right\} ,
\]
with
\[
\gamma_{L^{2}}^{\pm}=\lim_{t\rightarrow\pm\infty}\log\left\Vert e^{tX_{F}}\right\Vert _{L^{\infty}}^{1/t}.
\]
For example, considering the vector field $X$ itself and the trivial
bundle $F=M\times\mathbb{C}$, since $\mathrm{div}_{dm}X=0$, we have
that $X=-X^{\dagger}$ is skew symmetric in $L^{2}\left(M\right)$
and $X$ has continuous spectrum on the imaginary axis $i\mathbb{R}$.
\end{rem}

For simplicity of notations, we will forget the vector bundle $F$,
that plays no role for this construction, i.e. we first consider only
a trivial bundle $F=M\times\mathbb{C}$. For any $u\in C^{\infty}\left(M\right)$
we define its $\mathcal{H}_{W}\left(M\right)$-norm
\begin{equation}
\left\Vert u\right\Vert _{\mathcal{H}_{W}\left(M\right)}:=\left\Vert W\mathcal{T}u\right\Vert _{L^{2}\left(T^{*}M\right)}\label{eq:def_Sobolev_norm}
\end{equation}
where $W$ is used as a multiplication operator on $\mathcal{S}\left(T^{*}M\right)$.
Let $\mathcal{H}_{W}\left(M\right)$ be the Hilbert space obtained
by completion of $C^{\infty}\left(M\right)$ with respect to this
norm and called \textbf{anisotropic Sobolev space}: 
\begin{equation}
\mathcal{H}_{W}\left(M\right):=\overline{\left\{ u\in C^{\infty}\left(M\right),\quad\left\Vert u\right\Vert _{\mathcal{H}_{W}\left(M\right)}<\infty\right\} }.\label{eq:def_H_W}
\end{equation}
In other words $W\mathcal{T}:\mathcal{H}_{W}\left(M\right)\rightarrow L^{2}\left(T^{*}M\right)$
is an isometry. To define $\mathcal{H}_{W}\left(M;F\right)$ we proceed
similarly by completion of the space of smooth sections $C^{\infty}\left(M;F\right)$.
For more details or information, we refer to \cite[def 4.37]{faure_tsujii_Ruelle_resonances_density_2016}. 
\begin{rem}
The function $W$ in (\ref{eq:def_W}) satisfies
\begin{equation}
\exists C>0,\forall\rho\in T^{*}M,\,C^{-1}\left\langle \left|\rho\right|_{g_{M}}\right\rangle ^{-\left|r\right|}\leq W\left(\rho\right)\leq C\left\langle \left|\rho\right|_{g_{M}}\right\rangle ^{\left|r\right|}.\label{eq:order_W}
\end{equation}
with
\begin{equation}
r=\frac{R}{2}\left(1-\gamma\right)>0.\label{eq:order_r}
\end{equation}
Consequently
\begin{equation}
H^{r}\left(M\right)\subset\mathcal{H}_{W}\left(M\right)\subset H^{-r}\left(M\right).\label{eq:H_W}
\end{equation}
where $H^{r}\left(M\right):=\mathrm{Op}\left(\left\langle \left|\xi\right|_{g_{M}}\right\rangle ^{-r}\right)\left(L^{2}\left(M\right)\right)$
is the standard \href{https://en.wikipedia.org/wiki/Sobolev_space}{Sobolev space}
of order $r$.
\end{rem}

\subsubsection{\label{subsec:Restriction-in-the}Decay of norm outside the trapped
set}

Recall the truncation operator $\mathrm{Op}\left(\chi_{\Sigma,\sigma}\right)$
at distance $\sigma$ from the set $\Sigma$ defined in (\ref{eq:def_op_Lambda}).
\begin{cBoxB}{}
\begin{thm}[Decay outside the trapped set]
\label{thm:decay}\cite[Thm 5.13]{faure_tsujii_Ruelle_resonances_density_2016}For
any $\Lambda>0$, we can choose $R\gg1$ large enough as a parameter
of the function $W$ in (\ref{eq:def_W}), such that $\exists C>0,\forall t\geq0,$
$\exists\sigma_{t},\forall\sigma>\sigma_{t}$ we have
\begin{equation}
\left\Vert e^{tX}\left(\mathrm{Id}-\mathrm{Op}\left(\chi_{\Sigma,\sigma}\right)\right)\right\Vert _{\mathcal{H}_{W}\left(M;F\right)}\leq Ce^{-\Lambda t},\label{eq:decay_outside}
\end{equation}
\begin{equation}
\left\Vert \left(\mathrm{Id}-\mathrm{Op}\left(\chi_{\Sigma,\sigma}\right)\right)e^{tX}\right\Vert _{\mathcal{H}_{W}\left(M;F\right)}\leq Ce^{-\Lambda t}.\label{eq:decay_outside-1}
\end{equation}
\end{thm}

\end{cBoxB}

\begin{rem}
\label{rem:Consider-the-decomposition}Consider the decomposition
$e^{tX}=e^{tX}\left(\mathrm{Id}-\mathrm{Op}\left(\chi_{\Sigma,\sigma}\right)\right)+e^{tX}\mathrm{Op}\left(\chi_{\Sigma,\sigma}\right)$.
Theorem \ref{thm:decay} means that the first component $e^{tX}\left(\mathrm{Id}-\mathrm{Op}\left(\chi_{\Sigma,\sigma}\right)\right)$
decays exponentially fast as $O\left(e^{-\Lambda t}\right)$ with
an arbitrarily large rate $\Lambda\gg1$. We deduce that in order
to describe the dominant effect of the operator $e^{tX}$ and Ruelle
eigenvalues on the spectral domain $\mathrm{Re}\left(z\right)>-\Lambda$,
one has to study the second component $e^{tX}\mathrm{Op}\left(\chi_{\Sigma,\sigma}\right)$
namely the dynamics $e^{tX}$ restricted to a neighborhood of the
trapped set $\Sigma=E_{0}^{*}\backslash\left\{ 0\right\} $. Then,
from Theorem \ref{Thm:approx_exptX_from_N}, we only need to study
the operator $\tilde{\mathrm{Op}}\left(d\tilde{\phi}_{N}^{t}\right)$
acting on the normal bundle over the trapped set $\Sigma$. This is
the subject of the section \ref{subsec:Description-of-the} below.
\end{rem}

\subsubsection{Discrete Ruelle spectrum in anisotropic Sobolev spaces $\mathcal{H}_{W}\left(M;F\right)$}

In paper \cite[thm 5.14]{faure_tsujii_Ruelle_resonances_density_2016},
Theorem \ref{thm:decay} is used to prove the following theorem. Recall
that the function $W$ in (\ref{eq:def_W}) depends on the exponent
$R$ and that the order $r=\frac{R}{2}\left(1-\gamma\right)$ is defined
from $R$ in (\ref{eq:order_r}).

\begin{cBoxB}{}
\begin{thm}
\label{thm:If--is}\textbf{``Group property and discrete spectrum''}\cite[thm 2.11]{faure_tsujii_Ruelle_resonances_density_2016}\textbf{.}
For a general Anosov vector field $X$ on $M$, the family of operators
\[
e^{tX_{F}}:\mathcal{H}_{W}\left(M;F\right)\rightarrow\mathcal{H}_{W}\left(M;F\right),\qquad t\in\mathbb{R},
\]
form a \textbf{strongly continuous group} and the generator $X_{F}$
has\textbf{ discrete spectrum} denoted $\mathrm{spect}\left(X_{F}\right)$,
on $C-r\lambda_{\mathrm{min}}\leq\mathrm{Re}\left(z\right)\leq C'$
with $C,C'$ independent of the parameter $R$. This spectrum is \textbf{intrinsic}
(i.e. does not depend on $W$) and called \textbf{future Ruelle spectrum}
in \cite[Thm 2.11]{faure_tsujii_Ruelle_resonances_density_2016}.
\end{thm}

\end{cBoxB}

See Figure \ref{fig:Band-structure-of}.

\subsection{\label{subsec:Description-of-the}Description of the operator $e^{tX}$
near the set $\Sigma$ with the bundle $N_{s}$}

\subsubsection{Decomposition $N=N_{u}\oplus N_{s}$}

From the Anosov property, we have the decomposition $\mathrm{Ker}\mathcal{A}\eq{\ref{eq:one_form}}E_{u}\oplus E_{s}$
into Hölder continuous sub-bundles, that gives a refined decomposition
of (\ref{eq:decomp_K_K0_N}) into isotropic subspaces:

\begin{cBoxB}{}
\begin{lem}
\label{lem:LetEach-component-}Let
\[
N_{u}:=N\cap d\pi^{-1}\left(E_{u}\right),\quad N_{s}:=N\cap d\pi^{-1}\left(E_{s}\right),
\]
\begin{equation}
N=N_{u}\oplus N_{s},\label{eq:N_Nu_Ns}
\end{equation}
\begin{equation}
K_{u}:=K\cap d\pi^{-1}\left(E_{u}\right),\quad K_{s}:=K\cap d\pi^{-1}\left(E_{s}\right),\quad K=K_{u}\oplus K_{s},\label{eq:def_Nus}
\end{equation}
\[
\mathrm{dim}K_{u}=\mathrm{dim}K_{s}=\mathrm{dim}N_{u}=\mathrm{dim}N_{s}=d,\quad\mathrm{dim}K_{0}=2,
\]
Each component $K_{u},K_{s},\tilde{E}_{0},\tilde{E}_{0}^{*},N_{u},N_{s}$
is a $\Omega-$isotropic vector space, invariant under the dynamics
$d\tilde{\phi}^{t}$. The indices $u,s,0$ denotes respectively instability,
stability or neutrality. See Figure \ref{fig:N_Sigma}.
\end{lem}

\end{cBoxB}

\begin{rem}
\label{rem:unstable}~
\begin{itemize}
\item Notice that due to (\ref{eq:project_phi}) that reverses the direction
of dynamics, $N_{u}$ is stable for $\left(d\tilde{\phi}^{t}\right)_{t\geq0}$
and $N_{s}$ is unstable.
\item The symplectic normal bundle $N\rightarrow\Sigma$ is smooth, but
the isotropic sub-bundle $N_{s}\rightarrow\Sigma$ is only $\beta_{s}$-Hölder
continuous.
\end{itemize}
\end{rem}

\begin{proof}
From the fact explained in section \ref{subsec:Contact-Anosov-flow}
that $E_{s},E_{u}$ are $d\mathcal{A}-$isotropic and that $d\pi$
is a symplectomorphism in (\ref{eq:dpi_K}) and (\ref{eq:dpi_N}),
we deduce that $N_{u},N_{s},K_{u},K_{s}$ are $\Omega-$isotropic.
\end{proof}

\subsubsection{\label{subsec:Taylor-expansion-along}Transfer operators on $N_{s}$}

We start from Theorem \ref{Thm:approx_exptX_from_N}. We have the
splitting $N=N_{s}\oplus N_{u}$ in (\ref{eq:N_Nu_Ns}) of the symplectic
bundle $N$ over $\Sigma$ into Lagrangian sub-bundles. We consider
the operator $\tilde{\mathrm{Op}}\left(d\tilde{\phi}_{N}^{t}\right)$
on the right hand side of (\ref{eq:exp_tX_N}) and our purpose is
to express it as an operator acting only on the Lagrangian sub-bundle
$N_{s}$. The bundles $N_{u},N_{s}$ are invariant under the dynamics
$d\tilde{\phi}^{t}$, hence the flow map splits accordingly (recall
Remark \ref{rem:unstable})
\[
d\tilde{\phi}_{N}^{t}=d\tilde{\phi}_{N_{u}}^{t}+d\tilde{\phi}_{N_{s}}^{t}\quad:N\rightarrow N.
\]

\begin{notation}
\label{rem:The-bundle-}The bundle $N_{s}\rightarrow\Sigma$ defined
in (\ref{eq:def_Nus}) is Hölder continuous with exponent $\beta$.
For simplicity we denote
\begin{align}
\mathcal{S}_{\beta}\left(N_{s}\right) & :=C^{\beta}\left(\Sigma;\mathcal{S}\left(N_{s}\left(.\right)\right)\right)\label{eq:notation}
\end{align}
as the set of functions $f\in C^{\beta}\left(\Sigma;\mathcal{S}\left(N_{s}\left(.\right)\right)\right)$,
$\beta$-Hölder continuous on $\Sigma$ with fast decay in frequency
$\omega$, valued in the \href{https://en.wikipedia.org/wiki/Schwartz_space}{Schwartz space}
of functions $\mathcal{S}\left(N_{s}\left(\rho\right)\right)$ for
each $\rho\in\Sigma$. This notation as already been introduced in
(\ref{eq:def_S_beta}) and will also concerns later other space of
sections over $\Sigma$. Notice that $\mathcal{S}\left(N_{s}\left(.\right)\right)\rightarrow\Sigma$
is an infinite rank bundle vector over $\Sigma$.
\end{notation}

\subsubsection{The bundle $\mathcal{F}\left(N_{s}\right)\rightarrow\Sigma$}

Let us denote $\left|\mathrm{det}N_{s}\right|^{-1/2}$ the dual of
the half density bundle in $N_{s}$ \cite[p.509]{guillemin_77}.

\begin{cBoxA}{}
\begin{defn}
Let
\begin{equation}
\mathcal{F}\left(N_{s}\right):=\left|\mathrm{det}N_{s}\right|^{-1/2}\otimes\mathcal{S}_{\beta}\left(N_{s}\right)\label{eq:def_F-1}
\end{equation}
and
\begin{equation}
X_{\mathcal{F}}:C^{\beta}\left(\Sigma;\mathcal{F}\left(N_{s}\right)\right)\rightarrow C^{\beta}\left(\Sigma;\mathcal{F}\left(N_{s}\right)\right)\label{eq:def_X_Fk}
\end{equation}
be the derivation of sections of $\mathcal{F}$ over the vector field
$\tilde{X}$, that is generator of the group of operators, $\forall t\in\mathbb{R}$,
\begin{equation}
e^{tX_{\mathcal{F}}}:=\left|\mathrm{det}\left(d\tilde{\phi}_{N_{s}}^{t}\right)\right|^{-1/2}\left(d\tilde{\phi}_{N_{s}}^{t}\right)^{-\circ}:C^{\beta}\left(\Sigma;\mathcal{F}\left(N_{s}\right)\right)\rightarrow C^{\beta}\left(\Sigma;\mathcal{F}\left(N_{s}\right)\right).\label{eq:def_XF-1}
\end{equation}
\end{defn}

\end{cBoxA}

\begin{rem}
The derivation $X_{\mathcal{F}}$ over $\tilde{X}$ in (\ref{eq:def_X_Fk})
is well defined since the bundle $\mathcal{F}\left(N_{s}\right)$
is constructed from the bundle $N_{s}$ and the given vector field
$X$ on $M$ induces a natural derivation on $N_{s}$ hence $X_{\mathcal{F}}$
on $\mathcal{F}\left(N_{s}\right)$. Recall from Remark \ref{rem:unstable}
that the linear map $d\tilde{\phi}_{N_{s}}^{t}$ is expanding on $N_{s}$,
hence $\left|\mathrm{det}\left(d\tilde{\phi}_{N_{s}}^{t}\right)\right|>1$.
The operator $X_{\mathcal{F}}$ has a domain that includes sections
smooth along the flow direction. Recall that for simplicity of notation
we have ignored the bundle $F$ in (\ref{eq:def_F-1}).

Now we will relate the operator $\tilde{\mathrm{Op}}\left(d\tilde{\phi}_{N}^{t}\right)$
on the right hand side of (\ref{eq:exp_tX_N}) to the operator $e^{tX_{\mathcal{F}}}$
defined in (\ref{eq:def_XF-1}). In (\ref{eq:def_Bargman_B}) we define
the Bargman transform 
\[
\mathcal{B}_{N_{s}}:\mathcal{S}_{\beta}\left(N_{s}\right)\rightarrow\mathcal{S}\left(N_{s}\oplus N_{s}^{*}\right).
\]
As explained in Lemma \ref{lem:If--is}, we have an $\left(\Omega,g\right)-$isomorphism
\[
\Psi:N=N_{s}\oplus N_{s}^{\perp_{g}}\rightarrow N_{s}\oplus N_{s}^{*}.
\]
But $N_{u}\neq N_{s}^{\perp_{g}}$ in general, so we introduce $S_{N_{s},N_{u}}:N\rightarrow N$
the linear (shear) symplectic map defined as identity on $N_{s}$
and any vector of $N_{s}^{\perp_{g}}$ is mapped to its projection
on $N_{u}$ parallel to $N_{s}$:
\[
S_{N_{s},N_{u}}\left(N_{s}\right)=N_{s},\quad S_{N_{s},N_{u}}\left(N_{s}^{\perp_{g}}\right)=N_{u}.
\]
Notice that the bundle of maps $S_{N_{s},N_{u}},\Psi$ is uniformly
bounded over $\Sigma$. The operator $\tilde{\mathrm{Op}}\left(S_{N_{s},N_{u}}\Psi^{-1}\right):\mathcal{S}\left(N_{s}\oplus N_{s}^{*}\right)\rightarrow\mathcal{S}\left(N\right)$
is defined as in (\ref{eq:def_op_tilde}) and we set
\end{rem}

\begin{equation}
\mathcal{B}_{N_{s},N_{u}}:=\tilde{\mathrm{Op}}\left(S_{N_{s},N_{u}}\Psi^{-1}\right)\mathcal{B}_{N_{s}}\quad:\mathcal{S}_{\beta}\left(N_{s}\right)\rightarrow\mathcal{S}\left(N\right).\label{eq:def_B_Ns_Nu}
\end{equation}

\begin{cBoxB}{}
\begin{lem}
We have
\begin{equation}
\tilde{\mathrm{Op}}\left(d\tilde{\phi}_{N}^{t}\right)=\mathcal{B}_{N_{s},N_{u}}e^{tX_{\mathcal{F}}}\mathcal{B}_{N_{s},N_{u}}^{\dagger}.\label{eq:B_exp_B}
\end{equation}
 
\end{lem}

\end{cBoxB}

\begin{proof}
Let $\Phi_{N_{s}}:=\left(d\tilde{\phi}_{N_{s}}^{t}\right)\oplus\left(d\tilde{\phi}_{N_{s}}^{t}\right)^{*-1}:N_{s}\oplus N_{s}^{*}\rightarrow N_{s}\oplus N_{s}^{*}$.
From Lemma \ref{lem:Let-,-be} with the setting $\phi=\left(d\tilde{\phi}_{N_{s}}^{t}\right)^{-1}$,
we have
\begin{align}
\tilde{\mathrm{Op}}\left(\Phi_{N_{s}}\right) & \eq{\ref{eq:Op_tilde_PHI}}\left|\mathrm{det}\left(d\tilde{\phi}_{N_{s}}^{t}\right)\right|^{-1/2}\mathcal{B}_{N_{s}}\left(d\tilde{\phi}_{N_{s}}^{t}\right)^{-\circ}\mathcal{B}_{N_{s}}^{\dagger}\nonumber \\
 & \eq{\ref{eq:def_XF-1}}\mathcal{B}_{N_{s}}e^{tX_{\mathcal{F}}}\mathcal{B}_{N_{s}}^{\dagger}\quad:\mathcal{S}\left(N_{s}\oplus N_{s}^{*}\right)\rightarrow\mathcal{S}\left(N_{s}\oplus N_{s}^{*}\right)\label{eq:op_tilde_Phi_Ns}
\end{align}
By construction, we have the conjugation
\[
d\tilde{\phi}_{N}^{t}=\left(S_{N_{s},N_{u}}\Psi^{-1}\right)\circ\Phi_{N_{s}}\circ\left(S_{N_{s},N_{u}}\Psi^{-1}\right)^{-1}\quad:N\rightarrow N
\]
that gives
\begin{align}
\tilde{\mathrm{Op}}\left(d\tilde{\phi}_{N}^{t}\right) & \eq{\ref{eq:homom},\ref{eq:op_tilde_Phi_Ns}}\tilde{\mathrm{Op}}\left(S_{N_{s},N_{u}}\Psi^{-1}\right)\mathcal{B}_{N_{s}}e^{tX_{\mathcal{F}}}\mathcal{B}_{N_{s}}^{\dagger}\tilde{\mathrm{Op}}\left(S_{N_{s},N_{u}}\Psi^{-1}\right)^{\dagger}\nonumber \\
 & \eq{\ref{eq:def_B_Ns_Nu}}\mathcal{B}_{N_{s},N_{u}}e^{tX_{\mathcal{F}}}\mathcal{B}_{N_{s},N_{u}}^{\dagger}.\label{eq:Op_tilde_dPhi_N}
\end{align}
\end{proof}
We finish this section with the description of the operator $e^{tX}$
near the  set $\Sigma$ using the bundle $N_{s}$.

\begin{cBoxA}{}
\begin{defn}
\label{def:We-define-the-1}We define the operators

\begin{equation}
\mathcal{T}_{N_{s}}:=\mathcal{B}_{N_{s},N_{u}}^{\dagger}\mathcal{T}_{N}\eq{\ref{eq:def_T_N}}\mathcal{B}_{N_{s},N_{u}}^{\dagger}\chi_{\Sigma}^{\mu}\widetilde{\left(\exp_{N}\right)^{\circ}}\mathcal{T}\quad:C^{\infty}\left(M\right)\rightarrow C^{\beta}\left(\Sigma;\mathcal{F}\left(N_{s}\right)\right)\label{eq:def_T_Ns}
\end{equation}
\begin{equation}
\mathcal{T}_{N_{s}}^{\Delta}:=\mathcal{T}_{N}^{\Delta}\mathcal{B}_{N_{s},N_{u}}\eq{\ref{eq:def_T_N_Delta}}\mathcal{T}^{\dagger}\widetilde{\left(\exp_{N}^{-1}\right)^{\circ}}\chi_{\Sigma}^{\mu}\mathcal{B}_{N_{s},N_{u}}\quad:C^{\beta}\left(\Sigma;\mathcal{F}\left(N_{s}\right)\right)\rightarrow C^{\infty}\left(M\right).\label{eq:def_T_Ns_Delta}
\end{equation}
\end{defn}

\end{cBoxA}

\begin{cBoxB}{}
\begin{thm}
\label{Thm:approx_exptX_from_Ns}For any fixed $t\in\mathbb{R}$,
we have
\begin{equation}
e^{tX}\approx\mathcal{T}_{N_{s}}^{\Delta}\,\Upsilon_{t}^{1/2}e^{tX_{\mathcal{F}}}\mathcal{T}_{N_{s}}\label{eq:exp_tX_N-1}
\end{equation}
\end{thm}

\end{cBoxB}

Recall that the metaplectic correction $\Upsilon_{t}^{1/2}>0$ defined
in (\ref{eq:def_Upsilon_t}) is a function on $\Sigma$ considered
as a multiplication operator.
\begin{proof}
We write
\begin{align}
e^{tX} & \underset{(\ref{eq:exp_tX_N})}{\approx}\mathcal{T}_{N}^{\Delta}\,\Upsilon_{t}^{1/2}\tilde{\mathrm{Op}}\left(d\tilde{\phi}_{N}^{t}\right)\mathcal{T}_{N}\label{eq:exp_tX_Ns}\\
 & \eq{\ref{eq:B_exp_B},\ref{eq:def_T_Ns},\ref{eq:def_T_Ns_Delta}}\mathcal{T}_{N_{s}}^{\Delta}\,\Upsilon_{t}^{1/2}e^{tX_{\mathcal{F}}}\mathcal{T}_{N_{s}}.\nonumber 
\end{align}
\end{proof}

\subsection{The bundle map $e^{tX_{\mathcal{F}}}$ and sub-bundle $\mathcal{F}_{k}\left(N_{s}\right)$}

The expression (\ref{eq:exp_tX_N-1}) will be convenient to get some
properties of the dynamical operator $e^{tX}$ from the study of properties
of the operator $e^{tX_{\mathcal{F}}}$. For example, one important
property that we will first exploit is that $d\tilde{\phi}_{N_{s}}^{t}:N_{s}\rightarrow N_{s}$
is a linear map hence the operators $\left(d\tilde{\phi}_{N_{s}}^{t}\right)^{-\circ}$
and $e^{tX_{\mathcal{F}}}$ in (\ref{eq:def_XF-1}) preserve the finite
rank sub-bundle of homogeneous polynomials of degree $k\in\mathbb{N}$.
This property is responsible for the band structure of the spectrum.

\subsubsection{Taylor projectors $T_{k}$}

Since $d\tilde{\phi}_{N_{s}}^{t}:N_{s}\rightarrow N_{s}$ is a vector
bundle homomorphism, the push-forward operator $\left(d\tilde{\phi}_{N_{s}}^{t}\right)^{-\circ}$
in (\ref{eq:def_XF-1}) keeps invariant the space of homogeneous polynomial
functions of order $k\in\mathbb{N}$ denoted $\mathrm{Pol}_{k}\left(N_{s}\right)\subset\mathcal{S}'_{\beta}\left(N_{s}\right)$.
To express this property, we first consider the vector bundle $E_{s}\rightarrow M$
and the bundle map of finite rank ``Taylor projectors'' over $\mathrm{Id}$,
from Definition \ref{def:T_k}, denoted
\begin{equation}
T_{k}:\mathcal{S}_{\beta}\left(E_{s}\right)\rightarrow\mathrm{Pol}_{k}\left(E_{s}\right)\subset\mathcal{S}'_{\beta}\left(E_{s}\right)\label{eq:def_Tk}
\end{equation}
and
\begin{equation}
T_{\left[0,K\right]}:=\sum_{k=0}^{K}T_{k},\qquad T_{\geq\left(K+1\right)}=\mathrm{Id}-T_{\left[0,K\right]},\label{eq:def_TK+1}
\end{equation}
satisfying
\begin{equation}
T_{k}T_{k'}=\delta_{k=k'}T_{k}.\label{eq:delta_T}
\end{equation}
Using the projection map $d\pi:N_{s}\rightarrow E_{s}$ in (\ref{eq:dpi_N}),
we lift these operators to the vector bundle $\mathcal{S}_{\beta}\left(N_{s}\right)\rightarrow\Sigma$:
\begin{equation}
\tilde{T}_{k}:\mathcal{S}_{\beta}\left(N_{s}\right)\rightarrow\mathrm{Pol}_{k}\left(N_{s}\right)\subset\mathcal{S}'_{\beta}\left(N_{s}\right).\label{eq:def_Tk-1}
\end{equation}
We define $\tilde{T}_{\left[0,K\right]}$ and $\tilde{T}_{\geq\left(K+1\right)}$
similarly to (\ref{eq:def_TK+1}). We have (in the next equation and
after, a strict notation should be $\mathrm{Id}_{\left|\mathrm{det}N_{s}\right|^{-1/2}}\otimes\tilde{T}_{k}$
instead of $\tilde{T}_{k}$)
\begin{equation}
\left[e^{tX_{\mathcal{F}}},\tilde{T}_{k}\right]=0.\label{eq:commut_T}
\end{equation}

\begin{cBoxA}{}
\begin{defn}
We denote
\begin{equation}
\mathcal{F}_{k}\left(N_{s}\right):=\left|\mathrm{det}N_{s}\right|^{-1/2}\otimes\mathcal{\mathrm{Pol}}_{k}\left(N_{s}\right)\label{eq:def_F-1-1}
\end{equation}
that is a finite rank bundle over $\Sigma$. The space of sections
$C^{\beta}\left(\Sigma;\mathcal{F}_{k}\left(N_{s}\right)\right)$
is the image of the bundle map $\tilde{T}_{k}:C^{\beta}\left(\Sigma;\mathcal{F}\left(N_{s}\right)\right)\rightarrow C^{\beta}\left(\Sigma;\mathcal{F}_{k}\left(N_{s}\right)\right)$.

We denote $X_{\mathcal{F}_{k}}$ the generator $X_{\mathcal{F}}$
in (\ref{eq:def_X_Fk}) restricted to the invariant space of sections
$C^{\beta}\left(\Sigma;\mathcal{F}_{k}\left(N_{s}\right)\right)$.
\end{defn}

\end{cBoxA}

More generally, for $0\leq k_{1}\leq k_{2}$, we denote the direct
sum bundle
\begin{equation}
\mathcal{F}_{\left[k_{1},k_{2}\right]}\left(N_{s}\right):=\bigoplus_{k=k_{1}}^{k_{2}}\mathcal{F}_{k}\left(N_{s}\right),\label{eq:def_F_0K}
\end{equation}
and denote $X_{\mathcal{F}_{\left[k_{1},k_{2}\right]}}$the generator
in $C^{\beta}\left(\Sigma;\mathcal{F}_{\left[k_{1},k_{2}\right]}\left(N_{s}\right)\right)$.

\subsubsection{Spectrum of the derivation $X_{\mathcal{F}_{k}}$ in $L^{2}\left(\Sigma;\mathcal{F}_{k}\left(N_{s}\right)\right)$}

The operators $\left(e^{tX_{\mathcal{F}_{k}}}\right)_{t\in\mathbb{R}}$
in (\ref{eq:def_XF-1}) form a group of point-wise linear bundle maps
over $\tilde{\phi}^{t}:\Sigma\rightarrow\Sigma$ that preserves the
volume. Consequently all $L^{p}$ norms $\left\Vert e^{tX_{\mathcal{F}_{k}}}\right\Vert _{L^{p}\left(\Sigma;\mathcal{F}_{k}\right)}$
for $p\in\left[1,\infty\right]$ are equal, as it is for general multiplication
operators.
\begin{cBoxA}{}
\begin{defn}
\label{lem:For-any-,}For any $k\in\mathbb{N}$, let
\begin{equation}
\gamma_{k}^{\pm}:=\lim_{t\rightarrow\pm\infty}\log\left\Vert e^{tX_{\mathcal{F}_{k}}}\right\Vert _{L^{2}\left(\Sigma;\mathcal{F}_{k}\left(N_{s}\right)\right)}^{1/t}\label{eq:def_gamma_}
\end{equation}
\end{defn}

\end{cBoxA}

\begin{rem}
The values $\gamma_{k}^{\pm}$ defined in (\ref{eq:def_gamma_}) with
$L^{2}\left(\Sigma;\mathcal{F}_{k}\left(N_{s}\right)\right)$ coincide
with the Definition (\ref{eq:def_gamma_-2}) using $L^{\infty}\left(M;\mathcal{F}_{k}\left(E_{s}\right)\right)$.
\end{rem}

\begin{cBoxB}{}
\begin{lem}
\label{lem:We-have-the}We have the estimates $\forall\epsilon>0,\exists C_{\epsilon}>0,\forall t\geq0,$
\begin{equation}
\left\Vert e^{tX_{\mathcal{F}_{k}}}\right\Vert _{L^{2}\left(\Sigma;\mathcal{F}_{k}\left(N_{s}\right)\right)}\leq C_{\epsilon}e^{t\left(\gamma_{k}^{+}+\epsilon\right)},\qquad\left\Vert e^{-tX_{\mathcal{F}_{k}}}\right\Vert _{L^{2}\left(\Sigma;\mathcal{F}_{k}\left(N_{s}\right)\right)}\leq C_{\epsilon}e^{-t\left(\gamma_{k}^{-}-\epsilon\right)}.\label{eq:bound-2}
\end{equation}
Equivalently $\forall\epsilon>0,\exists C_{\epsilon}>0$,$\forall z\in\mathbb{C}$,
$\mathrm{Re}\left(z\right)>\gamma_{k}^{+}+\epsilon$ or $\mathrm{Re}\left(z\right)<\gamma_{k}^{-}-\epsilon$,
\begin{equation}
\left\Vert \left(z-X_{\mathcal{F}_{k}}\right)^{-1}\right\Vert _{L^{2}\left(\Sigma;\mathcal{F}_{k}\left(N_{s}\right)\right)}\leq C_{\epsilon}.\label{eq:bound-3}
\end{equation}
Consequently the spectrum of the operator $X_{\mathcal{F}_{k}}$ in
$L^{2}\left(\Sigma;\mathcal{F}_{k}\left(N_{s}\right)\right)$ is contained
in the vertical band
\begin{equation}
B_{k}:=\left\{ \mathrm{Re}\left(z\right)\in\left[\gamma_{k}^{-},\gamma_{k}^{+}\right]\right\} .\label{eq:spect}
\end{equation}
\end{lem}

\end{cBoxB}

\begin{rem}
We will explain in a forthcoming paper that the spectrum of $X_{\mathcal{F}_{k}}$
in $L^{2}\left(\Sigma;\mathcal{F}_{k}\left(N_{s}\right)\right)$ consists
only of \href{https://en.wikipedia.org/wiki/Essential_spectrum}{essential spectrum}
that fills the whole band $B_{k}$.
\end{rem}

\begin{proof}
Eq.(\ref{eq:def_gamma_}) implies (\ref{eq:bound-2}) and also (\ref{eq:bound-3})
and claim for (\ref{eq:spect}) by writing the convergent expression
for $\mathrm{Re}\left(z\right)>\gamma_{k}^{+}+\epsilon$:
\[
\left(z-X_{\mathcal{F}_{k}}\right)^{-1}=\int_{0}^{\infty}e^{t\left(X_{\mathcal{F}_{k}}-z\right)}dt
\]
that gives, if $\epsilon'<\epsilon$, 
\begin{align*}
\left\Vert \left(z-X_{\mathcal{F}_{k}}\right)^{-1}\right\Vert  & \leq\int_{0}^{\infty}e^{-\mathrm{Re}\left(z\right)t}\left\Vert e^{tX_{\mathcal{F}_{k}}}\right\Vert dt\leq\int_{0}^{\infty}e^{-\mathrm{Re}\left(z\right)t}C_{\epsilon'}e^{t\left(\gamma_{k}^{+}+\epsilon'\right)}dt\leq C_{\epsilon}.
\end{align*}
Similarly for $\mathrm{Re}\left(z\right)<\gamma_{k}^{-}-\epsilon$,
we have $\left(z-X_{\mathcal{F}_{k}}\right)^{-1}=-\int_{-\infty}^{0}e^{t\left(X_{\mathcal{F}_{k}}-z\right)}dt$,
and get also (\ref{eq:bound-3}).
\end{proof}

\subsubsection{\label{subsec:Sobolev-space-}Sobolev space $\mathcal{H}_{\mathcal{W}}\left(N_{s}\right)$
and remainder term}

The operator $e^{tX_{\mathcal{F}}}$ acts in the infinite rank bundle
$\mathcal{F}\left(N_{s}\right)$ in (\ref{eq:def_F-1}). We explain
here how to ``extract'' the finite codimension part $e^{tX_{\mathcal{F}}}\tilde{T}_{\geq\left(K+1\right)}$.
We define a weight function $\mathcal{W}:N\rightarrow\mathbb{R}^{+}$
similar to $W:T^{*}M\rightarrow\mathbb{R}^{+}$ in (\ref{eq:def_W}),
with the same parameters $0<\gamma<1$, $R>0$, as follows. For $\rho\in\Sigma$,
$v=v_{u}+v_{s}\in N\left(\rho\right)=N_{u}\left(\rho\right)\oplus N_{s}\left(\rho\right)$
with $v_{u}\in N_{u}\left(\rho\right)$, $v_{s}\in N_{s}\left(\rho\right)$,
let
\begin{equation}
\mathcal{W}\left(v\right):=\frac{\left\langle h_{\gamma}\left(v\right)\left\Vert v_{s}\right\Vert _{g_{\rho}}\right\rangle ^{R}}{\left\langle h_{\gamma}\left(v\right)\left\Vert v_{u}\right\Vert _{g_{\rho}}\right\rangle ^{R}}\label{eq:def_W-1}
\end{equation}
with
\[
h_{\gamma}\left(v\right)=\left\langle \left\Vert v\right\Vert _{g_{\rho}}\right\rangle ^{-\gamma}.
\]
The metric $g$ on $T^{*}M$ is moderate \cite[Lemma 4.12]{faure_tsujii_Ruelle_resonances_density_2016}.
This implies that under the exponential map $\exp_{N}:N\rightarrow T^{*}M$
defined in (\ref{eq:exp_N}), the functions $\mathcal{W}$ and $W$
are equivalent in the neighborhood of the trapped set defined in (\ref{eq:def_Chi_mu}),
i.e. $\exists C\geq1,\exists\omega_{0}>0,\forall v\in N$ with $\left\Vert v\right\Vert _{g}\leq\left\langle \omega\right\rangle ^{\mu/2}$,
$\omega=\boldsymbol{\omega}\left(\pi\left(v\right)\right)$, $\left|\omega\right|\geq\omega_{0}$,
we have
\begin{align}
\frac{1}{C}\mathcal{W}\left(v\right)\leq W\left(\exp_{N}\left(v\right)\right)\leq C\mathcal{W}\left(v\right).\label{eq:equiv_W}
\end{align}

We define the Sobolev norm of $u\in\mathcal{S}_{\beta}\left(N_{s}\right)\eq{\ref{eq:notation}}C^{\beta}\left(\Sigma;\mathcal{S}\left(N_{s}\right)\right)$
by
\[
\left\Vert u\right\Vert _{\mathcal{H}_{\mathcal{W}}\left(N_{s}\right)}:=\left\Vert \mathcal{W}\mathcal{B}_{N_{s},N_{u}}u\right\Vert _{L^{2}\left(N\right)},
\]
and the anisotropic Sobolev space bundle by completion
\begin{align}
\mathcal{H}_{\mathcal{W}}\left(N_{s}\right) & :=\overline{\left\{ u\in\mathcal{S}_{\beta}\left(N_{s}\right),\quad\left\Vert u\right\Vert _{\mathcal{H}_{\mathcal{W}}\left(N_{s}\right)}<\infty\right\} }.\label{eq:def_HW_Ns}
\end{align}

\begin{cBoxB}{}
\begin{lem}
\label{lem:For--and}For any $\epsilon>0$, $K\in\mathbb{N}$, we
can choose $R$ in (\ref{eq:def_W-1}) large enough and $\omega_{0}>0$
so that $\exists C_{K,\epsilon}>0,\forall t\geq0$,
\begin{equation}
\left\Vert e^{tX_{\mathcal{F}}}\tilde{T}_{\geq\left(K+1\right)}\chi_{\left|\omega\right|\geq\omega_{0}}\right\Vert _{\mathcal{H}_{\mathcal{W}}\left(N_{s}\right)}\leq C_{K,\epsilon}e^{\left(\gamma_{K+1}^{+}+\epsilon\right)t}.\label{eq:norm_Lw_out}
\end{equation}
\end{lem}

\end{cBoxB}

\begin{proof}
Let $K\in\mathbb{N}$, $\epsilon>0$ and $t\geq0$. Recall that $e^{tX_{\mathcal{F}}}$
is a bundle map over $\tilde{\phi}^{t}:\Sigma\rightarrow\Sigma$.
Let $\rho\in\Sigma$ with $\boldsymbol{\omega}\left(\rho\right)\geq\omega_{0}$
and $m=\pi\left(\rho\right)\in M$. We first observe that the operator
\[
e^{tX_{\mathcal{F}}}\tilde{T}_{\geq\left(K+1\right)}:\mathcal{S}\left(N_{s}\left(\rho\right)\right)\rightarrow\mathcal{S}\left(N_{s}\left(\tilde{\phi}^{t}\left(\rho\right)\right)\right)
\]
is isomorphic to
\[
e^{tX_{\mathcal{F}}}\tilde{T}_{\geq\left(K+1\right)}:\mathcal{S}\left(E_{s}\left(m\right)\right)\rightarrow\mathcal{S}\left(E_{s}\left(\phi^{t}\left(m\right)\right)\right).
\]
with $m\in M$ on a compact space providing uniformity over $\Sigma$.
We apply Proposition \ref{prop:For-any-} in the appendix giving that
$\forall\epsilon>0,$$\exists C_{K,\epsilon}>0,\forall t\geq0$, $\forall u\in\mathcal{S}_{\beta}\left(N_{s}\right)$,
\begin{align*}
\left\Vert e^{tX_{\mathcal{F}}}\tilde{T}_{\geq\left(K+1\right)}\chi_{\left|\omega\right|\geq\omega_{0}}u\right\Vert _{\mathcal{H}_{\mathcal{W}}\left(N_{s}\right)}^{2} & =\int_{\Sigma}\left\Vert \left(e^{tX_{\mathcal{F}}}\tilde{T}_{\geq\left(K+1\right)}\chi_{\left|\omega\right|\geq\omega_{0}}u\right)\left(\rho\right)\right\Vert _{\mathcal{H}_{\mathcal{W}}\left(N_{s}\right)}^{2}d\rho\\
 & \ineq{\ref{eq:norm_Lw_out-1}}\left(C_{K,\epsilon}e^{\left(\gamma_{K+1}^{+}+\epsilon\right)t}\right)^{2}\int_{\Sigma}\left\Vert \left(\chi_{\left|\omega\right|\geq\omega_{0}}u\right)\left(\rho\right)\right\Vert _{\mathcal{H}_{\mathcal{W}}\left(N_{s}\right)}^{2}d\rho\\
 & \leq\left(C_{K,\epsilon}e^{\left(\gamma_{K+1}^{+}+\epsilon\right)t}\right)^{2}\left\Vert u\right\Vert _{\mathcal{H}_{\mathcal{W}}\left(N_{s}\right)}^{2},
\end{align*}
giving (\ref{eq:norm_Lw_out}).
\end{proof}
\begin{cBoxB}{}
\begin{prop}
For any $K\in\mathbb{N}$, we can choose $R$ in (\ref{eq:def_W-1})
large enough so that for every $k\in\left[0,K\right]$, there exists
$C_{k}$ such that
\begin{equation}
\left\Vert \tilde{T}_{k}\right\Vert _{\mathcal{H}_{\mathcal{W}}\left(N_{s}\right)}\leq C_{k}.\label{eq:T_k_bounded}
\end{equation}
\end{prop}

\end{cBoxB}

\begin{proof}
As in the proof of Lemma \ref{lem:For--and}, we use the fact that
we have a bundle map and at each point of the base, we use Lemma \ref{lem:Let-.-If}
that is uniform over $\Sigma$. 
\end{proof}

\subsection{Symbols and F.I.O. on $\Sigma$}

Recall $\mathcal{F}_{k}\left(E_{s}\right)\eq{\ref{eq:def_Fk}}\left|\mathrm{det}E_{s}\right|^{-1/2}\otimes\mathrm{Pol}_{k}\left(E_{s}\right)$
that is a $C^{\beta}$ continuous bundle over $M$ and more generally
$\mathcal{F}_{[k_{1},k_{2}]}\left(E_{s}\right)$ in (\ref{eq:def_F_0K}).

\begin{cBoxA}{}
\begin{defn}
\label{def:symbols}A \textbf{symbol} is a function on $M$ operator
valued in $L\left(\mathcal{F}_{\left[0,K\right]}\left(E_{s}\right)\right)$
(or $L\left(\mathcal{F}\left(E_{s}\right)\right)$):
\[
a\in C^{\beta}\left(M;L\left(\mathcal{F}_{\left[0,K\right]}\left(E_{s}\right)\right)\right).
\]
Using the projection map $\sqrt{\omega}d\pi:N_{s}\rightarrow E_{s}$
in (\ref{eq:dpi_N}) that is invertible, we obtain a lifted function
denoted
\[
\tilde{a}\in C^{\beta}\left(\Sigma;L\left(\mathcal{F}_{\left[0,K\right]}\left(N_{s}\right)\right)\right),
\]
considered below as a multiplicative operator.
\end{defn}

\end{cBoxA}

\begin{rem}
Here is precisely how we get $\tilde{a}$ from $a$: for $\rho\in\Sigma$,
$m=\pi\left(\rho\right)\in M$, one has $\left(\sqrt{\omega}d\pi_{\rho}\right)^{\circ}:\mathcal{S}\left(E_{s}\left(m\right)\right)\rightarrow\mathcal{S}\left(N_{s}\left(\rho\right)\right)$
and set
\[
\tilde{a}:\rho\in\Sigma\rightarrow\tilde{a}\left(\rho\right):=\left(\sqrt{\omega}d\pi_{\rho}\right)^{\circ}a\left(\sqrt{\omega}d\pi_{\rho}\right)^{-\circ}.
\]
\end{rem}

\begin{example}
Later we will use Definition (\ref{eq:def_Op_a}) only for symbols
of the form 
\begin{equation}
a=T_{k}\text{ (giving \ensuremath{\tilde{a}=\tilde{T}_{k}}),}\text{ or }a=T_{\left[0,K\right]}\text{ or }a\in\mathrm{End}\left(\mathcal{F}_{\left[0,K\right]}\right)\text{ or }a=T_{\geq\left(K+1\right)}\text{ or }a=\mathrm{Id}_{\mathcal{F}}.\label{eq:list_symbols}
\end{equation}
\end{example}

Recall the operators $\mathcal{T}_{N_{s}}:C^{\infty}\left(M\right)\rightarrow C^{\beta}\left(\Sigma;\mathcal{F}\left(N_{s}\right)\right)$
and $\mathcal{T}_{N_{s}}^{\Delta}:C^{\beta}\left(\Sigma;\mathcal{F}\left(N_{s}\right)\right)\rightarrow C^{\infty}\left(M\right)$
defined in (\ref{eq:def_T_Ns}) and (\ref{eq:def_T_Ns_Delta}).

\begin{cBoxA}{}
\begin{defn}[F.I.O., quantum operator]
\label{def:The-quantization-of-1}For a symbol $a\in C^{\beta}\left(M;L\left(\mathcal{F}\left(E_{s}\right)\right)\right)$
and $t\in\mathbb{R}$, the \textbf{quantization} of the symplectic
bundle map $e^{tX_{\mathcal{F}}}\tilde{a}:C^{\beta}\left(\Sigma;\mathcal{F}\left(N_{s}\right)\right)\rightarrow C^{\beta}\left(\Sigma;\mathcal{F}\left(N_{s}\right)\right)$
is the F.I.O. operator
\begin{align}
\mathrm{Op}_{\Sigma}\left(e^{tX_{\mathcal{F}}}\tilde{a}\right) & :=\mathcal{T}_{N_{s}}^{\Delta}\Upsilon_{t}^{1/2}e^{tX_{\mathcal{F}}}\tilde{a}\mathcal{T}_{N_{s}}\quad:C^{\infty}\left(M\right)\rightarrow C^{\infty}\left(M\right)\label{eq:def_Op_a}
\end{align}
\end{defn}

\end{cBoxA}

\begin{rem}
The term F.I.O. stands for \href{https://en.wikipedia.org/wiki/Fourier_integral_operator}{Fourier integral operator},
that is the usual name for this kind of operator. 
\end{rem}

Using the new notation (\ref{eq:def_Op_a}), previous Theorem \ref{Thm:approx_exptX_from_Ns}
can be rephrased as follows.

\begin{cBoxB}{}
\begin{thm}[\textbf{Approximation of the dynamics by a quantum operator}]
\textbf{\label{thm:Approximation-of-the}}For any $t\in\mathbb{R}$,
\begin{align}
e^{tX} & \approx\mathrm{Op}_{\Sigma}\left(e^{tX_{\mathcal{F}}}\right).\label{eq:expression}
\end{align}
\end{thm}

\end{cBoxB}

\subsubsection{Continuity theorem, boundness estimates with respect to time}
\begin{rem}
In the rest of this section we take any $K\in\mathbb{N}$ and fix
it. Then we choose a weight function $W$ with exponent $R$ such
that (\ref{eq:norm_Lw_out}) holds true and consider symbols $a\in C^{\beta}\left(M;L\left(\mathcal{F}_{\left[0,K\right]}\left(E_{s}\right)\right)\right)$.
\end{rem}

\begin{cBoxB}{}
\begin{thm}[continuity theorem for F.I.O.]
\label{thm:continuity_thm}There exists $C>0$, such that for any
symbol $a\in C^{\beta}\left(M;L\left(\mathcal{F}_{\left[0,K\right]}\left(E_{s}\right)\right)\right)$
and any $t\in\mathbb{R}$,
\begin{equation}
\left\Vert \mathrm{Op}_{\Sigma}\left(e^{tX_{\mathcal{F}}}\tilde{a}\right)\right\Vert _{\mathcal{H}_{W}\left(M\right)}\leq C\left\Vert e^{tX_{\mathcal{F}}}\tilde{a}\right\Vert _{\mathcal{H}_{\mathcal{W}}\left(\mathcal{F}\right)}.\label{eq:bound_OP_Sigma}
\end{equation}
\end{thm}

\end{cBoxB}

\begin{proof}
We first give an equivalent expression for the operator as we already
saw in Theorem \ref{Thm:For-any-}. This expression is more complicated
in appearance but more useful for the proof to work uniformly w.r.t.
time $t$.
\begin{lem}[Equivalent expression]
We have
\begin{equation}
\mathrm{Op}_{\Sigma}\left(e^{tX_{\mathcal{F}}}\tilde{a}\right)\eq{\ref{eq:def_Op_a}}\mathcal{T}_{N_{s}}^{\Delta}\Upsilon_{t}^{1/2}e^{tX_{\mathcal{F}}}\tilde{a}\mathcal{T}_{N_{s}}\approx\mathcal{T}_{\Sigma}^{\Delta}\left(\tilde{\mathrm{Op}}\left(d\tilde{\phi}_{T\Sigma}^{t}\right)\otimes\left(\mathcal{B}_{N_{s},N_{u}}\left(e^{tX_{\mathcal{F}}}\tilde{a}\right)\mathcal{B}_{N_{s},N_{u}}^{\dagger}\right)\right)\mathcal{T}_{\Sigma}\label{eq:equiv}
\end{equation}
\end{lem}

\begin{proof}
We have already obtained 
\begin{align*}
\mathcal{T}_{N_{s}}^{\Delta}\,\Upsilon_{t}^{1/2}e^{tX_{\mathcal{F}}}\mathcal{T}_{N_{s}} & \underset{(\ref{eq:exp_tX_N-1},\ref{eq:exp_tX_Lamda})}{\approx}\mathcal{T}_{\Sigma}^{\Delta}\tilde{\mathrm{Op}}\left(\left(d\tilde{\phi}^{t}\right)_{/\Sigma}\right)\mathcal{T}_{\Sigma}\\
 & \eq{\ref{eq:fact}}\mathcal{T}_{\Sigma}^{\Delta}\left(\tilde{\mathrm{Op}}\left(d\tilde{\phi}_{T\Sigma}^{t}\right)\otimes_{\Sigma}\tilde{\mathrm{Op}}\left(d\tilde{\phi}_{N}^{t}\right)\right)\mathcal{T}_{\Sigma}\\
 & \eq{\ref{eq:Op_tilde_dPhi_N}}\mathcal{T}_{\Sigma}^{\Delta}\left(\tilde{\mathrm{Op}}\left(d\tilde{\phi}_{T\Sigma}^{t}\right)\otimes\left(\mathcal{B}_{N_{s},N_{u}}e^{tX_{\mathcal{F}}}\mathcal{B}_{N_{s},N_{u}}^{\dagger}\right)\right)\mathcal{T}_{\Sigma}.
\end{align*}
By inserting the bundle map $\tilde{a}$ we obtain (\ref{eq:equiv}).
\end{proof}
We have defined $\mathcal{H}_{\mathcal{W}}\left(N_{s}\right)$ in
(\ref{eq:def_HW_Ns}). Similarly, we can define $\mathcal{H}_{\mathcal{W}}\left(N\right):=L^{2}\left(N;\mathcal{W}^{2}\right)$
with the weight $\mathcal{W}$ on $N$ in (\ref{eq:def_W-1}) and
define $\mathcal{H}_{\mathcal{W}}\left(T_{\Sigma}T^{*}M\right)$ by
the weight $\mathcal{W}$ extended trivially to $T_{\Sigma}T^{*}M\eq{\ref{eq:decomp_K_K0_N}}T\Sigma\oplus N$
by $\mathcal{W}\left(u,v\right)=\mathcal{W}\left(v\right)$ for any
$u\in T_{\rho}\Sigma$, $v\in N_{\rho}$.
\begin{lem}
\label{Bound_Check-1}$\exists C>0$,
\begin{equation}
\left\Vert \mathcal{T}_{\Sigma}\right\Vert _{\mathcal{H}_{W}\left(M\right)\rightarrow\mathcal{H}_{\mathcal{W}}\left(T_{\Sigma}T^{*}M\right)}\leq C,\qquad\left\Vert \mathcal{T}_{\Sigma}^{\Delta}\right\Vert _{\mathcal{H}_{\mathcal{W}}\left(T_{\Sigma}T^{*}M\right)\rightarrow\mathcal{H}_{W}\left(M\right)}\leq C.\label{eq:bound_T_check-1}
\end{equation}
\end{lem}

\begin{proof}
Recall the expression $\mathcal{T}_{\Sigma}\eq{\ref{eq:def_T_Sigma}}\chi_{\Sigma}^{\mu}r_{/\Sigma}\left(\widetilde{\exp^{\circ}}\right)\mathcal{T}$
where, from (\ref{eq:equiv_W}), the operator $\chi_{\Sigma}^{\mu}r_{/\Sigma}\left(\widetilde{\exp^{\circ}}\right)$
satisfies $\exists C>0$,
\[
\left\Vert \chi_{\Sigma}^{\mu}r_{/\Sigma}\left(\widetilde{\exp^{\circ}}\right)\right\Vert _{L^{2}\left(T^{*}M;W\right)\rightarrow\mathcal{H}_{\mathcal{W}}\left(T_{\Sigma}T^{*}M\right)}\leq C.
\]
Since $\mathcal{T}:\mathcal{H}_{W}\left(M\right)\rightarrow L^{2}\left(T^{*}M;W\right)$
is an isometry, we deduce the first bound in (\ref{eq:bound_T_check-1}).
Similarly from the expression $\mathcal{T}_{\Sigma}^{\Delta}\eq{\ref{eq:def_T_sigma_Delta}}\mathcal{T}^{\dagger}\widetilde{\left(\exp_{N}^{-1}\right)^{\circ}}r_{N}\chi_{\Sigma}^{\mu}$,
we deduce the second claim.
\end{proof}
\begin{lem}
We have
\begin{equation}
\left\Vert \tilde{\mathrm{Op}}\left(d\tilde{\phi}_{T\Sigma}^{t}\right)\otimes\left(\mathcal{B}_{N_{s},N_{u}}\left(e^{tX_{\mathcal{F}}}\tilde{a}\right)\mathcal{B}_{N_{s},N_{u}}^{\dagger}\right)\right\Vert _{\mathcal{H}_{\mathcal{W}}\left(T_{\Sigma}T^{*}M\right)}=\left\Vert e^{tX_{\mathcal{F}}}\tilde{a}\right\Vert _{\mathcal{H}_{\mathcal{W}}\left(\mathcal{F}\right)}.\label{eq:equiv2}
\end{equation}
\end{lem}

\begin{proof}
From the definition of the norms above, we have
\begin{align*}
\left\Vert \tilde{\mathrm{Op}}\left(d\tilde{\phi}_{T\Sigma}^{t}\right)\otimes\left(\mathcal{B}_{N_{s},N_{u}}\left(e^{tX_{\mathcal{F}}}\tilde{a}\right)\mathcal{B}_{N_{s},N_{u}}^{\dagger}\right)\right\Vert _{\mathcal{H}_{\mathcal{W}}\left(T_{\Sigma}T^{*}M\right)} & =\left\Vert \tilde{\mathrm{Op}}\left(d\tilde{\phi}_{T\Sigma}^{t}\right)\right\Vert _{L^{2}}\left\Vert \left(\mathcal{B}_{N_{s},N_{u}}\left(e^{tX_{\mathcal{F}}}\tilde{a}\right)\mathcal{B}_{N_{s},N_{u}}^{\dagger}\right)\right\Vert _{\mathcal{H}_{\mathcal{W}}\left(N\right)}\\
 & =\left\Vert e^{tX_{\mathcal{F}}}\tilde{a}\right\Vert _{\mathcal{H}_{\mathcal{W}}\left(\mathcal{F}\right)}
\end{align*}
\end{proof}
We come back to the proof of Theorem \ref{thm:continuity_thm}. Then
\[
\left\Vert \mathrm{Op}_{\Sigma}\left(e^{tX_{\mathcal{F}}}\tilde{a}\right)\right\Vert _{\mathcal{H}_{W}\left(M\right)}\ineq{\ref{eq:equiv},\ref{eq:bound_T_check-1},\ref{eq:equiv2}}C\left\Vert e^{tX_{\mathcal{F}}}\tilde{a}\right\Vert _{\mathcal{H}_{\mathcal{W}}\left(\mathcal{F}\right)}
\]
giving (\ref{eq:bound_OP_Sigma}).
\end{proof}
Here are some example of Theorem \ref{thm:continuity_thm} that we
will use later.

\begin{cBoxB}{}
\begin{cor}
\label{Thm:Boundness}We have $\forall\epsilon>0,\exists C_{K,\epsilon}>0$,$\forall t\geq0$,
(positive time)
\begin{equation}
\left\Vert \mathrm{Op}_{\Sigma}\left(e^{tX_{\mathcal{F}}}\tilde{T}_{\geq\left(K+1\right)}\right)\right\Vert _{\mathcal{H}_{W}\left(M\right)}\leq C_{K,\epsilon}e^{t\left(\gamma_{K+1}^{+}+\epsilon\right)}\label{eq:bound-1}
\end{equation}
and for any $0\leq k_{1}\leq k_{2}\leq K$, $\forall\epsilon>0,\exists C_{K,\epsilon}>0$,$\forall t\geq0$,
\begin{equation}
\left\Vert \mathrm{Op}_{\Sigma}\left(e^{tX_{\mathcal{F}}}\tilde{T}_{\left[k_{1},k_{2}\right]}\right)\right\Vert _{\mathcal{H}_{W}\left(M\right)}\leq C_{K,\epsilon}e^{t\left(\gamma_{k_{1}}^{+}+\epsilon\right)}.\label{eq:bound-4-1}
\end{equation}
Also $\forall t\leq0$, (negative time)
\begin{equation}
\left\Vert \mathrm{Op}_{\Sigma}\left(e^{tX_{\mathcal{F}}}\tilde{T}_{\left[k_{1},k_{2}\right]}\right)\right\Vert _{\mathcal{H}_{W}\left(M\right)}\leq C_{K,\epsilon}e^{t\left(\gamma_{k_{2}}^{-}-\epsilon\right)}.\label{eq:bound-4}
\end{equation}
\end{cor}

\end{cBoxB}

\begin{proof}
To get (\ref{eq:bound-1}) we have that $\exists C,\forall\epsilon>0,\exists C_{K,\epsilon}>0$,$\forall t\geq0$,

\begin{align*}
\left\Vert \mathrm{Op}_{\Sigma}\left(e^{tX_{\mathcal{F}}}\tilde{T}_{\geq\left(K+1\right)}\right)\right\Vert _{\mathcal{H}_{W}\left(M\right)} & \ineq{\ref{eq:bound_OP_Sigma}}C\left\Vert e^{tX_{\mathcal{F}}}\tilde{T}_{\geq\left(K+1\right)}\right\Vert _{\mathcal{H}_{\mathcal{W}}\left(\mathcal{F}\right)}\\
 & \ineq{\ref{eq:norm_Lw_out}}C_{K,\epsilon}e^{\left(\gamma_{K+1}^{+}+\epsilon\right)t}.
\end{align*}
To get (\ref{eq:bound-4-1}), in the equations below we will use ({*})
that $\gamma_{k+1}^{+}\leq\gamma_{k}^{+},\forall k\in\mathbb{N}$.
We have that $\exists C,\forall\epsilon>0,\exists C_{K},C_{K,\epsilon}>0$,$\forall t\geq0$,

\begin{align*}
\left\Vert \mathrm{Op}_{\Sigma}\left(e^{tX_{\mathcal{F}}}\tilde{T}_{\left[k_{1},k_{2}\right]}\right)\right\Vert _{\mathcal{H}_{W}\left(M\right)} & \ineq{\ref{eq:bound_OP_Sigma}}C\left\Vert e^{tX_{\mathcal{F}}}\tilde{T}_{\left[k_{1},k_{2}\right]}\right\Vert _{\mathcal{H}_{\mathcal{W}}\left(\mathcal{F}\right)}\leq C\left\Vert e^{tX_{\mathcal{F}_{\left[k_{1},k_{2}\right]}}}\right\Vert _{L^{2}\left(\Sigma;\mathcal{F}_{\left[k_{1},k_{2}\right]}\right)}\\
 & \ineq{\ref{eq:bound-2}}C_{K,\epsilon}\max_{k\in\left[k_{1},k_{2}\right]}e^{t\left(\gamma_{k}^{+}+\epsilon\right)}\eq *C_{K,\epsilon}e^{t\left(\gamma_{k_{1}}^{+}+\epsilon\right)}
\end{align*}
Finally to get (\ref{eq:bound-4}), we repeat the same first lines
as above and write $\forall\epsilon>0,\exists C_{K,\epsilon}>0$,$\forall t\geq0$,
\begin{align*}
\left\Vert \mathrm{Op}_{\Sigma}\left(e^{-tX_{\mathcal{F}}}\tilde{T}_{\left[k_{1},k_{2}\right]}\right)\right\Vert _{\mathcal{H}_{W}\left(M\right)}\leq & C\left\Vert e^{-tX_{\mathcal{F}_{\left[k_{1},k_{2}\right]}}}\right\Vert _{L^{2}\left(\Sigma;\mathcal{F}_{\left[k_{1},k_{2}\right]}\right)}\\
\underset{(\ref{eq:bound-2})}{\leq}C_{K,\epsilon}\max_{k\in\left[k_{1},k_{2}\right]}e^{-t\left(\gamma_{k}^{-}-\epsilon\right)}\eq * & C_{K,\epsilon}e^{-t\left(\gamma_{k_{2}}^{-}-\epsilon\right)}.
\end{align*}
\end{proof}
\begin{rem}
In particular, taking $t=0$ we get for any $k\in\left[0,K\right]$,
\begin{equation}
\left\Vert \mathrm{Op}_{\Sigma}\left(\tilde{T}_{k}\right)\right\Vert _{\mathcal{H}_{W}\left(M\right)}\underset{(\ref{eq:bound-4-1})}{\leq}C_{K}.\label{eq:-3}
\end{equation}
\end{rem}

\subsubsection{Composition formula}

Recall $\mathcal{F}_{\left[0,K\right]}$ defined in (\ref{eq:def_F_0K}).
Recall that the operators $\mathcal{T}_{N_{s}},\mathcal{T}_{N_{s}}^{\Delta}$
in Definition \ref{def:We-define-the} depends on the exponent $0<\mu<1$
introduced in the cut-off (\ref{eq:def_Chi_mu}). Consequently the
quantization $\mathrm{Op}_{\Sigma}\left(.\right)$ defined in (\ref{eq:def_Op_a})
also depends on $\mu$. Recall that $\beta$ is the Hölder exponent
of $E_{u},E_{s}$ in (\ref{eq:Holder_exp}).

\begin{cBoxB}{}
\begin{thm}[\textbf{Composition formula}]
\textbf{\label{thm:Composition-formula-For}}Take $\mu<\beta$. For
any symbols $a,b\in C^{\beta}\left(M;L\left(\mathcal{F}_{\left[0,K\right]}\left(E_{s}\right)\right)\right)$,
any $t,t'\in\mathbb{R}$, we have
\begin{equation}
\mathrm{Op}_{\Sigma}\left(e^{tX_{\mathcal{F}}}\tilde{a}\right)\mathrm{Op}_{\Sigma}\left(e^{t'X_{\mathcal{F}}}\tilde{b}\right)\approx\mathrm{Op}_{\Sigma}\left(e^{tX_{\mathcal{F}}}\tilde{a}e^{t'X_{\mathcal{F}}}\tilde{b}\right).\label{eq:compos_operators}
\end{equation}
(where the right hand side term has the factor $\Upsilon_{t+t'}^{1/2}$
in its Definition (\ref{eq:def_Op_a})).
\end{thm}

\end{cBoxB}

\begin{proof}
For convenience, we first naturally extend the Definition \ref{def:For-two-family}
of the equivalence $\approx$ in (\ref{eq:def_approx}) to operators
on $\mathcal{S}\left(T^{*}M\right)$ and also $C^{\beta}\left(\Sigma;\mathcal{F}\right)$
(instead of $C^{\infty}\left(M\right)$ only). We define
\begin{align}
\check{\mathcal{T}}_{N_{s}} & :=\mathcal{B}_{N_{s},N_{u}}^{\dagger}\chi_{\Sigma}^{\mu}\widetilde{\left(\exp_{N}\right)^{\circ}}\quad:\mathcal{S}\left(T^{*}M\right)\rightarrow C^{\beta}\left(\Sigma;\mathcal{F}\right),\label{eq:def_T_check_Ns}\\
\check{\mathcal{T}}_{N_{s}}^{\Delta} & :=\widetilde{\left(\exp_{N}^{-1}\right)^{\circ}}\chi_{\Sigma}^{\mu}\mathcal{B}_{N_{s},N_{u}}\quad:C^{\beta}\left(\Sigma;\mathcal{F}\right)\rightarrow C\left(T^{*}M\right).\nonumber 
\end{align}
These operators enter in the following expressions
\begin{equation}
\mathcal{T}_{N_{s}}\eq{\ref{eq:def_T_Ns}}\check{\mathcal{T}}_{N_{s}}\mathcal{T},\qquad\mathcal{T}_{N_{s}}^{\Delta}\eq{\ref{eq:def_T_Ns_Delta}}\mathcal{T}^{\dagger}\check{\mathcal{T}}_{N_{s}}^{\Delta}.\label{eq:_expressions}
\end{equation}
For an operators $A:C\left(\Sigma;\mathcal{F}\right)\rightarrow C\left(\Sigma;\mathcal{F}\right)$,
we denote 
\begin{equation}
\tilde{\tilde{A}}:=\check{\mathcal{T}}_{N_{s}}^{\Delta}A\check{\mathcal{T}}_{N_{s}}:\mathcal{S}\left(T^{*}M\right)\rightarrow\mathcal{S}'\left(T^{*}M\right)\label{eq:def_tilde2}
\end{equation}
.
\begin{defn}
\label{def:For-two-family-1}For two operators $A,B:\mathcal{S}\left(T^{*}M\right)\rightarrow\mathcal{S}'\left(T^{*}M\right)$
we write
\begin{equation}
A\approx B\label{eq:def_approx-1}
\end{equation}
if there exists $m<0$, $t\in\mathbb{R}$, for any $N>0$, any $\sigma>0$,
there exists a constant $C_{\sigma,N,t}>0$ such that for any $\rho,\rho'\in T^{*}M$,
\[
\left|\langle\delta_{\rho'}|\chi_{\Sigma,\sigma}\left(A-B\right)\chi_{\Sigma,\sigma}\delta_{\rho}\rangle\right|\leq C_{\sigma,N,t}\left\langle \mathrm{dist}_{g}\left(\rho',\tilde{\phi}^{t}\left(\rho\right)\right)\right\rangle ^{-N}\left\langle \left|\rho\right|\right\rangle ^{m}.
\]
For two operators $A,B:C\left(\Sigma;\mathcal{F}\right)\rightarrow C\left(\Sigma;\mathcal{F}\right)$
we write $A\approx B$ if $\tilde{\tilde{A}}\approx\tilde{\tilde{B}}$. 
\end{defn}

Using operators $\mathcal{T}_{N_{s}},\mathcal{T}_{N_{s}}^{\Delta}$
defined in (\ref{eq:def_T_Ns},\ref{eq:def_T_Ns_Delta}), let
\begin{equation}
\mathcal{P}_{N_{s}}:=\mathcal{T}_{N_{s}}\mathcal{T}_{N_{s}}^{\Delta}\quad:C^{\beta}\left(\Sigma;\mathcal{F}\right)\rightarrow C^{\beta}\left(\Sigma;\mathcal{F}\right).\label{eq:def_P_Ns}
\end{equation}
\begin{lem}
We have
\begin{equation}
\mathcal{P}_{N_{s}}^{2}\approx\mathcal{P}_{N_{s}},\qquad\mathcal{P}_{N_{s}}\mathcal{T}_{N_{s}}\approx\mathcal{T}_{N_{s}}.\label{eq:equiv6}
\end{equation}
\end{lem}

\begin{proof}
Taking $t=0$ in (\ref{eq:exp_tX_N-1}) gives
\begin{equation}
\mathcal{T}_{N_{s}}^{\Delta}\mathcal{T}_{N_{s}}\approx\mathrm{Id}.\label{eq:equiv_Id}
\end{equation}
Then
\[
\mathcal{P}_{N_{s}}^{2}\eq{\ref{eq:def_P_Ns}}\mathcal{T}_{N_{s}}\left(\mathcal{T}_{N_{s}}^{\Delta}\mathcal{T}_{N_{s}}\right)\mathcal{T}_{N_{s}}^{\Delta}\underset{(\ref{eq:equiv_Id})}{\approx}\mathcal{T}_{N_{s}}\mathcal{T}_{N_{s}}^{\Delta}\eq{\ref{eq:def_P_Ns}}\mathcal{P}_{N_{s}}.
\]
\[
\mathcal{P}_{N_{s}}\mathcal{T}_{N_{s}}\eq{\ref{eq:def_P_Ns}}\mathcal{T}_{N_{s}}\left(\mathcal{T}_{N_{s}}^{\Delta}\mathcal{T}_{N_{s}}\right)\underset{(\ref{eq:equiv_Id})}{\approx}\mathcal{T}_{N_{s}}.
\]
\end{proof}
We show the following Lemma (similar to \cite[Lemma 4.47]{faure_tsujii_Ruelle_resonances_density_2016})
\begin{lem}[Basic Lemma]
For any symbol $a\in C^{\beta}\left(M;L\left(\mathcal{F}_{\left[0,K\right]}\left(E_{s}\right)\right)\right)$,
we have for $\mu<\beta$ that
\begin{equation}
\left[\tilde{a},\mathcal{P}_{N_{s}}\right]\approx0.\label{eq:commut-a_P_Ns}
\end{equation}
\end{lem}

\begin{proof}
Observe that
\begin{equation}
\check{\mathcal{T}}_{N_{s}}\check{\mathcal{T}}_{N_{s}}^{\Delta}\underset{(\ref{eq:def_T_check_Ns})}{\approx}\mathrm{Id},\qquad\check{\mathcal{T}}_{N_{s}}^{\Delta}\check{\mathcal{T}}_{N_{s}}\underset{(\ref{eq:def_T_check_Ns})}{\approx}\mathrm{\mathcal{P}}_{N}.\label{eq:approx_Id}
\end{equation}
The notation $\approx$ has been defined for operators. This definition
can been understood also for Schwartz kernel of operators lifted on
$T^{*}M$ and we use it below. We have $\forall N\geq0,\exists C_{N}>0,\forall\rho,\rho'\in T^{*}M$,
\begin{align}
\left|\langle\delta_{\rho'}|\tilde{\tilde{\mathcal{P}}}_{N_{s}}\delta_{\rho}\rangle\right| & \eq{\ref{eq:def_tilde2}}\left|\langle\delta_{\rho'}|\check{\mathcal{T}}_{N_{s}}^{\Delta}\mathcal{P}_{N_{s}}\check{\mathcal{T}}_{N_{s}}\delta_{\rho}\rangle\right|\eq{\ref{eq:def_P_Ns},\ref{eq:_expressions}}\left|\langle\delta_{\rho'}|\check{\mathcal{T}}_{N_{s}}^{\Delta}\check{\mathcal{T}}_{N_{s}}\mathcal{T}\mathcal{T}^{\dagger}\check{\mathcal{T}}_{N_{s}}^{\Delta}\check{\mathcal{T}}_{N_{s}}\delta_{\rho}\rangle\right|\label{eq:kernel_P_tild2}\\
 & \underset{(\ref{eq:approx_Id},\ref{eq:def_P_wave_packet_projector})}{\approx}\left|\langle\delta_{\rho'}|\mathcal{P}\delta_{\rho}\rangle\right|\ineq{\ref{eq:estimate_Bergman_kernel}}C_{N}\left\langle \mathrm{dist}_{g}\left(\rho',\rho\right)\right\rangle ^{-N}.\nonumber 
\end{align}
Then
\begin{align}
\left|\langle\delta_{\rho'}|\check{\mathcal{T}}_{N_{s}}^{\Delta}\left[\tilde{a},\mathcal{P}_{N_{s}}\right]\check{\mathcal{T}}_{N_{s}}\delta_{\rho}\rangle\right| & \underset{(\ref{eq:approx_Id},\ref{eq:def_tilde2})}{\approx}\left|\langle\delta_{\rho'}|\left[\tilde{\tilde{a}},\tilde{\tilde{\mathcal{P}}}_{N_{s}}\right]\delta_{\rho}\rangle\right|\label{eq:commut_a_PNs}\\
 & =\left|\int_{\rho''\in T^{*}M}\langle\delta_{\rho'}|\tilde{\tilde{a}}\delta_{\rho''}\rangle\langle\delta_{\rho''}|\tilde{\tilde{\mathcal{P}}}_{N_{s}}\delta_{\rho}\rangle-\langle\delta_{\rho'}|\tilde{\tilde{\mathcal{P}}}_{N_{s}}\delta_{\rho''}\rangle\langle\delta_{\rho''}|\tilde{\tilde{a}}\delta_{\rho}\rangle\right|.\nonumber 
\end{align}
The bundle $E_{s}$ is Hölder continuous with some exponent $\beta>0$
and consequently $\tilde{a}$ satisfies the following property (called
slowly varying property in \cite[def 4.40]{faure_tsujii_Ruelle_resonances_density_2016}),
using any local trivialization, for any $\varrho,\varrho'\in\Sigma$,
setting $x=\pi\left(\varrho\right),x'=\pi\left(\varrho'\right)\in M$,
we assume $\omega=\boldsymbol{\omega}\left(\varrho\right)\asymp\boldsymbol{\omega}\left(\varrho'\right)$,
we use $\left|x'-x\right|\underset{(\ref{eq:metric_g_tilde_in_coordinates})}{\lesssim}\left\langle \omega\right\rangle ^{-1/2}\mathrm{dist}_{g}\left(\varrho',\varrho\right)$,
and we have
\begin{equation}
\left\Vert \tilde{a}\left(\varrho'\right)-\tilde{a}\left(\varrho\right)\right\Vert \asymp\left\Vert a\left(x'\right)-a\left(x\right)\right\Vert \lesssim\left|x'-x\right|^{\beta}\asymp\left\langle \omega\right\rangle ^{-\beta/2}\left(\mathrm{dist}_{g}\left(\varrho',\varrho\right)\right)^{\beta}.\label{eq:ineq2}
\end{equation}
Hence for $\rho,\rho'\in T^{*}M$, due to the cutoff $\chi_{\Sigma}^{\mu}$
in (\ref{eq:def_Chi_mu-1}) at distance $\left\langle \omega\right\rangle ^{\mu/2}$
from $\Sigma$, writing $\rho\equiv\left(\varrho,v\right)\in T^{*}M$
with $\varrho\in\Sigma$, $v\in N_{\varrho}$, $\exp_{N}\left(v\right)=\rho,$
we have for $\mathrm{dist}_{g}\left(\rho',\rho\right)\lesssim1$ and
$\mathrm{dist}_{g}\left(\rho'',\rho\right)\lesssim1$ that 
\begin{equation}
\left|\langle\delta_{\rho'}|\tilde{\tilde{a}}\delta_{\rho''}\rangle-\langle\delta_{\rho''}|\tilde{\tilde{a}}\delta_{\rho}\rangle\right|\leq\left\langle \omega\right\rangle ^{\mu/2}\left|\tilde{a}\left(\varrho'\right)-\tilde{a}\left(\varrho\right)\right|\ineq{\ref{eq:ineq2}}\left\langle \omega\right\rangle ^{\mu/2}\left\langle \omega\right\rangle ^{-\beta/2}.\label{eq:Holder_irreg}
\end{equation}
We deduce
\[
\left|\langle\delta_{\rho'}|\check{\mathcal{T}}_{N_{s}}^{\Delta}\left[\tilde{a},\mathcal{P}_{N_{s}}\right]\check{\mathcal{T}}_{N_{s}}\delta_{\rho}\rangle\right|\ineq{\ref{eq:commut-a_P_Ns},\ref{eq:kernel_P_tild2}}C\left\langle \omega\right\rangle ^{-\left(\beta-\mu\right)/2}C_{N}\left\langle \mathrm{dist}_{g}\left(\rho',\rho\right)\right\rangle ^{-N}.
\]
From Definition \ref{def:For-two-family-1}, we get that $\left[\tilde{a},\mathcal{P}_{N_{s}}\right]\approx0$
if $\beta>\mu$.
\end{proof}
Observe that
\begin{equation}
\mathrm{Op}_{\Sigma}\left(e^{tX_{\mathcal{F}}}\right)\mathrm{Op}_{\Sigma}\left(e^{t'X_{\mathcal{F}}}\right)\underset{(\ref{eq:expression})}{\approx}e^{tX_{\mathcal{F}}}e^{t'X_{\mathcal{F}}}=e^{\left(t+t'\right)X_{\mathcal{F}}}\underset{(\ref{eq:expression})}{\approx}\mathrm{Op}_{\Sigma}\left(e^{\left(t+t'\right)X_{\mathcal{F}}}\right).\label{eq:equiv5}
\end{equation}
Also
\begin{equation}
\mathrm{Op}_{\Sigma}\left(\tilde{a}\right)\mathrm{Op}_{\Sigma}\left(\tilde{b}\right)\eq{\ref{eq:def_Op_a},\ref{eq:def_P_Ns}}\mathcal{T}_{N_{s}}^{\Delta}\tilde{a}\mathcal{P}_{N_{s}}\tilde{b}\mathcal{T}_{N_{s}}\underset{(\ref{eq:commut-a_P_Ns},\ref{eq:equiv6})}{\approx}\mathcal{T}_{N_{s}}^{\Delta}\tilde{a}\tilde{b}\mathcal{T}_{N_{s}}=\mathrm{Op}_{\Sigma}\left(\tilde{a}\tilde{b}\right).\label{eq:equiv4}
\end{equation}
and similarly
\begin{equation}
\mathrm{Op}_{\Sigma}\left(e^{tX_{\mathcal{F}}}\right)\mathrm{Op}_{\Sigma}\left(\tilde{a}\right)\approx\mathrm{Op}_{\Sigma}\left(e^{tX_{\mathcal{F}}}\tilde{a}\right)\label{eq:equiv1}
\end{equation}
\begin{equation}
\mathrm{Op}_{\Sigma}\left(\tilde{a}\right)\mathrm{Op}_{\Sigma}\left(e^{tX_{\mathcal{F}}}\right)\approx\mathrm{Op}_{\Sigma}\left(\tilde{a}e^{tX_{\mathcal{F}}}\right).\label{eq:equiv3}
\end{equation}
From these relations, we deduce (\ref{eq:compos_operators}) as follows.
We set $\tilde{a}'=e^{-t'X_{\mathcal{F}}}\tilde{a}e^{t'X_{\mathcal{F}}}$.
\begin{align*}
\mathrm{Op}_{\Sigma}\left(e^{tX_{\mathcal{F}}}\tilde{a}\right)\mathrm{Op}_{\Sigma}\left(e^{t'X_{\mathcal{F}}}\tilde{b}\right) & \underset{(\ref{eq:equiv1},\ref{eq:equiv3},\ref{eq:equiv4})}{\approx}\mathrm{Op}_{\Sigma}\left(e^{tX_{\mathcal{F}}}\right)\mathrm{Op}_{\Sigma}\left(e^{t'X_{\mathcal{F}}}\right)\mathrm{Op}_{\Sigma}\left(\tilde{a}'\tilde{b}\right)\\
 & \underset{(\ref{eq:equiv5})}{\approx}\mathrm{Op}_{\Sigma}\left(e^{tX_{\mathcal{F}}}e^{t'X_{\mathcal{F}}}\right)\mathrm{Op}_{\Sigma}\left(\tilde{a}'\tilde{b}\right)\underset{(\ref{eq:equiv1},\ref{eq:equiv3},\ref{eq:equiv4})}{\approx}\mathrm{Op}_{\Sigma}\left(e^{tX_{\mathcal{F}}}\tilde{a}e^{t'X_{\mathcal{F}}}\tilde{b}\right).
\end{align*}
\end{proof}
The next corollary follows as special cases of (\ref{eq:compos_operators}).

\begin{cBoxB}{}
\begin{cor}
\label{cor:As-particular-cases}Take $\mu<\beta$. We have
\begin{itemize}
\item \textbf{Egorov formula:} for a symbol $a\in C^{\beta}\left(M;\mathrm{End}\left(\mathcal{F}_{\left[0,K\right]}\left(E_{s}\right)\right)\right),$
$t\in\mathbb{R}$,
\begin{equation}
\mathrm{Op}_{\Sigma}\left(e^{tX_{\mathcal{F}}}\right)\mathrm{Op}_{\Sigma}\left(\tilde{a}\right)\approx\mathrm{Op}_{\Sigma}\left(e^{t\mathrm{ad}_{X_{\mathcal{F}}}}\tilde{a}\right)\mathrm{Op}_{\Sigma}\left(e^{tX_{\mathcal{F}}}\right).\label{eq:Egorov_formula}
\end{equation}
with the natural action $e^{t\mathrm{ad}_{X_{\mathcal{F}}}}\tilde{a}=e^{t\left[X_{\mathcal{F}},.\right]}\tilde{a}:=e^{tX_{\mathcal{F}}}\tilde{a}e^{-tX_{\mathcal{F}}}$.
\item \textbf{Approximate projectors}: for any $k,k'\leq K$,
\begin{equation}
\mathrm{Op}_{\Sigma}\left(\tilde{T}_{k}\right)\mathrm{Op}_{\Sigma}\left(\tilde{T}_{k'}\right)\approx\delta_{k=k'}\mathrm{Op}_{\Sigma}\left(\tilde{T}_{k}\right).\label{eq:-2}
\end{equation}
\item \textbf{Decomposition of the dynamics in components}: for any $k,k'\leq K$,
\begin{equation}
\mathrm{Op}_{\Sigma}\left(e^{tX_{\mathcal{F}}}\tilde{T}_{k}\right)\mathrm{Op}_{\Sigma}\left(e^{t'X_{\mathcal{F}}}\tilde{T}_{k'}\right)\approx\delta_{k=k'}\mathrm{Op}_{\Sigma}\left(e^{\left(t+t'\right)X_{\mathcal{F}}}\tilde{T}_{k}\right).\label{eq:alg}
\end{equation}
\end{itemize}
\end{cor}

\end{cBoxB}

\begin{rem}
The algebraic structure of (\ref{eq:alg}) will manifest itself later
in the band structure of the Ruelle spectrum. 
\end{rem}

\begin{proof}
To get (\ref{eq:Egorov_formula}), we write
\begin{align*}
\mathrm{Op}_{\Sigma}\left(e^{tX_{\mathcal{F}}}\right)\mathrm{Op}_{\Sigma}\left(\tilde{a}\right) & \underset{(\ref{eq:compos_operators})}{\approx}\mathrm{Op}_{\Sigma}\left(e^{tX_{\mathcal{F}}}\tilde{a}\right)=\mathrm{Op}_{\Sigma}\left(\left(e^{t\mathrm{ad}_{X_{\mathcal{F}}}}\tilde{a}\right)e^{tX_{\mathcal{F}}}\right)\\
 & \underset{(\ref{eq:compos_operators})}{\approx}\mathrm{Op}_{\Sigma}\left(e^{t\mathrm{ad}_{X_{\mathcal{F}}}}\tilde{a}\right)\mathrm{Op}_{\Sigma}\left(e^{tX_{\mathcal{F}}}\right)
\end{align*}
To get (\ref{eq:-2}), we write
\[
\mathrm{Op}_{\Sigma}\left(\tilde{T}_{k}\right)\mathrm{Op}_{\Sigma}\left(\tilde{T}_{k'}\right)\underset{(\ref{eq:compos_operators})}{\approx}\mathrm{Op}_{\Sigma}\left(\tilde{T}_{k}\tilde{T}_{k'}\right)\eq{\ref{eq:delta_T}}\delta_{k=k'}\mathrm{Op}_{\Sigma}\left(\tilde{T}_{k}\right).
\]
To get (\ref{eq:alg}), we write
\begin{align*}
\mathrm{Op}_{\Sigma}\left(e^{tX_{\mathcal{F}}}\tilde{T}_{k}\right)\mathrm{Op}_{\Sigma}\left(e^{t'X_{\mathcal{F}}}\tilde{T}_{k'}\right) & \underset{(\ref{eq:compos_operators})}{\approx}\mathrm{Op}_{\Sigma}\left(e^{tX_{\mathcal{F}}}\tilde{T}_{k}e^{t'X_{\mathcal{F}}}\tilde{T}_{k'}\right)\\
 & \eq{\ref{eq:commut_T}}\mathrm{Op}_{\Sigma}\left(e^{tX_{\mathcal{F}}}e^{t'X_{\mathcal{F}}}\tilde{T}_{k}\tilde{T}_{k'}\right)\eq{\ref{eq:delta_T}}\delta_{k=k'}\mathrm{Op}_{\Sigma}\left(e^{\left(t+t'\right)X_{\mathcal{F}}}\tilde{T}_{k}\right).
\end{align*}
\end{proof}

\subsubsection{Trace formula}

We present some trace formula for PDO (Pseudo Differential Operators),
i.e. FIO (Fourier Integral Operators) at $t=0$. We will use this
formula in section \ref{subsec:Approximate-projector}. It is also
possible to express $\mathrm{Tr}\left(\mathrm{Op}_{\Sigma}\left(\boldsymbol{1}_{\left[\omega,\omega'\right]}e^{tX_{\mathcal{F}}}\tilde{a}\right)\right)$
for $t\neq0$, but we don't need it in this paper.

\begin{cBoxB}{}
\begin{thm}[Trace formula]
\label{thm:Take-.-There}Take $\mu<\beta$. There exists $C>0$,
for any $\omega\geq\delta\geq1$, for any symbol $a\in C^{\beta}\left(M;L\left(\mathcal{F}_{\left[0,K\right]}\left(E_{s}\right)\right)\right)$,

\begin{equation}
\left\Vert \mathrm{Op}_{\Sigma}\left(\boldsymbol{1}_{\boldsymbol{\omega}\left(\rho\right)\in\left[\omega-\delta,\omega+\delta\right]}\tilde{a}\right)\right\Vert _{\mathrm{Tr}_{\mathcal{H}_{W}\left(M\right)}}\leq C\omega^{d}\delta.\label{eq:bound_OP_Sigma-1}
\end{equation}
\begin{equation}
\mathrm{Tr}\left(\mathrm{Op}_{\Sigma}\left(\boldsymbol{1}_{\boldsymbol{\omega}\left(\rho\right)\in\left[\omega-\delta,\omega+\delta\right]}\tilde{a}\right)\right)=\left(\frac{\omega^{d}}{\left(2\pi\right)^{d+1}}\left(2\delta\right)\right)\int_{M}\mathrm{Tr}\left(a\left(m\right)\right)dm+r\left(\omega,\delta\right),\label{eq:bound_OP_Sigma-1-1}
\end{equation}
with remainder
\[
\left|r\left(\omega,\delta\right)\right|\leq C\omega^{d}\delta\left(\left|\omega\right|^{-\left(\beta-\mu\right)/2}+\frac{\delta}{\omega}\right)+C\omega^{d}.
\]
\end{thm}

\end{cBoxB}

\begin{rem}
Later in (\ref{eq:last}) we will use (\ref{eq:bound_OP_Sigma-1-1})
as follows. For any small $c'>0$ we will take $\delta\gg1$ and $\omega\gg\delta$
large enough so that the remainder $r\left(\omega,\delta\right)$
is smaller than $c'\omega^{d}\delta$, i.e. negligible w.r.t. the
first term.
\end{rem}

\begin{proof}
Since the symbol $\boldsymbol{1}_{\boldsymbol{\omega}\left(\rho\right)\in\left[\omega-\delta,\omega+\delta\right]}\tilde{a}$
is fiber-wise, we have
\begin{align}
\mathrm{Op}_{\Sigma}\left(\boldsymbol{1}_{\boldsymbol{\omega}\left(\rho\right)\in\left[\omega-\delta,\omega+\delta\right]}\tilde{a}\right) & \eq{\ref{eq:def_Op_a}}\mathcal{T}_{N_{s}}^{\Delta}\boldsymbol{1}_{\boldsymbol{\omega}\left(\rho\right)\in\left[\omega-\delta,\omega+\delta\right]}\tilde{a}\mathcal{T}_{N_{s}}\nonumber \\
 & =\int_{\rho\in\Sigma}A\left(\rho\right)\boldsymbol{1}_{\boldsymbol{\omega}\left(\rho\right)\in\left[\omega-\delta,\omega+\delta\right]}\tilde{a}\left(\rho\right)B\left(\rho\right)\frac{d\rho}{\left(2\pi\right)^{d+1}}\label{eq:int1}
\end{align}
with for $\rho\in\Sigma$, the finite dimensional vector space $\mathcal{F}_{\left[0,K\right]}\left(N_{s}\left(\rho\right)\right)$
and the finite rank linear operators
\begin{align}
A\left(\rho\right) & :=\mathcal{T}_{N_{s}}^{\Delta}\tilde{T}_{\left[0,K\right]}\left(\rho\right)\quad:\mathcal{F}_{\left[0,K\right]}\left(N_{s}\left(\rho\right)\right)\rightarrow C^{\infty}\left(M\right)\label{eq:def_A-1}
\end{align}
\begin{align}
B\left(\rho\right) & :=\tilde{T}_{\left[0,K\right]}\left(\rho\right)\mathcal{T}_{N_{s}}\quad:C^{\infty}\left(M\right)\rightarrow\mathcal{F}_{\left[0,K\right]}\left(N_{s}\left(\rho\right)\right)\label{eq:def_B-2}
\end{align}
Where $\tilde{T}_{\left[0,K\right]}\left(\rho\right)$ is the bundle
map projector onto $\mathcal{F}_{\left[0,K\right]}\left(N_{s}\left(\rho\right)\right)$,
defined in (\ref{eq:def_Tk-1}). We have $\exists C>0,\forall\rho\in\Sigma,$
\[
\left\Vert A\left(\rho\right)\right\Vert \leq C,\qquad\left\Vert B\left(\rho\right)\right\Vert \leq C.
\]
and
\begin{equation}
\left\Vert \boldsymbol{1}_{\boldsymbol{\omega}\left(\rho\right)\in\left[\omega-\delta,\omega+\delta\right]}\tilde{a}\left(\rho\right)\right\Vert _{\mathrm{Tr}_{\mathcal{H}_{W}\left(M\right)}}\leq C\label{eq:C}
\end{equation}
For $\omega\geq\delta>0$, we have 
\begin{equation}
\int_{\Sigma}\boldsymbol{1}_{\boldsymbol{\omega}\left(\rho\right)\in\left[\omega-\delta,\omega+\delta\right]}d\rho\eq{\ref{eq:dvol_E0*}}\mathrm{Vol}\left(M\right)\int_{\omega-\delta}^{\omega+\delta}\omega'^{d}d\omega'=\mathrm{Vol}\left(M\right)\omega^{d}\left(2\delta\right)\left(1+O\left(\frac{\delta}{\omega}\right)\right).\label{eq:int_sig}
\end{equation}
Hence
\begin{align*}
\left\Vert \mathrm{Op}_{\Sigma}\left(\boldsymbol{1}_{\boldsymbol{\omega}\left(\rho\right)\in\left[\omega-\delta,\omega+\delta\right]}\tilde{a}\right)\right\Vert _{\mathrm{Tr}_{\mathcal{H}_{W}\left(M\right)}} & \ineq{\ref{eq:int1}}C\int_{\Sigma}\left\Vert \boldsymbol{1}_{\boldsymbol{\omega}\left(\rho\right)\in\left[\omega-\delta,\omega+\delta\right]}\tilde{a}\left(\rho\right)\right\Vert _{\mathrm{Tr}_{\mathcal{H}_{W}\left(M\right)}}\frac{d\rho}{\left(2\pi\right)^{d+1}}\\
 & \ineq{\ref{eq:C},\ref{eq:int_sig}}C\omega^{d}\delta,
\end{align*}
giving (\ref{eq:bound_OP_Sigma-1}). Then
\[
\mathrm{Tr}\left(\mathrm{Op}_{\Sigma}\left(\boldsymbol{1}_{\boldsymbol{\omega}\left(\rho\right)\in\left[\omega-\delta,\omega+\delta\right]}\tilde{a}\right)\right)\eq{\ref{eq:int1}}\int_{\Sigma}\mathrm{Tr}\left(\boldsymbol{1}_{\boldsymbol{\omega}\left(\rho\right)\in\left[\omega-\delta,\omega+\delta\right]}\tilde{a}\left(\rho\right)B\left(\rho\right)A\left(\rho\right)\right)\frac{d\rho}{\left(2\pi\right)^{d+1}},
\]
where
\begin{align*}
B\left(\rho\right)A\left(\rho\right) & \eq{\ref{eq:def_A-1},\ref{eq:def_B-2},\ref{eq:def_T_Ns_Delta},\ref{eq:def_T_Ns}}\tilde{T}_{\left[0,K\right]}\left(\rho\right)\mathcal{B}_{N_{s},N_{u}}^{\dagger}\chi_{\Sigma}^{\mu}\widetilde{\left(\exp_{N}\right)^{\circ}}\mathcal{T}\mathcal{T}^{\dagger}\widetilde{\left(\exp_{N}^{-1}\right)^{\circ}}\chi_{\Sigma}^{\mu}\mathcal{B}_{N_{s},N_{u}}\tilde{T}_{\left[0,K\right]}\left(\rho\right)\\
 & =\tilde{T}_{\left[0,K\right]}\left(\rho\right)+O\left(\left\langle \omega\right\rangle ^{-\left(\beta-\mu\right)/2}\right),
\end{align*}
where $\boldsymbol{\omega}\left(\rho\right)^{-\frac{1}{2}\left(\beta-\mu\right)}$
comes from Hölder irregularity as in (\ref{eq:Holder_irreg}). Hence
\[
\mathrm{Tr}\left(\mathrm{Op}_{\Sigma}\left(\boldsymbol{1}_{\boldsymbol{\omega}\left(\rho\right)\in\left[\omega-\delta,\omega+\delta\right]}\tilde{a}\right)\right)=\int_{\Sigma}\mathrm{Tr}\left(\boldsymbol{1}_{\boldsymbol{\omega}\left(\rho\right)\in\left[\omega-\delta,\omega+\delta\right]}\tilde{a}\left(\rho\right)\right)\frac{d\rho}{\left(2\pi\right)^{d+1}}+R
\]
with remainder
\begin{align*}
\left|R\right| & \leq C\int_{\Sigma}\boldsymbol{1}_{\boldsymbol{\omega}\left(\rho\right)\in\left[\omega-\delta,\omega+\delta\right]}\boldsymbol{\omega}\left(\rho\right)^{-\left(\beta-\mu\right)/2}d\rho+C\omega^{d}\\
 & \ineq{\ref{eq:int_sig}}C\omega^{d}\delta\omega^{-\left(\beta-\mu\right)/2}+C\omega^{d}.
\end{align*}
 $C\omega^{d}$ is due to sharp cutoff in frequency, and main term
\[
\int_{\Sigma}\mathrm{Tr}\left(\boldsymbol{1}_{\boldsymbol{\omega}\left(\rho\right)\in\left[\omega-\delta,\omega+\delta\right]}\tilde{a}\left(\rho\right)\right)\frac{d\rho}{\left(2\pi\right)^{d+1}}\eq{\ref{eq:dvol_E0*},\ref{eq:int_sig}}\frac{\omega^{d}\left(2\delta\right)}{\left(2\pi\right)^{d+1}}\left(1+O\left(\frac{\delta}{\omega}\right)\right)\int_{M}\mathrm{Tr}\left(a\left(m\right)\right)dm.
\]
\end{proof}
We have obtained (\ref{eq:bound_OP_Sigma-1-1}).

\section{\label{subsec:Band-spectrum-of}Proof of Theorem \ref{Thm: bands}
(band spectrum)}

We will follow the strategy presented in the proof of \cite[Lemma 5.16]{faure_tsujii_Ruelle_resonances_density_2016}.
Let $K\in\mathbb{N}$, $\epsilon>0$ be fixed and consider $z\in\mathbb{C}$
in the spectral plane such that
\begin{equation}
\gamma_{K+1}^{+}+\epsilon<\mathrm{Re}\left(z\right)<\gamma_{K}^{-}-\epsilon.\label{eq:Re_z}
\end{equation}
We put 
\begin{equation}
\omega=\mathrm{Im}\left(z\right).\label{eq:def_omega}
\end{equation}
Let $\delta>0$. We consider the frequency intervals
\begin{equation}
J_{\omega}:=\left[\omega-1,\omega+1\right],\qquad J'_{\omega,\delta}:=\left[\omega-1-\delta,\omega+1+\delta\right],\label{eq:J_delta}
\end{equation}
so that $J_{\omega}\subset J'_{\omega,\delta}$. Let $\sigma>0$.
We consider the following partition of the cotangent bundle
\begin{equation}
T^{*}M=\Omega_{0}\cup\Omega_{1}\cup\Omega_{2}\label{eq:partition}
\end{equation}
with, as in (\ref{eq:def_Lambda}),
\begin{equation}
\Omega_{0}:=\left\{ \rho\in T^{*}M\,\mid\quad\left\Vert \rho_{u}+\rho_{s}\right\Vert _{g_{\rho}}\leq\sigma\text{ and }\boldsymbol{\omega}\left(\rho\right)\in J'_{\omega,\delta}\right\} ,\label{eq:def_Omega_sigma}
\end{equation}
\begin{equation}
\Omega_{1}:=\left\{ \rho\in T^{*}M\,\mid\quad\left\Vert \rho_{u}+\rho_{s}\right\Vert _{g_{\rho}}>\sigma\text{ and }\boldsymbol{\omega}\left(\rho\right)\in J'_{\omega,\delta}\right\} ,\label{eq:def_Omega_1}
\end{equation}
\[
\Omega_{2}:=\left\{ \rho\in T^{*}M\,\mid\quad\boldsymbol{\omega}\left(\rho\right)\notin J'_{\omega,\delta}\right\} .
\]
For $j=0,1,2$, we define $\boldsymbol{1}_{\Omega_{j}}:T^{*}M\rightarrow\left[0,1\right]$
as being the characteristic function of the set $\Omega_{j}$ above
and as in (\ref{eq:def_op_Lambda}), we set
\[
\mathrm{Op}\left(\boldsymbol{1}_{\Omega_{j}}\right):=\mathcal{T}^{\dagger}\boldsymbol{1}_{\Omega_{j}}\mathcal{T}\quad:C^{\infty}\left(M\right)\rightarrow C^{\infty}\left(M\right),
\]
giving a resolution of identity
\begin{align}
\mathrm{Id}_{\mathcal{H}_{W}\left(M\right)}\eq{\ref{eq:partition}} & \mathrm{Op}\left(\boldsymbol{1}_{\Omega_{0}}\right)+\mathrm{Op}\left(\boldsymbol{1}_{\Omega_{1}}\right)+\mathrm{Op}\left(\boldsymbol{1}_{\Omega_{2}}\right)\nonumber \\
= & \mathrm{Op}_{\Sigma}\left(T_{\left[0,K\right]}\right)\mathrm{Op}\left(\boldsymbol{1}_{\Omega_{0}}\right)+\mathrm{Op}_{\Sigma}\left(T_{\geq\left(K+1\right)}\right)\mathrm{Op}\left(\boldsymbol{1}_{\Omega_{0}}\right)+\mathrm{Op}\left(\boldsymbol{1}_{\Omega_{1}}\right)+\mathrm{Op}\left(\boldsymbol{1}_{\Omega_{2}}\right)\label{eq:Id_decomp}\\
 & +\left(\mathrm{Id}-\mathrm{Op}_{\Sigma}\left(\mathrm{Id}_{\mathcal{F}}\right)\right)\mathrm{Op}\left(\boldsymbol{1}_{\Omega_{0}}\right).
\end{align}
We will now construct the resolvent of $X$ at $z$ from the contribution
of each term in (\ref{eq:Id_decomp}), except for the last term. For
this last term, we observe that
\[
\mathrm{Id}-\mathrm{Op}_{\Sigma}\left(\mathrm{Id}_{\mathcal{F}}\right)\underset{(\ref{eq:expression})}{\approx}0.
\]
Since $\sigma$ is fixed in (\ref{eq:def_Omega_sigma}) we have  that
\begin{equation}
\lim_{\left|\omega\right|\rightarrow\infty}\left\Vert \left(\mathrm{Id}-\mathrm{Op}_{\Sigma}\left(\mathrm{Id}_{\mathcal{F}}\right)\right)\mathrm{Op}\left(\boldsymbol{1}_{\Omega_{0}}\right)\right\Vert _{\mathcal{H}_{W}\left(M\right)}\eq{\ref{eq:result_of_Shur}}0.\label{eq:r_omega}
\end{equation}

\subsection{Contribution of $\Omega_{0}$}

Recall that $\mathrm{Op}\left(\boldsymbol{1}_{\Omega_{0}}\right)$
defined from $\Omega_{0}$ in (\ref{eq:def_Omega_sigma}) depends
on $\omega$ defined in (\ref{eq:def_omega}).

\begin{cBoxB}{}
\begin{lem}
\label{lem:We-have-,,}We have (for negative time) $\forall\epsilon>0,\exists C_{\epsilon}>0$,$\forall t\leq0$,
$\exists\omega_{t}>0$,$\forall\omega>\omega_{t}$,
\begin{equation}
\left\Vert e^{tX}\mathrm{Op}_{\Sigma}\left(\tilde{T}_{\left[0,K\right]}\right)\mathrm{Op}\left(\boldsymbol{1}_{\Omega_{0}}\right)\right\Vert _{\mathcal{H}_{W}\left(M\right)}\leq C_{\epsilon}e^{t\left(\gamma_{K}^{-}-\epsilon\right)},\label{eq:exp-tX}
\end{equation}
and (for positive time) $\forall\epsilon>0,\exists C_{\epsilon}>0$,$\forall t\geq0$,
$\exists\omega_{t}>0$,$\forall\omega>\omega_{t}$,
\begin{equation}
\left\Vert e^{tX_{F}}\mathrm{Op}_{\Sigma}\left(\tilde{T}_{\left[0,K\right]}\right)\mathrm{Op}\left(\boldsymbol{1}_{\Omega_{0}}\right)\right\Vert _{\mathcal{H}_{W}\left(M\right)}\leq C_{\epsilon}e^{t\left(\gamma_{0}^{+}+\epsilon\right)},\label{eq:exp-tX_0}
\end{equation}
\begin{equation}
\left\Vert e^{tX}\mathrm{Op}_{\Sigma}\left(\tilde{T}_{\geq\left(K+1\right)}\right)\mathrm{Op}\left(\boldsymbol{1}_{\Omega_{0}}\right)\right\Vert _{\mathcal{H}_{W}\left(M\right)}\leq C_{\epsilon}e^{t\left(\gamma_{K+1}^{+}+\epsilon\right)}.\label{eq:exp-tXK+1}
\end{equation}
\end{lem}

\end{cBoxB}

\begin{proof}
Let $t\leq0$. We have
\begin{align}
e^{tX}\mathrm{Op}_{\Sigma}\left(\tilde{T}_{\left[0,K\right]}\right)\underset{(\ref{eq:expression})}{\approx} & \mathrm{Op}_{\Sigma}\left(e^{tX_{\mathcal{F}}}\right)\mathrm{Op}_{\Sigma}\left(\tilde{T}_{\left[0,K\right]}\right)\nonumber \\
\underset{(\ref{eq:compos_operators})}{\approx} & \mathrm{Op}_{\Sigma}\left(e^{tX_{\mathcal{F}}}\tilde{T}_{\left[0,K\right]}\right)\label{eq:ineq}
\end{align}
Now we use the cut-off in frequency $\mathrm{Op}\left(\boldsymbol{1}_{\Omega_{0}}\right)$.
We have
\begin{equation}
\left\Vert \mathrm{Op}\left(\boldsymbol{1}_{\Omega_{0}}\right)\right\Vert _{\mathcal{H}_{W}\left(M\right)}\leq C,\label{eq:Op_Omega0_bounded}
\end{equation}
and
\begin{align*}
\left\Vert e^{tX}\mathrm{Op}_{\Sigma}\left(\tilde{T}_{\left[0,K\right]}\right)\mathrm{Op}\left(\boldsymbol{1}_{\Omega_{0}}\right)\right\Vert _{\mathcal{H}_{W}\left(M\right)}\underset{(\ref{eq:ineq},\ref{eq:result_of_Shur})}{\leq} & \left\Vert \mathrm{Op}_{\Sigma}\left(e^{tX_{\mathcal{F}}}\tilde{T}_{\left[0,K\right]}\right)\mathrm{Op}\left(\boldsymbol{1}_{\Omega_{0}}\right)\right\Vert _{\mathcal{H}_{W}\left(M\right)}+r\left(\omega\right)\\
\leq & \left\Vert \mathrm{Op}_{\Sigma}\left(e^{tX_{\mathcal{F}}}\tilde{T}_{\left[0,K\right]}\right)\right\Vert _{\mathcal{H}_{W}\left(M\right)}\left\Vert \mathrm{Op}\left(\boldsymbol{1}_{\Omega_{0}}\right)\right\Vert _{\mathcal{H}_{W}\left(M\right)}+r\left(\omega\right)\\
\underset{(\ref{eq:bound-4},\ref{eq:Op_Omega0_bounded})}{\leq} & C_{K,\epsilon'}e^{t\left(\gamma_{K}^{-}-\epsilon'\right)}+r\left(\omega\right),
\end{align*}
 with some function $r\left(\omega\right)$ that satisfies $\lim_{\left|\omega\right|\rightarrow\infty}r\left(\omega\right)=0$.
This gives (\ref{eq:exp-tX}). Similarly, using (\ref{eq:bound-4-1})
we get (\ref{eq:exp-tX_0}) and using (\ref{eq:bound-1}) we get (\ref{eq:exp-tXK+1}).
\end{proof}

\paragraph{Approximate resolvent $R_{\mathrm{in}}^{\left(0\right)}\left(z\right)$
and $R_{\mathrm{out}}^{\left(0\right)}\left(z\right)$:}

For $T\geq0$ we define
\begin{equation}
R_{\mathrm{in}}^{\left(0\right)}\left(z\right):=-\left(\int_{-T}^{0}e^{-tz}e^{tX}dt\right)\mathrm{Op}_{\Sigma}\left(\tilde{T}_{\left[0,K\right]}\right)\mathrm{Op}\left(\boldsymbol{1}_{\Omega_{0}}\right).\label{eq:def_R0}
\end{equation}
Then
\begin{align}
\left\Vert R_{\mathrm{in}}^{\left(0\right)}\left(z\right)\right\Vert _{\mathcal{H}_{W}\left(M\right)}\underset{(\ref{eq:def_R0},\ref{eq:exp-tX})}{\leq} & \int_{0}^{T}e^{t\mathrm{Re}\left(z\right)}C_{\epsilon'}e^{-t\left(\gamma_{K}^{-}-\epsilon'\right)}dt\underset{(\ref{eq:Re_z})}{\leq}C_{\epsilon'}\int_{0}^{T}e^{-\left(\epsilon-\epsilon'\right)t}dt\leq C_{\epsilon}.\label{eq:b_R0}
\end{align}
where we have taken $\epsilon'=\frac{1}{2}\epsilon$. We also have
\begin{align*}
\left(z-X\right)R_{\mathrm{in}}^{\left(0\right)}\left(z\right) & \eq{\ref{eq:def_R0}}\left(\mathrm{Id}-e^{Tz}e^{-TX}\right)\mathrm{Op}_{\Sigma}\left(\tilde{T}_{\left[0,K\right]}\right)\mathrm{Op}\left(\boldsymbol{1}_{\Omega_{0}}\right)
\end{align*}
and
\begin{align}
\left\Vert \left(z-X\right)R_{\mathrm{in}}^{\left(0\right)}\left(z\right)-\mathrm{Op}_{\Sigma}\left(\tilde{T}_{\left[0,K\right]}\right)\mathrm{Op}\left(\boldsymbol{1}_{\Omega_{0}}\right)\right\Vert _{\mathcal{H}_{W}\left(M\right)} & \underset{(\ref{eq:exp-tX})}{\leq}e^{T\mathrm{Re}\left(z\right)}C_{\epsilon'}e^{-T\left(\gamma_{K}^{-}-\epsilon'\right)}\label{eq:R0_in}\\
 & \underset{(\ref{eq:Re_z})}{\leq}C_{\epsilon'}e^{-\left(\epsilon-\epsilon'\right)T}\leq C_{\epsilon}e^{-\frac{1}{2}\epsilon T}
\end{align}
with $\epsilon=\frac{1}{2}\epsilon'$. Similarly we define
\begin{equation}
R_{\mathrm{out}}^{\left(0\right)}\left(z\right):=\left(\int_{0}^{T}e^{-tz}e^{tX}dt\right)\mathrm{Op}_{\Sigma}\left(\tilde{T}_{\geq\left(K+1\right)}\right)\mathrm{Op}\left(\boldsymbol{1}_{\Omega_{0}}\right)\label{eq:def_R0-1}
\end{equation}
and get
\begin{equation}
\left\Vert R_{\mathrm{out}}^{\left(0\right)}\left(z\right)\right\Vert _{\mathcal{H}_{W}\left(M\right)}\underset{(\ref{eq:exp-tXK+1},\ref{eq:Re_z})}{\leq}C_{\epsilon}\label{eq:b_R0out}
\end{equation}
and
\begin{equation}
\left\Vert \left(z-X\right)R_{\mathrm{out}}^{\left(0\right)}\left(z\right)-\mathrm{Op}_{\Sigma}\left(\tilde{T}_{\geq\left(K+1\right)}\right)\mathrm{Op}\left(\boldsymbol{1}_{\Omega_{0}}\right)\right\Vert _{\mathcal{H}_{W}\left(M\right)}\underset{(\ref{eq:exp-tXK+1},\ref{eq:Re_z})}{\leq}C_{\epsilon}e^{-\frac{1}{2}\epsilon T}.\label{eq:R0_out}
\end{equation}

\subsection{Contribution of $\Omega_{1}$}

Recall that $\mathrm{Op}\left(\boldsymbol{1}_{\Omega_{1}}\right)$
defined from $\Omega_{1}$ in (\ref{eq:def_Omega_1}) depends on $\sigma>0$.
Let $\sigma'>0$. 
\[
e^{tX}\mathrm{Op}\left(\boldsymbol{1}_{\Omega_{1}}\right)=e^{tX}\left(\mathrm{Id}-\mathrm{Op}\left(\chi_{\Sigma,\sigma'}\right)\right)\mathrm{Op}\left(\boldsymbol{1}_{\Omega_{1}}\right)+e^{tX}\mathrm{Op}\left(\chi_{\Sigma,\sigma'}\right)\mathrm{Op}\left(\boldsymbol{1}_{\Omega_{1}}\right)
\]
For the first term, from Theorem \ref{thm:decay}, we have $\forall\Lambda>0,\exists C>0,\forall t\geq0,$
$\exists\sigma_{t},\forall\sigma'>\sigma_{t}$,
\[
\left\Vert e^{tX}\left(\mathrm{Id}-\mathrm{Op}\left(\chi_{\Sigma,\sigma'}\right)\right)\mathrm{Op}\left(\boldsymbol{1}_{\Omega_{1}}\right)\right\Vert _{\mathcal{H}_{W}\left(M\right)}\underset{(\ref{eq:decay_outside})}{\leq}\frac{1}{2}Ce^{-\Lambda t}.
\]
For the second term, since the support of the symbols are disjoint,
we have $\forall\Lambda>0,\exists C>0,\forall t\geq0,$ $\forall\sigma',\exists\sigma>\sigma'$,$\exists\omega_{t}>0,\forall\left|\omega\right|>\omega_{t},$
\[
\left\Vert \mathrm{Op}\left(\chi_{\Sigma,\sigma'}\right)\mathrm{Op}\left(\boldsymbol{1}_{\Omega_{1}}\right)\right\Vert _{\mathcal{H}_{W}\left(M\right)}\leq\frac{1}{2}Ce^{-\Lambda t}
\]
Taking $\Lambda>0$ large enough, we deduce that $\exists C>0,\forall t\geq0,$
$\exists\sigma_{t},\forall\sigma>\sigma_{t},\exists\omega_{t}>0,\forall\left|\omega\right|>\omega_{t},$
\begin{equation}
\left\Vert e^{tX}\mathrm{Op}\left(\boldsymbol{1}_{\Omega_{1}}\right)\right\Vert _{\mathcal{H}_{W}\left(M\right)}\leq Ce^{t\gamma_{K+1}^{+}}.\label{eq:Xi_Omega1}
\end{equation}

\paragraph{Approximate resolvent:}

For $T\geq0$ we define
\begin{equation}
R^{\left(1\right)}\left(z\right):=\int_{0}^{T}e^{-tz}e^{tX}\mathrm{Op}\left(\boldsymbol{1}_{\Omega_{1}}\right)dt\label{eq:def_R1}
\end{equation}
and similarly to (\ref{eq:b_R0}) and (\ref{eq:R0_in}) we get
\begin{equation}
\left\Vert R^{\left(1\right)}\left(z\right)\right\Vert _{\mathcal{H}_{W}\left(M\right)}\underset{(\ref{eq:Xi_Omega1},\ref{eq:Re_z})}{\leq}C_{\epsilon},\label{eq:b_R1}
\end{equation}
and
\begin{equation}
\left\Vert \left(z-X\right)R^{\left(1\right)}\left(z\right)-\mathrm{Op}\left(\boldsymbol{1}_{\Omega_{1}}\right)\right\Vert _{\mathcal{H}_{W}\left(M\right)}\underset{(\ref{eq:Xi_Omega1},\ref{eq:Re_z})}{\leq}Ce^{-\epsilon T}.\label{eq:R1}
\end{equation}

\subsection{Contribution of $\Omega_{2}$}

We define
\[
R^{\left(2\right)}\left(z\right):=\mathrm{Op}\left(\frac{1}{z-i\omega\left(.\right)}\boldsymbol{1}_{\Omega_{2}}\right)
\]
We have the following properties. Recall that $\delta>0$ enters in
the Definition (\ref{eq:J_delta}) of $J'_{\omega,\delta}$, hence
for $\Omega_{2}$.
\begin{lem}
\cite[Lemma 5.17]{faure_tsujii_Ruelle_resonances_density_2016}There
exists $C>0$ such that for any $\omega\in\mathbb{R},\delta>0$,

\begin{equation}
\left\Vert R^{\left(2\right)}\left(z\right)\right\Vert _{\mathcal{H}_{W}\left(M\right)}\leq\frac{C}{\delta},\label{eq:b_R2}
\end{equation}
and
\begin{equation}
\left\Vert \left(z-X\right)R^{\left(2\right)}\left(z\right)-\mathrm{Op}\left(\chi_{\Omega_{2,\sigma}}\right)\right\Vert _{\mathcal{H}_{W}\left(M\right)}\leq\frac{C}{\delta}.\label{eq:R2}
\end{equation}
\end{lem}

\subsection{Sum of contributions}

The final approximate resolvent is defined by
\[
R\left(z\right):=R_{\mathrm{in}}^{\left(0\right)}\left(z\right)+R_{\mathrm{out}}^{\left(0\right)}\left(z\right)+R^{\left(1\right)}\left(z\right)+R^{\left(2\right)}\left(z\right).
\]
From previous estimates we have
\[
\left\Vert R\left(z\right)\right\Vert _{\mathcal{H}_{W}\left(M\right)}\underset{(\ref{eq:b_R0},\ref{eq:b_R0out},\ref{eq:b_R1},\ref{eq:b_R2})}{\leq}3C_{\epsilon}+\frac{C}{\delta},
\]
and
\[
\left\Vert \left(z-X\right)R\left(z\right)-\mathrm{Id}\right\Vert _{\mathcal{H}_{W}\left(M\right)}\underset{(\ref{eq:Id_decomp},\ref{eq:r_omega},\ref{eq:R0_in},\ref{eq:R0_out},\ref{eq:R1},\ref{eq:R2})}{\leq}r\left(\omega\right)+2C_{\epsilon}e^{-\frac{1}{2}\epsilon T}+Ce^{-\epsilon T}+\frac{C}{\delta},
\]
with $\lim_{\left|\omega\right|\rightarrow\infty}r\left(\omega\right)=0$.
Hence we can take in this order $T\gg1$, $\sigma\gg1$, $\delta\gg1$,$\left|\omega\right|\gg1$,
large enough so that
\[
\left(z-X\right)R\left(z\right)=\mathrm{Id}-r_{2}\left(z\right)
\]
with 
\[
\left\Vert r_{2}\left(z\right)\right\Vert _{\mathcal{H}_{W}\left(M\right)}<\frac{1}{2}.
\]
Hence $\left(\mathrm{Id}-r_{2}\left(z\right)\right)^{-1}$ is bounded.
We set $\tilde{R}\left(z\right):=R\left(z\right)\left(\mathrm{Id}-r_{2}\left(z\right)\right)^{-1}$
and get that
\[
\left(z-X\right)\tilde{R}\left(z\right)=\mathrm{Id}.
\]
With a similar construction on the left $\tilde{R}_{l}\left(z\right)\left(z-X\right)=\mathrm{Id}$,
we deduce that 
\[
\left\Vert \left(z-X\right)^{-1}\right\Vert \leq C_{\epsilon}
\]
is bounded by $C_{\epsilon}>0$ that does not depend on $\omega$.
We have shown (\ref{eq:gaps-1}) and (\ref{eq:resolvent-1}). This
finishes the proof of Theorem \ref{Thm: bands}.

\section{\label{subsec:Proof-of-Theorem-2}Proof of Theorem \ref{thm:1_emergence_of_QM}
(emergence of quantum dynamics)}

Recall that the space $\mathcal{H}_{W}\left(M\right)$ in (\ref{eq:def_H_W})
depends on the weight $W$ in (\ref{eq:def_W}) that itself depends
on a parameter $R\in\mathbb{R}$. Let $K\in\mathbb{N}$ and $\gamma_{K+1}^{+}$
defined in (\ref{eq:def_gamma_}). Using Theorem \ref{thm:If--is},
we can take $R$ large enough such that the generator $X$ in $\mathcal{H}_{W}\left(M\right)$
has discrete Ruelle spectrum on $\mathrm{Re}\left(z\right)>\gamma_{K+1}^{+}$
. 

We will use the approximate projector $\mathrm{Op}_{\Sigma}\left(\tilde{T}_{\left[0,K\right]}\right)$
defined from the symbol (\ref{eq:def_TK+1}) and (\ref{eq:def_Op_a}),
the operator $\mathrm{Op}\left(\chi_{\omega}\right)$ in (\ref{eq:def_Xi_low})
that selects the low frequencies. We have the decomposition

\begin{align}
e^{tX} & =\mathrm{Op}\left(\chi_{\Sigma,\sigma}\right)e^{tX}\mathrm{Op}_{\Sigma}\left(\tilde{T}_{\left[0,K\right]}\right)\mathrm{Op}\left(\chi_{\Sigma,\sigma}\right)\left(\mathrm{Id}-\mathrm{Op}\left(\chi_{\omega}\right)\right)\label{eq:eq}\\
 & +\mathrm{Op}\left(\chi_{\Sigma,\sigma}\right)e^{tX}\left(\mathrm{Id}-\mathrm{Op}_{\Sigma}\left(\tilde{T}_{\left[0,K\right]}\right)\right)\mathrm{Op}\left(\chi_{\Sigma,\sigma}\right)\left(\mathrm{Id}-\mathrm{Op}\left(\chi_{\omega}\right)\right)\\
 & +\left(\mathrm{Id}-\mathrm{Op}\left(\chi_{\Sigma,\sigma}\right)\right)e^{tX}\left(\mathrm{Id}-\mathrm{Op}\left(\chi_{\omega}\right)\right)+\mathrm{Op}\left(\chi_{\Sigma,\sigma}\right)e^{tX}\left(\mathrm{Id}-\mathrm{Op}\left(\chi_{\Sigma,\sigma}\right)\right)\left(\mathrm{Id}-\mathrm{Op}\left(\chi_{\omega}\right)\right)\label{eq:line2}\\
 & +e^{tX}\mathrm{Op}\left(\chi_{\omega}\right)
\end{align}
\[
\]
For the first term in the right hand side of (\ref{eq:eq}) we have
\begin{align}
e^{tX}\mathrm{Op}_{\Sigma}\left(\tilde{T}_{\left[0,K\right]}\right) & \underset{(\ref{eq:expression})}{\approx}\mathrm{Op}_{\Sigma}\left(e^{tX_{\mathcal{F}}}\right)\mathrm{Op}_{\Sigma}\left(\tilde{T}_{\left[0,K\right]}\right)\nonumber \\
 & \underset{(\ref{eq:compos_operators})}{\approx}\mathrm{Op}_{\Sigma}\left(e^{tX_{\mathcal{F}}}\tilde{T}_{\left[0,K\right]}\right)\label{eq:eq3}
\end{align}
hence
\begin{align}
\mathrm{Op}\left(\chi_{\Sigma,\sigma}\right)e^{tX}\mathrm{Op}_{\Sigma}\left(\tilde{T}_{\left[0,K\right]}\right) & \mathrm{Op}\left(\chi_{\Sigma,\sigma}\right)\left(\mathrm{Id}-\mathrm{Op}\left(\chi_{\omega}\right)\right)\label{eq:eq4}\\
 & =\mathrm{Op}\left(\chi_{\Sigma,\sigma}\right)\mathrm{Op}_{\Sigma}\left(e^{tX_{\mathcal{F}}}\tilde{T}_{\left[0,K\right]}\right)\mathrm{Op}\left(\chi_{\Sigma,\sigma}\right)\left(\mathrm{Id}-\mathrm{Op}\left(\chi_{\omega}\right)\right)+r_{t,\omega}
\end{align}
with
\begin{equation}
\lim_{\left|\omega\right|\rightarrow\infty}\left\Vert r_{t,\omega}\right\Vert _{\mathcal{H}_{W}\left(M\right)}\eq{\ref{eq:result_of_Shur}}0.\label{eq:r_t_omega}
\end{equation}
For the second term in the right hand side of (\ref{eq:eq}), we proceed
exactly as for the proof of Eq.(\ref{eq:exp-tXK+1}) in Lemma \ref{lem:We-have-,,}
above, with the difference that we replace $\mathrm{Op}\left(\boldsymbol{1}_{\Omega_{0}}\right)$
by $\mathrm{Op}\left(\chi_{\Sigma,\sigma}\right)\left(\mathrm{Id}-\mathrm{Op}\left(\chi_{\omega}\right)\right)$.
We get that $\forall\epsilon>0,\exists C_{\epsilon}>0$,$\forall t\geq0$,
$\exists\omega_{t}>0$,$\forall\omega>\omega_{t}$,
\[
\left\Vert \mathrm{Op}\left(\chi_{\Sigma,\sigma}\right)e^{tX}\left(\mathrm{Id}-\mathrm{Op}_{\Sigma}\left(\tilde{T}_{\left[0,K\right]}\right)\right)\mathrm{Op}\left(\chi_{\Sigma,\sigma}\right)\left(\mathrm{Id}-\mathrm{Op}\left(\chi_{\omega}\right)\right)\right\Vert _{\mathcal{H}_{W}\left(M\right)}\leq C_{\epsilon}e^{t\left(\gamma_{K+1}^{+}+\epsilon\right)}
\]
Let $\Lambda\geq-\gamma_{K+1}^{+}$. For the terms in line (\ref{eq:line2}),
we have $\exists C>0,\forall t\geq0,$ $\exists\sigma_{t},\forall\sigma>\sigma_{t}$,
\[
\left\Vert \left(\mathrm{Id}-\mathrm{Op}\left(\chi_{\Sigma,\sigma}\right)\right)e^{tX}\left(\mathrm{Id}-\mathrm{Op}\left(\chi_{\omega}\right)\right)\right\Vert _{\mathcal{H}_{W}\left(M\right)}\ineq{\ref{eq:decay_outside-1}}Ce^{-\Lambda t}\leq Ce^{t\gamma_{K+1}^{+}},
\]
\[
\left\Vert \mathrm{Op}\left(\chi_{\Sigma,\sigma}\right)e^{tX}\left(\mathrm{Id}-\mathrm{Op}\left(\chi_{\Sigma,\sigma}\right)\right)\left(\mathrm{Id}-\mathrm{Op}\left(\chi_{\omega}\right)\right)\right\Vert _{\mathcal{H}_{W}\left(M\right)}\ineq{\ref{eq:decay_outside-1}}Ce^{-\Lambda t}\leq Ce^{t\gamma_{K+1}^{+}}.
\]
The operator $\mathrm{Op}\left(\chi_{\omega}\right)$ is compact hence
we can decompose the last term in the right hand side of (\ref{eq:eq})
and the last term in (\ref{eq:eq3}) as
\begin{equation}
\left(e^{tX}-\mathrm{Op}\left(\chi_{\Sigma,\sigma}\right)\mathrm{Op}_{\Sigma}\left(e^{tX_{\mathcal{F}}}\tilde{T}_{\left[0,K\right]}\right)\mathrm{Op}\left(\chi_{\Sigma,\sigma}\right)\right)\mathrm{Op}\left(\chi_{\omega}\right)=R_{t}+r''\label{eq:eq2}
\end{equation}
where $R_{t}$ is finite rank and $\left\Vert r''\right\Vert _{\mathcal{H}_{W}\left(M\right)}\leq C_{\epsilon}e^{t\left(\gamma_{K+1}^{+}+\epsilon\right)}$.
We get
\[
e^{tX}\eq{\ref{eq:eq},\ref{eq:eq4},\ref{eq:eq2}}\mathrm{Op}\left(\chi_{\Sigma,\sigma}\right)\mathrm{Op}_{\Sigma}\left(e^{tX_{\mathcal{F}}}\tilde{T}_{\left[0,K\right]}\right)\mathrm{Op}\left(\chi_{\Sigma,\sigma}\right)+R_{t}+r''+r_{t,\omega},
\]
where $\left|\omega\right|$ is taken large enough in (\ref{eq:r_t_omega})
so that $\left\Vert r_{t,\omega}\right\Vert _{\mathcal{H}_{W}\left(M\right)}\leq Ce^{t\left(\gamma_{K+1}^{+}+\epsilon\right)}$
also. We have obtained $\exists C>0$, $\forall t>0$, $\exists\sigma_{t}>0$,$\forall\sigma>\sigma_{t}$,
$\exists R_{\sigma,t}$ finite rank operator, 
\[
\left\Vert e^{tX}-\mathrm{Op}\left(\chi_{\Sigma,\sigma}\right)\mathrm{Op}_{\Sigma}\left(e^{tX_{\mathcal{F}}}\tilde{T}_{\left[0,K\right]}\right)\mathrm{Op}\left(\chi_{\Sigma,\sigma}\right)+R_{\sigma,t}\right\Vert _{\mathcal{H}_{W}\left(M\right)}\leq Ce^{t\left(\gamma_{K+1}^{+}+\epsilon\right)}.
\]
giving Theorem \ref{thm:1_emergence_of_QM}.
\begin{rem}
by the effects of weight $W$ we can replace $\mathrm{Op}\left(\chi_{\Sigma,\sigma}\right)\mathrm{Op}_{\Sigma}\left(e^{tX_{\mathcal{F}}}\tilde{T}_{\left[0,K\right]}\right)\mathrm{Op}\left(\chi_{\Sigma,\sigma}\right)$
by $\mathrm{Op}_{\Sigma}\left(e^{tX_{\mathcal{F}}}\tilde{T}_{\left[0,K\right]}\right)$
in (\ref{eq:emerging-1}).
\end{rem}

\section{\label{sec:Proof-of-Theorem_Weyl}Proof of Theorem \ref{thm:Weyl-law-for}
(Weyl law)}

The aim of this section is to prove Theorem \ref{thm:Weyl-law-for}.
We do this in three steps.

\subsection{\label{subsec:Approximate-projector}Approximate projector}

For $\omega\in\mathbb{R}$, $\delta>0$ and $k\in\mathbb{N}$, we
define an approximate projector $P$ for the union of band $k'\in\left[0,k\right]$
and frequency interval $\left[\omega-\delta,\omega+\delta\right]$
as follows. We will relate its trace to a symplectic volume. We denote
$\boldsymbol{1}_{\left[\omega-\delta,\omega+\delta\right]}:C^{\beta}\left(\Sigma;\mathcal{F}_{\left[0,k\right]}\right)\rightarrow C^{\beta}\left(\Sigma;\mathcal{F}_{\left[0,k\right]}\right)$
the multiplication operator by the characteristic function for $\rho\in\Sigma$,
\[
\boldsymbol{1}_{\left[\omega-\delta,\omega+\delta\right]}\left(\rho\right):=\begin{cases}
1 & \text{ if }\boldsymbol{\omega}\left(\rho\right)\in\left[\omega-\delta,\omega+\delta\right]\\
0 & \text{ if not.}
\end{cases}.
\]
We define the operator
\begin{equation}
P:=\mathrm{Op}_{\Sigma}\text{\ensuremath{\left(\boldsymbol{1}_{\left[\omega-\delta,\omega+\delta\right]}\tilde{T}_{\left[0,k\right]}\right)}}\quad:C^{\infty}\left(M\right)\rightarrow C^{\infty}\left(M\right).\label{eq:P_k-1}
\end{equation}
The next lemma shows that the trace of $P$ is related to the symplectic
volume of the support of function $\boldsymbol{1}_{\left[\omega-\delta,\omega+\delta\right]}$
on the trapped set $\Sigma$.

\begin{cBoxB}{}
\begin{lem}
The operator $P$ is trace class in $\mathcal{H}_{W}\left(M\right)$
and $\exists C>0$,$\forall\omega\geq\delta\geq1$,
\begin{equation}
\left\Vert P\right\Vert _{\mathrm{Tr}}\leq C\omega^{d}\delta,\qquad\left\Vert P\right\Vert _{\mathcal{H}_{W}}\leq C,\label{eq:trace_bound}
\end{equation}
\begin{equation}
\left|\mathrm{Tr}\left(P\right)-\mathrm{rank}\left(\mathcal{F}_{\left[0,k\right]}\right)\frac{\mathrm{Vol}\left(M\right)\omega^{d}}{\left(2\pi\right)^{d+1}}\left(2\delta\right)\right|\leq C\omega^{d}\delta\left(\left|\omega\right|^{-\left(\beta-\mu\right)/2}+\frac{\delta}{\omega}\right)+C\omega^{d}\label{eq:trace-1}
\end{equation}
\end{lem}

\end{cBoxB}

\begin{rem}
Later in (\ref{eq:last}) we will use (\ref{eq:trace-1}) as follows.
For any small $c'>0$ we will take $\delta\gg1$ and $\omega\gg\delta$
large enough so that the right hand side of (\ref{eq:trace-1}) is
smaller than $c'\omega^{d}\delta$.
\end{rem}

\begin{proof}
We have
\[
\left\Vert P\right\Vert _{\mathrm{Tr}}\eq{\ref{eq:P_k-1}}\left\Vert \mathrm{Op}_{\Sigma}\text{\ensuremath{\left(\boldsymbol{1}_{\left[\omega-\delta,\omega+\delta\right]}\tilde{T}_{\left[0,k\right]}\right)}}\right\Vert _{\mathrm{Tr}}\ineq{\ref{eq:bound_OP_Sigma-1}}C\omega^{d}\delta,
\]
\[
\left\Vert P\right\Vert _{\mathcal{H}_{W}\left(M\right)}\eq{\ref{eq:P_k-1}}\left\Vert \mathrm{Op}_{\Sigma}\text{\ensuremath{\left(\boldsymbol{1}_{\left[\omega-\delta,\omega+\delta\right]}\tilde{T}_{\left[0,k\right]}\right)}}\right\Vert _{\mathcal{H}_{W}\left(M\right)}\ineq{\ref{eq:bound_OP_Sigma}}C.
\]
Eq. (\ref{eq:trace-1}) comes from (\ref{eq:bound_OP_Sigma-1-1})
because $T_{\left[0,k\right]}\left(m\right)$ is a projector onto
$\mathrm{Pol}_{\left[0,k\right]}\left(E_{s}\left(m\right)\right)$
hence
\[
\mathrm{Tr}\left(T_{\left[0,k\right]}\left(m\right)\right)\eq{\ref{eq:def_Tk}}\mathrm{dim}\left(\mathrm{Pol}_{\left[0,k\right]}\left(E_{s}\left(m\right)\right)\right)\eq{\ref{eq:def_F-1}}\mathrm{rank}\left(\mathcal{F}_{\left[0,k\right]}\right)
\]
and
\[
\int_{M}\mathrm{Tr}\left(T_{\left[0,k\right]}\left(m\right)\right)dm=\mathrm{Vol}\left(M\right)\mathrm{rank}\left(\mathcal{F}_{\left[0,k\right]}\right).
\]
\end{proof}

\subsection{Truncated resolvent}

\begin{cBoxA}{}
\begin{defn}
For $z\in\mathbb{C}$, $\Lambda>0$, let
\begin{equation}
\mathcal{R}\left(z\right):=\left(z-\left(X-\Lambda P\right)\right)^{-1}\label{eq:def_RR}
\end{equation}
\begin{equation}
D\left(z\right):=\left(z-X\right)^{-1}-\mathcal{R}\left(z\right)\label{eq:def_D}
\end{equation}
\end{defn}

\end{cBoxA}

\begin{rem}
A similar construction of truncated resolvent $\mathcal{R}\left(z\right)$
has been used in \cite[Proof of Prop. 5.14]{faure_tsujii_Ruelle_resonances_density_2016}.
\end{rem}

Let $\delta'>0$ and define the rectangle in the spectral plane
\begin{equation}
\tilde{R}\left(\omega,\delta,\delta'\right):=\left[\gamma_{k}^{-}-2\epsilon,\gamma_{0}^{+}+2\epsilon\right]\times i\left[\omega-\delta+\delta',\omega+\delta-\delta'\right]\label{eq:def_Rtilde}
\end{equation}

The next lemma is similar to \cite[Lemma 5.16]{faure_tsujii_Ruelle_resonances_density_2016}
so we omit the proof.

\begin{cBoxB}{}
\begin{lem}
\label{lem:ResolventBound}There exists $C>0$ such that, if $\delta'>0$
in (\ref{eq:def_Rtilde}) and $\Lambda>0$ in (\ref{eq:def_RR}) are
sufficiently large, then for any $\omega\in\mathbb{R},\delta>\delta'>0$,
$z\in\tilde{R}\left(\omega,\delta,\delta'\right)$, 
\begin{equation}
\left\Vert \mathcal{R}\left(z\right)\right\Vert _{\mathcal{H}_{W}\left(M\right)}\leq C.\label{eq:R_bounded}
\end{equation}
\end{lem}

\end{cBoxB}

\begin{rem}
A consequence of Lemma \ref{lem:ResolventBound} and expression (\ref{eq:def_D})
is that the poles of $D\left(z\right)$ in $\tilde{R}\left(\omega,\delta,\delta'\right)$
coincide with the poles of the resolvent $\left(z-X\right)^{-1}$,
i.e. Ruelle resonances.
\end{rem}

Let us consider the following union of two horizontal bands (see Figure
\ref{fig:boxes})
\begin{align}
B_{\pm}\left(\omega,\delta,\delta'\right) & :=\left[\gamma_{k}^{-}-\epsilon,\gamma_{0}^{+}+\epsilon\right]\times i\left[\omega\pm\left(\delta-2\delta'\right),\omega\pm\left(\delta-\delta'\right)\right],\label{eq:def_B}\\
B\left(\omega,\delta,\delta'\right) & :=B_{+}\left(\omega,\delta,\delta'\right)\cup B_{-}\left(\omega,\delta,\delta'\right).
\end{align}

\begin{center}
\begin{figure}
\begin{centering}
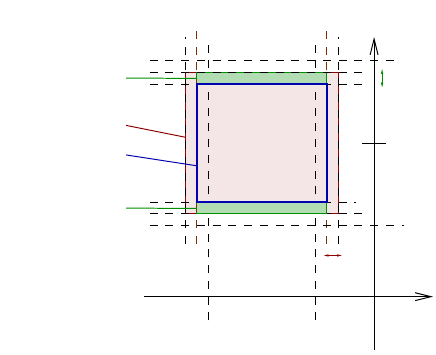
\par\end{centering}
\caption{\label{fig:boxes}Picture in the spectral plane for the rectangular
domain $\tilde{R}\left(\omega,\delta,\delta'\right)$ defined in (\ref{eq:def_Rtilde}),
the sub-domains $B_{\pm}\left(\omega,\delta,\delta'\right)$ defined
in (\ref{eq:def_B}) and $R\left(\omega,\delta,\delta'\right)$ defined
later in (\ref{eq:def_R_omega_delta}).}
\end{figure}
\par\end{center}

We will need later (in the proof of Lemma \ref{lem:We-have-,}), the
following proposition.

\begin{cBoxB}{}
\begin{prop}
\label{lem:For-any--1}The operator $D\left(z\right):\mathcal{H}_{W}\left(M\right)\rightarrow\mathcal{H}_{W}\left(M\right)$
is meromorphic w.r.t. $z\in\tilde{R}\left(\omega,\delta,\delta'\right)$
and trace class outside of its poles of first order. In particular
$\mathrm{Tr}\left(D\left(z\right)\right)\in L_{\mathrm{loc}}^{1}$
is locally integrable and with setting $z:=x+iy$, we have $\exists C>0,\forall\delta>1,\forall\omega>\delta,$
\begin{equation}
\int_{B\left(\omega,\delta,\delta'\right)}\left|\mathrm{Tr}\left(D\left(z\right)\right)\right|dxdy\leq C\delta'\omega^{d}\ln\delta.\label{eq:int_trace}
\end{equation}
\end{prop}

\end{cBoxB}

\begin{proof}
We first show few lemmas.
\begin{proof}
~

\begin{cBoxB}{}
\begin{lem}
$D\left(z\right)$ is trace class outside of its poles and
\begin{equation}
\mathrm{Tr}\left(D\left(z\right)\right)=\partial_{z}\ln F\left(z\right)\label{eq:Tr_D}
\end{equation}
with
\begin{align}
A\left(z\right) & :=\mathrm{Id}-\mathcal{R}\left(z\right)\Lambda P,\label{eq:def_A}
\end{align}
\begin{equation}
F\left(z\right):=\mathrm{det}\left(A\left(z\right)\right).\label{eq:def_F}
\end{equation}
\end{lem}

\end{cBoxB}

$D\left(z\right)$ is meromorphic because both $\left(z-X\right)^{-1}$
and $\mathcal{R}\left(z\right)$ are meromorphic, see \cite[Lemma 5.16]{faure_tsujii_Ruelle_resonances_density_2016}.
$\mathcal{R}\left(z\right)$ is even holomorphic on the rectangle
$\tilde{R}\left(\omega,\delta,\delta'\right)$, from Lemma \ref{lem:ResolventBound}.
Then
\begin{align}
D\left(z\right)\eq{\ref{eq:def_D}} & \left(z-X\right)^{-1}-\mathcal{R}\left(z\right)\nonumber \\
= & \left(z-\left(X-\Lambda P\right)\right)^{-1}\left(\left(z-\left(X-\Lambda P\right)\right)-\left(z-X\right)\right)\left(z-X\right)^{-1}\nonumber \\
= & \mathcal{R}\left(z\right)\Lambda P\left(z-X\right)^{-1}.\label{eq:R-Rk}
\end{align}
In this last expression $P$ is trace class. $\left(z-X\right)^{-1}$
and $\mathcal{R}\left(z\right)$ are bounded outside of their poles
hence $D\left(z\right)$ is trace class outside the poles of $\left(z-X\right)^{-1}$
and $\mathcal{R}\left(z\right)$. We have
\begin{align}
D\left(z\right) & \eq{\ref{eq:def_D}}\left(z-X\right)^{-1}-\mathcal{R}\left(z\right)=\left(z-\left(X-\Lambda P\right)+\Lambda P\right)^{-1}-\mathcal{R}\left(z\right)\nonumber \\
 & \eq{\ref{eq:def_RR}}\mathcal{R}\left(z\right)\left(\mathrm{Id}-\mathcal{R}\left(z\right)\Lambda P\right)^{-1}-\mathcal{R}\left(z\right)\eq{\ref{eq:def_A}}\mathcal{R}\left(z\right)\left(A\left(z\right)^{-1}-\mathrm{Id}\right).\label{eq:help}
\end{align}
We compute
\begin{equation}
\partial_{z}A\left(z\right)\eq{\ref{eq:def_A}}\mathcal{R}\left(z\right)^{2}\Lambda P\eq{\ref{eq:def_A}}\mathcal{R}\left(z\right)\left(\mathrm{Id}-A\left(z\right)\right)\eq{\ref{eq:help}}D\left(z\right)A\left(z\right),\label{eq:int_A}
\end{equation}
then we have
\begin{align*}
\partial_{z}\ln A\left(z\right) & =\left(\partial_{z}A\left(z\right)\right)A\left(z\right)^{-1}\eq{\ref{eq:int_A}}D\left(z\right),
\end{align*}
\[
\ln F\left(z\right)\eq{\ref{eq:def_F}}\ln\mathrm{det}\left(A\left(z\right)\right)=\mathrm{Tr}\left(\ln A\left(z\right)\right),
\]
\[
\partial_{z}\ln F\left(z\right)=\mathrm{Tr}\left(\partial_{z}\ln A\left(z\right)\right)=\mathrm{Tr}\left(D\left(z\right)\right).
\]
We have obtained (\ref{eq:Tr_D}).
\end{proof}
For later use, we also notice that
\begin{equation}
\mathcal{R}\left(z\right)\eq{\ref{eq:R-Rk}}A\left(z\right)\left(z-X\right)^{-1}.\label{eq:R-Rk2}
\end{equation}

\begin{cBoxB}{}
\begin{lem}
There exists $C>0$, such that for every $z\in B\left(\omega,\delta,\delta'\right)$,
\begin{equation}
\left|F\left(z\right)\right|\leq e^{C\left(\ln\delta\right)\omega^{d}}.\label{eq:F_upper_bound}
\end{equation}
and for every $z\in B\left(\omega,\delta,\delta'\right)$ with $\mathrm{Re}\left(z\right)=\gamma_{k}^{-}-\epsilon$,
\begin{equation}
\left|F\left(z\right)\right|\geq e^{-C\left(\ln\delta\right)\omega^{d}}.\label{eq:F_lower_bound}
\end{equation}
\end{lem}

\end{cBoxB}

\begin{proof}
For $z\in B_{-}\left(\omega,\delta,\delta'\right)$, writing $\lambda_{j}\left(z\right)$,
$j=1,2,\cdots,J,$ for the eigenvalues of the trace class operator
$\mathcal{R}\left(z\right)\Lambda P$, we have by \href{https://en.wikipedia.org/wiki/Trace_class\#Lidskii's_theorem}{Lidskii's theorem}
\begin{align}
\log\left|F\left(z\right)\right| & =\log\left|\mathrm{det}\left(\mathrm{Id}-\mathcal{R}\left(z\right)\Lambda P\right)\right|=\sum_{j}\log\left|1-\lambda_{j}\left(z\right)\right|\nonumber \\
 & \leq\sum_{j}\log\left(1+\left|\lambda_{j}\left(z\right)\right|\right)\leq\sum_{j}\left|\lambda_{j}\left(z\right)\right|\leq\left\Vert \mathcal{R}\left(z\right)\Lambda P\right\Vert _{\mathrm{Tr}}.\label{eq:log_F}
\end{align}
For the last inequality, see e.g. \cite[p.64]{gohberg-00}. Notice
that $\left\Vert \mathcal{R}\left(z\right)\Lambda P\right\Vert _{\mathrm{Tr}}\underset{(\ref{eq:trace_bound},\ref{eq:R_bounded})}{\leq}C\delta\omega^{d}$
but this is not enough to get (\ref{eq:F_upper_bound}). We will improve
the last bound as follows. We decompose the frequency interval $\left[\omega-\delta,\omega+\delta\right]$
as the union of $\lfloor\delta\rfloor$ (the integer part of $\delta$)
intervals of bounded length $l=\frac{2\delta}{\left[\delta\right]}$,
so that 
\begin{equation}
\left[\omega-\delta,\omega+\delta\right]=\bigcup_{w\in\left\{ 1,2,\ldots,\lfloor\delta\rfloor\right\} }I_{w}\label{eq:decomp_inter}
\end{equation}
with
\begin{equation}
I_{w}:=\left[\omega-\delta+\left(w-1\right)l,\omega-\delta+wl\right].\label{eq:def_Iw}
\end{equation}
Parallel to (\ref{eq:P_k-1}) we define
\[
P_{w}:=\check{\mathcal{T}}_{\sigma}^{\Delta}\check{\mathrm{Op}}_{\Sigma}\left(T_{\left[0,k\right]}\right)\chi_{I_{w}}\check{\mathcal{T}}_{\sigma}\quad:C^{\infty}\left(M\right)\rightarrow C^{\infty}\left(M\right)
\]
i.e. an approximate projector on frequency interval of size $l$ and
we write
\begin{equation}
P=\sum_{w}P_{w}.\label{eq:P_Pw}
\end{equation}
Similarly to (\ref{eq:trace_bound}) we have
\[
\left\Vert P_{w}\right\Vert _{\mathrm{Tr}}\leq C\omega^{d}
\]
and
\begin{equation}
\left\Vert \mathcal{R}\left(z\right)\Lambda P_{w}\right\Vert _{\mathrm{Tr}}\leq\frac{1}{w}C\omega^{d}\label{eq:estimate}
\end{equation}
uniformly for $1\le w\le\lfloor\delta\rfloor$. Hence
\[
\left\Vert \mathcal{R}\left(z\right)\Lambda P\right\Vert _{\mathrm{Tr}}\leq\sum_{w=1}^{\left[\delta\right]}\left\Vert \mathcal{R}\left(z\right)\Lambda P_{w}\right\Vert _{\mathrm{Tr}}\leq C\omega^{d}\ln\delta.
\]
We get
\[
\left|F\left(z\right)\right|\underset{(\ref{eq:log_F})}{\leq}e^{\left\Vert \mathcal{R}\left(z\right)\Lambda P\right\Vert _{\mathrm{Tr}}}\leq e^{C\omega^{d}\ln\delta}.
\]
We have obtained (\ref{eq:F_upper_bound}). To prove the second claim
\eqref{eq:F_lower_bound}, we consider an arbitrary $z\in B\left(\omega,\delta,\delta'\right)$
with $\mathrm{Re}\left(z\right)=\gamma_{k}^{-}-\epsilon$. From Corollary
\ref{Thm: bands} we have $\left\Vert \left(z-X\right)^{-1}\right\Vert \leq C$.
Then
\[
\left(\mathrm{Id}-\mathcal{R}\left(z\right)\Lambda P\right)^{-1}=A\left(z\right)^{-1}\eq{\ref{eq:R-Rk2}}\left(z-X\right)^{-1}\mathcal{R}\left(z\right)^{-1}\eq{\ref{eq:def_RR}}\left(\mathrm{Id}+\left(z-X\right)^{-1}\Lambda P\right)
\]
is uniformly bounded from (\ref{eq:trace_bound}). Hence $\left|1-\lambda_{j}\left(z\right)\right|>c$
with some $c>0$ independent on $\omega$ and therefore we have $\left|\log\left|1-\lambda_{j}\left(z\right)\right|\right|<C'\left|\lambda_{j}\left(z\right)\right|$
for some $C'>0$. Thus 
\[
-\log\left|F\left(z\right)\right|=-\sum_{j}\log\left|1-\lambda_{j}\left(z\right)\right|\leq C'\sum_{j}\left|\lambda_{j}\left(z\right)\right|\leq C'\left\Vert \mathcal{R}\left(z\right)\Lambda P\right\Vert _{\mathrm{Tr}}\leq C'\omega^{d}\ln\delta,
\]
giving (\ref{eq:F_lower_bound}).
\end{proof}
Let $\mathrm{spect}\left(X\right)$ denotes the discrete Ruelle spectrum
of $X$.

\begin{cBoxB}{}
\begin{lem}
We have $\exists C>0,\forall\omega$,$\forall z\in B\left(\omega,\delta,\delta'\right),$
\begin{equation}
\left|\partial_{z}\ln F\left(z\right)-\sum_{z_{j}\in\mathrm{spect}\left(X\right)\cap B\left(\omega,\delta,\delta'\right)}\frac{1}{z-z_{j}}\right|\leq C\left(\ln\delta\right)\omega^{d}\label{eq:lnF}
\end{equation}
where the sum over $z_{j}\in\mathrm{spect}\left(X\right)\cap B\left(\omega,\delta,\delta'\right)$
is counted with multiplicity of the Ruelle eigenvalues.
\end{lem}

\end{cBoxB}

\begin{proof}
This is lemma $\alpha$ in \cite[p.56]{titchmarsh1986theory}. We
reproduce the proof here. From (\ref{eq:R_bounded}), $F\left(z\right)$
is a well defined holomorphic function in the domains $B\left(\omega,\delta,\delta'\right)$
and from (\ref{eq:R-Rk2}), the zeroes of $F\left(z\right)$ coincide
up to multiplicities with the poles of $\left(z-X\right)^{-1}$, i.e.
Ruelle eigenvalues $z_{j}\in\mathrm{spect}\left(X\right)\cap B\left(\omega,\delta,\delta'\right)$.
Using Weyl upper bound $O\left(\omega^{d}\right)$ on the density
in \cite[Thm.3.6]{faure_tsujii_Ruelle_resonances_density_2016}, note
that the number of such Ruelle eigenvalues is bounded by $C\delta'\omega^{d}\le C\delta\omega^{d}$.
From (\ref{eq:F_lower_bound}), we can fix some $z_{0}\in B\left(\omega,\delta,\delta'\right)$
with $\mathrm{Re}\left(z_{0}\right)=\gamma_{k}^{-}-\epsilon$ so that
$\left|F\left(z_{0}\right)\right|\underset{(\ref{eq:F_lower_bound})}{\geq}e^{-C'\delta\omega^{d}}.$
The function
\begin{equation}
G\left(z\right):=\frac{F\left(z\right)\prod_{z_{j}}\left(z_{0}-z_{j}\right)}{F\left(z_{0}\right)\prod_{z_{j}}\left(z-z_{j}\right)}.\label{eq:def_G}
\end{equation}
$G\left(z\right)$ is holomorphic on $B\left(\omega,\delta,\delta'\right)$
and from estimate (\ref{eq:F_upper_bound}) on the boundary of $B\left(\omega,\delta,\delta'\right)$
and \href{https://en.wikipedia.org/wiki/Maximum_modulus_principle}{maximum modulus principle}
we get
\begin{equation}
\forall z\in B\left(\omega,\delta,\delta'\right),\quad\left|G\left(z\right)\right|\underset{(\ref{eq:F_upper_bound})}{\leq}e^{C\omega^{d}\ln\delta}.\label{eq:bound_G}
\end{equation}
Moreover $G\left(z\right)$ has no zero on on the simply connected
domain $B\left(\omega,\delta,\delta'\right)$ hence
\begin{equation}
g\left(z\right):=\ln G\left(z\right)\label{eq:def_g}
\end{equation}
is well defined as a holomorphic function on $B\left(\omega,\delta,\delta'\right)$
with $g\left(z_{0}\right)=0$ and satisfies $\mathrm{Re}g\left(z\right)\underset{(\ref{eq:bound_G})}{\leq}C\omega^{d}\ln\delta$.
By \href{https://en.wikipedia.org/wiki/Borel\%E2\%80\%93Carath\%C3\%A9odory_theorem}{Borel-Carathéodory theorem}\footnote{Here we do not need to give the constants and the proof is simple:
$g$ maps $B\left(\omega,\delta,\delta'\right)$ on the half-plane
$\mathrm{Re}\left(z\right)<C\delta\omega^{d}$. We post-compose by
a \href{https://en.wikipedia.org/wiki/M\%C3\%B6bius_transformation}{Moebius transformation}
$f$ which maps this half-plane to the unit disk and maps $g\left(z_{0}\right)=0$
to $0$. Then the image of $B\left(\omega,\delta,\delta'\right)$
by $f\circ g$ is contained in a disk $\left|z\right|<r<1$. This
implies that the image $g\left(B\left(\omega,\delta,\delta'\right)\right)$
is contained in a (large) disk $\left|z\right|<C'\delta\omega^{d}$.} this implies
\[
\left|g\left(z\right)\right|<C'\omega^{d}\ln\delta
\]
on a region slightly larger than $B\left(\omega,\delta,\delta'\right)$
and therefore by \href{https://en.wikipedia.org/wiki/Cauchy's_integral_formula}{Cauchy integral formula}
\[
\left|\partial_{z}g\left(z\right)\right|\eq{\ref{eq:def_g},\ref{eq:def_G}}\left|\partial_{z}\ln F\left(z\right)-\sum_{z_{j}}\frac{1}{z-z_{j}}\right|<C''\omega^{d}\ln\delta
\]
giving the conclusion.
\end{proof}
Now we finish with the proof of Lemma \ref{lem:For-any--1}. Using
Weyl upper bound on the density in \cite[Thm. 2.3]{faure_tsujii_Ruelle_resonances_density_2016},
we have, with $z=x+iy$, that
\begin{equation}
\sum_{z_{j}\in\mathrm{spect}\left(X\right)\cap B\left(\omega,\delta,\delta'\right)}\int_{B\left(\omega,\delta,\delta'\right)}\frac{1}{\left|z-z_{j}\right|}dxdy<C\left(\delta'\omega^{d}\right)\ln\delta'.\label{eq:bound}
\end{equation}
Then
\begin{align*}
\int_{B\left(\omega,\delta,\delta'\right)}\left|\mathrm{Tr}\left(D\left(z\right)\right)\right|dxdy & \eq{\ref{eq:Tr_D}}\int_{B\left(\omega,\delta,\delta'\right)}\left|\partial_{z}\ln F\left(z\right)\right|dxdy\\
 & \underset{(\ref{eq:lnF},\ref{eq:bound})}{\leq}C\delta'\omega^{d}\ln\delta+C\delta'\omega^{d}\ln\delta'\leq2C\delta'\omega^{d}\ln\delta.
\end{align*}
We have obtained (\ref{eq:int_trace}).
\end{proof}

\subsection{Counting resonances by argument principle}

From Lemma \ref{lem:ResolventBound}, the truncated resolvent $\mathcal{R}\left(z\right)$
is bounded on the rectangle $\tilde{R}\left(\omega,\delta,\delta'\right)$
and hence by the Definition (\ref{eq:def_D}) the poles of $z\rightarrow\mathrm{Tr}\,D\left(z\right)$
coincides with the poles of $\left(z-X\right)^{-1}$, i.e., the Ruelle
resonance. In the following, we count the number of Ruelle resonance
in the rectangle $\tilde{R}\left(\omega,\delta,\delta'\right)$ by
using a slightly generalized version of the argument principle. We
learned this method from the paper of S. Dyatlov \cite[Section 11]{dyatlov_resonance_band_2013}. 

In the following, we suppose that $\delta>0,\delta'>0$ is fixed so
that $2\delta'<\delta$ and we consider the limit $|\omega|\to\infty$.
We take a $C^{\infty}$ function $f_{\delta,\delta'}:\mathbb{R}\rightarrow\left[0,1\right]$
that is $f_{\delta,\delta'}\left(\omega'\right)=1$ for $\left|\omega'\right|<\delta-2\delta'$
and $f_{\delta,\delta'}\left(\omega'\right)=0$ for $\left|\omega'\right|>\delta-\delta'$
and $\left|\partial_{\omega'}f\right|<C\delta'^{-1}$. Then for $\omega\in\mathbb{R}$,
$z\in\mathbb{C}$, let
\[
f_{\omega,\delta,\delta'}\left(z\right):=f_{\delta,\delta'}\left(\mathrm{Im}\left(z-i\omega\right)\right)
\]
Let us define the rectangle in the spectral plane (see Figure \ref{fig:boxes})
\begin{equation}
R\left(\omega,\delta,\delta'\right):=\left[\gamma_{k}^{-}-\epsilon,\gamma_{0}^{+}+\epsilon\right]\times i\left[\omega-\delta+2\delta',\omega+\delta-2\delta'\right].\label{eq:def_R_omega_delta}
\end{equation}
In the following we suppose that the $k$-th band is separated from
$(k+1)$-st band i.e.
\[
\gamma_{k+1}^{+}<\gamma_{k}^{-}.
\]
Note that $R\left(\omega,\delta,\delta'\right)$ is a proper subset
of $\tilde{R}\left(\omega,\delta,\delta'\right)$ defined in (\ref{eq:def_Rtilde}).
We have that $\forall z\in R\left(\omega,\delta,\delta'\right)$,
$f_{\omega,\delta,\delta'}\left(z\right)=1$ and $\forall z\in\tilde{R}\left(\omega,\delta,\delta'\right)$,
\begin{equation}
\left|\partial_{\overline{z}}f_{\omega,\delta,\delta'}\left(z\right)\right|<C\delta'^{-1}.\label{eq:almost_anal}
\end{equation}
We consider the following paths $\Gamma:=\Gamma_{\mathrm{-}}\cup\Gamma_{+}$
on $\mathbb{C}$, parameterized as
\[
\Gamma_{-}:\omega\in\mathbb{R}\rightarrow z=\gamma_{k}^{-}-\epsilon+i\omega\in\mathbb{C},
\]
\[
\Gamma_{+}:\omega\in\mathbb{R}\rightarrow z=\gamma_{0}^{+}+\epsilon+i\omega\in\mathbb{C}.
\]

From (\ref{eq:R_bounded}) and (\ref{eq:resolvent-1}), there is no
pole of $D(z)$ on $\Gamma\cap\mathrm{supp}\left(f_{\omega,\delta,\delta'}\right)$
supposing that $\left|\omega\right|$ is sufficiently large. We know
that $D\left(z\right)$ is in trace class. Hence the integral $\int_{\Gamma_{+}-\Gamma_{-}}f_{\omega,\delta,\delta'}\left(z\right)\mathrm{Tr}\left(D\left(z\right)\right)dz$
is well defined. The next lemma shows that this integral is close
to the number of Ruelle resonances in the rectangle $R\left(\omega,\delta,\delta'\right)$.

\begin{cBoxB}{}
\begin{lem}
\label{lem:We-have-,}We have $\exists C>0,$$\forall\delta>1,\forall\omega>\delta,\forall0<\delta'<\delta$,
\begin{equation}
\left|\left(\frac{1}{2\pi i}\int_{\Gamma_{+}-\Gamma_{-}}f_{\omega,\delta,\delta'}\left(z\right)\mathrm{Tr}\left(D\left(z\right)\right)dz\right)-\sharp\left(\mathrm{spect}\left(X\right)\cap\tilde{R}\left(\omega,\delta,\delta'\right)\right)\right|<C\omega^{d}\max\left(\ln\delta,\delta'\right)\label{eq:res4}
\end{equation}
\end{lem}

\end{cBoxB}

\begin{proof}
From the Definition of $f_{\omega,\delta,\delta'}$ above, we have
$f_{\omega,\delta,\delta'}\left(z\right)=0$ if $\left|\mathrm{Im}\left(z\right)-\omega\right|>\delta-\delta'$,
so we can close the contour integral on horizontal segments $\mathrm{Im}\left(z\right)=\omega\pm\left(\delta-\delta'\right)$.
The residues of the poles of $D\left(z\right)$ inside the contour
are the spectral projector onto eigenspaces. Hence \href{https://en.wikipedia.org/wiki/Cauchy\%27s_integral_formula\#Smooth_functions}{Cauchy integral formula for smooth functions}
gives
\begin{align*}
\frac{1}{2\pi i}\int_{\Gamma_{+}-\Gamma_{-}}f_{\omega,\delta,\delta'}\left(z\right)\mathrm{Tr}\left(D\left(z\right)\right)dz= & \sum_{z_{j}\in\mathrm{spect}\left(X\right)\cap\tilde{R}\left(\omega,\delta,\delta'\right)}f_{\omega,\delta,\delta'}\left(z_{j}\right)\\
 & +\frac{1}{\pi}\int_{\tilde{R}\left(\omega,\delta,\delta'\right)}\left(\partial_{\overline{z}}f_{\omega,\delta,\delta'}\right)\left(z\right)\mathrm{Tr}\left(D\left(z\right)\right)dxdy
\end{align*}
with $z=x+iy$. We have $\exists C>0,$$\forall\delta>0,\forall\delta'<\frac{1}{2}\delta,\forall\omega>\delta$,
\begin{align}
\left|\int_{\tilde{R}\left(\omega,\delta,\delta'\right)}\left(\partial_{\overline{z}}f_{\omega,\delta,\delta'}\right)\left(z\right)\mathrm{Tr}\left(D\left(z\right)\right)dxdy\right|\leq & \delta'^{-1}\int_{B\left(\omega,\delta,\delta'\right)}\left|\mathrm{Tr}\left(D\left(z\right)\right)\right|dxdy\nonumber \\
\underset{(\ref{eq:almost_anal},\ref{eq:int_trace})}{\leq} & C\delta'^{-1}\left(\ln\delta\right)\delta'\omega^{d}=C\omega^{d}\ln\delta.\label{eq:use}
\end{align}
From properties of $f_{\omega,\delta,\delta'}$ and using Weyl upper
bound on the density in \cite[Thm.2.3]{faure_tsujii_Ruelle_resonances_density_2016}
we have
\[
\left|\sum_{z_{j}\in\mathrm{spect}\left(X\right)\cap\tilde{R}\left(\omega,\delta,\delta'\right)}f_{\omega,\delta,\delta'}\left(z_{j}\right)-\sharp\left\{ z_{j}\in\mathrm{spect}\left(X\right)\cap\tilde{R}\left(\omega,\delta,\delta'\right)\right\} \right|\leq C\omega^{d}\delta'.
\]
We have obtained (\ref{eq:res4}).
\end{proof}

\subsection{Relation with the symplectic volume}

The last step is to relate the integral $\frac{1}{2\pi i}\int_{\Gamma_{+}-\Gamma_{-}}f_{\omega,\delta,\delta'}\left(z\right)\mathrm{Tr}\left(D\left(z\right)\right)dz$
in (\ref{eq:res4}) to the trace of $P$ in (\ref{eq:trace-1}). Recall
Definition (\ref{eq:P_k-1}) of $P$. We define the operator 
\[
P_{0}:=\mathrm{Op}_{\Sigma}\text{\ensuremath{\left(\boldsymbol{1}_{\left[\omega-\delta+\delta'/2,\omega+\delta-\delta'/2\right]}T_{\left[0,k\right]}\right)}}.
\]
For $z\in R\left(\omega,\delta,\delta'\right)$, let
\[
d(z)=\min\left\{ \mathrm{Im}(z)-(\omega+\delta+\delta'/2),\mathrm{Im}(z)+(\omega-\delta-\delta'/2)\right\} \ge\delta'/2.
\]

\begin{cBoxB}{}
\begin{lem}[\textbf{Approximate expressions of $D\left(z\right)$}]
\textbf{\label{lem:Approximate-expressions-of-3}}There exists $C>0$
such that for any $T>0$, $\omega>\delta>\delta'>1$ and any $z\in\mathbb{C}$
with $\mathrm{Im}\left(z\right)\in\left[\omega-\delta+\delta',\omega+\delta-\delta'\right]$
and $\mathrm{Re}\left(z\right)=\gamma_{0}^{+}+\epsilon$, we have
\begin{equation}
\left\Vert D\left(z\right)-\left(\int_{0}^{T}e^{-tz}e^{t\left(X-\Lambda\right)}dt\right)\left(\Lambda P_{0}\right)\left(\int_{0}^{T}e^{-tz}e^{tX}dt\right)\right\Vert _{\mathrm{Tr}}\leq C\omega^{d}\left(e^{-\epsilon T}+\frac{\delta'}{\left(d\left(z\right)\right)^{2}}\right)\label{eq:approx_D-4}
\end{equation}

If instead $\mathrm{Re}\left(z\right)=\gamma_{k}^{-}-\epsilon$, we
have
\begin{equation}
\left\Vert D\left(z\right)-\left(\int_{0}^{T}e^{-tz}e^{t\left(X-\Lambda\right)}dt\right)\left(\Lambda P_{0}\right)\left(\int_{-T}^{0}e^{-tz}e^{tX}dt\right)\right\Vert _{\mathrm{Tr}}\leq C\omega^{d}\left(e^{-\epsilon T}+\frac{\delta'}{\left(d\left(z\right)\right)^{2}}\right)\label{eq:approx_D--2}
\end{equation}
\end{lem}

\end{cBoxB}

\begin{proof}
Let $z\in\mathbb{C}$ with $\mathrm{Im}\left(z\right)\in\left[\omega-\delta+\delta',\omega+\delta-\delta'\right]$.
In the following we let $T>0$ be sufficiently large. Suppose $\mathrm{Re}\left(z\right)=\gamma_{0}^{+}+\epsilon$.
As in (\ref{eq:P_Pw}), we decompose $P=\sum_{w\in W}P_{w}$ with
$W:=\left\{ 1,2,\ldots,\lfloor\delta\rfloor\right\} $. As in (\ref{eq:decomp_inter})
with interval $I_{w}$ defined in (\ref{eq:def_Iw}) we decompose
\[
\left[\omega-\delta+\delta'/2,\omega+\delta-\delta'/2\right]=\bigcup_{w\in W_{0}}I_{w}
\]
with a subset $W_{0}\subset W$. So we can write 
\[
P_{0}=\sum_{w\in W_{0}}P_{w}.
\]
We have
\[
D\left(z\right)\eq{\ref{eq:R-Rk}}\sum_{w\in W}D_{w}\left(z\right)
\]
with
\[
D_{w}\left(z\right):=\Lambda\mathcal{R}\left(z\right)P_{w}\left(z-X\right)^{-1}.
\]
For $w\in W_{0}$, let
\[
\tilde{D}_{w}\left(z\right):=\Lambda\left(\int_{0}^{T}e^{-tz}e^{t\left(X-\Lambda\right)}dt\right)P_{w}\left(\int_{0}^{T}e^{-tz}e^{tX}dt\right).
\]
Then
\begin{align*}
\left(z-X+\Lambda\right)\tilde{D}_{w}\left(z\right)\left(z-X\right) & =\Lambda\left[e^{-tz}e^{t\left(X-\Lambda\right)}\right]_{0}^{T}P_{w}\left[e^{-tz}e^{tX}\right]_{0}^{T}\\
 & =\Lambda P_{w}+\Lambda r_{w},
\end{align*}
with
\[
r_{w}=e^{T\left(-z+X-\Lambda\right)}P_{w}+P_{w}e^{T\left(-z+X\right)}+e^{T\left(-z+X-\Lambda\right)}P_{w}e^{T\left(-z+X\right)}.
\]
We deduce that

\[
\tilde{D}_{w}\left(z\right)=D_{w}\left(z\right)+\Lambda\mathcal{R}\left(z\right)r_{w}\left(z-X\right)^{-1}
\]
with
\begin{align*}
\mathcal{R}\left(z\right)r_{w}\left(z-X\right)^{-1} & =\mathcal{R}\left(z\right)e^{T\left(-z+X-\Lambda\right)}P_{w}\left(z-X\right)^{-1}+\mathcal{R}\left(z\right)P_{w}\left(z-X\right)^{-1}e^{T\left(-z+X\right)}\\
 & \qquad+\mathcal{R}\left(z\right)e^{T\left(-z+X-\Lambda\right)}P_{w}\left(z-X\right)^{-1}e^{T\left(-z+X\right)}.
\end{align*}
For the right hand side of the last equality, we have the following
estimates. We use the notation $\mathcal{O}_{\mathrm{Tr}}\left(\ast\right)$
for an operator whose trace norm is bounded by a constant multiple
of $\ast$. If $w\in W_{0}$ then as in (\ref{eq:estimate}), we have
\[
\mathcal{R}\left(z\right)r_{w}\left(z-X\right)^{-1}=\mathcal{O}_{\mathrm{Tr}}\left(\omega^{d}e^{-\epsilon T}\frac{1}{\mathrm{dist}\left(I_{w},\mathrm{Im}\left(z\right)\right)^{2}}\right),
\]
where $\mathrm{dist}\left(I_{w},\mathrm{Im}\left(z\right)\right)$
is the distance between $I_{w}$ and $\mathrm{Im}\left(z\right)$.
If $w\in W\backslash W_{0}$ then
\[
D_{w}\left(z\right)=\mathcal{O}_{\mathrm{Tr}}\left(\omega^{d}\frac{1}{\left(d\left(z\right)\right)^{2}}\right).
\]
Taking the sum with respect to $w\in W_{0}$, we see
\[
\sum_{w\in W_{0}}\left\Vert \mathcal{R}\left(z\right)r_{w}\left(z-X\right)^{-1}\right\Vert _{\mathrm{Tr}}=\sum_{w\in W_{0}}\mathcal{O}\left(\omega^{d}e^{-\epsilon T}\frac{1}{\mathrm{dist}\left(I_{w},\mathrm{Im}\left(z\right)\right)^{2}}\right)=\mathcal{O}\left(\omega^{d}e^{-\epsilon T}\right),
\]
\[
\sum_{w\in W\backslash W_{0}}\mathcal{O}_{\mathrm{Tr}}\left(\omega^{d}\frac{1}{\left(d\left(z\right)\right)^{2}}\right)=\mathcal{O}_{\mathrm{Tr}}\left(\omega^{d}\frac{\delta'}{\left(d\left(z\right)\right)^{2}}\right).
\]
Hence
\begin{align*}
\left\Vert D\left(z\right)-\left(\int_{0}^{T}e^{-tz}e^{t\left(X-\Lambda\right)}dt\right)\left(\Lambda P_{0}\right)\left(\int_{0}^{T}e^{-tz}e^{tX}dt\right)\right\Vert _{\mathrm{Tr}} & =\mathcal{O}\left(\omega^{d}\left(e^{-\epsilon T}+\frac{\delta'}{\left(d\left(z\right)\right)^{2}}\right)\right).
\end{align*}
\end{proof}
Now we give the key estimate. 
\begin{cBoxB}{}
\begin{lem}
For any arbitrarily small $c>0$, if we take $\delta$ large enough,
$\delta'>1$ large enough and $\delta/\delta'$ large enough, we have
\begin{equation}
\left|\left(\frac{1}{2\pi i}\int_{\Gamma_{+}-\Gamma_{-}}f_{\omega,\delta,\delta'}\left(z\right)\mathrm{Tr}\left(D\left(z\right)\right)dz\right)-\mathrm{Tr}\left(P\right)\right|<c\omega^{d}\delta\label{eq:Int_TrP}
\end{equation}
\end{lem}

\end{cBoxB}

\begin{proof}
We first consider the integration on $\Gamma_{+}$. From the last
lemma, we have

\begin{align}
\int_{\Gamma_{+}} & f_{\omega,\delta,\delta'}\left(z\right)\mathrm{Tr}\left(D\left(z\right)\right)dz=\int_{\mathbb{R}}f_{\delta,\delta'}\left(\omega'-\omega\right)\mathrm{Tr}\left(D\left(\gamma_{0}^{+}+\epsilon+i\omega'\right)\right)d\left(i\omega'\right)\nonumber \\
\eq{\ref{eq:approx_D-4}}i & \mathrm{Tr}\left[\int_{\omega'\in\mathbb{R}}f_{\delta,\delta'}\left(\omega'-\omega\right)\left(\int_{t'=0}^{T}e^{-t'\left(\gamma_{0}^{+}+\epsilon+i\omega'\right)}e^{t'\left(X-\Lambda\right)}dt'\right)\left(\Lambda P_{0}\right)\left(\int_{t=0}^{T}e^{-t\left(\gamma_{0}^{+}+\epsilon+i\omega'\right)}e^{tX}dt\right)\right]d\omega'\label{eq:f_om_delt}\\
 & +\mathcal{O}_{Tr}\left(\omega^{d}\left(\delta e^{-\epsilon T}+\delta'\right)\right).\nonumber 
\end{align}
Let
\begin{align*}
\tilde{f}_{\omega,\delta,\delta'}\left(t\right) & :=\int_{\mathbb{R}}e^{-t\left(\gamma_{0}^{+}+\epsilon+i\omega'\right)}f_{\delta,\delta'}\left(\omega'-\omega\right)d\omega'=e^{-t\left(\gamma_{0}^{+}+\epsilon+i\omega\right)}\hat{f}_{\delta,\delta'}\left(t\right),
\end{align*}
where $\hat{f}_{\delta,\delta'}\left(t\right):=\int_{\mathbb{R}}e^{-it\omega'}f_{\delta,\delta'}\left(\omega'\right)d\omega'$
satisfies a fast decay for $\left|t\right|\gg\delta'^{-1}$ as follows:
$\forall N>0,\exists C_{N}>0,\forall t,$
\begin{equation}
\left|\hat{f}_{\delta,\delta'}\left(t\right)\right|\leq C_{N}\left\langle \delta't\right\rangle ^{-N}\quad\text{{and}\ensuremath{\quad\int_{\mathbb{R}}\hat{f}_{\delta,\delta'}\left(t\right)dt=2\pi f_{\delta,\delta'}\left(0\right)=2\pi.}}\label{eq:f_hat-1}
\end{equation}
Then, with the change of variables $\left(t,t'\right)\rightarrow\left(t,s=t+t'\right)$,
we see
\begin{align*}
\int_{\Gamma_{+}} & f_{\omega,\delta,\delta'}\left(z\right)\mathrm{Tr}\left(D\left(z\right)\right)dz\\
\eq{\ref{eq:f_om_delt}}i & \int_{0}^{T}\int_{0}^{T}\tilde{f}_{\omega,\delta,\delta'}\left(t+t'\right)\mathrm{Tr}\left(e^{t'\left(X-\Lambda\right)}P_{0}\Lambda e^{tX}\right)dt'dt+\mathcal{O}_{Tr}\left(\omega^{d}\left(\delta e^{-\epsilon T}+\delta'\right)\right)\\
=i & \int_{0}^{2T}\left(\int_{t=\max\left\{ 0,s-T\right\} }^{\min\left\{ T,s\right\} }e^{-s\left(\gamma_{0}^{+}+\epsilon+i\omega\right)}\hat{f}_{\delta,\delta'}\left(s\right)\mathrm{Tr}\left(e^{\left(s-t\right)\left(X-\Lambda\right)}P_{0}\Lambda e^{tX}\right)dt\right)ds\\
 & +\mathcal{O}_{Tr}\left(\omega^{d}\left(\delta e^{-\epsilon T}+\delta'\right)\right).
\end{align*}
From (\ref{eq:f_hat-1}), by letting $\delta'$ be large (and $\hat{f}_{\delta,\delta'}\left(t\right)$
be more concentrate to $0$), we get

\[
\left|\int_{\Gamma_{+}}f_{\omega,\delta,\delta'}\left(z\right)\mathrm{Tr}\left(D\left(z\right)\right)dz\right|=\mathcal{O}_{Tr}\left(\delta'^{-1}\omega^{d}\delta\right)+\mathcal{O}_{Tr}\left(\omega^{d}\left(\delta e^{-\epsilon T}+\delta'\right)\right).
\]
Next we consider the integration on $\Gamma_{-}$. Similarly to the
argument above, we see
\begin{align*}
\int_{-\Gamma_{-}} & f_{\omega,\delta,\delta'}\left(z\right)\mathrm{Tr}\left(D\left(z\right)\right)dz=i\int f_{\delta,\delta'}\left(\omega'-\omega\right)\mathrm{Tr}\left(D\left(\gamma_{0}^{-}-\epsilon-i\omega'\right)\right)d\omega'\\
\eq{\ref{eq:approx_D-4}} & i\mathrm{Tr}\left[\int f_{\delta,\delta'}\left(\omega'-\omega\right)\left(\int_{0}^{T}e^{-t'\left(\gamma_{0}^{-}-\epsilon-i\omega'\right)}e^{t'\left(X-\Lambda\right)}dt'\right)\left(\Lambda P_{0}\right)\left(\int_{-T}^{0}e^{-t\left(\gamma_{0}^{-}-\epsilon-i\omega'\right)}e^{tX}dt\right)\right]d\omega'\\
 & +\mathcal{O}_{\mathrm{Tr}}\left(\omega^{d}\left(\delta e^{-\epsilon T}+\delta'\right)\right)\\
= & i\int_{s=-T}^{T}\left(\int_{t=\max\left\{ -T,s-T\right\} }^{\min\left\{ 0,s\right\} }e^{-s\left(\gamma_{0}^{-}-\epsilon-i\omega'\right)}\hat{f}_{\delta,\delta'}\left(s\right)\mathrm{Tr}\left(e^{s\left(X-\Lambda\right)}e^{-t\left(X-\Lambda\right)}\Lambda P_{0}e^{tX}\right)dt\right)ds\\
 & +\mathcal{O}_{\mathrm{Tr}}\left(\omega^{d}\left(\delta e^{-\epsilon T}+\delta'\right)\right)\\
= & 2\pi i\int_{t=-T}^{0}\mathrm{Tr}\left(e^{-t\left(X-\Lambda\right)}\Lambda P_{0}e^{tX}\right)dt+\mathcal{O}_{Tr}\left(\delta'^{-1}\omega^{d}\delta\right)+\mathcal{O}_{\mathrm{Tr}}\left(\omega^{d}\left(\delta e^{-\epsilon T}+\delta'\right)\right)\\
= & 2\pi i\mathrm{Tr}\left(\left[e^{t\Lambda}P_{0}\right]_{t=-T}^{0}\right)+\mathcal{O}_{Tr}\left(\delta'^{-1}\omega^{d}\delta\right)+\mathcal{O}_{\mathrm{Tr}}\left(\omega^{d}\left(\delta e^{-\epsilon T}+\delta'\right)\right)\\
= & 2\pi i\mathrm{Tr}\left(P_{0}\right)+\mathcal{O}_{Tr}(e^{-\epsilon T}\omega^{d})+\mathcal{O}_{Tr}\left(\delta'^{-1}\omega^{d}\delta\right)+\mathcal{O}_{\mathrm{Tr}}\left(\omega^{d}\left(\delta e^{-\epsilon T}+\delta'\right)\right)\\
= & 2\pi i\mathrm{Tr}\left(P\right)+\mathcal{O}_{\mathrm{Tr}}\left(\omega^{d}\delta'\right)+\mathcal{O}_{\mathrm{Tr}}(e^{-\epsilon T}\omega^{d})+\mathcal{O}_{\mathrm{Tr}}\left(\delta'^{-1}\omega^{d}\delta\right)+\mathcal{O}_{\mathrm{Tr}}\left(\omega^{d}\left(\delta e^{-\epsilon T}+\delta'\right)\right).
\end{align*}
Therefore, by letting $T$ large $\delta\gg\delta'\gg1$ and then
letting $\omega$ large, we obtain the required estimate.
\end{proof}

\subsection{\label{subsec:Weyl-law}Weyl law}

Using the previous estimates we get that

\begin{align}
\left|\sharp\left(\mathrm{spect}\left(X\right)\cap\tilde{R}\left(\omega,\delta,\delta'\right)\right)-\mathrm{rank}\left(\mathcal{F}_{\left[0,k\right]}\right)\frac{\mathrm{Vol}\left(M\right)\omega^{d}}{\left(2\pi\right)^{d+1}}\left(2\delta\right)\right|\nonumber \\
\underset{(\ref{eq:res4},\ref{eq:Int_TrP},\ref{eq:trace-1})}{<}C\omega^{d}\max\left\{ \delta',\ln\delta\right\} +c\omega^{d}\delta+C\omega^{d}\delta\left(\left|\omega\right|^{-\beta/2}+\frac{\delta}{\omega}\right)+C\omega^{d} & \left(C_{N}\sigma^{-N}+\left|\omega\right|^{-\beta/2}\sigma+\frac{\delta}{\omega}\right)\nonumber \\
< & c'\omega^{d}\delta.\label{eq:last}
\end{align}
with arbitrarily small $c'>0$ and $\omega\gg\delta\gg\delta'\gg1$.
Comparing this estimate for adjacent $k$'s, we obtain Theorem \ref{thm:Weyl-law-for}.

\section{\label{sec:Proof-of-Theorem-accumulation}Proof of Theorem \ref{thm:Ergodic-concentration-of}
(accumulation of eigenvalues on narrower bands)}

The aim of this section is to prove Theorem \ref{thm:Ergodic-concentration-of}.
We consider the operator $X_{F}$ defined in Section \ref{subsec:Transfer-operator}
where $X$ is a contact Anosov flow. For $k\in\mathbb{N},$ consider
the vector bundle $\pi:\mathcal{F}_{k}\rightarrow\Sigma$, with $\mathcal{F}_{k}=\left|\mathrm{det}N_{s}\right|^{-1/2}\otimes\mathrm{Pol}_{k}\left(N_{s}\right)\otimes F$
defined in (\ref{eq:def_F-1-1}), twisted by $F$, and the flow induced
from $X_{F}$,
\[
\phi_{k}^{t}:\mathcal{F}_{k}\rightarrow\mathcal{F}_{k},\quad t\in\mathbb{R},
\]
that is a linear bundle map over $\tilde{\phi}^{t}:\Sigma\rightarrow\Sigma$.
For a point $\rho\in\Sigma$, $t\in\mathbb{R}$, let
\[
g_{k,t}\left(\rho\right):=\ln\left(\left(\max_{v\in\pi^{-1}\left(\rho\right)}\frac{\left\Vert \phi_{k}^{t}\left(v\right)\right\Vert }{\left\Vert v\right\Vert }\right)^{1/t}\right),
\]
that defines a continuous function $g_{k,t}$ on $\Sigma$. Let $d\varrho$
be the Liouville measure (\ref{eq:dvol_E0*}) on $\Sigma$ and
\[
\mu_{k,t}:=\left(g_{k,t}^{\circ}\right)^{\dagger}\left(d\varrho\right)
\]
be the push-forward measure on $\mathbb{R}$ by the map $g_{k,t}$.
We have already defined $\gamma_{k}^{\pm}$ in (\ref{eq:def_gamma_})
that can be written as
\[
\gamma_{k}^{+}:=\lim_{t\rightarrow+\infty}\max_{\rho\in\Sigma}g_{k,t}\left(\rho\right)=\lim_{t\rightarrow+\infty}\max\left(\mathrm{supp}\left(\mu_{k,t}\right)\right),
\]
\[
\gamma_{k}^{-}:=\lim_{t\rightarrow-\infty}\min_{\rho\in\Sigma}g_{k,t}\left(\rho\right)=\lim_{t\rightarrow-\infty}\min\left(\mathrm{supp}\left(\mu_{k,t}\right)\right).
\]
Now we define the \textbf{maximal and minimal Lyapounov exponents
with respect to the Liouville measure} by
\begin{equation}
\check{\gamma}_{k}^{+}=\max\left(\mathrm{supp}\left(\lim_{t\rightarrow+\infty}\mu_{k,t}\right)\right),\label{eq:def_Lyapounov_exp}
\end{equation}
\[
\check{\gamma}_{k}^{-}=\min\left(\mathrm{supp}\left(\lim_{t\rightarrow-\infty}\mu_{k,t}\right)\right).
\]
One has
\[
\gamma_{k}^{-}\leq\check{\gamma}_{k}^{-}\leq\check{\gamma}_{k}^{+}\leq\gamma_{k}^{+}.
\]
From the definition of $\check{\gamma}_{k}^{-},\check{\gamma}_{k}^{+}$,
for any small $\epsilon>0$, if $I_{\epsilon}:=\left[\check{\gamma}_{k}^{-}-\epsilon,\check{\gamma}_{k}^{+}+\epsilon\right]$
and
\[
\Sigma_{\omega,\delta,I_{\epsilon},t}:=\left(g_{k,t}^{-1}\left(I_{\epsilon}\right)\cap\left[\omega,\omega+\delta\right]\right)\subset\Sigma,
\]
then we have that for any $c>0$, there exists $T_{c}>0$ such that
\begin{equation}
\left(d\varrho\right)\left(\Sigma\backslash\Sigma_{\omega,\delta,I_{\epsilon},T_{c}}\right)\leq c\delta\omega^{d}.\label{eq:measure}
\end{equation}

For the proof of Theorem \ref{thm:Ergodic-concentration-of}, we follow
\cite[Thm 1.3.11 and Chapter 8,]{faure-tsujii_prequantum_maps_12}.
Let $c>0$. Take $T_{c}>0$ large enough such that (\ref{eq:measure})
holds true. Then similarly to (\ref{eq:bound-4-1}) one gets $\forall\epsilon>0,\exists C_{\epsilon}>0$,$\forall t\geq0$,
$\exists\sigma_{t}>0$, $\forall\sigma>\sigma_{t},$ $\exists\omega_{\sigma}>0,\forall\omega>\omega_{\sigma}$,
\[
\left\Vert e^{tX}\mathrm{Op}_{\Sigma}\left(\boldsymbol{1}_{\Sigma_{\omega,\delta,I_{\epsilon},T_{c}}}T_{k}\right)\right\Vert _{\mathcal{H}_{W}\left(M\right)}\leq C_{\epsilon}e^{t\left(\check{\gamma}_{k}^{+}+\epsilon\right)},
\]
\[
\left\Vert e^{-tX}\mathrm{Op}_{\Sigma}\left(\boldsymbol{1}_{\Sigma_{\omega,\delta,I_{\epsilon},T_{c}}}T_{k}\right)\right\Vert _{\mathcal{H}_{W}\left(M\right)}\leq C_{\epsilon}e^{-t\left(\check{\gamma}_{k}^{-}-\epsilon\right)},
\]
\[
\left\Vert \mathrm{Op}_{\Sigma}\left(\boldsymbol{1}_{\Sigma\backslash\Sigma_{\omega,\delta,I_{\epsilon},T_{c}}}T_{k}\right)\right\Vert _{\mathrm{Tr}}=c\delta\omega^{d}+C\omega^{d}.
\]
Using arguments with perturbation of the resolvent as in previous
proofs, we deduce the Theorem.

\section{\label{sec:Horocycle-operators}Proof of Theorem \ref{thm:Quantization-of-()}
(horocycle operators)}

Suppose $s\in C^{\beta}\left(M;E_{s}\right)$ is a Hölder continuous
section of the stable bundle over $M$ and $u\in C^{\beta}\left(M;E_{u}\right)$
a section of the unstable bundle. Let $\pi:T^{*}M\rightarrow M$ be
the canonical projection. We have seen in (\ref{eq:def_Nus}) that
for $\rho\in\Sigma$, the differential $d\pi$ gives a linear isomorphism
$d\pi:N\left(\rho\right)=N_{s}\left(\rho\right)\oplus N_{u}\left(\rho\right)\rightarrow E_{u}\left(m\right)\oplus E_{s}\left(m\right)$
with $m=\pi\left(\rho\right)$. By pull back we get sections $\tilde{s}:=\left(d\pi\right)^{-1}\left(s\right)\in C\left(\Sigma;N_{s}\right)$
and $\tilde{u}:=\left(d\pi\right)^{-1}\left(u\right)\in C\left(\Sigma;N_{u}\right)$.
Recall from Section \ref{subsec:Contact-Anosov-flow} that $d\mathcal{A}$
is a symplectic form on $E_{u}\oplus E_{s}$ and from Definition \ref{def:KNA}
that $N\left(\rho\right)=N_{s}\left(\rho\right)\oplus N_{u}\left(\rho\right)$
is a $\Omega$-symplectic vector space and that $\Omega_{\rho}\left(\tilde{u},\tilde{s}\right)\eq{\ref{eq:theta_A}}\omega\left(\rho\right).\left(d\mathcal{A}\right)\left(u,s\right)$,
hence $\tilde{u}$ (so then $u$) defines a dual vector 
\begin{equation}
\tilde{u}^{*}:=\check{\Omega}\tilde{u}\eq{\ref{eq:def_Omega_check}}\Omega\left(\tilde{u},.\right)\eq{\ref{eq:theta_A}}\boldsymbol{\omega}\left(.\right)\,\left(d\mathcal{A}\right)\left(u,d\pi\left(.\right)\right)\in N_{s}^{*}\left(\rho\right).\label{eq:utilde*}
\end{equation}

For $k\in\mathbb{N}$, let $p\in C\left(\Sigma;\mathcal{F}_{k}\right)$
a continuous section of $\mathcal{F}_{k}\left(N_{s}\right)\eq{\ref{eq:def_F-1}}\left|\mathrm{det}N_{s}\right|^{-1/2}\otimes\mathrm{Pol}_{k}\left(N_{s}\right)\otimes F$
over $\Sigma$ where $\mathrm{Pol}_{k}\left(N_{s}\right)\equiv\mathrm{Sym}^{k}\left(N_{s}^{*}\otimes\ldots\otimes N_{s}^{*}\right)$
is equivalent to the \href{https://en.wikipedia.org/wiki/Symmetric_algebra}{symmetric tensor algebra}
of degree $k$.

\begin{cBoxA}{}
\begin{defn}
\label{def:For--we}For $s\in C^{\beta}\left(M;E_{s}\right)$ we define
\begin{equation}
\iota_{s}:\begin{cases}
C^{\beta}\left(\Sigma;\mathcal{F}_{k}\right) & \rightarrow C^{\beta}\left(\Sigma;\mathcal{F}_{k-1}\right)\\
p & \rightarrow\iota_{\tilde{s}}p=kp\left(\tilde{s},.,\ldots,.\right)
\end{cases}\label{eq:def_is}
\end{equation}
as the \href{https://en.wikipedia.org/wiki/Interior_product}{interior product},
i.e. the pointwise (i.e. independently in every fiber) contraction
in the first entry of $p$ by $\tilde{s}:=\left(d\pi\right)^{-1}\left(s\right)\in C^{\beta}\left(\Sigma;N_{s}\right)$,
but with a factor $k$. For $u\in C^{\beta}\left(M;E_{u}\right)$
we define
\[
u\varodot:\begin{cases}
C^{\beta}\left(\Sigma;\mathcal{F}_{k}\right) & \rightarrow C^{\beta}\left(\Sigma;\mathcal{F}_{k+1}\right)\\
p & \rightarrow\tilde{u}^{*}\varodot p=\mathrm{Sym}\left(\tilde{u}^{*}\otimes p\right)
\end{cases}
\]
as the symmetric point-wise tensor product of $p$ by $\tilde{u}^{*}:=\omega\left(d\mathcal{A}\right)\left(u,d\pi\left(.\right)\right)\in C^{\beta}\left(\Sigma;N_{s}^{*}\right)$.
\end{defn}

\end{cBoxA}

~

\begin{cBoxB}{}
\begin{lem}[\textbf{Weyl algebra}]
\textbf{\label{lem:Weyl-algebra.-For}}For $s\in C^{\beta}\left(M;E_{s}\right)$,
$u\in C^{\beta}\left(M;E_{u}\right)$ and $k\in\mathbb{N}$ we have
\begin{equation}
\left[\iota_{s},u\varodot\right]=\boldsymbol{\omega}\left(.\right)\left(d\mathcal{A}\right)\left(s,u\right)\mathrm{Id}_{\mathcal{F}_{k}}\quad:C^{\beta}\left(\Sigma;\mathcal{F}_{k}\right)\rightarrow C^{\beta}\left(\Sigma;\mathcal{F}_{k}\right)\label{eq:commut}
\end{equation}
where $\omega$ is the frequency function (\ref{eq:omega_function})
on $\Sigma$ and $\mathrm{Id}_{\mathcal{F}_{k}}$ is the identity
operator in fibers.
\end{lem}

\end{cBoxB}

\begin{proof}
This is point wise relation so we consider a point $\rho\in\Sigma$
with frequency $\omega\eq{\ref{eq:omega_function}}\boldsymbol{\omega}\left(\rho\right)$
and the vector space $N_{s}\left(\rho\right)$ defined in (\ref{eq:def_Nus}).
We use the isomorphism $\mathrm{Pol}_{k}\left(N_{s}\right)\equiv\mathrm{Sym}^{k}\left(N_{s}^{*}\otimes\ldots\otimes N_{s}^{*}\right)$
to derive (\ref{eq:commut}) that goes as follows. If $\left(e_{i}\right)_{i}$
is a basis of $N_{s}\left(\rho\right)$, we write $\tilde{s}=\sum_{i=1}^{d}\sigma_{i}e_{i}\in N_{s}\left(\rho\right)$
with components $\sigma=\left(\sigma_{i}\right)_{i}\in\mathbb{R}^{d}$.
For $p\in\mathrm{Sym}^{k}\left(N_{s}^{*}\otimes\ldots\otimes N_{s}^{*}\right)$,
we associate a degree $k$ polynomial $P\in\mathrm{Pol}_{k}\left(\mathbb{R}^{d}\right)$
by $P\left(\sigma\right):=p\left(\tilde{s},\tilde{s},\ldots\tilde{s}\right)$.
We define $e_{i}^{*}\varodot p:=\mathrm{Sym}\left(e_{i}^{*}\otimes p\right)$
and observe that $\left(e_{i}^{*}\varodot p\right)\left(\tilde{s}\right)=\sigma_{i}P\left(\sigma\right)$
and $\left(\iota_{e_{i}}p\right)\left(\tilde{s}\right)=\left(\partial_{\sigma_{i}}P\right)\left(\sigma\right)$.
Then the ``Weyl algebra relation '' on polynomials $\left[\partial_{\sigma_{i}},\sigma_{j}\right]=\delta_{i=j}\mathrm{Id}$
gives $\left[\iota_{e_{i}},e_{j}^{*}\varodot\right]=\delta_{i=j}\mathrm{Id}$
hence for $\tilde{s}\in N_{s}$, $\tilde{u}^{*}\in N_{s}^{*}$ we
have $\left[\iota_{\tilde{s}},\tilde{u}^{*}\varodot\right]=\tilde{u}^{*}\left(\tilde{s}\right)\mathrm{Id}$.
With (\ref{eq:utilde*}), we get (\ref{eq:commut}).
\end{proof}

\paragraph{Proof of Theorem \ref{thm:Quantization-of-()}:}

Let $K\in\mathbb{N}$. We will use $\tilde{T}_{\left[0,K\right]}$
to restrict to the finite rank bundle of homogeneous polynomials of
degree less than $K$. For $k<K$, let $\tilde{T}_{k}$ the projector
on homogeneous polynomials of degree $k$. We have
\begin{align*}
\left[\mathrm{Op}_{\Sigma}\left(\iota_{s}\tilde{T}_{\left[0,K\right]}\right),\mathrm{Op}_{\Sigma}\left(\left(u\varodot\right)\tilde{T}_{\left[0,K\right]}\right)\right]\mathrm{Op}_{\Sigma}\left(\tilde{T}_{k}\right) & \underset{(\ref{eq:compos_operators})}{\approx}\mathrm{Op}_{\Sigma}\left(\left[\iota_{s},u\varodot\right]\tilde{T}_{k}\right)\eq{\ref{eq:commut}}\mathrm{Op}_{\Sigma}\left(\boldsymbol{\omega}\left(.\right)\left(d\mathcal{A}\right)\left(s,u\right)\tilde{T}_{k}\right)
\end{align*}
that gives (\ref{eq:weyl_algebra_quantized}). We have
\begin{align*}
e^{tX_{F}}\mathrm{Op}_{\Sigma}\left(\iota_{s}\tilde{T}_{k}\right) & \underset{(\ref{eq:expression})}{\approx}\mathrm{Op}_{\Sigma}\left(e^{tX_{\mathcal{F}}}\right)\mathrm{Op}_{\Sigma}\left(\iota_{s}\tilde{T}_{k}\right)\underset{(\ref{eq:compos_operators})}{\approx}\mathrm{Op}_{\Sigma}\left(e^{tX_{\mathcal{F}}}\iota_{s}\tilde{T}_{k}\right)\\
 & =\mathrm{Op}_{\Sigma}\left(\iota_{d\phi^{t}s}e^{tX_{\mathcal{F}}}\tilde{T}_{k}\right)\underset{(\ref{eq:commut_T},\ref{eq:compos_operators})}{\approx}\mathrm{Op}_{\Sigma}\left(\iota_{d\phi^{t}s}\tilde{T}_{k}\right)\mathrm{Op}_{\Sigma}\left(e^{tX_{\mathcal{F}}}\right)\\
 & \underset{(\ref{eq:expression})}{\approx}\mathrm{Op}_{\Sigma}\left(\iota_{d\phi^{t}s}\tilde{T}_{k}\right)e^{tX_{F}},
\end{align*}
giving (\ref{eq:commut-1-1-2}) and similarly to get (\ref{eq:commut-1-1-1-1}).

\appendix

\section{\label{sec:General-notations-used}General notations used in this
paper}
\begin{itemize}
\item \textbf{Dual $*$.} If $E$ is a finite dimensional vector space and
$E^{*}$ its \href{https://en.wikipedia.org/wiki/Dual_space}{dual space},
$x\in E,\xi\in E^{*}$ we denote the duality by $\langle\xi|x\rangle\in\mathbb{R}$.
If $E_{1},E_{2}$ are vector spaces and $A:E_{1}\rightarrow E_{2}$
is a linear map, we will denote $A^{*}:E_{2}^{*}\rightarrow E_{1}^{*}$
the induced dual map on the dual spaces defined by $\left\langle A^{*}\xi|x\right\rangle =\langle\xi|Ax\rangle,\forall x\in E_{1},\xi\in E_{2}^{*}$.
\item \textbf{Bilinear map.} Let $\Omega:E\times E\rightarrow\mathbb{R}$
be a bilinear map on a vector space $E$. It defines a linear map
$\check{\Omega}:E\rightarrow E^{*}$ by 
\begin{equation}
\check{\Omega}\left(x\right)\left(.\right)=\Omega\left(x,.\right).\label{eq:def_Omega_check}
\end{equation}

\begin{itemize}
\item If $\check{\Omega}$ is invertible ($\Omega$ is said to be non degenerated),
it induces a bilinear form on $E^{*}$ denoted $\Omega^{-1}:E^{*}\times E^{*}\rightarrow\mathbb{R}$
and defined by $\Omega^{-1}\left(\xi_{1},\xi_{2}\right):=\Omega\left(\check{\Omega}^{-1}\xi_{1},\check{\Omega}^{-1}\xi_{2}\right)=\xi_{1}\left(\check{\Omega}^{-1}\xi_{2}\right)$. 
\item The bilinear map $\Omega:E\times E\rightarrow\mathbb{R}$ is anti-symmetric
iff 
\begin{equation}
\check{\Omega}^{*}=-\check{\Omega}.\label{eq:Omega_antis}
\end{equation}
The bilinear map $\Omega:E\times E\rightarrow\mathbb{R}$ is symmetric
iff $\check{\Omega}^{*}=\check{\Omega}$.
\end{itemize}
\item \textbf{Adjoint $\dagger$.} Suppose $A:E_{1}\rightarrow E_{2}$ is
a linear map and $\Omega_{1}$ (respect. $\Omega_{2}$) is a non degenerated
bilinear form on $E_{1}$ (respect. $E_{2}$). The $\Omega$-adjoint
of $A$ is $A^{\dagger_{\Omega}}:E_{2}\rightarrow E_{1}$ defined
by
\begin{equation}
A^{\dagger_{\Omega}}=\check{\Omega}_{1}^{-1}A^{*}\check{\Omega}_{2}\label{eq:adjoint}
\end{equation}
equivalently given by
\[
\Omega_{1}\left(A^{\dagger_{\Omega}}x_{1},x_{2}\right)=\Omega_{2}\left(x_{1},Ax_{2}\right),\forall x_{1}\in E_{1},\forall x_{2}\in E_{2}.
\]
In this paper bilinear forms will be either a symplectic form $\Omega\left(.,.\right)$
or a metric $g\left(.,.\right)$.
\item \textbf{Pull-back $\circ$.} If $f:M\rightarrow N$ is a smooth map
between two manifolds,
\begin{itemize}
\item We denote $f^{\circ}$ the \textbf{pull back} operator
\begin{equation}
f^{\circ}:\begin{cases}
C^{\infty}\left(N\right) & \rightarrow C^{\infty}\left(M\right)\\
u & \rightarrow u\circ f
\end{cases}\label{eq:def_pullback}
\end{equation}
and if $f$ is a diffeomorphism, we denote $f^{-\circ}:=\left(f^{-1}\right)^{\circ}$
the \textbf{push forward} operator, so that $f^{-\circ}f^{\circ}=\mathrm{Id}$. 
\item We denote $df:TM\rightarrow TN$ the differential of $f$ that is
a linear bundle map and denote $\left(df\right)^{*}:T^{*}N\rightarrow T^{*}M$
its dual.
\end{itemize}
\end{itemize}

\section{\label{sec:More-informations-on}More information about flows}

\subsection{\label{subsec:Transfer-operator-1}Transfer operator}

Suppose that $X$ is a vector field on $M$ and $\left(\phi^{t}\right)^{\circ}=e^{tX}$
is the pull back operator defined in (\ref{eq:def_exptX}). For any
smooth measure $d\mu$ on $M$, the $L^{2}\left(M,d\mu\right)$ scalar
product is defined by $\langle v|u\rangle_{L^{2}}:=\int_{M}\overline{u}vd\mu$
with $u,v\in C_{c}^{\infty}\left(M\right)$. Then the $L^{2}$-adjoint
$\left(\phi^{t}\right)^{\circ\dagger}$ of $\left(\phi^{t}\right)^{\circ}$,
defined by $\langle v|\left(\phi^{t}\right)^{\circ\dagger}u\rangle_{L^{2}}=\langle\left(\phi^{t}\right)^{\circ}v|u\rangle_{L^{2}},\forall u,v\in C_{c}^{\infty}\left(M\right)$
is given by
\[
\left(\phi^{t}\right)^{\circ\dagger}=\left|\mathrm{det}\left(d\phi^{-t}\right)\right|.\left(\phi^{-t}\right)^{\circ}=e^{tX^{\dagger}},\quad\text{with }X^{\dagger}=-X-\mathrm{div}_{\mu}X.
\]
$\left(\phi^{t}\right)^{\circ\dagger}$ is called the \textbf{Ruelle-Perron-Frobenius
operator }or \textbf{\href{https://en.wikipedia.org/wiki/Transfer_operator}{transfer operator}.
}Observe that $\left(\phi^{t}\right)^{\circ\dagger}$\textbf{ }\textbf{\small{}pushes
forward probability distributions} because
\[
\int_{M}\left(\phi^{t}\right)^{\circ\dagger}ud\mu=\langle\mathrm{1}|\left(\phi^{t}\right)^{\circ\dagger}u\rangle_{L^{2}}=\langle\underbrace{\left(\phi^{t}\right)^{\circ}\mathrm{1}}_{1}|u\rangle_{L^{2}}=\int_{M}ud\mu.
\]
In particular the evolution of Dirac measures gives $\left(\phi^{t}\right)^{\circ\dagger}\delta_{m}=\delta_{\phi^{t}\left(m\right)}$
and is equivalent to evolution of points under the flow map $\phi^{t}$.

\subsection{\label{subsec:More-general-pull}More general pull back operators
$e^{tX_{F}}$}

Suppose that $X$ is a smooth vector field on $M$ and $F\rightarrow M$
is a smooth vector bundle over $M$. Suppose that $X_{F}:C^{\infty}\left(M;F\right)\rightarrow C^{\infty}\left(M;F\right)$
is a derivation over $X$ fulfilling (\ref{eq:def_XF}).

\subsubsection{Expression in local frames}

For a general vector bundle $F$ of rank $r$, with respect to a local
frame $\left(e_{1},\ldots e_{r}\right)$, a section $u\in C^{\infty}\left(M;F\right)$
is expressed as $u\left(m\right)=\sum_{j=1}^{r}u_{j}\left(m\right)e_{j}\left(m\right)$
with $m\in M$ and components $u_{j}\left(m\right)\in\mathbb{C}$.
We introduce a matrix of potential functions $V_{k,j}\left(m\right)\in\mathbb{C}$
defined by
\begin{equation}
\left(X_{F}e_{k}\right)\left(m\right)=\sum_{j}V_{k,j}\left(m\right)e_{j}\left(m\right).\label{eq:def_Vkj}
\end{equation}
Then the operator $X_{F}$ in (\ref{eq:def_X_F}) is expressed as
\begin{equation}
\left(X_{F}u\right)\left(m\right)\eq{\ref{eq:def_XF}}\sum_{j=1}^{r}\left(\left(Xu_{j}\right)\left(m\right)+\sum_{k=1}^{r}V_{k,j}\left(m\right)u_{k}\left(m\right)\right)e_{j}\left(m\right).\label{eq:A_bundle}
\end{equation}
We see that the expression (\ref{eq:A_bundle}) generalizes (\ref{eq:XF_X_V}).

\subsubsection{Expression of $e^{tX_{F}}$ from a flow $\tilde{\phi}_{F}^{t}$ on
$F$}

An equivalent way to present the operator $X_{F}$ in (\ref{eq:def_XF}),
is to consider a vector field $\tilde{X}_{F}$ on $F$, a lift of
$X$, i.e. $\left(d\pi\right)\tilde{X}_{F}=X$, with the projection
$\pi:F\rightarrow M$, and such that the flow map generated by $\tilde{X}_{F}$
\[
\tilde{\phi}_{F}^{t}:F\rightarrow F
\]
i.e. $\left(\tilde{\phi}_{F}^{t}\right)^{\circ}=e^{t\tilde{X}_{F}}:\mathcal{S}\left(F\right)\rightarrow\mathcal{S}\left(F\right)$,
is a smooth linear bundle map over $\phi^{t}$ which means that for
every $m\in M,t\in\mathbb{R}$, $\tilde{\phi}_{F}^{t}\left(m\right):F\left(m\right)\rightarrow F\left(\phi^{t}\left(m\right)\right)$
is a linear map. See Figure \ref{fig:Transfer-operator}. This defines
a \textbf{group of pull back operators} with generator $X_{F}$ acting
on a section $u\in C^{\infty}\left(M;F\right)$ of the vector bundle
$F$:
\[
e^{tX_{F}}:C^{\infty}\left(M;F\right)\rightarrow C^{\infty}\left(M;F\right),\quad t\in\mathbb{R}.
\]
by
\begin{equation}
e^{tX_{F}}u:=\tilde{\phi}_{F}^{-t}\left(u\circ\phi^{t}\right).\label{eq:evol_XF}
\end{equation}
We get $e^{tX_{F}}\left(fu\right)=\tilde{\phi}_{F}^{-t}\left(\left(e^{tX}f\right)\left(u\circ\phi^{t}\right)\right)=\left(e^{tX}f\right)\left(e^{tX_{F}}u\right)$
and deduce (\ref{eq:def_XF}) by derivation w.r.t. $t$. Conversely
the flow map $\tilde{\phi}_{F}^{t}$ and its generator $\tilde{X}_{F}$
is defined from $X_{F}$ by (\ref{eq:evol_XF}) and use of any section
$u$.
\begin{center}
\begin{figure}
\begin{centering}
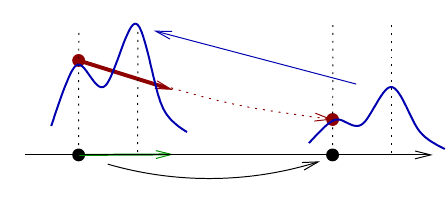
\par\end{centering}
\caption{\label{fig:Transfer-operator}pull back operator $e^{tX_{F}}$}
\end{figure}
\par\end{center}

\section{\label{sec:Bargmann-transform-and}Bargmann transform and Metaplectic
operators}

In this section we collect results that concerns the \href{https://en.wikipedia.org/wiki/Segal\%E2\%80\%93Bargmann_space}{Bargmann transform}
on a vector space $E$, the quantization of affine and linear symplectic
map, that are the unitary representations of the \href{https://en.wikipedia.org/wiki/Heisenberg_group}{Heisenberg group}
and the \href{https://en.wikipedia.org/wiki/Metaplectic_group}{metaplectic operators}.
All these definitions and results are well known and important in
many fields of mathematics and physics. They are at the core of micro-local
analysis. We present these definitions and results in a form that
is adapted to our paper that relies strongly on them. References are
\cite{folland-88},\cite{hall99},\cite[appendix 4.4]{lerner2011metrics},\cite[chap.3, chap.4]{faure-tsujii_prequantum_maps_12}.

\subsection{Weyl Heisenberg group}

Let $E$ be a real vector space, $n=\mathrm{dim}E$. For $x\in E$,
$\xi\in E^{*}$ we denote $\langle\xi|x\rangle\in\mathbb{R}$ the
duality. Let $dx\in\left|\Lambda^{n}\left(E\right)\right|$ be a \href{https://en.wikipedia.org/wiki/Density_on_a_manifold}{density}
on $E$ let $d\xi$ be the induced density on $E^{*}$. Let us denote
the \href{https://en.wikipedia.org/wiki/Fourier_transform}{Fourier transform}
\[
\mathcal{F}:\begin{cases}
\mathcal{S}\left(E\right) & \rightarrow\mathcal{S}\left(E^{*}\right)\\
u & \rightarrow v\left(\xi\right)=\frac{1}{\left(2\pi\right)^{n/2}}\int e^{-i\langle\xi|x\rangle}u\left(x\right)dx
\end{cases}
\]
Its $L^{2}$-adjoint is
\[
\mathcal{F}^{\dagger}:\begin{cases}
\mathcal{S}\left(E^{*}\right) & \rightarrow\mathcal{S}\left(E\right)\\
v & \rightarrow u\left(x\right)=\frac{1}{\left(2\pi\right)^{n/2}}\int e^{i\langle\xi|x\rangle}v\left(\xi\right)d\xi
\end{cases}
\]
and we have that $\mathcal{F}^{\dagger}\mathcal{F}=\mathrm{Id}_{\mathcal{S}\left(E\right)}$
hence $\mathcal{F}^{-1}=\mathcal{F}^{\dagger}$.

For $x\in E$, $\xi\in E^{*}$ we define the translation maps
\[
T_{x}:\begin{cases}
E & \rightarrow E\\
y & \rightarrow y+x
\end{cases},\qquad T_{\xi}:\begin{cases}
E^{*} & \rightarrow E^{*}\\
\xi' & \rightarrow\xi'+\xi
\end{cases}
\]
And denote $T_{x}^{-\circ},T_{\xi}^{-\circ}$ the push forward operators.
We will denote
\begin{equation}
\hat{T}_{x,\xi}:=T_{x}^{-\circ}\left(\mathcal{F}^{-1}T_{\xi}^{-\circ}\mathcal{F}\right)\label{eq:def_T_rho}
\end{equation}

\begin{cBoxB}{}
\begin{lem}
On \emph{$\mathcal{S}\left(E\right)$}, one has
\begin{equation}
\left(\mathcal{F}^{-1}T_{\xi}^{-\circ}\mathcal{F}\right)T_{x}^{-\circ}=e^{i\langle\xi|x\rangle}T_{x}^{-\circ}\left(\mathcal{F}^{-1}T_{\xi}^{-\circ}\mathcal{F}\right).\label{eq:weyl_heisenberg}
\end{equation}
\[
\hat{T}_{x,\xi}\hat{T}_{x',\xi'}=e^{i\langle\xi|x'\rangle}\hat{T}_{x+x',\xi+\xi'}.
\]
In particular $\hat{T}_{x,\xi}$ is a unitary operator in $L^{2}\left(E,dx\right)$
and
\begin{equation}
\hat{T}_{x,\xi}^{-1}=\hat{T}_{x,\xi}^{\dagger}=\hat{T}_{-\left(x,\xi\right)}e^{i\xi x}.\label{eq:inverse_T}
\end{equation}
\end{lem}

\end{cBoxB}

\begin{rem}
The family of operators $\hat{T}_{x,\xi}$ with $\left(x,\xi\right)\in T^{*}E=E\oplus E^{*}$
are called the Schrödinger unitary representation of the \href{https://en.wikipedia.org/wiki/Heisenberg_group}{Weyl Heisenberg group}.
\end{rem}

\begin{proof}
Let $x\in E$. Then
\[
\left(\mathcal{F}\delta_{x}\right)\left(\xi'\right)=\frac{1}{\left(2\pi\right)^{n/2}}e^{-i\xi'x},\quad\left(T_{\xi}^{-\circ}\mathcal{F}\delta_{x}\right)\left(\xi'\right)=\frac{1}{\left(2\pi\right)^{n/2}}e^{-i\left(\xi'-\xi\right)x},
\]
\[
\left(\mathcal{F}^{-1}T_{\xi}^{-\circ}\mathcal{F}\delta_{x}\right)\left(x'\right)=\frac{1}{\left(2\pi\right)^{n}}\int e^{i\xi'x}e^{-i\left(\xi'-\xi\right)x}d\xi'=e^{i\xi x}\delta_{x}\left(x'\right)
\]
Hence $\mathcal{F}^{-1}T_{\xi}^{-\circ}\mathcal{F}\delta_{x}=e^{i\xi x}\delta_{x}$.
Hence
\[
T_{x}^{-\circ}\left(\mathcal{F}^{-1}T_{\xi}^{-\circ}\mathcal{F}\right)\delta_{x'}=e^{i\xi x'}\delta_{x+x'},
\]
\[
\left(\mathcal{F}^{-1}T_{\xi}^{-\circ}\mathcal{F}\right)T_{x}^{-\circ}\delta_{x'}=e^{i\xi\left(x+x'\right)}\delta_{x+x'}=e^{i\langle\xi|x\rangle}T_{x}^{-\circ}\left(\mathcal{F}^{-1}T_{\xi}^{-\circ}\mathcal{F}\right)\delta_{x'}.
\]
Then
\begin{align*}
\hat{T}_{x,\xi}\hat{T}_{x',\xi'} & =T_{x}^{-\circ}\left(\mathcal{F}^{-1}T_{\xi}^{-\circ}\mathcal{F}\right)T_{x'}^{-\circ}\left(\mathcal{F}^{-1}T_{\xi'}^{-\circ}\mathcal{F}\right)\\
 & =e^{i\langle\xi|x'\rangle}T_{x}^{-\circ}T_{x'}^{-\circ}\left(\mathcal{F}^{-1}T_{\xi}^{-\circ}\mathcal{F}\right)\left(\mathcal{F}^{-1}T_{\xi'}^{-\circ}\mathcal{F}\right)\\
 & =e^{i\langle\xi|x'\rangle}\hat{T}_{x+x',\xi+\xi'}.
\end{align*}
\end{proof}

\subsection{\label{subsec:Bargmann-Transform}Bargmann Transform on a Euclidean
vector space $\left(E,g\right)$}

Let $E$ be a vector space with an Euclidean metric $g$. We have
the canonical identification
\[
T^{*}E=E\oplus E^{*},
\]
with the \textbf{canonical symplectic form} $\Omega$ on $E\oplus E^{*}$
given by: for $\left(x_{1},\xi_{1}\right),\left(x_{2},\xi_{2}\right)\in E\oplus E^{*}$,
\begin{equation}
\Omega\left(\left(x_{1},\xi_{1}\right),\left(x_{2},\xi_{2}\right)\right)=\langle\xi_{1}|x_{2}\rangle-\langle\xi_{2}|x_{1}\rangle.\label{eq:def_Omega}
\end{equation}
We denote $dx$ the density on $E$ induced by $g$.

\begin{cBoxA}{}
\begin{defn}
A \textbf{Gaussian wave packet in vertical gauge (V)} with parameters
$\left(x,\xi\right)\in E\oplus E^{*}$, is the function $\varphi_{x,\xi}^{\left(\mathrm{V}\right)}\in\mathcal{S}\left(E\right)$
given by: for $y\in E$,
\begin{equation}
\varphi_{x,\xi}^{\left(\mathrm{V}\right)}\left(y\right):=\pi^{-\frac{n}{4}}e^{i\langle\xi|y-x\rangle}e^{-\frac{1}{2}\left\Vert y-x\right\Vert _{g}^{2}}\label{eq:Gaussian_wave_packet}
\end{equation}
Similarly a \textbf{Gaussian wave packet in radial gauge (R) }is
\begin{equation}
\varphi_{x,\xi}^{\left(\mathrm{R}\right)}\left(y\right):=e^{\frac{i}{2}\langle\xi|x\rangle}\varphi_{x,\xi}^{\left(\mathrm{V}\right)}\left(y\right)=\pi^{-\frac{n}{4}}e^{i\langle\xi|y-\frac{1}{2}x\rangle}e^{-\frac{1}{2}\left\Vert y-x\right\Vert _{g}^{2}}\label{eq:Gaussian_wave_packet-V-R}
\end{equation}
\end{defn}

\end{cBoxA}

\begin{rem}
We have
\begin{itemize}
\item The terms 'vertical gauge' or 'radial gauge' come from a geometrical
construction of these waves packets using prequantization that we
do not describe here. But we can explain the terminology as follows.
The wave packet is seen as a map $\varphi^{\left(V\right)}:\left(x,\xi\right)\in E\oplus E^{*}\rightarrow\mathcal{H}=L^{2}\left(E\right)$
that induces a map to the projective space $\left[\varphi^{\left(V\right)}\right]:E\oplus E^{*}\rightarrow\mathbb{P}\left(\mathcal{H}\right)$
(i.e. ignoring phase) and a pull back line bundle $L\rightarrow E\oplus E^{*}$
with Hermitian connection (Levi Civita). The map $\varphi^{\left(V\right)}$
can be seen as a trivialization of this line bundle $L$. Its covariant
derivative is given by $D\varphi^{\left(V\right)}=Pd\varphi^{\left(V\right)}=\eta^{\left(V\right)}\varphi^{\left(V\right)}$
\cite[(1.5) p.540]{taylor_tome2}, with projector onto the fiber $P=\varphi^{\left(V\right)}\langle\varphi^{\left(V\right)}|.\rangle$
hence connection one-form
\[
\eta^{\left(V\right)}=\langle\varphi^{\left(V\right)}|d\varphi^{\left(V\right)}\rangle\eq{\ref{eq:Gaussian_wave_packet}}-i\xi dx.
\]
We observe that the kernel of this one-form is along $x=\mathrm{cste}$
on phase space $E\oplus E^{*}$, i.e. 'vertical'. Similarly for the
wave packet (\ref{eq:Gaussian_wave_packet-V-R}), considered as another
trivialization, we compute the connection one-form
\[
\eta^{\left(R\right)}=\langle\varphi^{\left(R\right)}|d\varphi^{\left(R\right)}\rangle\eq{\ref{eq:Gaussian_wave_packet-V-R}}\frac{i}{2}\left(-\xi dx+xd\xi\right),
\]
for which the kernel at point $\left(x,\xi\right)$ contains the 'radial'
vector $\left(x,\xi\right)$.
\item $\left\Vert \varphi_{x,\xi}^{\left(\mathrm{V}\right)}\right\Vert _{L^{2}\left(E,dx\right)}=\left\Vert \varphi_{x,\xi}^{\left(\mathrm{R}\right)}\right\Vert _{L^{2}\left(E,dx\right)}=1$
with $dx$ the density on $E$ associated to the metric $g$.
\item We have
\begin{equation}
\varphi_{x,\xi}^{\left(\mathrm{V}\right)}\eq{\ref{eq:def_T_rho}}\hat{T}_{x,\xi}\varphi_{0,0}^{\left(\mathrm{V}\right)}\label{eq:phi_T}
\end{equation}
\end{itemize}
\end{rem}

\begin{cBoxA}{}
\begin{defn}
The \textbf{Bargmann transform} in vertical gauge is the continuous
operator
\begin{equation}
\mathcal{B}_{E,\left(\mathrm{V}\right)}:\begin{cases}
\mathcal{S}\left(E\right) & \rightarrow\mathcal{S}\left(E\oplus E^{*}\right)\\
u & \rightarrow\left(\left(x,\xi\right)\rightarrow\langle\varphi_{x,\xi}^{\left(\mathrm{V}\right)}|u\rangle_{L^{2}\left(E;dx\right)}\right).
\end{cases}\label{eq:def_Bargman_B}
\end{equation}
\end{defn}

\end{cBoxA}

\begin{cBoxB}{}
\begin{lem}
\label{lem:With-respect-to}With respect to the density $\frac{1}{\left(2\pi\right)^{n}}dxd\xi$
on $E\oplus E^{*}$, the Bargmann transform $\mathcal{B}_{E,\left(\mathrm{V}\right)}:L^{2}\left(E;dx\right)\rightarrow L^{2}\left(E\oplus E^{*};\frac{1}{\left(2\pi\right)^{n}}dxd\xi\right)$
is an isometry. Its $L^{2}$-adjoint operator $\mathcal{B}_{\left(\mathrm{V}\right)}^{\dagger}:L^{2}\left(E\oplus E^{*};\frac{1}{\left(2\pi\right)^{n}}dxd\xi\right)\rightarrow L^{2}\left(E;dx\right)$
is a smoothing operator given by 
\begin{equation}
\left(\mathcal{B}_{\left(\mathrm{V}\right)}^{\dagger}v\right)\left(y\right)=\int_{E\oplus E^{*}}\varphi_{x,\xi}^{\left(\mathrm{V}\right)}\left(y\right)v\left(x,\xi\right)\frac{dxd\xi}{\left(2\pi\right)^{n}}\label{eq:B*}
\end{equation}
and satisfies the \textbf{resolution of identity}
\begin{equation}
\mathrm{Id}_{L^{2}\left(E\right)}=\mathcal{B}_{\left(\mathrm{V}\right)}^{\dagger}\mathcal{B}_{\left(\mathrm{V}\right)}=\int_{E\oplus E^{*}}\pi\left(x,\xi\right)\frac{dxd\xi}{\left(2\pi\right)^{n}}\label{eq:resol_ident_bargamnn}
\end{equation}
with $\pi\left(x,\xi\right):=\varphi_{x,\xi}^{\left(\mathrm{V}\right)}\langle\varphi_{x,\xi}^{\left(\mathrm{V}\right)}|.\rangle$
that is the rank one orthogonal projector in $L^{2}\left(E\right)$
onto $\mathbb{C}\varphi_{x,\xi}^{\left(\mathrm{V}\right)}$.
\end{lem}

\end{cBoxB}

\begin{proof}
We first check (\ref{eq:B*}). For $u\in\mathcal{S}\left(E\right)$
and $v\in\mathcal{S}\left(E\oplus E^{*}\right)$ we have
\begin{align*}
\langle u|\mathcal{B}_{\left(\mathrm{V}\right)}^{\dagger}v\rangle_{L^{2}} & =\langle\mathcal{B}_{\left(\mathrm{V}\right)}u|v\rangle_{L^{2}}=\int_{E\oplus E^{*}}\overline{\langle\varphi_{x,\xi}^{\left(\mathrm{V}\right)}|u\rangle_{L^{2}}}v\left(x,\xi\right)\frac{dxd\xi}{\left(2\pi\right)^{n}}\\
 & =\int_{E}\overline{u\left(y\right)}\int_{E\oplus E^{*}}\varphi_{x,\xi}^{\left(\mathrm{V}\right)}\left(y\right)v\left(x,\xi\right)\frac{dxd\xi}{\left(2\pi\right)^{n}}dy.
\end{align*}
Now we check that $\left\Vert \mathcal{B}_{\left(\mathrm{V}\right)}u\right\Vert ^{2}=\left\Vert u\right\Vert ^{2}$
that is equivalent to (\ref{eq:resol_ident_bargamnn}). We have
\begin{align*}
\left\Vert \mathcal{B}_{\left(\mathrm{V}\right)}u\right\Vert ^{2} & =\int_{E\oplus E^{*}}\left|\langle\varphi_{x,\xi}^{\left(\mathrm{V}\right)}|u\rangle_{L^{2}}\right|^{2}\frac{dxd\xi}{\left(2\pi\right)^{n}}\\
 & =\pi^{-\frac{n}{2}}\int e^{-i\langle\xi|y-x\rangle}e^{-\frac{1}{2}\left\Vert y-x\right\Vert _{g}^{2}}u\left(y\right)dye^{i\langle\xi|y'-x\rangle}e^{-\frac{1}{2}\left\Vert y'-x\right\Vert _{g}^{2}}\overline{u\left(y'\right)}dy'\frac{dxd\xi}{\left(2\pi\right)^{n}}
\end{align*}
But
\[
\int e^{i\langle\xi|y'-y\rangle}d\xi=\left(2\pi\right)^{n}\delta\left(y'-y\right)
\]
\[
\int_{E}e^{-\left\Vert y\right\Vert ^{2}}dy=\pi^{\frac{n}{2}}
\]
hence
\[
\left\Vert \mathcal{B}_{\left(\mathrm{V}\right)}u\right\Vert ^{2}=\pi^{-\frac{n}{2}}\int e^{-\left\Vert y-x\right\Vert _{g}^{2}}\left|u\left(y\right)\right|^{2}dydx=\int\left|u\left(y\right)\right|^{2}dy=\left\Vert u\right\Vert ^{2}.
\]
\end{proof}
By using $\varphi_{x,\xi}^{\left(\mathrm{R}\right)}$ in radial gauge
instead of $\varphi_{x,\xi}^{\left(\mathrm{V}\right)}$, we can define
similarly operators $\mathcal{B}_{\left(\mathrm{R}\right)},\mathcal{B}_{\left(\mathrm{R}\right)}^{\dagger}$. 

Let
\begin{equation}
\mathcal{P}_{\left(\mathrm{V}\right)}:=\mathcal{B}_{\left(\mathrm{V}\right)}\mathcal{B}_{\left(\mathrm{V}\right)}^{\dagger}\quad:\mathcal{S}\left(E\oplus E^{*}\right)\rightarrow\mathcal{S}\left(E\oplus E^{*}\right)\label{eq:Bergman_projector-1}
\end{equation}
\begin{equation}
\mathcal{P}_{\left(\mathrm{R}\right)}:=\mathcal{B}_{\left(\mathrm{R}\right)}\mathcal{B}_{\left(\mathrm{R}\right)}^{\dagger}\quad:\mathcal{S}\left(E\oplus E^{*}\right)\rightarrow\mathcal{S}\left(E\oplus E^{*}\right)\label{eq:Bergman_projector}
\end{equation}
be the orthogonal projector in $L^{2}\left(E\oplus E^{*}\right)$
onto their closed images $\mathrm{Im}\left(\mathcal{B}_{\left(\mathrm{V}\right)}\right),\mathrm{Im}\left(\mathcal{B}_{\left(\mathrm{R}\right)}\right)$
respectively, called the \textbf{Bergman projector}. Remark that Lemma
\ref{lem:With-respect-to} holds true for $\mathcal{B}_{\left(\mathrm{R}\right)}$
as well.

The metric $g$ on $E$ induces a canonical metric on $E\oplus E^{*}$
denoted $\boldsymbol{g}:=g\oplus g^{-1}$.

\begin{cBoxB}{}
\begin{lem}
The Schwartz kernel of $\mathcal{P}_{\left(\mathrm{V}\right)},\mathcal{P}_{\left(\mathrm{R}\right)}$
are given by

\begin{equation}
\langle\delta_{x',\xi'}|\mathcal{P}_{\left(\mathrm{V}\right)}\delta_{x,\xi}\rangle=\exp\left(\frac{i}{2}\left(\xi'+\xi\right)\left(x'-x\right)-\frac{1}{4}\left\Vert \left(x',\xi'\right)-\left(x,\xi\right)\right\Vert _{g\oplus g^{-1}}^{2}\right)\label{eq:Kernel_of_P_vertical_Gauge}
\end{equation}
\begin{equation}
\langle\delta_{x',\xi'}|\mathcal{P}_{\left(\mathrm{R}\right)}\delta_{x,\xi}\rangle=\exp\left(-\frac{i}{2}\Omega\left(\left(x',\xi'\right),\left(x,\xi\right)\right)-\frac{1}{4}\left\Vert \left(x',\xi'\right)-\left(x,\xi\right)\right\Vert _{g\oplus g^{-1}}^{2}\right)\label{eq:Kernel_of_P_radial_Gauge}
\end{equation}
\end{lem}

\end{cBoxB}

\begin{proof}
We have
\begin{align*}
\langle\delta_{x',\xi'}|\mathcal{P}_{\left(\mathrm{V}\right)}\delta_{x,\xi}\rangle & =\langle\varphi_{x',\xi'}^{\left(\mathrm{V}\right)}|\varphi_{x,\xi}^{\left(\mathrm{V}\right)}\rangle=\pi^{-\frac{n}{2}}\int e^{-i\langle\xi'|y-x'\rangle}e^{-\frac{1}{2}\left\Vert y-x'\right\Vert _{g}^{2}}e^{i\langle\xi|y-x\rangle}e^{-\frac{1}{2}\left\Vert y-x\right\Vert _{g}^{2}}dy\\
 & =\pi^{-\frac{n}{2}}e^{i\left(\langle\xi'|x'\rangle-\langle\xi|x\rangle\right)-\frac{1}{2}\left\Vert x'\right\Vert _{g}^{2}-\frac{1}{2}\left\Vert x\right\Vert _{g}^{2}}\int e^{-\left\Vert y\right\Vert _{g}^{2}+\langle b|y\rangle}dy
\end{align*}
with $b=\check{g}\left(x'+x\right)+i\left(\xi-\xi'\right)\in E^{*}$.
We use the Gaussian integral
\begin{equation}
\int_{E}e^{-\alpha\left\Vert y\right\Vert _{g}^{2}+\langle b|y\rangle}dy=\left(\frac{\pi}{\alpha}\right)^{\frac{n}{2}}e^{\frac{1}{4\alpha}\left\Vert b\right\Vert _{g^{-1}}^{2}},\label{eq:Gaussian_integral}
\end{equation}
with $\alpha=1$, giving
\begin{align*}
\langle\delta_{x',\xi'}|\mathcal{P}_{\left(\mathrm{V}\right)}\delta_{x,\xi}\rangle & =e^{i\left(\langle\xi'|x'\rangle-\langle\xi|x\rangle\right)-\frac{1}{2}\left\Vert x'\right\Vert ^{2}-\frac{1}{2}\left\Vert x\right\Vert ^{2}}e^{\frac{1}{4}\left(\left\Vert x'+x\right\Vert ^{2}-\left\Vert \xi-\xi'\right\Vert ^{2}+2i\langle\xi-\xi'|x'+x\rangle\right)}\\
 & =e^{i\left(\langle\xi'|x'\rangle-\langle\xi|x\rangle+\frac{1}{2}\langle\xi-\xi'|x'+x\rangle\right)}e^{-\frac{1}{4}\left(\left\Vert x'-x\right\Vert ^{2}+\left\Vert \xi-\xi'\right\Vert ^{2}\right)}\\
 & =e^{i\frac{1}{2}\langle\xi'+\xi|x'-x\rangle}e^{-\frac{1}{4}\left\Vert \left(x',\xi'\right)-\left(x,\xi\right)\right\Vert _{g+g^{-1}}^{2}}
\end{align*}
We have
\begin{align*}
\langle\delta_{x',\xi'}|\mathcal{P}_{\left(\mathrm{R}\right)}\delta_{x,\xi}\rangle & =e^{i\frac{1}{2}\left(\langle\xi|x\rangle-\langle\xi'|x'\rangle\right)}\langle\delta_{x',\xi'}|\mathcal{P}_{\left(\mathrm{V}\right)}\delta_{x,\xi}\rangle\\
 & =e^{i\frac{1}{2}\left(-\langle\xi'|x\rangle+\langle\xi|x'\rangle\right)}e^{-\frac{1}{4}\left\Vert \left(x',\xi'\right)-\left(x,\xi\right)\right\Vert _{g+g^{-1}}^{2}}\\
 & =e^{-i\frac{1}{2}\Omega\left(\left(x',\xi'\right),\left(x,\xi\right)\right)}e^{-\frac{1}{4}\left\Vert \left(x',\xi'\right)-\left(x,\xi\right)\right\Vert _{g+g^{-1}}^{2}}
\end{align*}
\end{proof}
\noindent\fcolorbox{blue}{white}{\begin{minipage}[t]{1\columnwidth - 2\fboxsep - 2\fboxrule}%
\begin{lem}
Let $\pi:E\oplus E^{*}\rightarrow E$ denotes the projector on the
first component, and $\pi^{\circ}:\mathcal{S}\left(E\right)\rightarrow\mathcal{S}\left(E\oplus E^{*}\right)$
the pull back operator. We have
\[
\mathcal{B}_{\left(\mathrm{V}\right)}=\pi^{\frac{n}{4}}\mathcal{P}_{\left(\mathrm{V}\right)}\pi^{\circ}
\]
\end{lem}

\end{minipage}}
\begin{proof}
We compute the Schwartz kernel of both sides for $y\in E,\left(x',\xi'\right)\in E\oplus E^{*}$.
We have $\langle\delta_{x',\xi'}|\mathcal{B}_{\left(\mathrm{V}\right)}\delta_{y}\rangle=\overline{\varphi_{x',\xi'}^{\left(\mathrm{V}\right)}\left(y\right)}$
and
\begin{align*}
\langle\delta_{x',\xi'}|\mathcal{P}_{\left(\mathrm{V}\right)}\pi^{\circ}\delta_{y}\rangle & =\int e^{\frac{i}{2}\left(\xi'+\xi\right)\left(x'-x\right)-\frac{1}{4}\left(\left\Vert x'-x\right\Vert ^{2}+\left\Vert \xi-\xi'\right\Vert ^{2}\right)}\delta\left(x-y\right)\frac{dxd\xi}{\left(2\pi\right)^{n}}\\
 & =\left(2\pi\right)^{-n}e^{\frac{i}{2}\xi'\left(x'-y\right)-\frac{1}{4}\left\Vert x'-y\right\Vert ^{2}-\frac{1}{4}\left\Vert \xi'\right\Vert ^{2}}\int e^{\frac{i}{2}\xi\left(x'-y\right)+\frac{1}{2}\xi.\xi'-\frac{1}{4}\left\Vert \xi\right\Vert ^{2}}d\xi
\end{align*}
We use the Gaussian integral (\ref{eq:Gaussian_integral}) with $\alpha=\frac{1}{4}$,
$b=\frac{i}{2}\left(x'-y\right)+\frac{1}{2}\xi'$, 
\[
\frac{1}{4\alpha}\left\Vert b\right\Vert ^{2}=\frac{1}{4}\left(\left\Vert \xi'\right\Vert ^{2}-\left\Vert x'-y\right\Vert ^{2}\right)+\frac{i}{2}\left(x'-y\right)\xi'
\]
\begin{align*}
\langle\delta_{x',\xi'}|\mathcal{P}_{\left(\mathrm{V}\right)}\pi^{\circ}\delta_{y}\rangle & =\left(2\pi\right)^{-n}e^{\frac{i}{2}\xi'\left(x'-y\right)-\frac{1}{4}\left\Vert x'-y\right\Vert ^{2}-\frac{1}{4}\left\Vert \xi'\right\Vert ^{2}}\left(\frac{\pi}{\alpha}\right)^{\frac{n}{2}}e^{\frac{1}{4\alpha}\left\Vert b\right\Vert ^{2}}\\
 & =\left(4\pi\right)^{\frac{n}{2}}\left(2\pi\right)^{-n}\pi^{\frac{n}{4}}\pi^{-\frac{n}{4}}e^{i\xi'\left(x'-y\right)-\frac{1}{2}\left\Vert x'-y\right\Vert ^{2}}\\
 & =\pi^{-\frac{n}{4}}\overline{\varphi_{x',\xi'}^{\left(\mathrm{V}\right)}\left(y\right)}
\end{align*}
\end{proof}

\subsection{Linear map $\phi:\left(E_{2},g_{2}\right)\rightarrow\left(E_{1},g_{1}\right)$}

Suppose $\phi:\left(E_{2},g_{2}\right)\rightarrow\left(E_{1},g_{1}\right)$
is a linear invertible map between two metric spaces. Let 
\[
\Phi:=\phi^{-1}\oplus\phi^{*}:E_{1}\oplus E_{1}^{*}\rightarrow E_{2}\oplus E_{2}^{*}
\]
be the induced map (i.e. pull back of differential forms $\Phi:T^{*}E_{1}\rightarrow T^{*}E_{2}$).
We have defined in (\ref{eq:def_Bargman_B}), $\mathcal{B}_{E_{1}}:\mathcal{S}\left(E_{1}\right)\rightarrow\mathcal{S}\left(E_{1}\oplus E_{1}^{*}\right)$
and $\mathcal{B}_{E_{2}}:\mathcal{S}\left(E_{2}\right)\rightarrow\mathcal{S}\left(E_{2}\oplus E_{2}^{*}\right)$
where $\mathcal{B}_{E_{j}}$ is either $\mathcal{B}_{\left(\mathrm{V}\right)}$
or $\mathcal{B}_{\left(R\right)}$.

\begin{cBoxB}{}
\begin{lem}
\label{lem:Let--be}Suppose $\phi:\left(E_{2},g_{2}\right)\rightarrow\left(E_{1},g_{1}\right)$
is a linear invertible map between two Euclidean vector spaces. We
have that 
\begin{equation}
\phi^{\circ}=\Upsilon\left(\phi\right)\mathcal{B}_{E_{2}}^{\dagger}\Phi^{-\circ}\mathcal{B}_{E_{1}}\label{eq:def_Upsilon}
\end{equation}
with
\begin{equation}
\Upsilon\left(\phi\right):=\left(\mathrm{det}\left(\frac{1}{2}\left(\mathrm{Id}+\left(\phi^{-1}\right)^{\dagger}\phi^{-1}\right)\right)\right)^{1/2}\label{eq:def_d}
\end{equation}
where $\left(\phi^{-1}\right)^{\dagger}$ is the metric-adjoint of
$\phi^{-1}$ defined by $\langle x_{1}|\left(\phi^{-1}\right)^{\dagger}x_{2}\rangle_{g_{1}}=\langle\phi^{-1}x_{1}|x_{2}\rangle_{g_{2}}$
for any $x_{1}\in E_{1}$, $x_{2}\in E_{2}$.
\end{lem}

\end{cBoxB}

\begin{rem}
If $\phi$ is an isometry then $\phi^{\dagger}\phi=\mathrm{Id}$ and
$\Upsilon\left(\phi\right)=1$.
\end{rem}

\begin{proof}
We compute the Schwartz kernels. For $y'\in E_{1},z'\in E_{2}$,
\[
\langle\delta_{z'}|\mathcal{B}_{E_{2}}^{\dagger}\Phi^{-\circ}\mathcal{B}_{E_{1}}\delta_{y'}\rangle=\langle\mathcal{B}_{E_{2}}\delta_{z'}|\Phi^{-\circ}\mathcal{B}_{E_{1}}\delta_{y'}\rangle
\]
For $y\in E_{1},\eta\in E_{1}^{*}$,
\[
\left(\mathcal{B}_{E_{1}}\delta_{y'}\right)\left(y,\eta\right)=\langle\varphi_{y,\eta}^{\left(1\right)}|\delta_{y'}\rangle=\pi^{-\frac{n}{4}}e^{-i\langle\eta|y'-y\rangle}e^{-\frac{1}{2}\left\Vert y'-y\right\Vert _{g_{1}}^{2}}
\]
For $z\in E_{2},\xi\in E_{2}^{*}$,
\[
\left(\Phi^{-\circ}\mathcal{B}_{E_{1}}\delta_{y'}\right)\left(z,\xi\right)=\pi^{-\frac{n}{4}}e^{-i\langle\phi^{*-1}\xi|y'-\phi z\rangle}e^{-\frac{1}{2}\left\Vert y'-\phi z\right\Vert _{g_{1}}^{2}}
\]
\[
\left(\mathcal{B}_{E_{2}}\delta_{z'}\right)\left(z,\xi\right)=\pi^{-\frac{n}{4}}e^{-i\langle\xi|z'-z\rangle}e^{-\frac{1}{2}\left\Vert z'-z\right\Vert _{g_{2}}^{2}}
\]
Hence
\[
\langle\mathcal{B}_{E_{2}}\delta_{z'}|\Phi^{-\circ}\mathcal{B}_{E_{1}}\delta_{y'}\rangle=\pi^{-\frac{n}{2}}\int e^{i\langle\xi|z'-z\rangle}e^{-\frac{1}{2}\left\Vert z'-z\right\Vert _{g_{2}}^{2}}e^{-i\langle\phi^{*-1}\xi|y'-\phi z\rangle}e^{-\frac{1}{2}\left\Vert y'-\phi z\right\Vert _{g_{1}}^{2}}\frac{dzd\xi}{\left(2\pi\right)^{n}}
\]
We do the symplectic change of variables $\left(z,\xi\right)\in E_{2}\oplus E_{2}^{*}$$\rightarrow$$\left(y,\eta\right)=\Phi^{-1}\left(z,\xi\right)=\left(\phi\left(z\right),\phi^{*^{-1}}\xi\right)\in\left(E_{1}\oplus E_{1}^{*}\right)$,
i.e. $\xi=\phi^{*}\eta$, $z=\phi^{-1}y$, and get
\begin{align*}
\langle\delta_{z'}|\mathcal{B}_{E_{2}}^{\dagger}\Phi^{-\circ}\mathcal{B}_{E_{1}}\delta_{y'}\rangle & =\pi^{-\frac{n}{2}}\int e^{i\langle\phi^{*}\eta|z'-\phi^{-1}y\rangle}e^{-\frac{1}{2}\left\Vert z'-\phi^{-1}y\right\Vert _{g_{2}}^{2}}e^{-i\langle\eta|y'-y\rangle}e^{-\frac{1}{2}\left\Vert y'-y\right\Vert _{g_{1}}^{2}}\frac{dyd\eta}{\left(2\pi\right)^{n}}\\
 & =\pi^{-\frac{n}{2}}\int e^{i\langle\eta|\phi z'-y'\rangle}e^{-\frac{1}{2}\left\Vert \phi^{-1}\left(\phi z'-y\right)\right\Vert _{g_{2}}^{2}}e^{-\frac{1}{2}\left\Vert y'-y\right\Vert _{g_{1}}^{2}}\frac{dyd\eta}{\left(2\pi\right)^{n}}
\end{align*}
We have
\[
\int e^{i\langle\eta|\phi z'-y'\rangle}\frac{d\eta}{\left(2\pi\right)^{n}}=\delta\left(\phi z'-y'\right)=\langle\delta_{z'}|\phi^{\circ}\delta_{y'}\rangle
\]
hence
\begin{align*}
\langle\delta_{z'}|\mathcal{B}_{E_{2}}^{\dagger}\Phi^{-\circ}\mathcal{B}_{E_{1}}\delta_{y'}\rangle & =\langle\delta_{z'}|\phi^{\circ}\delta_{y'}\rangle\pi^{-\frac{n}{2}}\int e^{-\frac{1}{2}\left\Vert \phi^{-1}\left(y'-y\right)\right\Vert _{g_{2}}^{2}}e^{-\frac{1}{2}\left\Vert y'-y\right\Vert _{g_{1}}^{2}}dy\\
 & =\langle\delta_{z'}|\phi^{\circ}\delta_{y'}\rangle\pi^{-\frac{n}{2}}\int_{E}e^{-\frac{1}{2}\left\Vert \phi^{-1}Y\right\Vert _{g_{2}}^{2}}e^{-\frac{1}{2}\left\Vert Y\right\Vert _{g_{1}}^{2}}dY\\
 & =\langle\delta_{z'}|\phi^{\circ}\delta_{y'}\rangle\pi^{-\frac{n}{2}}\int e^{-\frac{1}{2}\langle Y|\left(\mathrm{Id}+\left(\phi^{-1}\right)^{\dagger}\phi^{-1}\right)Y\rangle}dY\\
 & =\langle\delta_{z'}|\phi^{\circ}\delta_{y'}\rangle\pi^{-\frac{n}{2}}\left(\frac{\left(2\pi\right)^{n}}{\mathrm{det}\left(\mathrm{Id}+\left(\phi^{-1}\right)^{\dagger}\phi^{-1}\right)}\right)^{1/2}\\
 & =\langle\delta_{z'}|\phi^{\circ}\delta_{y'}\rangle\Upsilon\left(\phi\right)^{-1}
\end{align*}
Giving (\ref{eq:def_Upsilon}).
\end{proof}
Later we will use the following Lemma, where the determinant is measured
with respect to local densities.

\begin{cBoxB}{}
\begin{lem}
With $\Upsilon\left(\phi\right)$ defined in (\ref{eq:def_d}), we
have
\begin{equation}
\Upsilon\left(\phi^{-1}\right)=\left|\mathrm{det}\phi\right|\Upsilon\left(\phi\right).\label{eq:dphi_inv}
\end{equation}
If $\phi=\phi_{1}\oplus\phi_{2}$ on $E=E_{1}\overset{\perp}{\oplus}E_{2}$,
then
\begin{equation}
\Upsilon\left(\phi\right)=\Upsilon\left(\phi_{1}\right)\Upsilon\left(\phi_{2}\right).\label{eq:product}
\end{equation}
\end{lem}

\end{cBoxB}

\begin{proof}
Notice that $\mathrm{det}\left(\phi^{\dagger}\phi\right)=\left|\mathrm{det}\phi\right|^{2}$.
Then
\begin{align*}
\Upsilon\left(\phi\right)^{2} & =\mathrm{det}\left(\frac{1}{2}\left(\mathrm{Id}+\phi^{-1\dagger}\phi^{-1}\right)\right)=\left(\mathrm{det}\left(\phi^{\dagger}\phi\right)\right)^{-1}\mathrm{det}\left(\frac{1}{2}\left(\phi^{\dagger}\phi+\mathrm{Id}\right)\right)=\left|\mathrm{det}\phi\right|^{-2}\Upsilon\left(\phi^{-1}\right)^{2}
\end{align*}
\end{proof}
\begin{cBoxB}{}
\begin{cor}
If $E$ is endowed with two metrics $g_{1},g_{2}$ then
\[
\mathrm{Id}_{\mathcal{S}\left(E\right)}=\Upsilon_{g_{2},g_{1}}\mathcal{B}_{g_{2}}^{\dagger_{g_{2}}}\mathcal{B}_{g_{1}}
\]
with
\begin{equation}
\Upsilon_{g_{2},g_{1}}=\left(\mathrm{det}\left(\frac{1}{2}\left(\mathrm{Id}+\mathrm{Id}^{\dagger}\right)\right)\right)^{1/2}=\prod_{j=1}^{n}\left(\frac{1+I_{j}}{2}\right)^{1/2}\label{eq:d_g1_g2}
\end{equation}
$\left(I_{j}\right)_{j}\in\mathbb{R}^{n}$ are inertial moments of
$g_{2}$ with respect to $g_{1}$ i.e. in suitable coordinates, $g_{1}=\sum_{j}dy_{j}^{2}$,
$g_{2}=\sum_{j}I_{j}dy_{j}^{2}$. Consequently
\[
\mathcal{P}_{g_{1},g_{2}}:=\Upsilon_{g_{2},g_{1}}\mathcal{B}_{g_{1}}\mathcal{B}_{g_{2}}^{\dagger_{g_{2}}}\quad:\mathcal{S}\left(E\oplus E^{*}\right)\rightarrow\mathcal{S}\left(E\oplus E^{*}\right)
\]
is a projector, with $\mathrm{Im}\left(\mathcal{P}_{g_{1},g_{2}}\right)=\mathrm{Im}\left(\mathcal{B}_{g_{1}}\right)$,
$\mathrm{Ker}\left(\mathcal{P}_{g_{1},g_{2}}\right)=\left(\mathrm{Im}\left(\mathcal{B}_{g_{2}}\right)\right)^{\perp}$.
\end{cor}

\end{cBoxB}

\begin{proof}
We apply Lemma \ref{lem:Let--be} with $\phi=\mathrm{Id}:\left(E,g_{2}\right)\rightarrow\left(E,g_{1}\right)$
and $\Phi=\mathrm{Id}$. We have
\[
\langle u|\left(\phi^{-1}\right)^{\dagger}v\rangle_{g_{1}}=\langle\phi^{-1}u|v\rangle_{g_{2}}=\langle u|v\rangle_{g_{2}}
\]
In coordinates such that $g_{1}=\sum_{j}dy_{j}^{2}$, $g_{2}=\sum_{j}I_{j}dy_{j}^{2}$,
\[
\left(\mathrm{Id}\right)_{j,k}^{\dagger}=\left(\phi^{-1}\right)_{j,k}^{\dagger}=I_{j}\delta_{j=k}
\]
giving (\ref{eq:d_g1_g2}).
\end{proof}

\subsection{\label{subsec:Compatible-triple}Compatible triple $g,\Omega,J$}

Let us recall the definition of a \href{https://en.wikipedia.org/wiki/Almost_complex_manifold\#Compatible_triples}{compatible triple}
$g,\Omega,J$. Let $F$ be a vector space endowed with a symplectic
bilinear form $\Omega$ and an Euclidean metric $g$. Let 
\begin{equation}
J:=\check{g}^{-1}\check{\Omega}\quad:F\rightarrow F\label{eq:def_J}
\end{equation}
 with $\check{\Omega}:F\rightarrow F^{*}$, $\check{g}:F\rightarrow F^{*}$
defined in Section \ref{sec:General-notations-used}.

\begin{cBoxA}{}
\begin{defn}
\label{def:Compatible}We say that $\Omega,g$ are compatible structures
on $F$ if $J:=\check{g}^{-1}\check{\Omega}:F\rightarrow F$ is an
\href{https://en.wikipedia.org/wiki/Almost_complex_manifold}{almost complex structure}
on $F$, i.e. 
\begin{equation}
J^{2}=-\mathrm{Id}\label{eq:J2}
\end{equation}
.
\end{defn}

\end{cBoxA}

\begin{cBoxB}{}
\begin{prop}
From (\ref{eq:def_J}), we have that for any $u,v\in F$
\begin{align}
\Omega\left(u,v\right)=g\left(Ju,v\right).\label{eq:g_Omega}
\end{align}
If moreover (\ref{eq:J2}) holds true, we have that for any $u,v\in F$,
\begin{equation}
g\left(Ju,Jv\right)=g\left(u,v\right),\qquad\Omega\left(Ju,Jv\right)=\Omega\left(u,v\right).\label{eq:J_comp}
\end{equation}
\end{prop}

\end{cBoxB}

\begin{proof}
We have
\[
g\left(Ju,v\right)\eq{\ref{eq:def_J}}\langle\check{g}\check{g}^{-1}\check{\Omega}u|v\rangle=\Omega\left(u,v\right).
\]
 Suppose $J^{2}=\check{g}^{-1}\check{\Omega}\check{g}^{-1}\check{\Omega}=-\mathrm{Id}$
hence $\check{g}=-\check{\Omega}\check{g}^{-1}\check{\Omega}$. Then
\[
g\left(Ju,Jv\right)=\langle\check{g}\check{g}^{-1}\check{\Omega}u|\check{g}^{-1}\check{\Omega}v\rangle=\langle-\check{\Omega}\check{g}^{-1}\check{\Omega}u|v\rangle=g\left(u,v\right)
\]
\[
\Omega\left(Ju,Jv\right)=\langle\check{\Omega}\check{g}^{-1}\check{\Omega}u|\check{g}^{-1}\check{\Omega}v\rangle=\langle-\check{\Omega}\check{g}^{-1}\check{\Omega}\check{g}^{-1}\check{\Omega}u|v\rangle=\Omega\left(u,v\right)
\]
\end{proof}
\begin{example}
Let $\left(E,g\right)$ be a vector space with Euclidean metric $g$.
On $E\oplus E^{*}$, the canonical symplectic form $\Omega$ in (\ref{eq:def_Omega}),
is compatible with the induced metric $\boldsymbol{g}:=g\oplus g^{-1}$.
\end{example}

\begin{cBoxB}{}
\begin{lem}
\label{lem:If--is}Let $\left(F,\Omega,g\right)$ be a vector space
with a symplectic structure $\Omega$ and compatible metric $g$.
If $E\subset F$ is a linear Lagrangian subspace, let $E^{\perp_{g}}$
be the orthogonal subspace with respect to $g$. Then $E^{\perp_{g}}=J\left(E\right)$
is also Lagrangian and $E=J\left(E^{\perp_{g}}\right)$. Hence we
have an isomorphism $\varPsi$ that is symplectic (for $\Omega$)
and isometric (for $g$)
\begin{equation}
\varPsi:\begin{cases}
F=E\oplus E^{\perp_{g}} & \rightarrow E\oplus E^{*}\\
\left(v,u\right) & \mapsto\left(v,\check{\Omega}u\right)
\end{cases}\label{eq:identif_F_E}
\end{equation}
where $E\oplus E^{*}$ is endowed with the canonical symplectic form
(\ref{eq:def_Omega}) and with the metric $\boldsymbol{g}=g\oplus g^{-1}$.
\end{lem}

\end{cBoxB}

\begin{proof}
Let $v_{1},v_{2}\in E$ Lagrangian. Then
\[
g\left(Jv_{1},v_{2}\right)\eq{\ref{eq:g_Omega}}\Omega\left(Jv_{1},Jv_{2}\right)\eq{\ref{eq:J_comp}}\Omega\left(v_{1},v_{2}\right)\eq{\mathrm{Lag.}}0.
\]
Considering dimensions, $\mathrm{dim}E=\mathrm{dim}E^{\perp_{g}}=\frac{1}{2}\mathrm{dim}F$,
we deduce $E^{\perp_{g}}=J\left(E\right)$ and $E=J\left(E^{\perp_{g}}\right)$,
since $J^{2}=-\mathrm{Id}$. From compatibility $J$ is a symplectic
map hence $E^{\perp_{g}}$ is Lagrangian as well.
\end{proof}

\subsection{Bergman projector on a symplectic vector space $\left(F,\Omega,g\right)$}

Let $\left(F,\Omega,g\right)$ be a vector space with a symplectic
structure $\Omega$ and compatible metric $g$.

\begin{cBoxA}{}
\begin{defn}
Let $\mathcal{P}_{F}:\mathcal{S}\left(F\right)\rightarrow\mathcal{S}\left(F\right)$
be the operator defined by its Schwartz kernel: for any $\rho,\rho'\in F$,
\begin{equation}
\langle\delta_{\rho'}|\mathcal{P}_{F}\delta_{\rho}\rangle=\exp\left(-\frac{i}{2}\Omega\left(\rho',\rho\right)-\frac{1}{4}\left\Vert \rho'-\rho\right\Vert _{g}^{2}\right).\label{eq:def_P}
\end{equation}
\end{defn}

\end{cBoxA}

\begin{cBoxB}{}
\begin{lem}
$\mathcal{P}_{F}$ is an orthogonal projector in $L^{2}\left(F,\frac{d\rho}{\left(2\pi\right)^{n}}\right)$,
called the \textbf{Bergman projector}.
\end{lem}

\end{cBoxB}

\begin{proof}
We choose any $E\subset F$ Lagrangian linear subspace. Then $\mathcal{P}_{F}\eq{\ref{eq:identif_F_E}}\Psi^{\circ}\circ\mathcal{P}_{E\oplus E^{*},\left(\mathrm{R}\right)}\circ\Psi^{-\circ}$
is conjugated to the Bergman projector in radial gauge defined in
(\ref{eq:Kernel_of_P_radial_Gauge}).
\end{proof}

\subsection{The metaplectic decomposition of $F\oplus F^{*}$}

Let $\left(F,\Omega\right)$ be a linear symplectic space. Let us
denote $\boldsymbol{\Omega}$ the canonical symplectic form on $F\oplus F^{*}$
as in (\ref{eq:def_Omega}). Recall $\check{\Omega}:F\rightarrow F^{*}$
defined in Section \ref{sec:General-notations-used} and that $\check{\Omega}^{*}=-\check{\Omega}$.

\begin{cBoxB}{}
\begin{lem}[Metaplectic decomposition]
\label{lem:def_K_N} Let $\left(F,\Omega\right)$ be a linear symplectic
space and 
\begin{equation}
K:=\mathrm{graph}\left(\check{\Omega}\right),\quad N:=\mathrm{graph}\left(-\check{\Omega}\right).\label{eq:def_K_N}
\end{equation}
\end{lem}

We have 
\begin{equation}
F\oplus F^{*}=K\overset{\perp_{\boldsymbol{\Omega}}}{\oplus}N.\label{eq:FF*KN}
\end{equation}
where the right hand side is an \textbf{$\boldsymbol{\Omega}$}-orthogonal
decomposition into linear symplectic subspaces. See Figure \ref{fig:symplectic_decomp}.
Explicitly for any $\left(x,\xi\right)\in F\oplus F^{*}$, $\left(x,\xi\right)=\left(\nu,\check{\Omega}\left(\nu\right)\right)+\left(\zeta,-\check{\Omega}\left(\zeta\right)\right)\in K\oplus N$
is given by
\begin{align}
\nu & =\frac{1}{2}\left(x+\check{\Omega}^{-1}\left(\xi\right)\right)\label{eq:nu_zeta}\\
\zeta & =\frac{1}{2}\left(x-\check{\Omega}^{-1}\left(\xi\right)\right)
\end{align}
\end{cBoxB}

\begin{figure}
\centering{}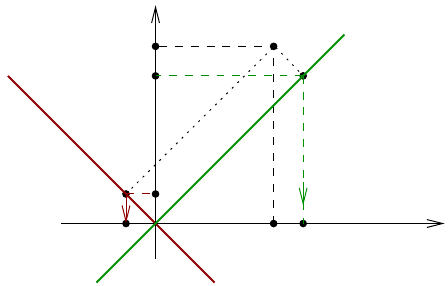\caption{\label{fig:symplectic_decomp} Picture for the orthogonal decomposition
$F\oplus F^{*}=K\protect\overset{\perp}{\oplus}N$ in (\ref{eq:FF*KN}).}
\end{figure}

\begin{rem}
The metaplectic decomposition (\ref{eq:FF*KN}) is central in this
paper. It has also been used in \cite[Prop.6]{fred-PreQ-06} and in
\cite[Prop.2.2.9]{faure-tsujii_prequantum_maps_12}.
\end{rem}

\begin{proof}
If $\left(x,\xi\right)=\left(x,\check{\Omega}\left(x\right)\right)\in\mathrm{graph}\left(\check{\Omega}\right)$
and $\left(x',\xi'\right)=\left(x',-\check{\Omega}\left(x'\right)\right)\in\mathrm{graph}\left(-\check{\Omega}\right)$
then 
\begin{align*}
\boldsymbol{\Omega}\left(\left(x,\xi\right),\left(x',\xi'\right)\right) & =\langle\xi|x'\rangle-\langle x|\xi'\rangle=\langle\check{\Omega}\left(x\right)|x'\rangle+\langle x|\check{\Omega}\left(x'\right)\rangle\\
 & =\Omega\left(x,x'\right)+\Omega\left(x',x\right)=0
\end{align*}
Moreover if $\left(x,\xi\right)=\left(x,\check{\Omega}\left(x\right)\right)=\left(x,-\check{\Omega}\left(x\right)\right)\in\mathrm{graph}\left(\check{\Omega}\right)\cap\mathrm{graph}\left(-\check{\Omega}\right)$
then $\check{\Omega}\left(x\right)=-\check{\Omega}\left(x\right)$
hence $\check{\Omega}\left(x\right)=0$ hence $x=0$. This orthogonality
and transversality implies that $\mathrm{graph}\left(\check{\Omega}\right),\mathrm{graph}\left(-\check{\Omega}\right)$
are symplectic. Let us prove (\ref{eq:nu_zeta}). We write 
\begin{equation}
\begin{cases}
x & =\nu+\zeta\\
\xi & =\check{\Omega}\left(\nu\right)-\check{\Omega}\left(\zeta\right)=\check{\Omega}\left(\nu-\zeta\right)
\end{cases}\Leftrightarrow\begin{cases}
\nu & =\frac{1}{2}\left(x+\check{\Omega}^{-1}\left(\xi\right)\right)\\
\zeta & =\frac{1}{2}\left(x-\check{\Omega}^{-1}\left(\xi\right)\right)
\end{cases}\label{eq:x_nu_zeta}
\end{equation}
\end{proof}
\begin{cBoxB}{}
\begin{lem}
\label{lem:Let--be-1}Let $\left(F,\Omega,g\right)$ be a vector space
with a symplectic structure $\Omega$ and a compatible metric $g$.
Then (\ref{eq:FF*KN}) is orthogonal for the metric $\boldsymbol{g=}g\oplus g^{-1}$
as well. The maps
\begin{equation}
\sqrt{2}\pi_{K}:\left(x,\check{\Omega}\left(x\right)\right)\in\left(K,\boldsymbol{\Omega},\boldsymbol{g}\right)\rightarrow x\in\left(F,\Omega,g\right)\label{eq:pi_K}
\end{equation}
\begin{equation}
\sqrt{2}\pi_{N}:\left(x,\check{\Omega}^{*}\left(x\right)\right)\in\left(N,\boldsymbol{\Omega},\boldsymbol{g}\right)\rightarrow x\in\left(F,-\Omega,g\right)\label{eq:pi_N}
\end{equation}
are isomorphism for the respective bilinear forms.
\end{lem}

\end{cBoxB}

\begin{proof}
$K\overset{\perp}{\oplus}N$ is also orthogonal for the metric $\boldsymbol{g}$
because
\begin{align*}
\boldsymbol{g}\left(\left(x,\xi\right),\left(x',\xi'\right)\right) & =g\left(x,x'\right)+g^{-1}\left(\xi,\xi'\right)=g\left(x,x'\right)+g^{-1}\left(\check{\Omega}\left(x\right),-\check{\Omega}\left(x'\right)\right)\\
 & =g\left(x,x'\right)-g^{-1}\left(\check{\Omega}\left(x\right),\check{\Omega}\left(x'\right)\right)=0
\end{align*}
because $g^{-1}\left(\check{\Omega}\left(x\right),\check{\Omega}\left(x'\right)\right)=g\left(x,x'\right)$
from compatibility between $\Omega$ and $g$.

If $\left(x,\xi\right)=\left(x,\check{\Omega}\left(x\right)\right)\in K$
and $\left(x',\xi'\right)=\left(x',\check{\Omega}\left(x'\right)\right)\in K$
then
\begin{align*}
\boldsymbol{\Omega}\left(\left(x,\xi\right),\left(x',\xi'\right)\right) & =\langle\xi|x'\rangle-\langle x|\xi'\rangle=\langle\check{\Omega}\left(x\right)|x'\rangle-\langle x|\check{\Omega}\left(x'\right)\rangle=\langle\check{\Omega}\left(x\right)|x'\rangle-\langle\check{\Omega}^{*}\left(x\right)|x'\rangle\\
 & =\langle\check{\Omega}\left(x\right)|x'\rangle+\langle\check{\Omega}\left(x\right)|x'\rangle=2\Omega\left(x,x'\right)
\end{align*}
\begin{align*}
\boldsymbol{g}\left(\left(x,\xi\right),\left(x',\xi'\right)\right) & =g\left(x,x'\right)+g^{-1}\left(\xi,\xi'\right)=g\left(x,x'\right)+g^{-1}\left(\check{\Omega}\left(x\right),\check{\Omega}\left(x'\right)\right)=2g\left(x,x'\right)
\end{align*}
and a similar computation on $N$, if $\check{\Omega}$ is replaced
by $-\check{\Omega}$.
\end{proof}
Let $\mathcal{B}_{F}:\mathcal{S}\left(F\right)\rightarrow\mathcal{S}\left(F\oplus F^{*}\right)$
the Bargman transform in radial Gauge. Let $\mathcal{P}_{F\oplus F^{*}}:=\mathcal{B}_{F}\mathcal{B}_{F}^{\dagger}$.
Let $\mathcal{P}_{K},\mathcal{P}_{N}$ defined as in (\ref{eq:def_P})
for the respective spaces $\left(K,\Omega_{K},g_{K}\right)$, $\left(N,\Omega_{N},g_{N}\right)$
defined in (\ref{eq:def_K_N}) with metrics induced from $\boldsymbol{\Omega},\boldsymbol{g}$.
From (\ref{eq:FF*KN}) we have the natural identification 
\[
\mathcal{S}\left(F\oplus F^{*}\right)=\mathcal{S}\left(K\oplus N\right)=\mathcal{S}\left(K\right)\otimes\mathcal{S}\left(N\right).
\]

\begin{cBoxB}{}
\begin{lem}
\label{lem:We-have-}We have
\begin{equation}
\mathcal{P}_{F\oplus F^{*}}=\mathcal{P}_{K}\otimes\mathcal{P}_{N}\label{eq:PKNB}
\end{equation}
\end{lem}

\end{cBoxB}

\begin{proof}
Since $\boldsymbol{\Omega}_{F\oplus F^{*}}=\Omega_{K}\oplus\Omega_{N}$
and $\boldsymbol{g}=g_{K}\oplus g_{N}$, (\ref{eq:PKNB}) follows
from the expression of the kernels (\ref{eq:Kernel_of_P_radial_Gauge})
and (\ref{eq:def_P}) that coincide for both sides.
\end{proof}
\begin{cBoxB}{}
\begin{lem}
\label{lem:We-have-unitary} From the linear maps $\pi_{K},\pi_{N}$
in (\ref{eq:pi_K}), we have $L^{2}-$unitary operators
\begin{equation}
U_{K}:=\left(\sqrt{2}\pi_{K}\right)^{\circ}\,:\mathrm{Im}\left(\mathcal{P}_{F}\right)\rightarrow\mathrm{Im}\left(\mathcal{P}_{K}\right)\label{eq:def_U_K}
\end{equation}
\[
U_{N}:=\left(\sqrt{2}\pi_{N}\right)^{\circ}\,:\mathrm{Im}\left(C\mathcal{P}_{F}\mathcal{C}\right)\rightarrow\mathrm{Im}\left(\mathcal{P}_{N}\right)
\]
where $\mathcal{C}:u\in\mathcal{S}\left(F\right)\rightarrow\overline{u}\in\mathcal{S}\left(F\right)$
is the conjugation operator.
\end{lem}

\end{cBoxB}

\begin{proof}
This is a direct consequence of Lemma \ref{lem:Let--be-1} and expression
(\ref{eq:def_P}).
\end{proof}

\subsection{\label{subsec:Metaplectic-decomposition-of}Metaplectic decomposition
of a linear symplectic map}

In this section, let 
\[
\Phi:\left(F_{1},\Omega_{1}\right)\rightarrow\left(F_{2},\Omega_{2}\right)
\]
a linear invertible symplectic map and $g_{j}$ a compatible metric
on $F_{j}$ for $j=1,2$.

Let us first give a definition of metaplectic operator that will be
useful. The metaplectic correction $\Upsilon\left(\Phi\right)>0$
has been defined in (\ref{eq:def_d}). The reader may consult \cite[section 13.2]{treves2022analytic}
for a general reference about metaplectic operators. 

\begin{cBoxA}{}
\begin{defn}
\label{def:Let--a}
\begin{itemize}
\item We define the \textbf{metaplectic operator (or quantization) of $\Phi$}
by
\begin{equation}
\tilde{\mathrm{Op}}\left(\Phi\right):=\left(\Upsilon\left(\Phi\right)\right)^{1/2}\mathcal{P}_{F_{2}}\Phi^{-\circ}\mathcal{P}_{F_{1}}\quad:\mathrm{Im}\left(\mathcal{P}_{F_{1}}\right)\rightarrow\mathrm{Im}\left(\mathcal{P}_{F_{2}}\right)\label{eq:def_op_tilde}
\end{equation}
with the metaplectic correction $\Upsilon\left(\Phi\right)>0$ defined
in (\ref{eq:def_d}).
\item If $E_{1}\subset F_{1}$, $E_{2}\subset F_{2}$ are Lagrangian linear
subspace, and using the identification (\ref{eq:identif_F_E}), we
define
\begin{align}
\mathrm{Op}\left(\Phi\right): & =\left(\Upsilon\left(\Phi\right)\right)^{1/2}\mathcal{B}_{E_{2}}^{\dagger}\Phi^{-\circ}\mathcal{B}_{E_{1}}\quad:\mathcal{S}\left(E_{1}\right)\rightarrow\mathcal{S}\left(E_{2}\right)\label{eq:def_Op}
\end{align}
that is conjugated to $\tilde{\mathrm{Op}}\left(\Phi\right)_{/\mathrm{Im}\mathcal{B}_{E_{1}}\rightarrow\mathrm{Im}\mathcal{B}_{E_{2}}}$
since 
\begin{equation}
\tilde{\mathrm{Op}}\left(\Phi\right)=\mathcal{B}_{E_{2}}\mathrm{Op}\left(\Phi\right)\mathcal{B}_{E_{1}}^{\dagger}.\label{eq:op_tilde_op}
\end{equation}
\end{itemize}
\end{defn}

\end{cBoxA}

\begin{rem}
We will show in Proposition \ref{prop:The-operator-} that $\tilde{\mathrm{Op}}\left(\Phi\right):L^{2}\left(F_{1}\right)\rightarrow L^{2}\left(F_{2}\right)$
is $L^{2}-$unitary (hence $\mathrm{Op}\left(\Phi\right):L^{2}\left(E_{1}\right)\rightarrow L^{2}\left(E_{2}\right)$
is also unitary).
\end{rem}

Let 
\begin{equation}
\tilde{\Phi}:=\Phi^{-1}\oplus\Phi^{*}:\quad F_{2}\oplus F_{2}^{*}\rightarrow F_{1}\oplus F_{1}^{*}\label{eq:Phi_phi_minus}
\end{equation}
the induced map on cotangent spaces. Since $\Phi$ is symplectic
we have that $\tilde{\Phi}$ preserves the decompositions (\ref{eq:FF*KN}),
$F_{j}\oplus F_{j}^{*}=K_{j}\oplus N_{j}$, with $j=1,2$ and we denote
its components by
\begin{equation}
\tilde{\Phi}=\Phi_{K}\oplus\Phi_{N}.\label{eq:phi_tilde}
\end{equation}
Beware that $\Phi_{K},\Phi_{N}$ are conjugated to $\Phi^{-1}$, i.e.
$\Phi^{-1}\pi_{N_{2}}=\pi_{N_{1}}\Phi_{N}:N_{2}\rightarrow F_{1}$
and $\Phi^{-1}\pi_{K_{2}}=\pi_{K_{1}}\Phi_{K}:K_{2}\rightarrow F_{1}$.

\begin{cBoxB}{}
\begin{thm}
\label{lem:metaplectic_decomp}\cite{fred-PreQ-06}\cite{faure-tsujii_prequantum_maps_12}We
have
\begin{equation}
\mathcal{B}_{F_{1}}\Phi^{\circ}\mathcal{B}_{F_{2}}^{\dagger}=\tilde{\mathrm{Op}}\left(\tilde{\Phi}\right)=\tilde{\mathrm{Op}}\left(\Phi_{K}\right)\otimes\tilde{\mathrm{Op}}\left(\Phi_{N}\right)\quad:\mathrm{Im}\left(\mathcal{P}_{F_{2}\oplus F_{2}^{*}}\right)\rightarrow\mathrm{Im}\left(\mathcal{P}_{F_{1}\oplus F_{1}^{*}}\right).\label{eq:decomp_phi}
\end{equation}
\end{thm}

\end{cBoxB}

\begin{proof}
We have
\begin{equation}
\Upsilon\left(\tilde{\Phi}\right)\eq{\ref{eq:Phi_phi_minus},\ref{eq:product}}\Upsilon\left(\Phi^{-1}\right)\Upsilon\left(\Phi^{*}\right)\eq{\ref{eq:dphi_inv}}\left(\Upsilon\left(\Phi\right)\right)^{2}.\label{eq:upsilon_phi_tilde}
\end{equation}
Hence
\begin{align*}
\mathcal{B}_{F_{1}}\Phi^{\circ}\mathcal{B}_{F_{2}}^{\dagger} & \eq{\ref{eq:def_Upsilon},,\ref{eq:Bergman_projector}}\Upsilon\left(\Phi\right)\mathcal{P}_{F_{1}\oplus F_{1}^{*}}\tilde{\Phi}^{-\circ}\mathcal{P}_{F_{2}\oplus F_{2}^{*}}\\
 & \eq{\ref{eq:def_op_tilde}}\Upsilon\left(\Phi\right)\left(\Upsilon\left(\tilde{\Phi}\right)\right)^{-1/2}\tilde{\mathrm{Op}}\left(\tilde{\Phi}\right)\eq{\ref{eq:upsilon_phi_tilde}}\tilde{\mathrm{Op}}\left(\tilde{\Phi}\right),
\end{align*}
and
\begin{align*}
\tilde{\mathrm{Op}}\left(\tilde{\Phi}\right) & \eq{\ref{eq:def_op_tilde}}\left(\Upsilon\left(\tilde{\Phi}\right)\right)^{1/2}\mathcal{P}_{F_{2}\oplus F_{2}^{*}}\tilde{\Phi}^{-\circ}\mathcal{P}_{F_{1}\oplus F_{1}^{*}}\\
 & \eq{\ref{eq:phi_tilde},\ref{eq:PKNB},\ref{eq:product}}\left(\Upsilon\left(\Phi_{K}\right)\Upsilon\left(\Phi_{N}\right)\right)^{1/2}\left(\mathcal{P}_{K_{1}}^{\dagger}\Phi_{K}^{-\circ}\mathcal{P}_{K_{2}}\right)\otimes\left(\mathcal{P}_{N_{1}}^{\dagger}\Phi_{N}^{-\circ}\mathcal{P}_{N_{2}}\right)\\
 & \eq{\ref{eq:def_op_tilde}}\tilde{\mathrm{Op}}\left(\Phi_{K}\right)\otimes\tilde{\mathrm{Op}}\left(\Phi_{N}\right).
\end{align*}
\end{proof}
\begin{rem}
Lemma \ref{lem:metaplectic_decomp} is central to the analysis in
this paper. It is a factorization formula for the map $\Phi$ and
somehow gives a square root of $\Phi^{\circ}$.
\end{rem}

\begin{cBoxB}{}
\begin{prop}
For any linear symplectic maps $\Phi_{2,1}:\left(F_{1},\Omega_{1}\right)\rightarrow\left(F_{2},\Omega_{2}\right)$,
$\Phi_{3,2}:\left(F_{2},\Omega_{2}\right)\rightarrow\left(F_{3},\Omega_{3}\right)$
with compatible metrics $g_{j}$ on $F_{j}$, $j=1,2,3$, we have
\begin{equation}
\tilde{\mathrm{Op}}\left(\Phi_{3,2}\Phi_{2,1}\right)=\tilde{\mathrm{Op}}\left(\Phi_{3,2}\right)\tilde{\mathrm{Op}}\left(\Phi_{2,1}\right).\label{eq:homom}
\end{equation}
\end{prop}

\end{cBoxB}

\begin{proof}
We have
\begin{align*}
\tilde{\mathrm{Op}}\left(\left(\Phi_{3,2}\Phi_{2,1}\right)_{K}\right)\otimes\tilde{\mathrm{Op}}\left(\left(\Phi_{3,2}\Phi_{2,1}\right)_{N}\right) & \eq{\ref{eq:decomp_phi}}\mathcal{B}_{F_{3}}\left(\Phi_{3,2}\Phi_{2,1}\right)^{-\circ}\mathcal{B}_{F_{1}}^{\dagger}\\
 & =\mathcal{B}_{F_{3}}\Phi_{3,2}^{-\circ}\Phi_{2,1}^{-\circ}\mathcal{B}_{F_{1}}^{\dagger}\eq{\ref{eq:resol_ident_bargamnn}}\left(\mathcal{B}_{F_{3}}\Phi_{3,2}^{-\circ}\mathcal{B}_{F_{2}}^{\dagger}\right)\left(\mathcal{B}_{F_{2}}\Phi_{2,1}^{-\circ}\mathcal{B}_{F_{1}}^{\dagger}\right)\\
 & \eq{\ref{eq:decomp_phi}}\left(\tilde{\mathrm{Op}}\left(\left(\Phi_{3,2}\right)_{K}\right)\tilde{\mathrm{Op}}\left(\left(\Phi_{2,1}\right)_{K}\right)\right)\otimes\left(\tilde{\mathrm{Op}}\left(\left(\Phi_{3,2}\right)_{N}\right)\tilde{\mathrm{Op}}\left(\left(\Phi_{2,1}\right)_{N}\right)\right)
\end{align*}
Considering each factor gives (\ref{eq:homom}).
\end{proof}
\begin{cBoxB}{}
\begin{prop}
\label{prop:The-operator-}The operator $\tilde{\mathrm{Op}}\left(\Phi\right):\mathrm{Im}\left(\mathcal{P}_{F_{1}}\right)\rightarrow\mathrm{Im}\left(\mathcal{P}_{F_{2}}\right)$
defined in (\ref{eq:def_op_tilde}) is $L^{2}-$unitary.
\end{prop}

\end{cBoxB}

\begin{rem}
Since $\mathcal{B}_{E_{j}}:L^{2}\left(E_{j}\right)\rightarrow\mathrm{Im}\mathcal{B}_{E_{j}}=\mathrm{Im}\mathcal{P}_{F_{j}}$
is unitary this implies that $\mathrm{Op}\left(\Phi\right):L^{2}\left(E_{1}\right)\rightarrow L^{2}\left(E_{2}\right)$
is also unitary.
\end{rem}

\begin{proof}
We have $\Upsilon\left(\mathrm{Id}\right)=1$ hence $\tilde{\mathrm{Op}}\left(\Phi\right)\left(\mathrm{Id}\right)\eq{\ref{eq:def_op_tilde}}\mathcal{P}=\mathrm{Id}_{/\mathrm{Im}\mathcal{P}}$.
Also
\begin{equation}
\Upsilon\left(\Phi^{-1}\right)\eq{\ref{eq:dphi_inv}}\left|\mathrm{det}\Phi\right|\Upsilon\left(\Phi\right)=\Upsilon\left(\Phi\right).\label{eq:ups_inv}
\end{equation}
We have
\[
\tilde{\mathrm{Op}}\left(\Phi^{-1}\right)\tilde{\mathrm{Op}}\left(\Phi\right)\eq{\ref{eq:homom}}\tilde{\mathrm{Op}}\left(\mathrm{Id}\right)=\mathrm{Id}_{/\mathrm{Im}\mathcal{P}},
\]
so $\tilde{\mathrm{Op}}\left(\Phi\right)^{-1}=\tilde{\mathrm{Op}}\left(\Phi^{-1}\right)$
and
\begin{align*}
\tilde{\mathrm{Op}}\left(\Phi\right)^{\dagger} & \eq{\ref{eq:def_op_tilde}}\left(\Upsilon\left(\Phi\right)\right)^{1/2}\mathcal{P}\Phi^{\circ}\mathcal{P}\eq{\ref{eq:ups_inv}}\left(\Upsilon\left(\Phi^{-1}\right)\right)^{1/2}\mathcal{P}\left(\Phi^{-1}\right)^{-\circ}\mathcal{P}\\
 & \eq{\ref{eq:def_op_tilde}}\tilde{\mathrm{Op}}\left(\Phi^{-1}\right)=\tilde{\mathrm{Op}}\left(\Phi\right)^{-1}.
\end{align*}
\end{proof}

The next lemma considers a special case for the symplectic map $\Phi$.

\begin{cBoxB}{}
\begin{lem}
\cite[proof of Prop 4.3.1 p.79]{faure-tsujii_prequantum_maps_12}\label{lem:Let-,-be}Let
$\left(E_{1},g_{1}\right)$,$\left(E_{2},g_{2}\right)$ be Euclidean
vector spaces, and $\phi:E_{2}\rightarrow E_{1}$ an invertible linear
map. We denote $\phi^{\circ}:\mathcal{S}\left(E_{1}\right)\rightarrow\mathcal{S}\left(E_{2}\right)$
the pull-back operator and $\Phi:=\phi^{-1}\oplus\phi^{*}:E_{1}\oplus E_{1}^{*}\rightarrow E_{2}\oplus E_{2}^{*}$
the induced symplectic map on cotangent spaces. $\mathrm{Op}\left(\Phi\right)$
has been defined in (\ref{eq:def_Op}) and $\tilde{\mathrm{Op}}\left(\Phi\right)$
in (\ref{eq:def_op_tilde}). We have 
\begin{equation}
\mathrm{Op}\left(\Phi\right)=\left|\mathrm{det}\phi\right|^{1/2}\phi^{\circ}.\label{eq:Op(f)}
\end{equation}
\begin{equation}
\tilde{\mathrm{Op}}\left(\Phi\right)=\left|\mathrm{det}\phi\right|^{1/2}\mathcal{B}_{E_{2}}\phi^{\circ}\mathcal{B}_{E_{1}}^{\dagger}\label{eq:Op_tilde_PHI}
\end{equation}
\end{lem}

\end{cBoxB}

\begin{proof}
We have
\begin{align}
\Upsilon\left(\Phi\right) & \eq{\ref{eq:product}}\Upsilon\left(\phi\right)\Upsilon\left(\phi^{-1}\right)\eq{\ref{eq:dphi_inv}}\left|\mathrm{det}\phi\right|\left(\Upsilon\left(\phi\right)\right)^{2}\label{eq:dPhi_dphi}
\end{align}
Then
\[
\mathrm{Op}\left(\Phi\right)\eq{\ref{eq:def_Op}}\left(\Upsilon\left(\Phi\right)\right)^{1/2}\mathcal{B}_{E_{2}}^{\dagger}\Phi^{-\circ}\mathcal{B}_{E_{1}}\eq{\ref{eq:def_Upsilon}}\left(\Upsilon\left(\Phi\right)\right)^{1/2}\left(\Upsilon\left(\phi\right)\right)^{-1}\phi^{\circ}\eq{\ref{eq:dPhi_dphi}}\left|\mathrm{det}\phi\right|^{1/2}\phi^{\circ}
\]
\end{proof}

\subsection{Symplectic spinors}

The previous constructions correspond to what is sometimes called
\textbf{\href{https://en.wikipedia.org/wiki/Symplectic_spinor_bundle}{symplectic spinors}}
in the literature. We summarize what we have obtained and comment
on this.

We start from a linear symplectic space $\left(F,\Omega\right)$.
We have the decomposition $F\oplus F^{*}=K\overset{\perp_{\boldsymbol{\Omega}}}{\oplus}N$
in (\ref{eq:FF*KN}). We take a $\Omega-$compatible metric $g$ on
$F$, so that we can define $\mathcal{B}_{F}:L^{2}\left(F\right)\rightarrow\mathrm{Im}\mathcal{B}_{F}\subset L^{2}\left(F\oplus F^{*}\right)$
in radial Gauge from (\ref{eq:def_Bargman_B}), but also the operators
$\mathcal{P}_{F},\mathcal{P}_{F\oplus F*},\mathcal{P}_{K},\mathcal{P}_{N}$.

Recall from Lemma \ref{lem:We-have-unitary}, we have $L^{2}-$unitary
isomorphisms
\[
U_{K}:\mathrm{Im}\mathcal{P}_{F}\rightarrow\mathrm{Im}\mathcal{P}_{K},\quad U_{N}:\mathrm{Im}\left(\mathcal{C}\mathcal{P}_{F}\mathcal{C}\right)\rightarrow\mathrm{Im}\mathcal{P}_{N},
\]
and from (\ref{eq:PKNB}), we have a $L^{2}-$unitary isomorphism
\[
\mathcal{B}_{F}:L^{2}\left(F\right)\rightarrow\mathrm{Im}\mathcal{P}_{K}\otimes\mathrm{Im}\mathcal{P}_{N}.
\]

This suggests to give the following definition and Theorem.

\begin{cBoxA}{}
\begin{defn}
For a linear symplectic space $\left(F,\Omega\right)$ with $\Omega-$compatible
metric $g$, we define the space of \textbf{symplectic spinors} $\mathrm{Spin}_{\pm}\left(F\right)$
by
\[
\mathrm{Spin}_{+}\left(F\right):=\mathrm{Im}\mathcal{P}_{F}\subset L^{2}\left(F\right),\qquad\mathrm{Spin}_{-}\left(F\right):=\mathrm{Im}\left(\mathcal{C}\mathcal{P}_{F}\mathcal{C}\right)\subset L^{2}\left(F\right).
\]
We define the $L^{2}-$unitary operator
\begin{equation}
U_{\mathrm{spin}}:=\left(U_{K}^{\dagger}\otimes U_{N}^{\dagger}\right)\circ\mathcal{B}_{F}\quad:L^{2}\left(F\right)\rightarrow\mathrm{Spin}_{+}\left(F\right)\otimes\mathrm{Spin}_{-}\left(F\right).\label{eq:def_Uspin}
\end{equation}
\end{defn}

\end{cBoxA}

\begin{cBoxB}{}
\begin{thm}
Let $\Phi:\left(F_{1},\Omega_{1}\right)\rightarrow\left(F_{2},\Omega_{2}\right)$
a linear invertible symplectic map. We have the commutative diagram
with unitary operators
\begin{equation}
\xymatrix{L^{2}\left(F_{1}\right)\ar[rrr]\sp(0.4){\Phi^{-\circ}}\ar[d]\sp(0.4){U_{\mathrm{spin}}} &  &  & L^{2}\left(F_{2}\right)\ar[d]\sp(0.4){U_{\mathrm{spin}}}\\
\mathrm{Spin}_{+}\left(F_{1}\right)\otimes\mathrm{Spin}_{-}\left(F_{1}\right)\ar[rrr]\sp(0.5){\tilde{\mathrm{Op}}\left(\Phi\right)\otimes\left(\mathcal{C}\tilde{\mathrm{Op}}\left(\Phi\right)\mathcal{C}\right)} &  &  & \mathrm{Spin}_{+}\left(F_{2}\right)\otimes\mathrm{Spin}_{-}\left(F_{2}\right)
}
\label{eq:SPin_Spin}
\end{equation}
\end{thm}

\end{cBoxB}

\begin{rem}
The operator $\Phi^{-\circ}$ in the upper line is the push forward
of the symplectic map and can be interpreted as the ``classical dynamics''.
The lower line is the product of two metaplectic operators or quantum
operators. So this diagram shows that
\begin{itemize}
\item In Eq.(\ref{eq:def_Uspin}), the classical space $L^{2}\left(F\right)$
is identified as the product of two quantum spinor spaces.
\item In Eq. (\ref{eq:SPin_Spin}), the ``classical dynamics'' can be
factorized as the product of the quantum dynamics times its complex
conjugation.
\end{itemize}
\end{rem}

\begin{proof}
Since $\Phi\pi_{N_{1}}=\pi_{N_{2}}\Phi_{N}^{-1}:N_{1}\rightarrow F_{2}$,
we have $\Phi^{\circ}=U_{N_{1}}^{\dagger}\Phi_{N}^{-\circ}U_{N_{2}}=U_{K_{1}}^{\dagger}\Phi_{K}^{-\circ}U_{K_{2}}:L^{2}\left(F_{2}\right)\rightarrow L^{2}\left(F_{1}\right)$.
From Lemma \ref{lem:We-have-unitary}, we have
\[
\mathcal{C}\mathcal{P}_{F}\mathcal{C}=U_{N}^{\dagger}\mathcal{P}_{N}U_{N}.
\]
Hence
\begin{align*}
U_{1}^{\dagger}\tilde{\mathrm{Op}}\left(\Phi_{N}\right)U_{N_{2}} & \eq{\ref{eq:def_op_tilde}}\left(\Upsilon\left(\Phi\right)\right)^{1/2}U_{N_{1}}^{\dagger}\mathcal{P}_{N_{1}}\Phi_{N}^{-\circ}\mathcal{P}_{N_{2}}U_{N_{2}}\\
 & =\left(\Upsilon\left(\Phi\right)\right)^{1/2}U_{N_{1}}^{\dagger}\mathcal{P}_{N_{1}}U_{N_{1}}\Phi^{\circ}U_{N_{2}}^{\dagger}\mathcal{P}_{N_{2}}U_{N_{2}}=\left(\Upsilon\left(\Phi\right)\right)^{1/2}\mathcal{C}\mathcal{P}_{F_{1}}\mathcal{C}\Phi^{\circ}\mathcal{C}\mathcal{P}_{F_{2}}\mathcal{C}\\
 & \eq{\ref{eq:ups_inv}}\mathcal{C}\left(\Upsilon\left(\Phi^{-1}\right)\right)^{1/2}\mathcal{P}_{F_{1}}\Phi^{\circ}\mathcal{P}_{F_{2}}\mathcal{C}\eq{\ref{eq:def_op_tilde}}\mathcal{C}\tilde{\mathrm{Op}}\left(\Phi^{-1}\right)\mathcal{C}
\end{align*}
and similarly $U_{K_{1}}^{\dagger}\tilde{\mathrm{Op}}\left(\Phi_{K}\right)U_{K_{2}}=\tilde{\mathrm{Op}}\left(\Phi^{-1}\right)$.
So we get
\begin{align*}
\Phi^{\circ} & \eq{\ref{eq:decomp_phi}}\mathcal{B}_{F_{1}}^{\dagger}\left(\tilde{\mathrm{Op}}\left(\Phi_{K}\right)\otimes\tilde{\mathrm{Op}}\left(\Phi_{N}\right)\right)\mathcal{B}_{F_{2}}\qquad:L^{2}\left(F_{2}\right)\rightarrow L^{2}\left(F_{1}\right)\\
 & =\mathcal{B}_{F_{1}}^{\dagger}\left(\left(U_{K_{1}}\tilde{\mathrm{Op}}\left(\Phi^{-1}\right)U_{K_{2}}^{\dagger}\right)\otimes\left(U_{N_{1}}\mathcal{C}\tilde{\mathrm{Op}}\left(\Phi^{-1}\right)\mathcal{C}U_{N_{2}}^{\dagger}\right)\right)\mathcal{B}_{F_{2}}\\
 & \eq{\ref{eq:def_Uspin}}U_{\mathrm{spin}}^{\dagger}\left(\tilde{\mathrm{Op}}\left(\Phi^{-1}\right)\otimes\left(\mathcal{C}\tilde{\mathrm{Op}}\left(\Phi^{-1}\right)\mathcal{C}\right)\right)U_{\mathrm{spin}}
\end{align*}
giving (\ref{eq:SPin_Spin}).
\end{proof}

\subsection{Some useful decompositions}

Here we give a relation that is simple here in the linear setting
and is used in the main text as an approximation, in the proof of
Theorem \ref{Thm:approx_exptX_from_N}. Let us denote $\tilde{F}=F\oplus F^{*}\eq{\ref{eq:FF*KN}}K\oplus N$
and $\exp:T\tilde{F}\rightarrow\tilde{F}$ the exponential map. Now
we restrict the base to $K\subset\tilde{F}$ and consider $N\subset T_{K}\tilde{F}$
as a sub-bundle $N\rightarrow K$ of $T_{K}\tilde{F}\rightarrow K$.
We denote $\exp_{N}:N\rightarrow\tilde{F}$ the exponential map that
is an isomorphism.

As before, we have the operators $\mathcal{B}_{F}:\mathcal{S}\left(F\right)\rightarrow\mathcal{S}\left(\tilde{F}\right)$,
$\mathcal{B}_{F}^{\dagger}:\mathcal{S}\left(\tilde{F}\right)\rightarrow\mathcal{S}\left(F\right)$,
$\widetilde{\exp_{N}^{\circ}}:\mathcal{S}\left(\tilde{F}\right)\rightarrow\mathcal{S}\left(N\right)$,
$\widetilde{\left(\exp_{N}^{-1}\right)^{\circ}}:\mathcal{S}\left(N\right)\rightarrow\mathcal{S}\left(\tilde{F}\right)$. 

Let $\Phi:F\rightarrow F$ be a linear symplectic map. We have $\Phi^{\circ}:\mathcal{S}\left(F\right)\rightarrow\mathcal{S}\left(F\right)$.
The map $\tilde{\Phi}:\tilde{F}\rightarrow\tilde{F}$ that can be
seen as the bundle map $\tilde{\Phi}_{N}:N\rightarrow N$ over $\Phi_{K}$.
Then $\tilde{\mathrm{Op}}\left(\tilde{\Phi}_{N}\right):\mathcal{S}\left(N\right)\rightarrow\mathcal{S}\left(N\right)$
is a bundle map over $\Phi_{K}$. The metaplectic correction $\left(\Upsilon\left(\Phi_{K}\right)\right)^{1/2}>0$
has been defined in (\ref{eq:def_d}) and in fact we have $\Upsilon\left(\Phi_{K}\right)=\Upsilon\left(\Phi_{N}\right)=\Upsilon\left(\Phi\right)$.

\begin{cBoxB}{}
\begin{lem}
\label{lem:decomp}We have
\begin{equation}
\Phi^{\circ}=\mathcal{B}_{F}^{\dagger}\widetilde{\left(\exp_{N}^{-1}\right)^{\circ}}\left(\Upsilon\left(\Phi_{K}\right)\right)^{1/2}\tilde{\mathrm{Op}}\left(\tilde{\Phi}_{N}\right)\widetilde{\exp_{N}^{\circ}}\mathcal{B}_{F}.\label{eq:expression_Phi_circ}
\end{equation}
Equivalently
\begin{equation}
\tilde{\mathrm{Op}}\left(\tilde{\Phi}\right)\eq{\ref{eq:decomp_phi}}\mathcal{P}_{F}\widetilde{\left(\exp_{N}^{-1}\right)^{\circ}}\left(\Upsilon\left(\Phi_{K}\right)\right)^{1/2}\tilde{\mathrm{Op}}\left(\tilde{\Phi}_{N}\right)\widetilde{\exp_{N}^{\circ}}\mathcal{P}_{F}.\label{eq:expression_Phi_circ-1}
\end{equation}
\end{lem}

\end{cBoxB}

\begin{rem}
As the proof belows shows, the metaplectic correction $\left(\Upsilon\left(\Phi_{K}\right)\right)^{1/2}>0$
in (\ref{eq:expression_Phi_circ}) is due to the fact that $\tilde{\mathrm{Op}}\left(\tilde{\Phi}_{N}\right)$
over $\Phi_{K}$ does not preserves the space $\mathrm{Im}\mathcal{P}_{K}$.
\end{rem}

\begin{proof}
We have
\begin{align*}
\Phi^{\circ} & \eq{\ref{eq:resol_ident_bargamnn}}\mathcal{B}_{F}^{\dagger}\mathcal{B}_{F}\Phi^{\circ}\mathcal{B}_{F}^{\dagger}\mathcal{B}_{F}\eq{\ref{eq:decomp_phi}}\mathcal{B}_{F}^{\dagger}\left(\tilde{\mathrm{Op}}\left(\Phi_{K}\right)\otimes\tilde{\mathrm{Op}}\left(\Phi_{N}\right)\right)\mathcal{B}_{F}\\
 & \eq{\ref{eq:def_op_tilde}}\mathcal{B}_{F}^{\dagger}\left(\left(\left(\Upsilon\left(\Phi_{K}\right)\right)^{1/2}\mathcal{P}_{K}\Phi_{K}^{-\circ}\mathcal{P}_{K}\right)\otimes\tilde{\mathrm{Op}}\left(\Phi_{N}\right)\right)\mathcal{B}_{F}
\end{align*}
We use that $\mathcal{B}_{F}\eq{\ref{eq:Bergman_projector},\ref{eq:resol_ident_bargamnn}}\mathcal{P}_{F}\mathcal{B}_{F}\eq{\ref{eq:PKNB}}\left(\mathcal{P}_{K}\otimes\mathcal{P}_{N}\right)\mathcal{B}_{F}$
and similarly $\mathcal{B}_{F}^{\dagger}=\mathcal{B}_{F}^{\dagger}\left(\mathcal{P}_{K}\otimes\mathcal{P}_{N}\right)$.
Hence we can remove the operators $\mathcal{P}_{K}$ and get
\begin{align*}
\Phi^{\circ} & =\mathcal{B}_{F}^{\dagger}\left(\left(\Upsilon\left(\Phi_{K}\right)\right)^{1/2}\Phi_{K}^{-\circ}\otimes\tilde{\mathrm{Op}}\left(\Phi_{N}\right)\right)\mathcal{B}_{F}\\
 & =\left(\Upsilon\left(\Phi_{K}\right)\right)^{1/2}\mathcal{B}_{F}^{\dagger}\widetilde{\exp_{N}^{\circ}}^{-1}\tilde{\mathrm{Op}}\left(\tilde{\Phi}_{N}\right)\widetilde{\exp_{N}^{\circ}}\mathcal{B}_{F}.
\end{align*}
where the last line expresses that $\tilde{\mathrm{Op}}\left(\tilde{\Phi}_{N}\right):\mathcal{S}\left(N\right)\rightarrow\mathcal{S}\left(N\right)$
is a bundle map over $\Phi_{K}$. Using (\ref{eq:decomp_phi}) and
(\ref{eq:Bergman_projector}) we deduce (\ref{eq:expression_Phi_circ-1}).
\end{proof}
In the next Lemma, we use the (twisted) operator $\widetilde{\exp^{\circ}}:\mathcal{S}\left(\tilde{F}\right)\rightarrow\mathcal{S}\left(T\tilde{F}\right)$.
We also use the restriction operators $r_{0}:\mathcal{S}\left(T\tilde{F}\right)\rightarrow\mathcal{S}\left(\tilde{F}\right)$,
$r_{/K}:\mathcal{S}\left(T\tilde{F}\right)\rightarrow\mathcal{S}\left(T_{K}\tilde{F}\right)$,
$r_{N}:\mathcal{S}\left(T_{K}\tilde{F}\right)\rightarrow\mathcal{S}\left(N\right)$.
We have
\[
\mathrm{Id}_{\mathcal{S}\left(\tilde{F}\right)}=r_{0}\widetilde{\exp^{\circ}}=\widetilde{\left(\exp_{N}^{-1}\right)^{\circ}}r_{N}r_{/K}\widetilde{\exp^{\circ}},
\]
where the second equality can be seen as a generalization of the first
equality for the respective decompositions $\tilde{F}=\left\{ 0\right\} \oplus\tilde{F}=K\oplus N$.
More generally we have the following Lemma.

\begin{cBoxB}{}
\begin{lem}
\label{lem:For-the-symplectic}For the symplectic map $\tilde{\Phi}\eq{\ref{eq:phi_tilde}}\Phi_{K}\oplus\Phi_{N}:\tilde{F}\rightarrow\tilde{F}$,
using $\tilde{\mathrm{Op}}\left(\Phi\right):\mathcal{S}\left(T\tilde{F}\right)\rightarrow\mathcal{S}\left(T\tilde{F}\right)$
and $\tilde{\mathrm{Op}}_{/K}\left(\Phi\right):\mathcal{S}\left(T_{K}\tilde{F}\right)\rightarrow\mathcal{S}\left(T_{K}\tilde{F}\right)$
as bundle map operators, we have the decomposition
\begin{equation}
r_{0}\tilde{\mathrm{Op}}\left(\Phi\right)\widetilde{\exp^{\circ}}=\widetilde{\left(\exp_{N}^{-1}\right)^{\circ}}r_{N}\tilde{\mathrm{Op}}_{/K}\left(\Phi\right)r_{/K}\widetilde{\exp^{\circ}}.\label{eq:decomp_for_thm}
\end{equation}
\end{lem}

\end{cBoxB}

\subsection{\label{subsec:Taylor-operators-}Taylor operators $T_{k}$ on $\mathcal{S}\left(E\right)$}

Let $E$ an Euclidean vector space, $n=\mathrm{dim}E$. If $\left(e_{1},\ldots e_{n}\right)$
is a basis of $E$, we write $x=\sum_{i=1}^{n}x_{i}e_{i}\in E$. A
polynomial on $E$ is $P\left(x\right):=p\left(x_{1},\ldots x_{n}\right)$
with $p\in\mathbb{C}\left[x_{1},\ldots x_{n}\right]$ a polynomial
on $\mathbb{R}^{n}$. For $\alpha=\left(\alpha_{1},\ldots,\alpha_{n}\right)\in\mathbb{N}^{n}$,
we write $x^{\alpha}:=x_{1}^{\alpha_{1}}\ldots x_{n}^{\alpha_{n}}$
a monomial of degree $\left|\alpha\right|:=\alpha_{1}+\ldots\alpha_{n}$.
For $k\in\mathbb{N}$, we denote $\mathrm{Pol}_{k}\left(E\right)$
the space of \href{https://en.wikipedia.org/wiki/Homogeneous_polynomial}{homogeneous polynomials}
of degree $k$, that is independent on the basis. $P\in\mathrm{Pol}_{k}\left(E\right)$
can be written $P\left(x\right)=\sum_{\alpha\in\mathbb{N}^{k},\left|\alpha\right|=k}P_{\alpha}x^{\alpha}$
with components $P_{\alpha}\in\mathbb{C}$. We denote $\alpha!:=\alpha_{1}!\ldots\alpha_{n}!$
and $\delta^{\left(\alpha\right)}:=\delta_{0}^{\left(\alpha_{1}\right)}\left(x_{1}\right)\ldots\delta_{0}^{\left(\alpha_{n}\right)}\left(x_{n}\right)$
the Dirac distribution on $E$ (with $\alpha$-derivatives). We have
that for any $\alpha,\alpha'\in\mathbb{N}^{n}$, $\langle\frac{1}{\alpha'!}\delta_{0}^{\left(\alpha'\right)}|x^{\alpha}\rangle=\delta_{\alpha'=\alpha}$.
Hence the set $\left(x^{\alpha}\right)_{\alpha}$ forms a basis of
$\mathrm{Pol}\left(E\right)$ and $\left(\langle\delta^{\left(\alpha\right)}|.\rangle\right)_{\alpha}$
is the dual basis. See Section \ref{sec:Horocycle-operators} for
comments about the equivalence $\mathrm{Pol}_{k}\left(E\right)\equiv\mathrm{Sym}\left(\left(E^{*}\right)^{\otimes k}\right)$. 

\begin{cBoxA}{}
\begin{defn}
\label{def:T_k}For $k\in\mathbb{N}$, let
\[
T_{k}:\mathcal{S}'\left(E\right)\rightarrow\mathrm{Pol}_{k}\left(E\right)\subset\mathcal{S}'\left(E\right)
\]
be the projector onto $\mathrm{Pol}_{k}\left(E\right)$ with kernel
$\oplus_{k'\neq k}\mathrm{Pol}_{k}\left(E\right)$. We have
\[
\forall k,k',\quad T_{k'}T_{k}=T_{k}\delta_{k'=k},
\]
and explicitly, with respect to a basis of $E$,
\begin{equation}
T_{k}=\sum_{\alpha\in\mathbb{N}^{k},\left|\alpha\right|=k}x^{\alpha}\langle\frac{1}{\alpha!}\delta_{0}^{\left(\alpha\right)}|.\rangle.\label{eq:def_T_k}
\end{equation}
\end{defn}

\end{cBoxA}

The next lemma shows  how this relation is changed by a small norm
operator if we insert a truncation at large distance $\sigma$ from
the origin. We put $\chi_{\sigma}\left(\rho\right)=1$ for $\left|\rho\right|_{\boldsymbol{g}}\leq\sigma$,
$\chi_{\sigma}\left(\rho\right)=0$ for $\left|\rho\right|_{\boldsymbol{g}}>\sigma$
and $\mathrm{Op}\left(\chi_{\sigma}\right):=\mathcal{B}^{\dagger}\mathcal{M}_{\chi_{\sigma}}\mathcal{B}$.

\begin{cBoxB}{}
\begin{lem}
\label{lem:For-,-we}For $k,k'\in\mathbb{N}$, we have $\forall N,\exists C_{N}>0,\forall\sigma>0$,
\[
R_{\sigma}:=\mathrm{Op}\left(\chi_{\sigma}\right)\left(T_{k'}\mathrm{Op}\left(\chi_{\sigma}\right)T_{k}-T_{k}\delta_{k'=k}\right)\mathrm{Op}\left(\chi_{\sigma}\right)
\]
satisfies
\[
\left|\langle\delta_{\rho'}|\mathcal{B}^{\dagger}R_{\sigma}\mathcal{B}\delta_{\rho}\rangle\right|\leq\left\langle \mathrm{dist}_{\boldsymbol{g}/\sigma^{2}}\left(\rho',\rho\right)\right\rangle ^{-N}C_{N}\sigma^{-N}\left\langle \mathrm{dist}_{\boldsymbol{g}/\sigma^{2}}\left(\rho,0\right)\right\rangle ^{-N},
\]

Consequently $\left\Vert R_{\sigma}\right\Vert \leq C_{N}\sigma^{-N}$
for any $N>0$, with $C_{N}>0$.
\end{lem}

\end{cBoxB}

\begin{proof}
We have
\begin{align}
\langle\delta_{\rho'}|\mathcal{B}R\mathcal{B}^{\dagger}\delta_{\rho}\rangle & \eq{\ref{eq:def_T_k}}\chi_{\sigma}\left(\rho'\right)\chi_{\sigma}\left(\rho\right)\label{eq:SK}\\
 & \sum_{\alpha'\in\mathbb{N}^{k},\left|\alpha'\right|=k',\alpha\in\mathbb{N}^{k},\left|\alpha\right|=k,}\langle\varphi_{\rho'}|x^{\alpha'}\rangle\left(\langle\frac{1}{\alpha'!}\delta_{0}^{\left(\alpha'\right)}|\mathrm{Op}\left(\chi_{\sigma}\right)x^{\alpha}\rangle-\delta_{\alpha'=\alpha}\right)\langle\frac{1}{\alpha!}\delta_{0}^{\left(\alpha\right)}|\varphi_{\rho}\rangle.
\end{align}
We have $\forall N,\exists C_{N},\forall\sigma>0$,
\[
\left|\langle\frac{1}{\alpha'!}\delta_{0}^{\left(\alpha'\right)}|\mathcal{B}^{\dagger}\mathcal{M}_{\chi_{\sigma}}\mathcal{B}x^{\alpha}\rangle-\delta_{\alpha'=\alpha}\right|\leq C_{N}\sigma^{-N},
\]
and due to truncations $\chi_{\sigma}\left(\rho'\right)\chi_{\sigma}\left(\rho\right)$
in (\ref{eq:SK}) we deduce that
\[
\left|\langle\delta_{\rho'}|\mathcal{B}^{\dagger}R\mathcal{B}\delta_{\rho}\rangle\right|\leq\left\langle \mathrm{dist}_{\boldsymbol{g}/\sigma^{2}}\left(\rho',\rho\right)\right\rangle ^{-N}C_{N}\sigma^{-N}\left\langle \mathrm{dist}_{\boldsymbol{g}/\sigma^{2}}\left(\rho,0\right)\right\rangle ^{-N},
\]
and $\left\Vert R_{\sigma}\right\Vert \leq C_{N}\sigma^{-N}$ from
example (\ref{eq:result_of_Shur}).
\end{proof}

\subsection{Analysis on $T\left(E\oplus E^{*}\right)$}

In this section we ``lift'' the analysis from $E\oplus E^{*}$ to
$T\left(E\oplus E^{*}\right)$. This might been seem a bit artificial
for vector spaces, but it will be useful in this paper as a linearized
model for manifolds.

\subsubsection{Bargman transform}

Let $\left(E,g\right)$ be a linear Euclidean space. Let $F:=E\oplus E^{*}$with
induced metric $\boldsymbol{g}=g\oplus g^{-1}$. One has $TF=T\left(E\oplus E^{*}\right)=\left(E\oplus E^{*}\right)\oplus\left(E\oplus E^{*}\right)$.
The exponential map is
\[
\exp:\begin{cases}
TF & \rightarrow F\\
\left(\rho,\rho'\right) & \rightarrow\rho+\rho'
\end{cases}
\]
Let us define the ``twisted pull back operator'' $\widetilde{\exp^{\circ}}:\mathcal{S}\left(F\right)\rightarrow\mathcal{S}\left(TF\right)$
by its Schwartz kernel as follows. For $\left(\rho_{1},\rho_{1}'\right)\in TF$,
$\rho_{2}\in F$, $\rho_{1}=\left(x_{1},\xi_{1}\right)$,$\rho_{1}'=\left(x_{1}',\xi_{1}'\right)$,
\[
\langle\delta_{\rho_{1},\rho_{1}'}|\widetilde{\exp^{\circ}}\delta_{\rho_{2}}\rangle=\delta_{\rho_{2}-\left(\rho_{1}+\rho_{1}'\right)}e^{-i\xi_{1}x_{1}'}
\]
We define the ``restriction'' operator $r_{0}:\mathcal{S}\left(TF\right)\rightarrow\mathcal{S}\left(F\right)$
by its Schwartz kernel
\[
\langle\delta_{\rho_{2}}|r_{0}\delta_{\rho_{1},\rho_{1}'}\rangle=\delta_{\rho_{2}-\rho_{1}}\delta_{\rho_{1}'}
\]
Let the ``horizontal space'' be
\begin{equation}
H:=F\oplus E.\label{eq:def_H}
\end{equation}
If $\mathcal{B}:\mathcal{S}\left(E\right)\rightarrow\mathcal{S}\left(F\right)$
is the Bargman transform defined in (\ref{eq:def_Bargman_B}), we
define the partial Bargman transform
\[
B:\mathcal{S}\left(H\right)\rightarrow\mathcal{S}\left(TF\right)
\]
by $B\left(u\left(\rho\right)\otimes v\left(x'\right)\right)=u\left(\rho\right)\otimes\left(\mathcal{B}v\right)\left(x'\right)$,
i.e. action on the second part only.

Let $\sigma>0$ and let us introduce the cutoff function that truncates
at distance $\sigma$ in the fiber $F$:
\[
\chi_{\sigma}:\begin{cases}
\mathcal{S}\left(TF\right) & \rightarrow\mathcal{S}'\left(TF\right)\\
u\left(\rho,\rho'\right) & \rightarrow\begin{cases}
u\left(\rho,\rho'\right) & \text{ if }\left\Vert \rho'\right\Vert _{\boldsymbol{g}}\leq\sigma\\
0 & \text{otherwise}
\end{cases}
\end{cases}
\]

Let
\[
\boldsymbol{B}_{\chi}^{\Delta}:=B^{\dagger}\chi_{\sigma}\widetilde{\exp^{\circ}}\quad:\mathcal{S}\left(F\right)\rightarrow\mathcal{S}\left(H\right)
\]
\[
\boldsymbol{B}:=r_{0}B\quad:\mathcal{S}\left(H\right)\rightarrow\mathcal{S}\left(F\right)
\]

\subsubsection{Linear map}

Let $A:E\rightarrow E$ a linear invertible map and $\tilde{A}:=A^{-1}\oplus A^{*}$
the induced map on $F:=E\oplus E^{*}$. Let $A^{\circ}:\mathcal{S}\left(E\right)\rightarrow\mathcal{S}\left(E\right)$
the pull back operator and $A_{H}:=\tilde{A}^{-1}\oplus A\quad:H\rightarrow H$
the induced map on $H$ in (\ref{eq:def_H}).

\begin{cBoxB}{}
\begin{lem}
\label{lem:Let-us-consider}Let us consider the difference operator
\[
R:=\boldsymbol{B}A_{H}^{\circ}\boldsymbol{B}_{\chi}^{\Delta}-\mathcal{B}A^{\circ}\mathcal{B}^{\dagger}\quad:\mathcal{S}\left(F\right)\rightarrow\mathcal{S}\left(F\right).
\]
Then $\forall N>0,\exists C_{N}>0$ such that for any $\sigma>0$,
$\rho,\rho'$,
\[
\left|\langle\delta_{\rho'}|R\delta_{\rho}\rangle\right|\leq C_{N}\left\langle \left\Vert \rho-\tilde{A}^{-1}\rho'\right\Vert \right\rangle ^{-N}\sigma^{-N}.
\]
\end{lem}

\end{cBoxB}

\begin{rem}
In particular, for $\sigma=\infty$, i.e. no cut-off, then $R=0$.
Another interesting particular case is $A=\mathrm{Id}$.
\end{rem}

\begin{proof}
We compute the Schwartz kernel.
\begin{align*}
\langle\delta_{\rho'}|\boldsymbol{B}A_{H}^{\circ}\boldsymbol{B}_{\chi}^{\Delta}\delta_{\rho}\rangle & =\langle\delta_{\rho'}|r_{0}BA_{H}^{\circ}B^{\dagger}\chi_{\sigma}\widetilde{\exp^{\circ}}\delta_{\rho}\rangle\\
 & =\int\langle\delta_{\rho}|r_{0}\delta_{\rho_{1},\rho_{1}'}\rangle\langle\delta_{\rho_{1},\rho_{1}'}|BA_{H}^{\circ}B^{\dagger}\delta_{\rho_{2},\rho_{2}'}\rangle\chi_{\sigma}\left(\rho_{2}'\right)\langle\delta_{\rho_{2},\rho_{2}'}|\widetilde{\exp^{\circ}}\delta_{\rho}\rangle\\
 & =\int\left(\delta_{\rho'-\rho_{1}}\delta_{\rho_{1}'}\right)\left(\delta_{\tilde{A}\rho_{2}-\rho_{1}}\langle\varphi_{\rho_{1}'}|A^{\circ}\varphi_{\rho_{2}'}\rangle\right)\chi_{\sigma}\left(\rho_{2}'\right)\left(\delta_{\rho-\left(\rho_{2}+\rho_{2}'\right)}e^{-i\xi_{2}x_{2}'}\right)\\
 & =\langle\varphi_{0}|A^{\circ}\varphi_{\rho-\tilde{A}^{-1}\rho'}\rangle e^{-i\left(\left(A^{*}\right)^{-1}\xi'\right)\left(x-Ax'\right)}\chi_{\sigma}\left(\rho-\tilde{A}^{-1}\rho'\right)
\end{align*}
because Dirac measures gave $\rho_{2}'=\rho-\rho_{2}=\rho-\tilde{A}^{-1}\rho'$,
$\xi_{2}=\left(A^{*}\right)^{-1}\xi'$ and $x_{2}'=x-Ax'$. Then
\begin{align*}
\langle\delta_{\rho'}|\boldsymbol{B}A_{H}^{\circ}\boldsymbol{B}_{\chi}^{\Delta}\delta_{\rho}\rangle & =\langle\varphi_{0}|A^{\circ}\varphi_{\rho-\tilde{A}^{-1}\rho'}\rangle e^{i\xi'\left(x'-A^{-1}x'\right)}\chi_{\sigma}\left(\rho-\tilde{A}^{-1}\rho'\right)\\
 & \eq{\ref{eq:phi_T}}\langle\varphi_{0}|A^{\circ}T_{x-Ax'}^{-\circ}\left(\mathcal{F}T_{\xi-A^{*-1}\xi'}^{-\circ}\mathcal{F}\right)\varphi_{0}\rangle e^{i\xi'\left(x'-A^{-1}x'\right)}\chi_{\sigma}\left(\rho-\tilde{A}^{-1}\rho'\right)\\
 & \eq{\ref{eq:weyl_heisenberg}}\langle\varphi_{0}|A^{\circ}\left(\mathcal{F}T_{-A^{*-1}\xi'}^{-\circ}\mathcal{F}\right)T_{-Ax'}^{-\circ}T_{x}^{-\circ}\left(\mathcal{F}T_{\xi}^{-\circ}\mathcal{F}\right)\varphi_{0}\rangle\\
 & \qquad\qquad e^{-i\left(-A^{*-1}\xi'\right)x}e^{i\xi'\left(-A^{-1}x'\right)}\chi_{\sigma}\left(\rho-\tilde{A}^{-1}\rho'\right)\\
 & =\eq{\ref{eq:weyl_heisenberg}}\langle\varphi_{0}|\left(\mathcal{F}T_{-\xi'}^{-\circ}\mathcal{F}\right)T_{-x'}^{-\circ}A^{\circ}T_{x}^{-\circ}\left(\mathcal{F}T_{\xi}^{-\circ}\mathcal{F}\right)\varphi_{0}\rangle\chi_{\sigma}\left(\rho-\tilde{A}^{-1}\rho'\right)\\
 & =\langle\varphi_{\rho'}|A^{\circ}\varphi_{\rho}\rangle\chi_{\sigma}\left(\rho-\tilde{A}^{-1}\rho'\right)=\langle\delta_{\rho'}|\mathcal{B}A^{\circ}\mathcal{B}^{\dagger}\delta_{\rho}\rangle\chi_{\sigma}\left(\rho-\tilde{A}^{-1}\rho'\right)
\end{align*}
If $\left\Vert \rho-\tilde{A}^{-1}\rho'\right\Vert \geq\sigma$ then
$\langle\delta_{\rho'}|R\delta_{\rho}\rangle=0$ otherwise
\[
\forall N>0,\exists C_{N}>0,\quad\left|\langle\varphi_{0}|A^{\circ}\varphi_{\rho-\tilde{A}^{-1}\rho'}\rangle\right|\leq C_{A}C_{N}\sigma^{-N}.
\]
\end{proof}

\subsubsection{Taylor projectors}

For $\rho\in F$, $\hat{T}_{\rho}$ has been defined in (\ref{eq:def_T_rho}).
Let us define
\[
T_{H}^{\left(k\right)}:=\hat{T}_{-\rho}T_{k}\hat{T}_{\rho}\quad:\mathcal{S}\left(H\right)\rightarrow\mathcal{S}\left(H\right)
\]
where $\rho\in F$ denotes the first variable and the operators acts
on the second variable $x'\in E$ only.

\begin{cBoxB}{}
\begin{lem}
\label{lem:Let-us-consider-1}Let us consider the difference operator
\[
R:=\boldsymbol{B}T_{H}^{\left(k\right)}\boldsymbol{B}_{\chi}^{\Delta}-\mathcal{B}T_{k}\mathcal{B}^{*}\quad:\mathcal{S}\left(F\right)\rightarrow\mathcal{S}\left(F\right).
\]
Then for $\left\Vert \rho-\rho'\right\Vert \leq\sigma$ we have $\langle\delta_{\rho'}|R\delta_{\rho}\rangle=0$.
\end{lem}

\end{cBoxB}

\begin{rem}
In particular for $\sigma=\infty$, i.e. no cut-off, then $R=0$.
\end{rem}

\begin{proof}
We repeat the lines of proof of Lemma \ref{lem:Let-us-consider}.
\begin{align*}
\langle\delta_{\rho'}|\boldsymbol{B}T_{H}^{\left(k\right)}\boldsymbol{B}_{\chi}^{\Delta}\delta_{\rho}\rangle & =\langle\delta_{\rho'}|r_{0}BT_{H}^{\left(k\right)}B^{*}\chi_{\sigma}\widetilde{\exp^{\circ}}\delta_{\rho}\rangle\\
 & =\langle\varphi_{0}|\hat{T}_{-\rho'}T_{k}\hat{T}_{\rho'}\varphi_{\rho-\rho'}\rangle e^{i\xi'\left(x'-x\right)}\chi_{\sigma}\left(\rho-\rho'\right)\\
 & \eq{\ref{eq:inverse_T}}\langle\varphi_{0}|\hat{T}_{\rho'}^{\dagger}e^{-i\xi'x'}T_{k}\hat{T}_{\rho'}\varphi_{\rho-\rho'}\rangle e^{i\xi'\left(x'-x\right)}\chi_{\sigma}\left(\rho-\rho'\right)\\
 & =\langle\varphi_{\rho'}|T_{k}\varphi_{\rho}\rangle\chi_{\sigma}\left(\rho-\rho'\right)=\langle\delta_{\rho'}|\mathcal{B}T_{k}\mathcal{B}^{\dagger}\delta_{\rho}\rangle\chi_{\sigma}\left(\rho-\rho'\right)
\end{align*}
\end{proof}

\section{\label{sec:Linear-expanding-maps}Linear contracting maps}

Let $\left(E,g\right)$ a finite dimensional vector space with Euclidean
metric $g$. In this section we consider a \textbf{linear }invertible
and\textbf{ contracting map} $\phi:\left(E,g\right)\rightarrow\left(E,g\right)$
i.e.
\begin{equation}
\lambda_{\pm}:=\lim_{t\rightarrow\pm\infty}\log\left\Vert \phi^{t}\right\Vert ^{1/t}\label{eq:def_gamma_-1-1}
\end{equation}
satisfy\footnote{Equivalently $\lambda_{-}\leq\lambda_{+}<0$ are the log of the minimal/maximal
modulus of eigenvalues of $\phi$ given by its Jordan decomposition.} $\lambda_{-}\leq\lambda_{+}<0$. Recall that $\mathrm{Op}\left(\Phi\right):\eq{\ref{eq:def_Op}}\left(\Upsilon\left(\Phi\right)\right)^{1/2}\mathcal{B}^{\dagger}\Phi^{-\circ}\mathcal{B}$
with $\Phi:=\phi^{-1}\oplus\phi^{*}:E\oplus E^{*}\rightarrow E\oplus E^{*}$
being the induced map on cotangent space and $\mathcal{B}:\mathcal{S}\left(E\right)\rightarrow\mathcal{S}\left(E\oplus E^{*}\right)$
defined in (\ref{eq:def_Bargman_B}). The pull back operator $\phi^{\circ}:\mathcal{S}\left(E\right)\rightarrow\mathcal{S}\left(E\right)$
has been defined in (\ref{eq:def_pullback}). The purpose of this
section is to study the spectrum of the following operator
\begin{equation}
\mathrm{Op}\left(\Phi\right)\eq{\ref{eq:Op(f)}}\left|\mathrm{det}\phi\right|^{1/2}\phi^{\circ}\qquad:\mathcal{S}\left(E\right)\rightarrow\mathcal{S}\left(E\right)\label{eq:Op_phi}
\end{equation}
on an adequate Hilbert space that contains $\mathcal{S}\left(E\right)$
(notice that $\mathrm{Op}\left(\Phi\right)$ is unitary in $L^{2}\left(E\right)$
and has essential spectrum on the unit circle).

For $k\in\mathbb{N}$, consider $T_{k}$ the finite rank projector
defined in (\ref{eq:def_T_k}). The vector space $\mathrm{Im}\left(T_{k}\right)\subset\mathcal{S}'\left(E\right)$
is finite dimensional. Since $\phi$ is a linear map, we have for
any $k\in\mathbb{N}$,
\begin{equation}
\left[\phi^{\circ},T_{k}\right]=0,\label{eq:phi,Tk}
\end{equation}
hence $\mathrm{Op}\left(\Phi\right):\mathrm{Im}\left(T_{k}\right)\rightarrow\mathrm{Im}\left(T_{k}\right)$
is invariant and has finite rank (we can compute explicitly its spectrum
from the spectrum of $\phi$). Let
\begin{equation}
\gamma_{k}^{\pm}:=\lim_{t\rightarrow\pm\infty}\log\left\Vert \mathrm{Op}\left(\Phi^{t}\right)_{/\mathrm{Im}\left(T_{k}\right)}\right\Vert ^{1/t}.\label{eq:def_gamma_-1}
\end{equation}
We can\footnote{In the simple $1-\mathrm{dim}$ case $\phi\left(x\right)=e^{\lambda}x$
on $\mathbb{R}$, with $\lambda_{\pm}=\lambda<0$, we get $\gamma_{k}^{+}=\gamma_{k}^{-}=\left(\frac{1}{2}+k\right)\lambda$.} compute $\gamma_{k}^{\pm}$ from the eigenvalues of $\phi$, see
Remark \cite[rem. 3.4.7]{faure-tsujii_prequantum_maps_12}. As in
(\ref{eq:estimates_lambda_k}), we have
\begin{equation}
\left(\frac{d}{2}+k\right)\lambda_{-}\leq\gamma_{k}^{-}\leq\gamma_{k}^{+}\leq\left(\frac{d}{2}+k\right)\lambda_{+},\label{eq:estimates_lambda_k-1}
\end{equation}
With $d=\mathrm{dim}E$. For every $k\in\mathbb{N}$, we have $\gamma_{k}^{-}\leq\gamma_{k}^{+}$,
$\gamma_{k+1}^{\pm}\leq\gamma_{k}^{\pm}$. The spectrum of $\mathrm{Op}\left(\Phi^{t}\right):\mathrm{Im}\left(T_{k}\right)\rightarrow\mathrm{Im}\left(T_{k}\right)$
is discrete and contained in the annulus $\left\{ z\in\mathbb{C},\,e^{t\gamma_{k}^{-}}\leq\left|z\right|\leq e^{t\gamma_{k}^{+}}\right\} $.
However we want to understand the action of $\mathrm{Op}\left(\Phi\right)$
on every function in $\mathcal{S}\left(E\right)$. For $K\in\mathbb{N}$,
let
\[
T_{\geq\left(K+1\right)}:=\mathrm{Id}_{\mathcal{S}\left(E\right)}-\left(\sum_{k=0}^{K}T_{k}\right).
\]
We will show below in Proposition \ref{prop:For-any-} that $\left\Vert \mathrm{Op}\left(\Phi^{t}\right)T_{\geq\left(K+1\right)}\right\Vert _{\mathcal{H}_{\mathcal{W}}\left(E\right)}\leq C_{\epsilon}e^{\left(\gamma_{K+1}^{+}+\epsilon\right)t}$
for some adequate norm $\left\Vert .\right\Vert _{\mathcal{H}_{\mathcal{W}}\left(E\right)}$
that we first define.

\subsection{Anisotropic Sobolev space $\mathcal{H}_{\mathcal{W}}\left(E\right)$}

Let $0<\gamma<1$, $R\in\mathbb{R}$. We define a weight function
$\mathcal{W}:E\oplus E^{*}\rightarrow\mathbb{R}^{+}$ similar to (\ref{eq:def_W-1})
as follows. For $\rho=\left(x,\xi\right)\in E\oplus E^{*}$, let
\[
h_{\gamma}\left(\rho\right)=\left\langle \left\Vert \rho\right\Vert _{g\oplus g^{-1}}\right\rangle ^{-\gamma}
\]
and
\begin{equation}
\mathcal{W}\left(\rho\right):=\frac{\left\langle h_{\gamma}\left(\rho\right)\left\Vert \xi\right\Vert _{g^{-1}}\right\rangle ^{R}}{\left\langle h_{\gamma}\left(\rho\right)\left\Vert x\right\Vert _{g}\right\rangle ^{R}}\label{eq:def_W-1-1}
\end{equation}
\begin{cBoxA}{}
\begin{defn}
For $u\in\mathcal{S}\left(E\right)$, we define the norm
\begin{equation}
\left\Vert u\right\Vert _{\mathcal{H}_{\mathcal{W}}\left(E\right)}:=\left\Vert \mathcal{W}\mathcal{B}u\right\Vert _{L^{2}\left(E\oplus E^{*}\right)},\label{eq:def_HWE-1}
\end{equation}
(where $\mathcal{W}$ denotes multiplication operator by $\mathcal{W}$)
and the \textbf{Sobolev space}
\begin{equation}
\mathcal{H}_{\mathcal{W}}\left(E\right):=\overline{\left\{ u\in\mathcal{S}\left(E\right)\right\} }\label{eq:def_HWE}
\end{equation}
where the completion is with the norm $\left\Vert u\right\Vert _{\mathcal{H}_{\mathcal{W}}\left(E\right)}$.
\end{defn}

\end{cBoxA}

\begin{rem}
~
\begin{itemize}
\item For $R=0$ we have $\mathcal{W}\equiv1$ hence $\mathcal{H}_{\mathcal{W}}\left(E\right)=L^{2}\left(E\right)$.
\item We have that $\mathcal{W}\left(0,\xi\right)\leq\left\langle \left|\xi\right|\right\rangle ^{r}$
and $\mathcal{W}\left(x,0\right)\leq\left\langle \left|x\right|\right\rangle ^{-r}$
with the order $r=R\left(1-\gamma\right)$ as in (\ref{eq:order_r})
(except for the factor $\frac{1}{2}$ that came from the metric).
\end{itemize}
For a \emph{non linear} map (Axiom A diffeomorphisms), it has been
explained in \cite[section 4]{faure_tsujii_Ruelle_resonances_density_2016}
that it is necessary to use a metric on phase space as (\ref{eq:metric_g_tilde_in_coordinates})
with $\delta^{\perp}\left(\eta\right)=\left\langle \left|\eta\right|\right\rangle ^{-\alpha^{\perp}}$
and exponent $\alpha^{\perp}\geq1/2$ to perform the micro-local analysis
and get Ruelle spectrum. In this section, we consider the special
case of a \emph{linear} map $\phi:E\rightarrow E$ is linear and for
that reason we can use the constant Euclidean metric $g\oplus g^{-1}$
on $T^{*}E=E\oplus E^{*}$ (i.e. exponent $\alpha^{\perp}=0$) to
get the Ruelle spectrum of the operator $\mathrm{Op}\left(\Phi\right)$
in (\ref{eq:Op_phi}).
\end{rem}

\subsection{Result}

The result that we will show is
\begin{cBoxB}{}
\begin{prop}
\label{prop:For-any-}For any $\epsilon>0,$ $K\in\mathbb{N}$, $R$
in (\ref{eq:def_W-1-1}) large enough so that 
\[
\lambda_{+}R\left(1-\gamma\right)<\gamma_{K+1}^{+},
\]
$\exists C_{\epsilon}>0,\forall t\geq0$,
\begin{equation}
\left\Vert \mathrm{Op}\left(\Phi^{t}\right)T_{\geq\left(K+1\right)}\right\Vert _{\mathcal{H}_{\mathcal{W}}\left(E\right)}\leq C_{\epsilon}e^{\left(\gamma_{K+1}^{+}+\epsilon\right)t},\label{eq:norm_Lw_out-1}
\end{equation}
with $\gamma_{K+1}^{+}$ defined in (\ref{eq:def_gamma_-1}).
\end{prop}

\end{cBoxB}

\subsubsection{Proof of Proposition \ref{prop:For-any-}}

We have $\tilde{\mathrm{Op}}\left(\Phi\right)\eq{\ref{eq:def_op_tilde}}\left(\Upsilon\left(\Phi\right)\right)^{1/2}\mathcal{P}\Phi^{-\circ}\mathcal{P}$
and let
\[
\tilde{\mathrm{Op}}_{\mathcal{W}}\left(\Phi\right):=\mathcal{W}\tilde{\mathrm{Op}}\left(\Phi\right)\mathcal{W}^{-1}.
\]
From (\ref{eq:op_tilde_op}) and (\ref{eq:def_HWE-1}) we have the
commutative diagram (with the notation of weighted norm $\left\Vert u\right\Vert _{L^{2}\left(E\oplus E^{*};\mathcal{W}^{2}\right)}:=\int\mathcal{W}^{2}\left(x,\xi\right)u\left(x,\xi\right)\frac{dxd\xi}{\left(2\pi\right)^{\mathrm{dim}E}}$)

\begin{equation}
\begin{array}{ccccc}
\mathcal{H}_{\mathcal{W}}\left(E\right) & \overset{\mathcal{B}}{\longrightarrow} & L^{2}\left(E\oplus E^{*};\mathcal{W}^{2}\right) & \overset{\mathcal{W}}{\longrightarrow} & L^{2}\left(E\oplus E^{*}\right)\\
\downarrow\mathrm{Op}\left(\Phi\right) &  & \downarrow\tilde{\mathrm{Op}}\left(\Phi\right) &  & \downarrow\tilde{\mathrm{Op}}_{\mathcal{W}}\left(\Phi\right)\\
\mathcal{H}_{\mathcal{W}}\left(E\right) & \overset{\mathcal{B}}{\longrightarrow} & L^{2}\left(E\oplus E^{*};\mathcal{W}^{2}\right) & \overset{\mathcal{W}}{\longrightarrow} & L^{2}\left(E\oplus E^{*}\right)
\end{array}\label{eq:comm_diagram}
\end{equation}
where horizontal arrows are isometries by definition. Hence, the study
of $\mathrm{Op}\left(\Phi^{t}\right):\mathcal{H}_{\mathcal{W}}\left(E\right)\rightarrow\mathcal{H}_{\mathcal{W}}\left(E\right)$
is equivalent to study $\tilde{\mathrm{Op}}_{\mathcal{W}}\left(\Phi^{t}\right):L^{2}\left(E\oplus E^{*}\right)\rightarrow L^{2}\left(E\oplus E^{*}\right)$.

\begin{cBoxB}{}
\begin{lem}
\label{lem:Decay-property-of}``\textbf{Decay property} of $\mathcal{W}$
with respect to $\Phi$''. There exists $C>0$, such that for any
$t\geq0$, there exists $C_{t}$ such that for any $\rho\in E\oplus E^{*}$
\begin{align}
\frac{\mathcal{W}\left(\Phi^{t}\left(\rho\right)\right)}{\mathcal{W}\left(\rho\right)} & \leq C\nonumber \\
 & \leq Ce^{-\Lambda t}\text{ if }\left\Vert \rho\right\Vert _{g}\geq C_{t}\label{eq:decay-1}
\end{align}
with 
\begin{equation}
\Lambda=-\lambda_{+}R\left(1-\gamma\right)>0.\label{eq:Lambda}
\end{equation}
\end{lem}

\end{cBoxB}

\begin{proof}
Similar to the proof in \cite[Thm 5.9]{faure_tsujii_Ruelle_resonances_density_2016}.
 Write $\rho=\left(x,\xi\right)\in T^{*}E=E\oplus E^{*}$. We consider
different zones of $T^{*}E$:
\begin{enumerate}
\item If $\left|x\right|,\left|\xi\right|\leq1$ then $W\left(\rho\right)\asymp1$.
\item If $\left|x\right|>1$ and $\left|\xi\right|\leq\left|x\right|^{\gamma}$
then $W\left(\rho\right)\asymp\left|x\right|^{-R\left(1-\gamma\right)}$.
\item If $\left|\xi\right|>1$ and $\left|x\right|\leq\left|\xi\right|^{\gamma}$
then $W\left(\rho\right)\asymp\left|\xi\right|^{R\left(1-\gamma\right)}$.
\item Otherwise $W\left(\rho\right)\asymp\left(\frac{\left|\xi\right|}{\left|x\right|}\right)^{R}$.
\end{enumerate}
We observe that for $t\geq0$, $\left|\phi^{-t}\left(x\right)\right|\geq e^{-\lambda_{+}t}\left|x\right|$
and $\left|\phi^{*t}\left(\xi\right)\right|\leq e^{\lambda_{+}t}\left|\xi\right|$
from which we deduce Lemma \ref{lem:Decay-property-of}.
\end{proof}

\paragraph{Truncation in phase space near the trapped set}

Let $\chi\in C_{c}^{\infty}\left(\mathbb{R}^{+};[0,1]\right)$ such
that
\begin{align*}
\chi\left(x\right) & =1\text{ if }x\leq1,\\
\chi\left(x\right) & =0\text{ if }x\geq2,
\end{align*}
Let $\sigma>0$. For $\rho\in E\oplus E^{*}$ let
\[
\chi_{\sigma}\left(\rho\right):=\chi\left(\frac{\left\Vert \rho\right\Vert _{g}}{\sigma}\right).
\]
\begin{cBoxB}{}
\begin{lem}
We have $\forall R>0,$$\forall\epsilon>0$,$\exists C>0,\forall t\geq0,$
$\exists\sigma_{t}>0$, $\forall\sigma>\sigma_{t}$, 
\begin{equation}
\left\Vert \tilde{\mathrm{Op}}_{\mathcal{W}}\left(\Phi^{t}\right)\left(1-\chi_{\sigma}\right)\right\Vert _{L^{2}}\leq Ce^{\left(-\Lambda+\epsilon\right)t}.\label{eq:norm_Lw_out-2}
\end{equation}
where $\Lambda$ is given in (\ref{eq:Lambda}). The operator $\tilde{\mathrm{Op}}_{\mathcal{W}}\left(\Phi^{t}\right)\chi_{\sigma}$
is Trace class in $L^{2}$. Consequently for any $t>0$, the \href{https://en.wikipedia.org/wiki/Essential_spectrum}{essential spectral radius}
is
\begin{equation}
r_{\mathrm{ess.}}\left(\tilde{\mathrm{Op}}_{\mathcal{W}}\left(\Phi^{t}\right)_{L^{2}\rightarrow L^{2}}\right)\leq e^{-\Lambda t}.\label{eq:ress_LWT}
\end{equation}
\end{lem}

\end{cBoxB}

\begin{proof}
For (\ref{eq:norm_Lw_out-2}), see \cite[thm 5.13]{faure_tsujii_Ruelle_resonances_density_2016}.
\end{proof}
\begin{cBoxB}{}
\begin{lem}
\label{lem:Let-.-If}Let $k\in\mathbb{N}$. If $k+1<r$ in (\ref{eq:order_r}),
then $T_{k}:\mathcal{H}_{\mathcal{W}}\left(E\right)\rightarrow\mathcal{H}_{\mathcal{W}}\left(E\right)$
is a bounded operator.
\end{lem}

\end{cBoxB}

\begin{proof}
We check that in (\ref{eq:def_T_k}), $\left\Vert x^{\alpha}\right\Vert _{\mathcal{H}_{\mathcal{W}}\left(E\right)}<\infty$
and $\left\Vert \delta_{0}^{\left(\alpha\right)}\right\Vert _{\mathcal{H}_{\mathcal{W}}\left(E\right)^{*}}<\infty$.
\end{proof}
\begin{cBoxB}{}
\begin{lem}
\label{lem:We-have}Let $K\in\mathbb{N}$ and $R>0$ large enough
such that $-\Lambda<\gamma_{K+1}^{+}$ in (\ref{eq:Lambda}). For
$t>0$, the \href{https://en.wikipedia.org/wiki/Spectral_radius}{spectral radius}
is
\[
r_{\mathrm{spec}}\left(\left(\mathrm{Op}\left(\Phi^{t}\right)T_{\geq\left(K+1\right)}\right)_{\mathcal{H}_{\mathcal{W}}\left(E\right)}\right)=e^{t\gamma_{K+1}^{+}}
\]
\end{lem}

\end{cBoxB}

\begin{rem}
Lemma \ref{lem:We-have} is similar to \cite[Claim (2) in Prop. 3.4.6]{faure-tsujii_prequantum_maps_12}.
\end{rem}

\begin{proof}
First, from (\ref{eq:ress_LWT}) and because $\mathrm{Op}\left(\Phi^{t}\right)T_{k}$
is finite rank, we have that 
\[
r_{\mathrm{ess}}\left(\left(\mathrm{Op}\left(\Phi^{t}\right)T_{\geq\left(K+1\right)}\right)_{\mathcal{H}_{\mathcal{W}}\left(E\right)}\right)\leq e^{-\Lambda t},
\]
i.e. $\mathrm{Op}\left(\Phi^{t}\right)T_{\geq\left(K+1\right)}:\mathcal{H}_{\mathcal{W}}\left(E\right)\rightarrow\mathcal{H}_{\mathcal{W}}\left(E\right)$
has discrete spectrum outside the disk of radius $e^{-\Lambda t}$.
From \href{https://en.wikipedia.org/wiki/Taylor\%27s_theorem}{Taylor-Lagrange remainder formula},
for any $u,v\in\mathcal{S}\left(E\right)$, we have
\begin{equation}
\left|\langle v|\mathrm{Op}\left(\Phi^{t}\right)T_{\geq\left(K+1\right)}u\rangle_{L^{2}}\right|\leq Ce^{t\gamma_{K+1}^{+}}\left\Vert u\right\Vert _{C^{K+1}}\left\Vert x^{K+1}v\right\Vert _{L^{1}}.\label{eq:decay_correl}
\end{equation}
i.e. correlation functions decay faster than $e^{t\gamma_{K+1}^{+}}$.
We deduce that $\mathrm{Op}\left(\Phi^{t}\right)T_{\geq\left(K+1\right)}$
has no spectrum on $\left|z\right|>e^{t\gamma_{K+1}^{+}}$.
\end{proof}
Finally Lemma \ref{lem:We-have} is equivalent to Proposition \ref{prop:For-any-}.

\subsection{\label{subsec:Discrete-Ruelle-Pollicott-spectr}Discrete Ruelle spectrum
in a simple toy model}

We give here a very simple example that illustrates Section \ref{sec:Linear-expanding-maps}
and some mechanisms that play a role in this paper. (See the lecture
notes \cite{faure_cours_cirm_2019} for further details). On $\mathbb{R}_{x}$
let us consider the vector field
\begin{equation}
X=-x\frac{\partial}{\partial x}\label{eq:X_xdx}
\end{equation}
that gives a flow $\phi^{t}\left(x\right)=e^{-t}x$ that is contracting
for $t>0$, hence we think $E_{s}=\mathbb{R}_{x}$ as a stable direction.
We observe that for any $k\in\mathbb{N}$, the monomial $x^{k}$ is
an eigenfunction of $X$ (hence of $e^{tX}$) with eigenvalue $\left(-k\right)$:
\[
Xx^{k}=\left(-k\right)x^{k}
\]
and the spectral projector is
\[
T_{k}=x^{k}\langle\frac{1}{k!}\delta^{\left(k\right)}|.\rangle,
\]
where $\delta^{\left(k\right)}$ is the $k$-th derivative of the
Dirac distribution (notice indeed from $\langle\frac{1}{k!}\delta^{\left(k\right)}|x^{k'}\rangle=\delta_{k=k'}$
that $\left(\delta^{\left(k\right)}\right)_{k}$ forms a dual basis
to $\left(x^{k'}\right)_{k'}$).

However $x^{k},\delta^{\left(k\right)}$ do not belong to $L^{2}\left(\mathbb{R}\right)$.
In the Hilbert space $L^{2}\left(\mathbb{R}\right)$ one has $X^{\dagger}=-X+1\Leftrightarrow\left(X-\frac{1}{2}\right)^{\dagger}=-\left(X-\frac{1}{2}\right)$
that implies that the spectrum of $X$ is on the vertical axis $\frac{1}{2}+i\mathbb{R}$
with some essential spectrum. Some better Hilbert space $\mathcal{H}_{\mathcal{W}}\left(\mathbb{R}\right)$
is constructed as in (\ref{eq:def_HWE}) in order to study the pull
back operator $\left(e^{tX}u\right)\left(x\right)=u\left(\phi^{t}\left(x\right)\right)=u\left(e^{-t}x\right)$
on functions $u\in\mathcal{S}\left(\mathbb{R}\right)$. We consider
the induced (pull-back) flow on the cotangent space $T^{*}E_{s}=T^{*}\mathbb{R}$
given by
\[
\tilde{\phi}^{t}\left(x,\xi\right)=\left(e^{t}x,e^{-t}\xi\right).
\]
For $0<h\ll1$, $R\geq0$ we define the ``escape function'' or ``Lyapounov
function'' on $T^{*}\mathbb{R}$ for $\tilde{\phi}^{t}$ that is
\[
\mathcal{W}\left(x,\xi\right)=\frac{\left\langle \sqrt{h}\xi\right\rangle ^{R}}{\left\langle \sqrt{h}x\right\rangle ^{R}}.
\]
Indeed, it satisfies $\frac{\mathcal{W}\circ\tilde{\phi}^{t}}{\mathcal{W}}\leq C$
everywhere and $\frac{\mathcal{W}\circ\tilde{\phi}^{t}}{\mathcal{W}}\leq e^{-Rt}$
far from the ``trapped set'' or ``non wandering set'' $\left(0,0\right)$.
In the Hilbert space
\[
\mathcal{H}_{\mathcal{W}}\left(\mathbb{R}\right):=\mathrm{Op}\left(\mathcal{W}^{-1}\right)L^{2}\left(\mathbb{R}\right),
\]
for $-R+\frac{1}{2}<-k$, we have that $x^{k}\in\mathcal{H}_{\mathcal{W}}\left(\mathbb{R}\right)$
and $\left\Vert T_{k}\right\Vert _{\mathcal{H}_{\mathcal{W}}}\leq C$
is a bounded operator. See figure \ref{fig:Spectrum-of-the}.For
large time $t\gg1$, the emerging behavior of $e^{tX}u$ is given
by
\begin{align*}
e^{tX}u & =\Pi_{0}u+O_{\mathcal{H}_{\mathcal{W}}}\left(e^{-t}\right)\\
 & =\boldsymbol{1}\,u\left(0\right)+O_{\mathcal{H}_{\mathcal{W}}}\left(e^{-t}\right)
\end{align*}
i.e. projection onto the constant function.
\begin{center}
\begin{figure}
\begin{centering}
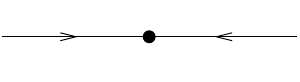$\qquad$$\qquad$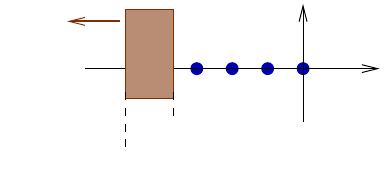
\par\end{centering}
\caption{\label{fig:Spectrum-of-the}In blue, discrete spectrum of the operator
(\ref{eq:X_xdx}), $X=-x\frac{\partial}{\partial x}$ in $\mathcal{H}_{\mathcal{W}}\left(\mathbb{R}\right)$.
In brown, the essential spectrum moves far away if $R\rightarrow+\infty$.}
\end{figure}
\par\end{center}

\bibliographystyle{plain}
\bibliography{/home/faure/articles/articles}

\end{document}